\documentclass{article}
\usepackage[english]{babel}
\usepackage[utf8]{inputenc}
\usepackage[T1]{fontenc}
\usepackage{lmodern}
\usepackage[a4paper]{geometry}
\usepackage{color}

\usepackage{hyphenat}
\allowhyphens

\usepackage{mathtools,amssymb,amsmath,amsthm}

\usepackage{dsfont}
\usepackage{enumitem}
\usepackage{array}
\usepackage{stmaryrd} 
\usepackage{subcaption}

\usepackage{varioref}
\usepackage[hidelinks,pdfusetitle]{hyperref}
\usepackage{cleveref}

\allowdisplaybreaks[1]

\providecommand{\keywords}[1]
{
  {\small	
  \textbf{\textit{Key words and phrases---}} #1}
}

\providecommand{\subjclass}[1]
{
  {\small	
  \textbf{\textit{2020 Mathematics Subject Classification---}} #1}
}

\usepackage{xspace}

\makeatletter
\newcommand\wrt{w.r.t.\@ifnextchar,{}{\ }}
\newcommand\aE{a.e.\@ifnextchar,{}{\ }}
\newcommand\as{a.s.\@ifnextchar,{}{\ }}
\newcommand\ps{p.s.\@ifnextchar,{}{\ }}
\newcommand\ie{i.e.\@ifnextchar,{}{\ }}
\newcommand\eg{e.g.\@ifnextchar,{}{\ }}
\newcommand\etal{et al.\@ifnextchar,{}{\ }}
\newcommand\etc{etc\@ifnextchar.{}{.\@}}
\makeatother

\theoremstyle{plain}
\newtheorem{Th}{Theorem}[section]
\newtheorem{theo}[Th]{Theorem}

\newtheorem{lemme}[Th]{Lemma}

\newtheorem{corol}[Th]{Corollary}

\newtheorem{prop}[Th]{Proposition}
\newtheorem{defin}[Th]{Definition}

\newtheorem{lemma}[Th]{Lemma}

\newtheorem{definition}[Th]{Definition}

\newtheorem{assumption}{Assumption}

\theoremstyle{definition}
\newtheorem{remark}[Th]{Remark}

\newtheorem*{remark*}{Remark}

\DeclareMathOperator{\argmax}{argmax}
\newcommand{\scalProd}[2]{\langle #1, #2 \rangle}

\newcommand{\inv}[1]{\frac{1}{#1}}

\DeclarePairedDelimiter{\ceil}{\lceil}{\rceil}
\DeclarePairedDelimiter{\floor}{\lfloor}{\rfloor}

\newcommand{\cF}{\mathcal{F}}

\newcommand{\cI}{\mathcal{I}}

\newcommand{\cL}{\mathcal{L}}

\newcommand{\cN}{\mathcal{N}}
\newcommand{\cO}{\mathcal{O}}

\newcommand{\cU}{\mathcal{U}}

\newcommand{\Prb}{\mathbb{P}}
\renewcommand{\P}{\mathbb{P}}
\newcommand{\Esp}{\mathbb{E}}

\newcommand{\ind}{\mathds{1}}

\newcommand{\Var}{\mathrm{Var}}
\newcommand{\Cov}{\mathrm{Cov}}

\newcommand{\N}{\mathbb{N}}
\newcommand{\Z}{\mathbb{Z}}

\newcommand{\R}{\mathbb{R}}

\newcommand{\drv}{\mathrm{d}}

\newcommand{\norm}[1]{\Vert#1\Vert}
\newcommand{\Bignorm}[1]{\left\Vert#1\right\Vert}
\newcommand{\BigNorm}[1]{\left\Vert#1\right\Vert}

\newcommand{\normTV}[1]{\Vert#1\Vert_{\text{TV}}}

\newcommand{\eps}{\varepsilon}

\newcommand{\leqlex}{\leq_{\mathrm{lex}}}

\usepackage{pict2e,picture}

\makeatletter
\DeclareRobustCommand{\bbDelta}{{\mathpalette\bb@Delta\relax}}
\newcommand{\bb@Delta}[2]{%
  \begingroup
  \sbox\z@{$\m@th#1\Delta$}%
  \dimendef\Dht=6 \dimendef\Dwd=8
  \setlength{\Dwd}{\wd\z@}%
  \setlength{\Dht}{\ht\z@}%
  \begin{picture}(\Dwd,\Dht)
  \put(0,0){$\m@th#1\Delta$}
  \put(.42\Dwd,.7\Dht){\line(10,-26){.25\Dht}}
  \end{picture}%
  \endgroup
}
\makeatother

\newcommand{\DeltaR}{{\bbDelta}}

\newcommand{\dgr}{d}

\newcommand{\T}{T}
\newcommand{\G}{G}

\newcommand{\rooot}{\partial}
\newcommand{\parent}{\mathrm{p}}

\newcommand{\SpaceX}{\mathcal{X}}
\newcommand{\SpaceY}{\mathcal{Y}}
\newcommand{\SpaceZ}{\mathcal{Z}}

\newcommand{\BorelX}{\mathcal{B}(\SpaceX)}
\newcommand{\BorelY}{\mathcal{B}(\SpaceY)}
\newcommand{\BorelZ}[1]{\mathcal{B}(#1)}

\newcommand{\Tpast}{T^{\infty}}

\newcommand{\Neighbor}{\mathcal{O}}

\newcommand{\Shape}{\mathcal{S}\mathit{h}}
\newcommand{\ShapeValue}{\mathcal{S}}
\newcommand{\ShapeSetValues}{{\mathcal{N}}}

\newcommand{\f}{\mathrm{f}}

\newcommand{\heightFunction}{h}
\newcommand{\height}[1]{\heightFunction(#1)}

\newcommand{\h}{\mathrm{h}}
\newcommand{\Hterm}{\mathrm{H}}

\newcommand{\ThetaNeighborhood}{\cO}

\newcommand{\dimTheta}{d}

\newcommand{\thetaTrue}{\theta^\star}
\newcommand{\Ptrue}{\mathbb{P}_{\thetaTrue}}
\newcommand{\Etrue}{\mathbb{E}_{\thetaTrue}}

\newcommand{\Pzeta}{\Prb_{\thetaTrue, \zeta}}
\newcommand{\Ezeta}{\Esp_{\thetaTrue, \zeta}}

\newcommand{\PtrueZeta}{\Prb_{\thetaTrue \bowtie \zeta}}

\title{Asymptotic properties of the maximum likelihood estimator for Hidden Markov Models indexed by binary trees}
\author{%
  Julien Weibel\footnote{Institut Denis Poisson,
Universit\'{e} d'Orl\'{e}ans,
Universit\'{e} de Tours,
CNRS,
France
and
CERMICS, \'{E}cole des Ponts, France.
    \textrm{\textbf{julien.weibel@normalesup.org}}}
    }
\date{September 8, 2024}

\begin{document}

\maketitle

\begin{abstract}
We consider hidden Markov models indexed by a binary tree
where the hidden state space is a general metric space.
We study the maximum likelihood estimator (MLE) 
	of the model parameters
	based only on the observed variables.
In both stationary and non-stationary regimes,
we prove strong consistency and 
asymptotic normality of the MLE
under standard assumptions.
Those standard assumptions imply
uniform exponential memorylessness properties of the initial distribution 
conditional on the observations.
The proofs rely on 
ergodic theorems for Markov chain indexed by trees 
with neighborhood-dependent functions.
\end{abstract}



\section{Introduction}

In this article, we consider a generalization of the hidden Markov chain/model (HMM)
	where the process is indexed by a binary tree,
	which we call hidden Markov tree (HMT).
The HMT is composed of a hidden process and an observed process.
The hidden process is a branching Markov process, that is, 
a random process $X = (X_u, u\in\T)$ with values in a metric space $\SpaceX$
indexed by a rooted tree $\T$ with the Markov property:
sibling nodes take independent and identically distributed values 
that depend only on the value of their parent node.
Note that the hidden process is sometimes called latent process in the literature.
Conditionally on the hidden process $X$, the observed process $Y = (Y_u, u\in\T)$,
with values in another metric space $\SpaceY$,
is composed of independent random variables $Y_u$ which only depends on $X_u$
	for all $u\in\T$.
See Definitions~\ref{def:Markov_proc} and \ref{def:HMT} below for a complete formal definitions.
In this article, we consider the case where the tree $\T$ is the (deterministic) complete infinite rooted binary tree,
that is, each vertex has exactly two children.
See Figure~\ref{fig_dependance_HMT} for a graphical representation of the dependance between
the variables composing the HMT process $(X,Y)$ indexed by $\T$.

\begin{figure}[h]
\centering
\includegraphics[height=5cm]{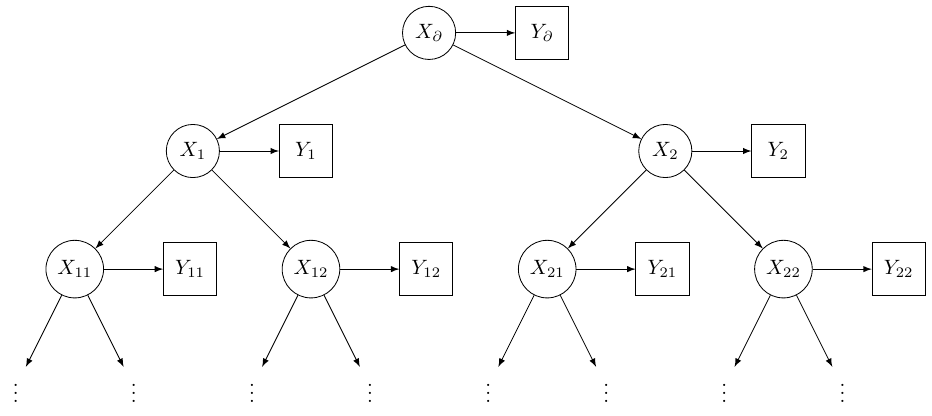}
\caption{Graph of dependance for variables of a HMT process indexed by the complete infinite rooted binary tree $\T$.
	The observed variables are represented inside square,
	while the hidden variables are represented inside circles.
	}
\label{fig_dependance_HMT}
\end{figure}

\subsection{Literature review}
	\label{section_Literature_review}

HMMs were first introduced by Baum and Petrie in \cite{baumStatisticalInferenceProbabilistic1966}
and were popularized by Rabiner's tutorial \cite{rabinerTutorialHiddenMarkov1989}.
Since then, HMMs have been used in a wide variety of applications such as
speech recognition \cite{yuAutomaticSpeechRecognition2015},
bioinformatics \cite{koskiHiddenMarkovModels2001},
finance \cite{mamonHiddenMarkovModels2014},
and time-series analysis
\cite{zucchiniHiddenMarkovModels2009};
see also \cite{bouguilaHiddenMarkovModels2022}
for a more global reference on HMMs applications.

HMTs were first introduced in \cite{CrouseHMT} to account for the multi-scale dependency of wavelet coefficients
	in statistical signal processing with applications in wavelet-based image processing 
\cite{rombergMultiscaleClassificationUsing2000,choiMultiscaleImageSegmentation2001,duarteWaveletdomainCompressiveSignal2008,
	shahdoostiImageDenoisingDual2017}.
After that, HMTs have been used in several application contexts such as
natural language processing \cite{graveHiddenMarkovTree2013,kondoHiddenMarkovTree2013},
flood mapping \cite{xieGeographicalHiddenMarkov2018},
medical imaging \cite{makhijaniAccelerated3DMERGE2012,huMicrocalcificationDiagnosisDigital2017,hanzouli-bensalahFrameworkBasedHidden2017},
plant growth modeling \cite{durandAnalysisPlantArchitecture2005},
and bioinformatics 
\cite{olariuModifiedVariationalBayes2009,biesingerDiscoveringMappingChromatin2013,nakashimaLineageEMAlgorithm2020}.

In practice, maximum likelihood estimation for HMMs
often relies on iterative numerical methods to approximate 
the maximum likelihood estimator (MLE).
Those methods are often based on
the expectation-maximization algorithm
which is an algorithm for models with missing data
and was popularized by Dempster \etal
\cite{dempsterMaximumLikelihoodIncomplete1977a} in a celebrated article.
For HMMs with finite hidden state space, 
the first presentation of a complete expectation-maximization strategy
is due to Baum \etal \cite{baumMaximizationTechniqueOccurring1970a},
and is the well-known ‘‘forward-backward’’ or Baum-Welch algorithm.
For more details on the expectation-maximization 
and ‘‘forward-backward’’ algorithms and their stochastic approximations, 
see \cite[Chapters~10 and~11]{CappeHMM}.
In the HMT case, the ‘‘forward-backward’’ algorithm
must be replaced by the ‘‘upward-downward’’ algorithm
developed in \cite{CrouseHMT}.
See also \cite{durandComputationalMethodsHidden2004} 
for alternative ‘‘upward-downward’’ recursive formulae
that can handle underflow issues implicitly.

\medskip

The statistical properties of the MLE for the HMM were first studied in \cite{baumStatisticalInferenceProbabilistic1966}
which proved consistency and asymptotic normality in the case where both the hidden and the observed processes
can only take finitely many values.
Those results were then successively extended in a series of articles
\cite{lerouxMaximumlikelihoodEstimationHidden1992,bickelAsymptoticNormalityMaximumlikelihood1998,jensenAsymptoticNormalityMaximum1999,leglandExponentialForgettingGeometric2000,doucAsymptoticsMaximumLikelihood2001}. 
An extension of all those results for HMMs  with autoregression
	(that is, when conditionally on the hidden Markov chain,
	the observed process is an inhomogeneous $s$-order Markov chain for some $s\in\N$)
was later developed in \cite{doucAsymptoticPropertiesMaximum2004},
which proved, using weaker assumptions, strong consistency and asymptotic normality of the MLE
for auto-regressive HMMs with compact hidden state space
and with possibly non-stationary regime.
The methods used in \cite{doucAsymptoticPropertiesMaximum2004}
relies on expressing the log-likelihood as an additive function of an extended Markov chain with infinite past
thanks to stationarity and using geometric ergodicity of this extended chain (extension to non-stationary regime is then made separately).
The method of \cite{doucAsymptoticPropertiesMaximum2004} was  adapted in 
\cite{kasaharaAsymptoticPropertiesMaximum2019}
under similar assumptions
to allow the transition densities of the hidden process to be zero valued.
Since the article \cite{doucAsymptoticPropertiesMaximum2004},
the strong consistency of the MLE was proved 
under weaker assumptions
in \cite{genon-catalotLerouxMethodGeneral2006,doucConsistencyMaximumLikelihood2011,doucMaximizingSetAsymptotic2016},
but no generalization has been made for the asymptotic normality of the MLE.

In this article, we will adapt the proof method of \cite{doucAsymptoticPropertiesMaximum2004} to the HMT case.
We shall also refer to the monograph \cite{CappeHMM} which exposes in details the theory of HMMs,
and in particular to its Chapter 12 which covers the strong consistency and asymptotic normality of the MLE,
under the same assumptions used in \cite{doucAsymptoticPropertiesMaximum2004},
for HMMs where the hidden state space is a general metric space.

\medskip

To adapt the proof method of \cite{doucAsymptoticPropertiesMaximum2004}
to the HMT case,
we will need almost sure (a.s.) and $L^2$ ergodic convergence results 
for branching Markov chains
under geometric ergodicity of the transition kernel
as in \cite{GuyonLimitTheorem,weibelErgodicTheoremBranching2024}.
Indeed, we will need variants of those results for 
neighborhood-dependent functions
(that is, the function associated to each vertex $u$
depends on variables $X_v$ for vertices $v$ in the neighborhood of $u$)
which we develop in \Cref{subsection_ergodic_theorem} and \Cref{appendix_ergodic_theorems}.

\subsection{New contribution}
\label{section_intro_new_contribution}

In this article, we consider the case where the distribution of the HMT
is parametrized by some vector $\theta$, that is,
the transition kernel $Q_\theta$ between the hidden variables
and the transition kernel $G_\theta$ from hidden variables to observed variables
both depend on $\theta$. 
As an example, if the hidden state space $\SpaceX$ is finite 
and $Y_u$ conditioned on $X_u$ is a Gaussian random variable for each $u\in\T$,
then $\theta$ could parametrized the transition matrix of the hidden process
	and the mean and variances of the Gaussian distribution associated to each hidden state values.
Our goal is to estimate the true parameter $\thetaTrue$ of the HMT process
among a compact set of possible parameters $\Theta\subset \R^{\dimTheta}$, for some integer $\dimTheta$,
using only the knowledge of the observed process $Y$ over $n$ generations of the tree.
Note that as our assumptions will imply uniform exponential memorylessness properties for the initial distribution,
we cannot try estimate the initial distribution. Denote $\rooot$ the root of the tree $\T$.
Thus, we assume that the distribution of the hidden root variable $X_\rooot$ 
is some unknown measure $\zeta$ which does not depend on $\theta$.
Denote by $\Pzeta$ the probability distribution of the HMT under the true parameter $\thetaTrue$
	when the initial unknown distribution of $X_\rooot$ is $\zeta$.
When $\zeta$ is the unique invariant measure of $Q_\theta$ (\ie in the stationary case),
we write $\Ptrue$ instead of $\Pzeta$.

To estimate the true parameter $\thetaTrue$ of the HMT, we will use the maximum likelihood estimator (MLE).
We will work with the likelihood conditioned on the hidden state of the root vertex $X_{\rooot}$.
The reason to do this is that the computation of the stationary distribution of the joint process $(X,Y)$,
and thus also the true likelihood, is intractable   in typical applications.
Note that the idea of conditioning on the initial hidden state was already used in \cite{doucAsymptoticPropertiesMaximum2004}
for HMMs with the same motivation,
and conditioning on initial observations in time series goes back at least to \cite{mannStatisticalTreatmentLinear1943}.
Remind that $\T$ denote the (deterministic) complete infinite rooted binary tree.
Denote $\T_n$ the tree $\T$ up to and including the $n$-th generation.
Hence, for any value $x\in\SpaceX$, we denote by $\ell_{n,x}(\theta)$ the log-likelihood under the parameter $\theta$
of the observed process $(Y_u, u\in \T_n)$ until the $n$-th generation of the tree $\T$
conditionally on $X_{\rooot} = x$
(see \eqref{eq_def_l_nx_theta_2} on page~\pageref{eq_def_l_nx_theta_2} for exact definition).
Then,  for any value $x\in\SpaceX$, 
we define the MLE $\hat \theta_{n,x}$ as the maximizer of $\ell_{n,x}$ over $\Theta$
(see \eqref{eq_def_MLE_hat_theta_nx}  on page~\pageref{eq_def_MLE_hat_theta_nx} for exact definition).

\medskip

Our goal is to study the asymptotic properties of the MLE.
We prove the strong consistency and the asymptotic normality of the MLE
in the stationary case in Sections~\ref{section_strong_consistency} and~\ref{section_asymptotic_normality}, respectively.
We then extend those results to the non-stationary case in  \Cref{section_non_stationary}.
In our results, the hidden state space $\SpaceX$ and the observed state space $\SpaceY$
are both general metric spaces.
We prove our results under the same assumptions used
in \cite{doucAsymptoticPropertiesMaximum2004} and in \cite[Chapter~12]{CappeHMM} for HMMs
with $L^1$ and $L^2$ integrability assumptions
replaced by $L^2$ and $L^4$ integrability assumptions, respectively,
to accommodate
the stronger assumptions needed in ergodic theorems for branching Markov chains.
See \Cref{rem_intro_diff_HMM_HMT} below for a discussion on 
the main differences between the HMM case as in \cite{doucAsymptoticPropertiesMaximum2004,CappeHMM}
and the HMT case we develop in this article.

We first state that strong consistency of the MLE holds under standard assumptions for HMMs.
Following \cite{doucAsymptoticPropertiesMaximum2004}, we assume a fully dominated model, that is,
the transition kernels $Q_\theta$ and $G_\theta$ admits densities $q_\theta$ and $g_\theta$ \wrt 
to common measures $\lambda$ and $\mu$, respectively (see \Cref{assump_HMM_1}).
We also assume (see \Cref{assump_HMM_2}) : 
\begin{equation}\label{eq_intro_bounds_q_theta}
0 < \sigma^- \leq \inf_{x,x'\in\SpaceX} q_\theta(x,x') \leq \sup_{x,x'\in\SpaceX} q_\theta(x,x') \leq \sigma^+ < \infty .
\end{equation}
This assumption is rather strong as it imposes a full connection for the hidden space,
see \cite{kasaharaAsymptoticPropertiesMaximum2019} for an extension of the method in \cite{doucAsymptoticPropertiesMaximum2004}
for HMMs where $q_\theta$ is allowed to be zero valued.
Nevertheless, this assumption implies 
the uniform exponential memorylessness properties with mixing rate $\rho:=1-\sigma^-/\sigma^+$ 
of the initial distribution conditional on the observations $(Y_u, u\in \T_n)$.
The other assumptions are more standard regularity assumptions for the densities $q_\theta$ and $g_\theta$
(see Assumptions~\ref{assump_HMM_1}-\ref{assump_HMM_4}),
and identifiability of the model.
We can now state the strong consistency of the MLE under those assumptions,
see Theorems~\ref{thm_Strong_consistency_MLE} and~\ref{thm_strong_constistency_non_stationary}
for the precise statements in the stationary and non-stationary case, respectively.

\begin{theo}[Strong consistency of the MLE]
	\label{intro:thm_Strong_consistency_MLE}
Under those assumptions of fully dominated model
with density satisfying \eqref{eq_intro_bounds_q_theta}
and other more standard regularity assumptions,
and under the assumption that the model is identifiable,
for any $x\in\SpaceX$,
the MLE $\hat \theta_{n,x}$ is strongly consistent, that is, 
	the sequence $(\hat \theta_{n,x})_{n\in\N}$ converges $\Pzeta$-almost surely  to the true parameter $\thetaTrue\in\Theta$.
\end{theo}

To prove asymptotic normality of the MLE, in addition to the assumptions used in \Cref{intro:thm_Strong_consistency_MLE},
we need existence and regularity assumptions for the gradient and the Hessian of the transition densities $q_\theta$ and $g_\theta$
(see Assumptions~\ref{assump_HMM_grad_1}-\ref{assump_HMM_grad_3}).
Denote by $\cI(\thetaTrue)$ the limiting Fisher information matrix of the model
(see \eqref{eq_def_Fisher_information} on page \pageref{eq_def_Fisher_information} for precise definition).
The proof of asymptotic normality in the non-stationary case is an extension of the stationary case.
The proof of asymptotic normality in the stationary case 
follows from a standard argument for asymptotic normality of the MLE
that relies on Theorem~\ref{intro:thm_Strong_consistency_MLE} 
and Theorems~\ref{intro:thm_convergence_score} and~\ref{intro:thm_LLN_observed_information} below.

The following theorem, which we only prove in the stationary case, states that
the normalized score $\vert \T_n \vert^{-1/2} \, \nabla_\theta \ell_{n,x}(\thetaTrue)$
has asymptotic normal fluctuations with covariance matrix $\cI(\thetaTrue)$,
see \Cref{thm_convergence_score} for the precise statement.
Note that the extra assumption in \Cref{intro:thm_convergence_score} (not present in the case of HMMs)
that $\rho < 1/\sqrt{2}$ for the mixing rate $\rho$ of the HMT process
comes from the approximation bounds used in the proof of this theorem.
See \Cref{rem_rho_smaller_than_2} below
for a discussion on this condition on $\rho$.

\begin{theo}[Asymptotic normality of the normalized score]
	\label{intro:thm_convergence_score}
Under the assumptions from \Cref{intro:thm_Strong_consistency_MLE} 
and existence and regularity assumptions for the gradient and the Hessian of the transition densities
(see Assumptions~\ref{assump_HMM_grad_1}-\ref{assump_HMM_grad_3}),
and under the assumption that $\rho < 1/\sqrt{2}$ for the mixing rate $\rho$ of the HMT process,
in the stationary case we have:
\begin{equation*}
\vert \T_n \vert^{-1/2} \, \nabla_\theta \ell_{n,x}(\thetaTrue) 
	\underset{n\to\infty}{\overset{(d)}{\longrightarrow}}
\cN(0, \cI(\thetaTrue))
\quad \text{under $\Ptrue$.}
\end{equation*}
\end{theo}

The following theorem states the locally uniform convergence $\Pzeta$-\as
of the normalized observed information $-\vert\T_n\vert^{-1} \nabla_\theta^2 \ell_{n,x}(\theta)$
towards the Fisher information matrix $\cI(\thetaTrue)$,
see Theorems~\ref{thm_LLN_observed_information} and~\ref{thm_LLN_observed_information_non_stationary}
for the precise statements in the stationary and non-stationary case, respectively.
Note that in this theorem we need the stronger assumption $\rho < 1/2$ for the mixing rate $\rho$ of the HMT process
as we use more restrictive approximation bounds in the proof of this theorem
than the ones used in the proof of \Cref{intro:thm_convergence_score}.

\begin{theo}[Convergence of the normalized observed information]
	\label{intro:thm_LLN_observed_information}
Under the assumptions from \Cref{intro:thm_convergence_score} on the HMT model,
and under the assumption that $\rho < 1/2$ for the mixing rate $\rho$ of the HMT process,
for all $x\in\SpaceX$, we have:
\begin{equation*}
\lim_{\delta\to 0} \lim_{n\to\infty} \sup_{\theta\in\Theta \,:\, \Vert \theta - \thetaTrue \Vert \leq \delta}
	\ \Bigl\Vert{-\vert\T_n\vert^{-1} \nabla_\theta^2 \ell_{n,x}(\theta) - \cI(\thetaTrue)}\Bigr\Vert = 0
\quad \text{$\Pzeta$-\as}
\end{equation*}
\end{theo}

In particular, combining Theorems~\ref{intro:thm_Strong_consistency_MLE} 
	and~\ref{intro:thm_LLN_observed_information},
we get that the normalized observed information
$-\vert\T_n\vert^{-1} \nabla_\theta^2 \ell_{n,x}(\hat\theta_{n,x})$
at the MLE $\hat\theta_{n,x}$ is a strongly consistent estimator
of the Fisher information matrix $\cI(\thetaTrue)$.

As announced above, 
following a standard argument for asymptotic normality of the MLE,
Theorems~\ref{intro:thm_Strong_consistency_MLE},
\ref{intro:thm_convergence_score} and~\ref{intro:thm_LLN_observed_information}
imply
the following theorem which states that the MLE 
has asymptotic normal fluctuations with covariance matrix $\cI(\thetaTrue)^{-1}$.
See Theorems~\ref{thm_asymptotic_normality} and~\ref{thm_asymptotic_normality_non_stationary}
for the precise statements in the stationary and non-stationary case, respectively.

\begin{theo}[Asymptotic normality of the MLE]
	\label{intro:thm_asymptotic_normality}
Under the assumptions from \Cref{intro:thm_convergence_score} on the HMT model,
that $\thetaTrue$ is an interior point of $\Theta$,
and the Fisher information matrix $\cI(\thetaTrue)$ is non-singular,
and under the assumption that $\rho < 1/2$ for the mixing rate $\rho$ of the HMT process,
we have the following convergence in distribution: 
\begin{align*}
\vert \T_n \vert^{1/2} \bigl( \hat{\theta}_{n} - \thetaTrue \bigr)
	\underset{n\to\infty}{\overset{(d)}{\longrightarrow}}
\cN(0, \cI(\thetaTrue)^{-1})
\quad \text{under $\Pzeta$,}
\end{align*}
where $\cN(0,M)$ denotes the centered Gaussian distribution with covariance matrix $M$.
\end{theo}

Note that the standard argument used in the proof of \Cref{intro:thm_asymptotic_normality}
implies that we have the following joint convergence in distribution:
\begin{equation*}
\left( \vert \T_n \vert^{1/2} \bigl( \hat{\theta}_{n} - \thetaTrue \bigr) , \,
\vert \T_n \vert^{-1/2} \, \nabla_\theta \ell_{n,x}(\thetaTrue)  \right)
	\underset{n\to\infty}{\overset{(d)}{\longrightarrow}}
(\cI(\thetaTrue)^{-1/2}\, G, \, \cI(\thetaTrue)^{1/2}\, G)
\quad \text{under $\Ptrue$,}
\end{equation*}
where $G$ is Gaussian random variable distributed as $\cN(0,I_{\dimTheta})$ with
$I_{\dimTheta}$ the identity matrix of dimension ${\dimTheta}\times {\dimTheta}$,
and $\cI(\thetaTrue)^{1/2}$ is a root matrix of $\cI(\thetaTrue)$.

\medskip

The following remark is a discussion on the condition
on the mixing rate $\rho$ of the HMT process $(X,Y)$
that appear in Theorems~\ref{intro:thm_asymptotic_normality},
\ref{intro:thm_convergence_score} and~\ref{intro:thm_LLN_observed_information}.

\begin{remark}[On the condition on the mixing rate $\rho$]
	\label{rem_rho_smaller_than_2}
Note that in central limit theorems for branching Markov chains,
three regimes with different asymptotic behaviors (and different normalization terms) for
$\rho < 1/\sqrt{2}$, $\rho = 1/\sqrt{2}$ and $\rho >1/\sqrt{2}$
were observed in \cite{PendaCLT}, 
corresponding to a competition between the ergodic mixing rate $\rho$ and the branching rate $2$ in $\T$,
see also \cite{athreyaLimitTheoremsMultitype1969,bitsekipendaDeviationInequalitiesModerate2014a,PendaCLTestim}.
However, the condition on $\rho$ disappears when we consider martingale increments 
in the central limit theorem for branching Markov chains,
see \cite{GuyonLimitTheorem,bercuAsymptoticAnalysisBifurcating2009,delmasDetectionCellularAging2010}.

In our case, the condition $\rho < 1/\sqrt{2}$ on the mixing rate $\rho$
that appears in \Cref{intro:thm_convergence_score}
is due to the coupling bounds and the grouping of terms used in the proof of \Cref{lemma_score_incr_L2_bound}
(the upper bounds at the end of the proof only add a constant multiplicative factor).
It is an open question whether or not some convergence would still hold in \Cref{intro:thm_convergence_score}
with $\rho \geq 1/\sqrt{2}$ even with a possibly stronger normalization term
and a possibly non Gaussian limit.
Nevertheless, note that the proof of \Cref{intro:thm_convergence_score} relies on decomposing
the score $\nabla_\theta \ell_{n,x}(\theta)$ as a sum of martingale increments,
which could indicate that convergence is possible for $\rho \geq 1/\sqrt{2}$.

Moreover,
the stronger condition $\rho < 1/2$ on the mixing rate $\rho$
that appears in \Cref{intro:thm_LLN_observed_information}, and thus in \Cref{intro:thm_asymptotic_normality},
is due to the coupling bounds from \Cref{lemma_bound_covar_terms_Gamma}
and the grouping of terms used in the proof of \Cref{lemma_Gamma_incr_L2_bound}
(the upper bounds in the rest of the proof only add a constant multiplicative factor).
It is an open question whether or not convergence would still hold in \Cref{intro:thm_LLN_observed_information}
and in \Cref{intro:thm_asymptotic_normality} with $\rho \geq 1/2$ 
even with a possibly stronger normalization term 
and a possibly non Gaussian limit in \Cref{intro:thm_asymptotic_normality}.
Also note that the condition $\rho < 1/2$ is used 
when proving that \Cref{intro:thm_asymptotic_normality} extends to the non-stationary case
to construct a coupling between a stationary HMT process and a non-stationary HMT process,
see \Cref{lemma_coupling_time}.
\end{remark}

In the following remark,
we discuss the main differences between the HMM case as in \cite{doucAsymptoticPropertiesMaximum2004,CappeHMM}
and the HMT case we develop in this article.

\begin{remark}[On main differences with the HMM case]
	\label{rem_intro_diff_HMM_HMT}

In both HMM and HMT cases, 
the study of the log-likelihood is based on decomposing it as a sum of increments,
and then extending the ‘‘past’’ seen by each variable.
However, while the extended ‘‘past’’ only spreads backwards in the HMM case,
the extended ‘‘past’’ in the HMT case is a subtree that also spreads laterally
due to the different topologies between the line $\Z$ and the binary tree,
see Figure~\ref{fig_Delta_u_k} on page~\pageref{fig_Delta_u_k} for an illustration.
See also Sections~\ref{subsection_ergodic_theorem} and~\ref{subsection_decomp_likelihood_ordering}
for the definition of those ‘‘past’’ and extended ‘‘past’’.
Moreover, 
due to the enumeration of vertices in the tree in a breadth-first-search manner,
those extended ‘‘past’’ do not have the same ‘‘shapes’’
for all vertices, see \Cref{subsection_ergodic_theorem}.
Also note that the infinite expanded ‘‘past’’ of a vertex
relies on a random infinite “backward spine” of left / right ancestors
(see Figure~\ref{fig_illustration_Tpast} on page~\pageref{fig_illustration_Tpast}),
which adds extra randomness to the ‘‘shape’’ of the ‘‘past’’.

Furthermore, contrary to the HMM case,
the lateral spreading of each vertex's ‘‘past’’ in the HMT case
implies that log-likelihood increments with infinite extended ‘‘pasts’’ 
do not form a branching Markov process.
For this reason, we need to work with log-likelihood increments
whose ‘‘past’’ is trimmed to a fixed common subtree height,
and only expand to infinite ‘‘past’’ in the limit.
To prove convergence for sums of log-likelihood increments
with trimmed ‘‘pasts’’ which have different shapes,
we need to develop new ergodic theorems
for branching Markov chains and neighborhood-dependent functions,
see \Cref{subsection_ergodic_theorem} and \Cref{appendix_ergodic_theorems}.

In the proof of asymptotic normality of the normalized score,
the score is decomposed as a sum of martingale increments
which is no longer stationary in the HMT case
due to the ‘‘pasts’’ of vertices having different shapes.
Thus, to apply the central limit theorem for martingales,
we first need to verify convergence for the quadratic variations
of the martingale increment sequences
and Lindeberg's condition.
Moreover, the computation of the approximation bounds
for the increments used to decompose the score and the observed information
are more involved and impose conditions on the value of the mixing rate $\rho$,
as already discussed in \Cref{rem_rho_smaller_than_2}.
This also implies that the proof scheme 
for convergence of the observed information matrix
needs to be modified
as we cannot have almost sure convergence for all the increments simultaneously,
and we must rely on $L^2$ convergence instead.

Lastly, as discussed in \Cref{section_Literature_review},
the results for HMMs in \cite{doucAsymptoticPropertiesMaximum2004}
allowed for autoregression
	(remind, that is, when conditionally on the hidden Markov chain,
	the observed process is an inhomogeneous $s$-order Markov chain for some $s\in\N$).
Our results for HMTs are stated for processes without autogression.
However, as our approach adapts the proof scheme of
\cite{doucAsymptoticPropertiesMaximum2004},
note that with straightforward modifications of our proofs, we could allow
for autoregression in HMT processes.
\end{remark}

\subsection{Organization of the paper}

The rest of the paper is organized as follows.
In Section~\ref{section_notations}, we define the notations used in this article, HMT processes
	and the log-likelihood for the HMT.
For the stationary case, we prove the strong consistency of the MLE in \Cref{section_strong_consistency},
	and its asymptotic normality in \Cref{section_asymptotic_normality}.
In \Cref{section_non_stationary}, we extend those results to the non-stationary case.
In \Cref{appendix_ergodic_theorems}, we develop
the ergodic theorems for branching Markov chains with neighborhood-dependent functions
needed in the proofs of the asymptotic properties of the MLE.

\section{Definition of HMT and notations}
\label{section_notations}

In this section, we first define the notations we use for the infinite complete binary tree $\T$.
We then define branching Markov chains and hidden Markov models (HMMs) indexed by the binary tree $\T$,
which we will simply call Hidden Markov Tree (HMT) models.
We continue with the basic assumptions we need to define the log-likelihood for the HMT.
Lastly, we present the ergodic theorems for branching Markov chains 
and neighborhood-dependent functions needed in this article,
whose proofs can be found in \Cref{appendix_ergodic_theorems}.

\subsection{Notations for trees}

Let $\T = \cup_{n\in\N} \{ 0,1\}^{n}$ denote the infinite complete plane rooted binary tree, 
that is the plane rooted tree where each vertex $u$ has exactly two children $u0$ and $u1$.
Denote by $\rooot$ the root vertex of $\T$ (which is the unique point in $\{0,1\}^{0}$).
If $u$ is distinct from the root, we denote by $\parent(u)$ its parent vertex.
We denote by $\height{u}$ its height, \ie the number of edges separating $u$ from the root $\rooot$.
(The height of the root $\rooot$ is zero.)
In particular, for $k\leq \height{u}$, note that $\parent^k(u)$ denotes the $k$-th ancestor of $u$.
For two vertices $u,v\in\T$, we denote by $u\land v$ the most recent common ancestor
	of $u$ and $v$,
	and by $\dgr(u,v)$ the graph-distance between $u$ and $v$ in $\T$,
	that is $\dgr(u,v) = \height{u} + \height{v} - 2 \height{u\land v}$.
For all $n\in\N$, denote by $\G_n$ the $n$-th generation of the tree,
that is vertices that are at distance $n$ from the root,
and denote by $\T_n = \cup_{0\leq k\leq n} \G_k$ the tree up to generation $n$.
For a vertex $u\in\T$, we denote by $\T(u)$ the subtree of $\T$ composed of descendants of $u$,
and for all $k\in\N$, we denote by $\T(u,k) = \T(u) \cap \T_{\height{u}+k}$ the subtree of $\T(u)$ 
	composed of descendants of $u$ at distance up to $k$ from $u$.
We will use the convention that for a subtree $T_{\text{sub}}$ of $T$, we write $T_{\text{sub}}^*$ for the subtree $T_{\text{sub}}$ without its root vertex,
for instance, $\T_n^* = \T_n \setminus \{\rooot\}$ and $\T(u)^* = \T(u) \setminus \{u\}$.
For a finite subset $A\subset \T$, we denote by $\vert A \vert$ its cardinal.

We will sometimes use Neveu's notation, which we define recursively:
a vertex $u\in T$ with height $\height{u}=n$ can be represented as a sequence $(u_{(j)})_{1\leq j\leq n}$
where $u$ is the $u_{(n)}$-th child of $\parent(u)$ and $\parent(u)$ can be represented by $(u_{(j)})_{1\leq j\leq n-1}$;
and the representation of the root $\rooot$ is the empty sequence.
Note that Neveu's notation can also be interpreted as encoding the path from the root $\rooot$ to the vertex $u$:
starting from the root $u_0=\rooot$, at each generation $j$ we go from $u_j$ to its $u_{(j+1)}$ child which we denote by $u_{j+1}$,
and at generation $n$ we get $u_n=u$.

For simplicity, we will write $u_{(k:n)} = (u_{(j)})_{k\leq j \leq n}$ and $u_{(k:n)} = (u_{(j)})_{k\leq j \leq n}$ for path sequences 
where $k,n\in \Z$ with $k<n$.

As $\T$ is a plane rooted tree, we can order its vertices in a breadth-first-search manner, that is,
the total order relation $\leq$ on $\T$ is defined for all $u,v\in\T$ as $u \leq v$
if $\height{u} < \height{v}$ or $\height{u} = \height{v}$ and $u \leqlex v$ (where $\leqlex$ is the lexicographical order on $\T$).
Moreover, we denote by $v < u$ if $u \leq v$ and $v\neq u$.

\subsection{Definition of HMT processes}

For a sequence $(x_u, u\in\T)$, for simplicity, we will write $x_{A} = (x_u , u\in A)$ for all subsets $A\subset \T$.
For a metric space $\SpaceX$, we will always assume it is equipped with its Borel $\sigma$-field $\BorelX$.

For a measure $\mu$ on a metric space $\SpaceX$ and an integrable function $f\in L^1(\mu)$,
we write $\mu(f) = \int_{\SpaceX} f\, \drv \mu$.
For two probability measures $\mu_1, \mu_2$ on a metric space $\SpaceX$,
we denote the total variation norm between them
by $\normTV{\mu_1 - \mu_2} = \sup_{A\subset \SpaceX} \vert \mu_1(A) - \mu_2(A) \vert$
(where $A$ ranges over all measurable subsets of $\SpaceX$).
We also remind the identities  
$\normTV{\mu_1 - \mu_2} = \inv{2} \sup_{f : \vert f \vert \leq 1} \vert \mu_1(f) - \mu_2(f) \vert
= \sup_{f : 0 \leq f  \leq 1} \vert \mu_1(f) - \mu_2(f) \vert$
(where $f$ is a measurable function).
Note that $\normTV{\mu_1 - \mu_2}$ takes value in $[0,1]$.

Denote by $X = (X_u, u\in \T)$ the hidden (stochastic) process with values in a metric space $\SpaceX$,
and by $Y = (Y_u, u\in \T)$ the observed (stochastic) process with values in a metric space $\SpaceY$.
We assume that the hidden process $X$ is a branching Markov process.
\begin{definition}[Branching Markov process]
	\label{def:Markov_proc}
The stochastic process $X$
	is called a \emph{(branching) Markov process}
	with transition (probability) kernel $Q$ on $(\SpaceX, \BorelX)$ 
		and initial (probability) distribution $\nu$ on $\SpaceX$
	if for all $n\in\N$, we have:
\begin{align*}
	\Prb(X_{\T_n} \in \drv x_{\T_n})
	& = \nu(\drv x_\rooot)\,  \Pi_{u\in \T_n^*} 
			Q(x_{\parent(u)}; \drv x_u)	. 
\end{align*}
\end{definition}

We can now define the Hidden Markov Tree process.

\begin{defin}[Hidden Markov Tree process]
	\label{def:HMT}
The stochastic process $(X,Y) = ((X_u,Y_u), u\in\T)$ is called a \emph{Hidden Markov Tree (HMT)}
with parameter $(Q,G,\nu)$ if:
\begin{enumerate}[label=(\roman*)]
\item
the hidden process $X = (X_u, u\in \T)$ is a branching Markov process
	with transition kernel $Q$ and initial distribution $\nu$,
\item
the observed process $Y = (Y_u, u\in \T)$ conditioned on the hidden process $X$
is composed of independent variables whose laws factorize using the transition (probability) kernel $G$ on $(\SpaceX, \BorelY)$,
that is for all $n\in\N$, we have:
\begin{equation*}
	\Prb( Y_{\T_n} = y_{\T_n} \,\vert\, X_{\T_n} = x_{\T_n} )
	= \Pi_{u\in \T_n} G( x_u; \drv y_u).
\end{equation*}
\end{enumerate}
\end{defin}

Note that the definitions of branching Markov chains and HMT processes
also work for non-plane rooted trees.

In particular, note that if $(X,Y) =((X_u,Y_u), u\in\T)$ is a HMT process, then the joint process $((X_u,Y_u), u\in\T)$
is also a branching Markov chain
(but the observed process $Y$ is not necessarily Markov).
The following fact, which we shall reuse later, illustrates the Markov property of the HMT process $(X,Y)$.
For any $k\in\N^*$, any $u\in\T$ with height at least $k$, and any subset $A\subset \T$, we have :
\begin{equation}\label{eq_illustration_Markov_prop}
\cL( X_u \,\vert\, Y_A, X_{\parent^k(u)} ) = \cL( X_u \,\vert\, Y_{A \cap \T(\parent^{k-1}(u))}, X_{\parent^k(u)} )
\end{equation}
where $\cL(R \,\vert\, S)$ denotes the distribution of a random variable $R$ conditionally on another random variable $S$.

\begin{figure}[t]
\centering
\includegraphics[height=5cm]{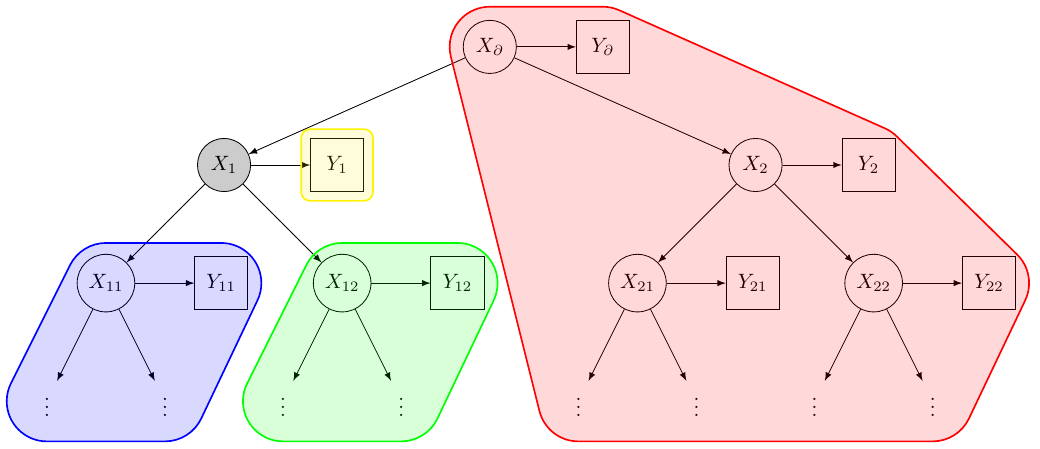}
\caption{Illustration of the Markov property for the HMT.
	Conditioning on $X_1$ (in grey) implies that the HMT process $(X,Y)$ becomes independent
	between the four connected components of variable-dependence tree from Figure~\ref{fig_dependance_HMT}
		where the vertex $X_1$ is removed,
	that is, $Y_{1}$, $(X_{\T(11)}, Y_{\T(11)})$, $(X_{\T(12)},Y_{\T(12)})$ and $(X_{\T\setminus\T(1)},Y_{\T\setminus\T(1)})$
	(respectively in yellow, blue, green and red) are independent
	conditionally on $X_1$. 
	}
\label{fig_Markov_prop_HMT}
\end{figure}

We say that a branching Markov process $X$ is \emph{stationary} 
if all its variables are identically distributed (that is, for all $u\in\T$, $X_u$ has the same distribution as $X_\rooot$),
or equivalently if its initial distribution $\nu$ is invariant for its transition kernel $Q$ (that is, $\nu Q = \nu$).
Moreover, if $(X,Y)$ is a HMT process and the hidden process $X$ is stationary,
then the joint process $(X,Y)$ is also stationary.

\subsection{Basic  assumptions and definition of the log-likelihood}

In this article, we consider the case where the kernels in the definition of the HMT $(Q_\theta,G_\theta,\nu_\theta)$
are parametrized by some vector $\theta$ that we want to estimate
using only the knowledge of the observed process $(Y_u, u\in \T_n)$ up to generation $n$.
We denote by $\Theta$ the set of all possible vector parameters $\theta$,
	which we assume to be a subset of $\R^{\dimTheta}$ for some integer $\dimTheta$.
And we denote by $\thetaTrue$ the true parameter of the HMT.

Through this article, with the exception of \Cref{section_non_stationary},
we assume that the hidden process $X$ is stationary.

\begin{assumption}[Stationarity]
	\label{assump_HMM_0}
The hidden process $(X_u , u\in\T)$ is stationary,
(and thus the joint process $((X_u,Y_u) , u\in\T)$ is also stationary).
\end{assumption}

We denote by $\Prb_\theta$ the probability distribution under the parameter $\theta$ of the stationary joint process $(X,Y)$,
and by $\Esp_\theta$ the corresponding expectation.

To prove asymptotic properties of the MLE for the HMT,
we will use assumptions similar to the HMM case in \cite[Chapter 12]{CappeHMM} and \cite{doucAsymptoticPropertiesMaximum2004}.
We first assume that the HMT model is fully dominated.
For two probability measures $\lambda,\mu$ on the same space,
we write $\lambda \ll \mu$ to denote that $\lambda$ is absolutely continuous \wrt to $\mu$.

\begin{assumption}[Fully dominated model, {\cite[Assumption 12.0.1]{CappeHMM}}]
	\label{assump_HMM_1}
~
\begin{enumerate}[label=(\roman*)]
\item\label{assump_HMM_1:item1}
There exists a probability measure $\lambda$ on $\SpaceX$ such that
for every $x\in\SpaceX$ and every $\theta\in\Theta$,
$Q_\theta(x,\cdot) \ll \lambda$,
with density $q_\theta(x,\cdot)$.
That is, $Q_\theta(x;A) = \int_{A} q_\theta(x,x')\, \lambda(\drv x')$ for all $A\in \BorelX$.
Moreover, the density function $q_\theta(\cdot,\cdot)$ is $\BorelX \otimes \BorelX$ measurable.
\item\label{assump_HMM_1:item2}
There exists a $\sigma$-finite measure $\mu$ on $\SpaceY$ such that
for every $x\in\SpaceX$ and every $\theta\in\Theta$,
$G_\theta(x,\cdot) \ll \mu$,
with density $g_\theta(x,\cdot)$.
That is, $G_\theta(x;B) = \int_{B} g_\theta(x,y)\, \mu(\drv y)$ for all $B\in \BorelY$.
Moreover, the density function $g_\theta(\cdot,\cdot)$ is $\BorelX \otimes \BorelY$ measurable.
\end{enumerate}
\end{assumption}

In addition to \Cref{assump_HMM_1}, we use the following regularity assumptions
on the density functions $q_\theta$ and $g_\theta$.

\begin{assumption}[Regularity, {\cite[Assumption 12.2.1]{CappeHMM}}]
	\label{assump_HMM_2}
~
\begin{enumerate}[label=(\roman*)]
\item\label{assump_HMM_2:item1}
The transition density $q_\theta$ is bounded:
there exist $\sigma^-, \sigma^+ \in (0, +\infty)$ such that
$\forall x,x'\in\SpaceX,\ \forall \theta\in\Theta$, 
$0 < \sigma^- \leq q_\theta(x,x') \leq \sigma^+ < +\infty$.
\item\label{assump_HMM_2:item2}
For every $y\in\SpaceY$, 
the function $\theta \mapsto \int_{\SpaceX} g_\theta(x,y) \, \lambda(\drv x)$
is bounded away from $0$ and $\infty$ uniformly on $\Theta$, that is,
$\sup_{\theta\in\Theta} \int_{\SpaceX} g_\theta(x,y) \, \lambda(\drv x) < \infty$
and $\inf_{\theta\in\Theta} \int_{\SpaceX} g_\theta(x,y) \, \lambda(\drv x) > 0$.
\item\label{assump_HMM_2:item3}
For every $x\in\SpaceX$ and $y\in\SpaceY$, we have $g_\theta(x,y)>0$.
\end{enumerate}
\end{assumption}
We will denote by $\rho = 1- \sigma^- / \sigma^+ \in (0,1)$ the \emph{mixing rate} of the hidden process $X$.

Note that \Cref{assump_HMM_2}-\ref{assump_HMM_2:item3}
is due to our choice of conditioning on $X_\rooot=x$ for some $x\in\SpaceX$ in the log-likelihood we analyse,
we discuss how to get rid of this assumption after the definition of the log-likelihood
at the end of this subsection.

As $\lambda$ is a probability measure and $q_\theta$ is the density of a probability kernel,
\Cref{assump_HMM_2}-\ref{assump_HMM_2:item1} implies
that $\sigma^- \leq 1 \leq \sigma^+$.
Moreover, \Cref{assump_HMM_2}-\ref{assump_HMM_2:item1} 
also implies the following result.
	
\begin{lemma}
	\label{lemma_Q_unif_geom_ergodic}
Assume that Assumptions~\ref{assump_HMM_1}-\ref{assump_HMM_1:item1} and \ref{assump_HMM_2}-\ref{assump_HMM_2:item1} hold.
Then, for all $\theta\in\Theta$, the transition kernel $Q_\theta$
has a unique invariant measure $\pi_\theta$
and is uniformly geometrically ergodic, that is:
\begin{equation*}
\forall x\in\SpaceX, \forall \theta\in\Theta, \forall n\geq 0, \quad
	\norm{Q^n_\theta(x,\cdot) - \pi_\theta}_{\text{TV}} \leq (1-\sigma^-)^n \leq \rho^n.
\end{equation*}
\end{lemma}

Note that due to structure of the HMT, \Cref{lemma_Q_unif_geom_ergodic} extends naturally
to the transition kernel of the joint process $(X,Y)$
with the same mixing rate $\rho$.
Moreover, note that \Cref{assump_HMM_2}-\ref{assump_HMM_2:item1} also implies that
$\pi_\theta \ll \lambda$ with density $\frac{\drv \pi_\theta}{\drv \lambda}$ 
taking value in $[ \sigma^-, \sigma^+]$.

\Cref{lemma_Q_unif_geom_ergodic} is due to \Cref{assump_HMM_2}-\ref{assump_HMM_2:item1} implying the Doeblin condition:
\begin{equation}\label{eq_Doeblin_cond_Q_theta}
\forall \theta\in\Theta, \forall x\in\SpaceX \qquad \sigma^- \lambda(\cdot) \leq Q_\theta(x;\cdot).
\end{equation}
As we will reuse Doeblin conditions later, before proving \Cref{lemma_Q_unif_geom_ergodic},
we give a quick summary on results for the Doeblin condition.
For a transition kernel $K$ on a metric space $\SpaceX$ (to itself), we define its \emph{Dobrushin coefficient} $\delta(K)$ as:
\begin{equation}\label{eq_def_Dobrushin_coefficient}
\delta(K) =  \sup_{x,x' \in \SpaceX} \normTV{K(x;\cdot) - K(x';\cdot)} .
\end{equation}
The Dobrushin coefficient gives the following coupling bound in the total variation norm.
(Note that the definition of the total variation norm $\normTV{\cdot}$ used in \cite[Chapter 4]{CappeHMM}
	differs by a factor $2$ from ours, see \cite[Lemma~4.3.5]{CappeHMM}.)

\begin{lemma}[{\cite[Lemma~4.3.8]{CappeHMM}}]
	\label{lemma_coumpling_bound_Dobrushin}
Let $\mu_1, \mu_2$ be two probability measures on a metric space $\SpaceX$,
and let $K$ be a transition kernel on $\SpaceX$.
Then, we have:
\begin{equation*}
\normTV{(\mu_1 - \mu_2)K} \leq \delta(K) \normTV{\mu_1 - \mu_2} \leq \delta(K).
\end{equation*}
\end{lemma}

Moreover, the Dobrushin coefficient is sub-multiplicative.

\begin{lemma}[{\cite[Proposition~4.3.10]{CappeHMM}}]
	\label{lemma_Dobrushin_sub_multiplicative}
The Dobrushin coefficient is sub-multiplicative.
That is, if $K,R$ are two transition kernels on a metric space $\SpaceX$,
then we have $\delta(K R) \leq \delta(K) \delta(R)$.
\end{lemma}

We know define the Doeblin condition.

\begin{defin}[Doeblin condition, {\cite[Definition~4.3.12]{CappeHMM}}]
	\label{def_Doeblin_condition}
We say that a transition kernel $K$ on a metric space $\SpaceX$ satisfies a \emph{Doeblin condition}
if there exist $\eps >0$ and a probability measure $\nu$ on $\SpaceX$ such that
for all $x\in\SpaceX$ and measurable subset $A\subset \SpaceX$, we have:
\begin{equation*}
K(x;A) \geq \eps \nu(A).
\end{equation*}
\end{defin}

The Doeblin condition gives an upper bound on the Dobrushin coefficient.

\begin{lemma}[{\cite[Lemma~4.3.13]{CappeHMM}}]
	\label{lemma_Doeblin_Dobrushin}
Let $K$ be transition kernel (on a metric space $\SpaceX$) that satisfies a Doeblin condition with $(\eps,\nu)$.
Then, we have $\delta(K) \leq 1-\eps$.
\end{lemma}

Lastly, the Doeblin condition implies the existence of a unique invariant probability measure,
as well as uniform geometric ergodicity.

\begin{lemma}[{\cite[Theorem~4.3.16]{CappeHMM}}]
	\label{lemma_exist_unif_geom_ergodic}
Let $K$ be a transition kernel on a metric space $\SpaceX$ that satisfies a Doeblin condition with $(\eps,\nu)$.
Then, $K$ admits a unique invariant probability measure $\pi$.
Moreover, for any probability measure $\zeta$ on $\SpaceX$, we have for all $n\in\N$:
\begin{equation*}
\normTV{\zeta K^n - \pi} \leq (1-\eps)^{n} \normTV{\zeta - \pi} .
\end{equation*}
\end{lemma}

\Cref{lemma_Q_unif_geom_ergodic} then follows immediately from \Cref{lemma_exist_unif_geom_ergodic}
and the uniform Doeblin condition \eqref{eq_Doeblin_cond_Q_theta}.

\begin{remark}[More properties of the transition kernel from the Doeblin condition]
For a transition kernel $K$ on a metric space $\SpaceX$,
the Doeblin condition also implies that $\SpaceX$ is an (accessible) $1$-small set.
In particular, we get that $K$ satisfies some extra classical properties (that we will not use here):
$K$
is positive (\ie irreducible and admits a unique invariant probability measure), strongly aperiodic and Harris recurrent
(see \cite[Chapter~9 and 10]{DoucMC} for definitions and details).
\end{remark}

\medskip

We will use the letter $p$ to denote (possibly conditional) probability density.
For instance, for any $\theta\in\Theta$, $y_{\T_n}\in\SpaceY^{\T_n}$ and $x_{\rooot}\in\SpaceX$, we denote by:
\begin{equation}\label{eq_def_exple_p_theta}
 p_\theta( y_{\T_n} \,\vert\, X_{\rooot} = x_{\rooot} ) 
	= g_\theta(x_{\rooot}, y_{\rooot})  \int_{\SpaceX^{ \T_n^* } } 
			\prod_{v\in\T_n^*} 
			q_\theta(x_{\parent(v)},x_v) g_\theta(x_v, y_v) \lambda(\drv x_v)	,
\end{equation}
the conditional density \wrt $\mu^{\otimes \T_n}$ under the parameter $\theta$ of $Y_{\T_n}=y_{\T_n}$
conditionally on $X_{\rooot}=x_{\rooot}$. 
Note that \Cref{assump_HMM_2} guarantees that $ p_\theta( y_{\T_n} \,\vert\, X_{\rooot} = x_{\rooot} ) $
is positive for all $y_{\T_n}\in\SpaceY^{\T_n}$ and $x_{\rooot}\in\SpaceX$.

We are now ready to define the log-likelihood.
As discussed in \Cref{section_intro_new_contribution},
we will analyze the log-likelihood of the observed process $(Y_u, u\in \T_n)$ up to generation $n$
	conditioned on the hidden value of the root $X_{\rooot} = x$ for some $x\in\SpaceX$. 
Thus, for any $x\in\SpaceX$, we define the log-likelihood function as:
\begin{equation}
	\label{eq_def_l_nx_theta}
\ell_{n,x}(\theta; y_{\T_n}) 
	:= \log \bigl( p_\theta( y_{\T_n} \,\vert\, X_{\rooot} = x ) 
		\bigr).
\end{equation}
We then define the log-likelihood that we will analyze as the following random variable
\begin{equation}\label{eq_def_l_nx_theta_2}
\ell_{n,x}(\theta) := \ell_{n,x}(\theta; Y_{\T_n}).
\end{equation}
For simplicity, we will write $\ell_{n,x}(\theta)$ instead of $\ell_{n,x}(\theta; Y_{\T_n})$
making the dependence on the observed variables $(Y_u , u\in\T_n)$ implicit.
We will keep this convention for all quantities considered in this article,
and only make the dependence explicit when necessary.
The MLE is then the maximizer over $\Theta$ of the log-likelihood $\ell_{n,x}$;
we post-pone the precise definition of the MLE to when we will first use it in \Cref{thm_Strong_consistency_MLE}.

\begin{remark}[On \Cref{assump_HMM_2}-\ref{assump_HMM_2:item3}]
Note that \Cref{assump_HMM_2}-\ref{assump_HMM_2:item3}
is due to our choice of conditioning on $X_\rooot=x$ for some $x\in\SpaceX$ in the log-likelihood $\ell_{n,x}(\theta)$ we analyse.
Indeed, without \Cref{assump_HMM_2}-\ref{assump_HMM_2:item3}, there could be a non-zero probability under $\Ptrue$
that $g_{\thetaTrue}(x,Y_{\rooot}) = 0$ for some $x\in\SpaceX$, implying $\ell_{n,x}(\thetaTrue) = -\infty$,
and thus preventing the MLE to converge to $\thetaTrue$.
Several modifications of the log-likelihood  $\ell_{n,x}(\theta)$
can be considered to get rid of \Cref{assump_HMM_2}-\ref{assump_HMM_2:item3}.
A first option would be to replace $p_\theta( y_{\T_n} \,\vert\, X_{\rooot} = x ) $
by $p_\theta( y_{\T_n^*} \,\vert\, X_{\rooot} = x ) $ in \eqref{eq_def_l_nx_theta}.
A second option would be to extend the tree $\T$ and the HMT $(X,Y)$ 
to add a parent vertex $\parent(\rooot)$ for the root vertex $\rooot$
(see \Cref{subsubsection_extension_tree_Tpast}),
and then replace $p_\theta( y_{\T_n} \,\vert\, X_{\rooot} = x ) $
by $p_\theta( y_{\T_n} \,\vert\, X_{\parent(\rooot)} = x ) $ in \eqref{eq_def_l_nx_theta}.
\end{remark}

\subsection{Ergodic theorems with neighborhood-dependent functions}
	\label{subsection_ergodic_theorem}

For all $u\in\T$ and $0 \leq k\leq \height{u}$, define the subtrees of $\T$:
\begin{equation*}
\Delta^*(u,k) = \{ v \in \T(\parent^k(u)) \,:\, v < u \} ,
\end{equation*}
and $\Delta(u,k) = \Delta^*(u,k) \cup \{ u \}$.
In particular, note that when $k=\height{u}$, we have that $\Delta^*(u) := \Delta^*(u,\height{u}) = \{ v \in \T \,:\, v < u \}$,
and we also write $\Delta(u) := \Delta(u,\height{u})$.
See \Cref{fig_Delta_u_k} for an illustration of those subtrees.
The subtree $\Delta^*(u)$ represents the past of the vertex $u$.	

\begin{figure}[t]
\centering

\begin{subfigure}[b]{0.45\textwidth}
         \centering
\includegraphics[width=4.4cm]{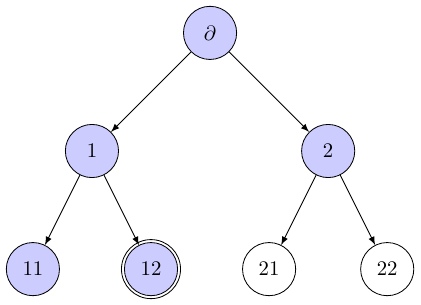}
\end{subfigure}
\begin{subfigure}[b]{0.45\textwidth}
         \centering
\includegraphics[width=4.4cm]{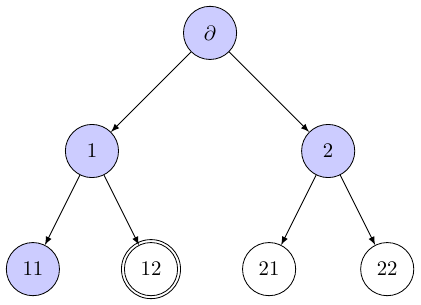}
\end{subfigure}\\
\vspace{0.6cm}
\begin{subfigure}[b]{0.45\textwidth}
         \centering
\includegraphics[width=4.4cm]{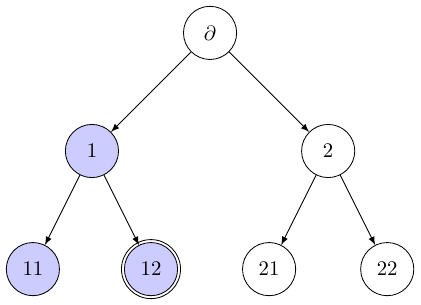}
\end{subfigure}
\begin{subfigure}[b]{0.45\textwidth}
         \centering
\includegraphics[width=4.4cm]{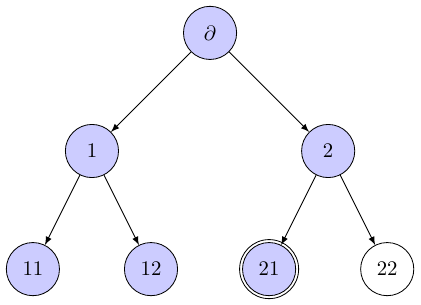}
\end{subfigure}

\caption{Illustration of the subtrees $\Delta(u,k)$ and $\Delta^*(u,k)$ where
	vertices in $\Delta(u,k)$ and $\Delta^*(u,k)$ are in blue
	and the vertex $u$ is inside a double circle.
	From left to right and top to bottom, using Neveu's notation, we have
	$\Delta(12)=\Delta(12,2)$, $\Delta^*(12)=\Delta^*(12,2)$, $\Delta(12,1)$ and $\Delta(21)=\Delta(21,2)$.
	Note that $\Delta(12)$ and $\Delta(21)$ have a different
	number of vertices,
	while vertices $12$ and $21$ are in the same generation.
	}
\label{fig_Delta_u_k}
\end{figure}

For the ergodic convergence results needed in this article,
we will need to consider different functions for each vertex $u\in\T$ depending on the ‘‘shape’’ of the subtree $\Delta(u,k)$
for some common $k\in\N$.
For $k\in\N$ and vertices $u,v\in\T$ both with height at least $k$,
	we say that $\Delta(u,k)$ and $\Delta(v,k)$ \emph{have the same shape}
	if they are equal up to translation, that is, if they are isomorphic as (finite) rooted plane trees.
For $k\in\N$ and any vertex $u\in\T$ with $\height{u}\geq k$,
there exists a unique $v_u\in\G_k$ such that $\Delta(u,k)$ and $\Delta(v_u)$ have the same shape,
and we thus define the \emph{shape} of $\Delta(u,k)$ as:
\begin{equation}\label{eq_def_shape_subtree}
\Shape\bigl( \Delta(u,k) \bigr) = \Delta(v_u).
\end{equation}
Note that as $\vert\Delta(v)\vert$ is different for each $v\in\G_k$,
	thus the shape of $\Delta(u,k)$ is characterized by its size.
For any $k\in\N$, we define the (finite) set $\ShapeSetValues_k$ of possible shapes for $\Delta(u,k)$ with $u\in\T$ as:
\begin{equation}\label{eq_def_set_possible_shapes}
\ShapeSetValues_k = \{ \Delta(v) \,:\, v\in\G_k \} .
\end{equation}

For any $k\in\N$, we define a collection of \emph{neighborhood-shape-dependent} functions as a collection of functions $(f_{\ShapeValue} : \SpaceZ^{\ShapeValue} \to \R)_{\ShapeValue\in\ShapeSetValues_k}$
where $\SpaceZ \in \{ \SpaceX, \SpaceY, \SpaceX\times\SpaceY \}$. For such a collection of functions, we will simply write $f_{\Delta(u,k)}$ 	instead of $f_{\Shape(\Delta(u,k))}$.  And we will also write $f_{\Delta(u,k)}(Y_{\Delta(u,k)})$ for the evaluation of $f_{\Delta(u,k)}$ on $Y_{\Delta(u,k)}$.
Note that indexing such a collection of functions with $\G_k$ or with $\ShapeSetValues_k$ 
is equivalent in light of \eqref{eq_def_set_possible_shapes}.

We prove the following ergodic convergence lemma for neighborhood-shape-dependent functions. 
The proof of this lemma relies on the theorems in \Cref{appendix_ergodic_theorems}.
Note that if $U_n$ is uniformly distributed over $\G_n$ with $n\geq k$,
then $\Shape(\Delta(u,k))$ is uniformly distributed over $\ShapeSetValues_k$.

\begin{lemme}[Ergodic theorem for neighborhood-dependent functions]
	\label{lemma_ergodic_convergence}
Assume that Assumptions~\ref{assump_HMM_0}--\ref{assump_HMM_2} hold.
Let $k\geq 0$.
Let $(f_{\ShapeValue} : \SpaceY^{\ShapeValue} \to \R)_{\ShapeValue\in\ShapeSetValues_k}$ 
be a collection of neighborhood-shape-dependent Borel functions that are in $L^2(\Ptrue)$.
Then, we have:
\begin{equation}\label{eq_ergodic_convergence_as_L2}
\lim_{n\to\infty}
\inv{\vert\T_n\vert} \sum_{u\in\T_n\setminus\T_{k-1}} 
	\!\! f_{\Delta(u,k)}(Y_{\Delta(u,k)})
= \Esp_{U_k} \otimes \Etrue\bigl[ f_{\Delta(U_k)}(Y_{\Delta(U_k)})   \bigr]
\quad \text{$\Ptrue$-\as and in $L^2(\Ptrue)$,}
\end{equation}
with the convention $\T_{-1} = \emptyset$,
and where $U_k$ is uniformly distributed over $\G_k$ and independent of the process $X$,
and $\Esp_{U_k} \otimes \Etrue$ denotes the joint expectation over $U_k$ and $X$ (under the true parameter $\thetaTrue$).

Moreover, there exist finite constants $C < \infty$ and $\alpha\in(0,1)$ such that:
\begin{equation}\label{eq_LLN_upper_bound_Esp_M_Gn_main_body}
\forall n\geq k, \quad \Esp\left[ \left( \inv{\vert\T_n\vert} \sum_{u\in\T_n\setminus\T_{k-1}} 
	f_{\Delta(u,k)}(Y_{\Delta(u,k)})
	- \Esp_{U_k} \otimes \Etrue\bigl[ f_{\Delta(U_k)}(Y_{\Delta(U_k)})   \bigr] \right)^2 \right] \leq C \alpha^n .
\end{equation}
\end{lemme}

Remark that in the left hand side of \eqref{eq_ergodic_convergence_as_L2} the subtrees $\Delta(u,k)$ are deterministic,
while the subtree $\Delta(U_k)$ is a random function of $U_k$.

\begin{proof}
Using \Cref{lemma_Q_unif_geom_ergodic}, remind 
that under Assumptions~\ref{assump_HMM_0}--\ref{assump_HMM_2},
the branching Markov process $(X,Y)$ is stationary 
and its transition kernel has a unique invariant probability 
and is uniformly geometrically ergodic. 
Hence, the lemma follows immediately from applying the ergodic
	Theorems~\ref{Ergodic_theorem_with_neighborhood_functions}
	and~\ref{Strong_LLN_with_neighborhood_functions} 
	for neighborhood-shape-dependent functions from the appendix.
\end{proof}

As $\T$ is a plane rooted tree, we can enumerate its vertices as a sequence $(v_j)_{j\in\N}$
in a breadth-first-search manner, that is, which is increasing for $<$
(note that $u_0 = \rooot$).
Note that if $V_n$ is uniformly distributed over 
$A_n := \{ v_j \,:\, \vert \T_{k-1} \vert < j \leq n \}
= \Delta(v_n) \setminus \T_{k-1}$,
then the distribution of $\Shape(\Delta(V_n,k))$ converges to the uniform distribution over $\ShapeSetValues_k$
as $n\to\infty$.
We will also need the following variant of \Cref{lemma_ergodic_convergence}
where $\T_n\setminus \T_{k-1}$ is replaced by $A_n$.

\begin{lemme}[Another ergodic theorem for neighborhood-dependent functions]
	\label{lemma_ergodic_convergence_2}
Assume that Assumptions~\ref{assump_HMM_0}--\ref{assump_HMM_2} hold.
Let $k\geq 0$.
Let $(f_{\ShapeValue} : \SpaceY^{\ShapeValue} \to \R)_{\ShapeValue\in\ShapeSetValues_k}$ 
be a collection of neighborhood-shape-dependent Borel functions that are in $L^2(\Ptrue)$.
Let $(v_j)_{j\in\N}$ be the sequence enumerating the vertices in $\T$ in a breadth-first-search manner.
For all $n> \vert\T_{k-1}\vert$, define $A_n = \Delta(v_n) \setminus \T_{k-1}$.
Then, we have:
\begin{equation*}
\lim_{n\to\infty}
\inv{n} \sum_{u\in A_n}
	f_{\Delta(u,k)}(Y_{\Delta(u,k)})
= \Esp_{U_k} \otimes \Etrue\bigl[ f_{\Delta(U_k)}(Y_{\Delta(U_k)})   \bigr]
\quad \text{in $L^2(\Ptrue)$,}
\end{equation*}
where $U_k$ is uniformly distributed over $\G_k$ and independent of the process $X$,
and $\Esp_{U_k} \otimes \Etrue$ denotes the joint expectation over $U_k$ and $X$ (under the true parameter $\thetaTrue$).
\end{lemme}

\begin{proof}
Using \Cref{lemma_Q_unif_geom_ergodic}, remind 
that under Assumptions~\ref{assump_HMM_0}--\ref{assump_HMM_2},
the branching Markov process $(X,Y)$ is stationary 
and its transition kernel has a unique invariant probability 
and is uniformly geometrically ergodic. 
Hence, the lemma follows immediately from applying the ergodic
	Theorem~\ref{Ergodic_theorem_with_neighborhood_functions}
	for neighborhood-shape-dependent functions from the appendix.
\end{proof}

\section{Strong consistency of the MLE}\label{section_strong_consistency}

In this section, we first define the extended tree $\Tpast$ to get an infinite past horizon
	and rewrite the log-likelihood as a sum of increments.
Then, we construct the log-likelihood increments with infinite past,
	which allows to define the contrast function.
We prove properties for this contrast function.
Finally, we prove the strong consistency of the MLE.

\subsection{Decomposition of the log-likelihood into increments}
	\label{subsection_decomp_likelihood_ordering}

\subsubsection{The extended tree \texorpdfstring{$\Tpast$}{} to get an infinite past horizon}
	\label{subsubsection_extension_tree_Tpast}

Remind that the subtree $\Delta^*(u)  = \Delta^*(u,\height{u})$ represents the past of the vertex $u$.

To get an infinite past horizon, we will consider an extended version of the tree $\T$.
Thus, we are going to define a random (countable) plane rooted tree $\Tpast$
that contains $\T$ as a subtree and is also rooted at $\rooot$ the root vertex of $\T$,
and where each vertex (including $\rooot$) has exactly one parent node and two children nodes.
To construct $\Tpast$, we start from $\T$ and add a line $L = \{ u_{-j} : j\in\N^*\}$ of ancestors for $\rooot$
(that is, $u_{-j} = \parent^j(\rooot)$ for $j\in\N$, where $u_0 = \rooot$),
and then for all $j\in\N^*$, we graft on $u_{-j}$ a copy $T^{(j)}$ of $\T$
(that is, $u_{-j}$ is the parent of the root vertex $\rooot^{(j)}$ of $\T^{(j)}$).
We extend the height function $\heightFunction$ from $\T$ to $\Tpast$ as follows:
for all $j\in\N^*$, we set $\height{u_{-j}} = -j$ and for all $u\in \T^{(j)}$,
we define $\height{u}$ as $-j$ plus the number of edges separating $u$ from $u_{-j}$.
For $u,v\in\Tpast$, denote by $u\land v$ their most recent common ancestor,
and by $\dgr(u,v) = \height{u} + \height{v} - 2 \height{u\land v}$ the graph distance between $u$ and $v$.
The definition of the subtrees $\Tpast(u)$ and $\Tpast(u,k)$ then naturally extend to $\Tpast$.

Thus, we have constructed the deterministic non-plane version of the tree $\Tpast$,
and we are left to define the random plane embedding of $\Tpast$.
That is, for each vertex $u\in\Tpast$, we have to define a possibly random ordering of its children.
As $\T$ is a plane rooted tree, note that if $u\in\T$ or $u\in\T^{(j)}$ for some $j\in\N^*$, then its children are already order deterministically.
Let $\cU = (\cU_{(j)})_{- \infty < j \leq 0}$ be a sequence of independent random variables with Bernoulli distribution of parameter $1/2$,
and which is independent of the HMT process $(X,Y)$.
For all $j\in\N$, we order the children of $u_{-j-1}$, that is $u_{-j}$ and $\delta^{(j+1)}$ (the root vertex of $\T^{(j)}$), as follows:
$u_{-j}$ is the left child of $u_{-j-1}$ if $\cU_{(-j)} = 0$, and is the right child otherwise.
Hence, we have constructed the random plane rooted tree $\Tpast$.
(Note that $\cU$ can be seen as the random shape of the backward spine of $\rooot$.)
See \Cref{fig_illustration_Tpast} for an illustration of 
the extended random plane rooted tree $\Tpast$. 
We denote by $\P_\cU$ the distribution of the random sequence $\cU$, and by $\Esp_\cU$ the corresponding expectation.

\begin{figure}[t]
\centering
\includegraphics[width=8cm]{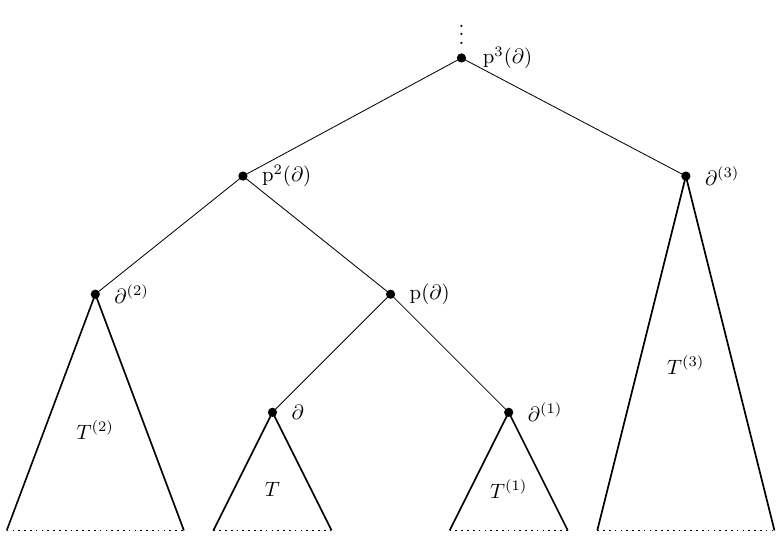}
\caption{Illustration of the construction of 
	the extended random plane rooted tree $\Tpast$
	(which is rooted at $\rooot$)
	for $\cU_{(-1)}=0$ and $\cU_{(-2)}=1$.
	}
\label{fig_illustration_Tpast}
\end{figure}

Note that the random plane embedding of $\Tpast$ allows to use Neveu's notation to represent the random path
between any vertex in the plane tree $\Tpast$ and one of its descendants
as a random sequence $\cU_{(k:n)}$ (which depends on $\cU)$ for some $k,n\in\Z$ with $k<n$.
The random breadth-first-search order relation $\leq := \leq_{\cU}$ can then be naturally extend from $\T$ to $\Tpast$
using the random plane embedding of $\Tpast$ (which depends on $\cU$):
we have $u \leq v$ for $u,v\in\Tpast$ if either $\height{u} < \height{v}$,
or $\height{u} = \height{v}$ and $U_{(k:n)} \leqlex V_{(k:n)}$ 
where $U_{(k:n)}$ (resp. $V_{(k:n)}$) is Neveu's notation for the random path (which depends on $\cU$)
from $u\land v$ to $u$ (resp. $v$)
with $k=\height{u\land v}+1$ and $n=\height{u}$.

Thanks to the stationarity assumption, for all $k\in\N$,
the HMT process $(X,Y)$ can be defined on the (rooted) tree $\Tpast(\parent^k(\rooot))$,
and thus by Kolmogorov's extension theorem,
the HMT process $(X,Y)$ can be defined on the whole tree $\Tpast$.
In particular, note that the stationarity assumption implies 
	that the distribution of the HMT process $(X,Y)$
	is invariant by translation on $\Tpast$,
	that is, is the same (up to translation) on $\T$ and on $\Tpast(u)$ for any $u\in\Tpast$.
Note that the extended process does not depend on $\cU$.
Thus, we will now assume that the HMT process $(X,Y)$ is defined on the whole tree $\Tpast$.

For all $u\in\Tpast$ and $k\in\N$,
define the subtrees (which are measurable functions of $\cU$):
\begin{equation*}
\DeltaR^*_{\cU}(u,k) = \{ v \in \Tpast(\parent^k(u)) \,:\, v <_{\cU} u \}, 
\end{equation*}
and $\DeltaR_{\cU}(u,k) = \DeltaR^*_{\cU}(u,k) \cup \{ u \}$.
For simplicity, we will write instead $\DeltaR^*(u,k)$ and $\DeltaR(u,k)$,
making the dependence on the random variable $\cU$ implicit,
and only make the dependence explicit when necessary.
The following fact illustrates that the shape and size of the $\DeltaR(u,k)$ do indeed depend on the value of $\cU$:
for $u=\rooot$ and $k=1$, note that $\DeltaR(\rooot,1)$ contains two vertices if $\cU_{(0)}=0$,
and contains three vertices if $\cU_{(0)}=1$.
\Cref{rem_rerooting_delta} below, which we shall reuse later, further illustrates the randomness of the set $\DeltaR(u,k)$.
However, for $u\in\T$ and $k\leq \height{u}$, we have that $\DeltaR(u,k) = \Delta(u,k)$ 
and $\DeltaR^*(u,k) = \Delta^*(u,k)$ are deterministic.
Also note that we have the following inclusions:
\begin{equation}\label{eq_inclusion_Delta_Tpast}
\Tpast(u,k-1) \subset \DeltaR(u,k) \subset \Tpast(u,k),
\end{equation}
where remind that the subtrees $\Tpast(u,k-1)$ and $\Tpast(u,k)$ are deterministic.

\begin{remark}
	\label{rem_rerooting_delta}
For a vertex $u=u_{(1:n)}$ in $\T$ with $\height{u}=n\geq k$, note that $\Delta(u,k)$, up to re-rooting (\ie up to translation),
	can be identified with $\DeltaR(\rooot,k)$ conditioned on $\cU_{(-k+1:0)} = u_{(n-k+1:n)}$.
In particular, when $U_n$ is a random vertex uniformly distributed over $\G_n$ for $n\geq k$,
we get the following equality between the distribution of the shapes
(that is, when the subtrees are seen up to translation / re-rooting)
for the subtrees $\DeltaR(\rooot,k)$, $\Delta(U_n,k)$ and $\Delta(U_k)$:
\begin{equation}\label{eq_equal_distrib_shapes}
\Shape\bigl( \DeltaR(\rooot, k) \bigr) 
\overset{\cL}{=} \Shape(\Delta(U_n,k))
\overset{\cL}{=} \Delta(U_k).
\end{equation}
Moreover, if $(f_{\ShapeValue} : \SpaceY^{\ShapeValue} \to \R)_{\ShapeValue\in\ShapeSetValues_k}$ 
is a collection of neighborhood-shape-dependent Borel functions that are in $L^2(\Ptrue)$
(as in Lemmas~\ref{lemma_ergodic_convergence} and \ref{lemma_ergodic_convergence_2}),
then we have:
\begin{equation}\label{eq_equality_expectation_root_and_U_k}
\Esp_\cU \otimes \Etrue \bigl[  f_{\DeltaR(\rooot,k)}(Y_{\DeltaR(\rooot,k)})   \bigr]
=  \Esp_{U_k} \otimes \Etrue\bigl[ f_{\Delta(U_k)}(Y_{\Delta(U_k)})   \bigr] ,
\end{equation}
where $\Esp_\cU \otimes \Etrue$ is the expectation corresponding to $\Prb_\cU \otimes \Ptrue$.
\end{remark}

\subsubsection{The log-likelihood as a sum of increments}
	\label{subsubsection_log_likelihood_sum_increments}

For any (possibly random) subtree $\Delta$ of $\Tpast$ with root vertex $w$, note that we have:
\begin{equation}\label{eq_def_p_theta_Y_Delta}
p_\theta(y_{\Delta} \,\vert\, X_w = x) = 
g_\theta(x, y_{\rooot}) \int_{\SpaceX^{\vert \Delta\vert -1}} 
			\prod_{v\in \Delta\setminus\{w\} } q_\theta(x_{\parent(v)},x_v) g_\theta(x_v, y_v) \lambda(\drv x_v).
\end{equation}
We will use the convention $p_\theta(Y_{\Delta} \,\vert\, X_w = x) = 1$ whenever $\Delta=\emptyset$ and $w$ is any vertex in $\Tpast$.
For all $u\in\T$, $k\in\N$, $x\in\SpaceX$ and $\theta\in\Theta$, using the conditional probabilities formula, define:
\begin{align}
\Hterm_{u,k,x}(\theta) & = 
\int_{\SpaceX} g_\theta(x_u, Y_u)  \, \Prb_\theta( X_u \in \drv x_u \,\vert\, Y_{\DeltaR^*(u,k)},  X_{\parent^k(u)}=x) \nonumber\\
& = \frac{p_\theta(Y_{\DeltaR(u,k)} \,\vert\, X_{\parent^k(u)} = x)}
	{p_\theta(Y_{\DeltaR^*(u,k)} \,\vert\, X_{\parent^k(u)} = x)} \cdot 	\label{eq_def_H_ukx}
\end{align}
We then define the log-likelihood contribution of node $u\in\T$ with past over $k\in\N$ generation as:
\begin{equation}\label{eq_def_h_ukx}
\h_{u,k,x}(\theta) = \log\bigl( \Hterm_{u,k,x}(\theta) \bigr)
	.
\end{equation}
Note that $\h_{u,k,x}(\theta)$ (resp. $\Hterm_{u,k,x}(\theta)$) is a random variable as a function of $Y_{\DeltaR(u,k)}$ 
with an implicit dependence on $\cU$ through $\DeltaR(u,k)$,
and that $\h_{u,k,x}(\theta)$ (resp. $\Hterm_{u,k,x}(\theta)$) does not depend on $\cU$ is $k\leq \height{u}$.

Hence, using \eqref{eq_def_l_nx_theta}, \eqref{eq_def_l_nx_theta_2}, \eqref{eq_def_H_ukx} 
	and  \eqref{eq_def_h_ukx} and a telescopic sum argument,
the log-likelihood of the observed variables $Y_{\T_n}$ can be rewritten as
the sum of the log-likelihood contributions defined in  \eqref{eq_def_h_ukx}:
\begin{equation}\label{eq_l_nx_as_sum_h_ukx}
\ell_{n,x}(\theta)  = \sum_{u\in\T_n} \h_{u, \height{u}, x}(\theta).
\end{equation}

\subsection{Construction of the log-likelihood increments with infinite past}

In this subsection, we construct the  log-likelihood increment functions with infinite past.

The following lemma states that, as the HMT is uniformly geometrically ergodic,
the tree forgets exponentially fast its starting state.
Recall the mixing ratio $\rho = 1 - \sigma^- / \sigma^+ \in (0,1)$
is defined just after \Cref{assump_HMM_2}.

\begin{lemme}[Exponential forgetting of the initial state]
	\label{lemme_exp_coupling_HMT}
Assume that Assumptions \ref{assump_HMM_1} and \ref{assump_HMM_2} hold.
We have for all $u\in\T$, $\theta\in\Theta$, $n\in\N$ and $y_{\T_n}\in \SpaceY^{ \T_n }$, 
and all initials distributions $\nu$ and $\nu'$ on $\SpaceX$,
that:
\begin{equation}\label{eq_exp_coupling_HMT}
\Bignorm{ \int_{\SpaceX} 
		\Prb_\theta\Bigl( X_u \in \cdot  \, \Bigm\vert \,  Y_{\T_n}=y_{\T_n}, X_{\rooot} = x  \Bigr) 
		[ \nu(\drv x) - \nu'(\drv x) ]
	}_{\mathrm{TV}}
\leq  \rho^{\height{u}} .
\end{equation}
\end{lemme}

For simplicity, \Cref{lemme_exp_coupling_HMT} is stated with $\rooot$ as the initial vertex,
but note that the results still holds when replacing $\rooot$ and $\T_n$
by $v$ and $\T(v,n)$ for any $v\in\Tpast$.
We shall reuse this fact later.

\begin{proof}
Fix some $u\in\T$, $\theta\in\Theta$, an integer $n$ and observables $y_{\T_n}\in \SpaceY^{ \T_n }$.
Denote by $u_0, \cdots, u_k$ with $k=\height{u}$ the vertices on the path from $\rooot$ to $u$.
The proof relies on the fact that conditionally on $Y_{\T_n} = y_{\T_n}$,
the sequence $(X_{u_j})_{0\leq j\leq k}$ is an inhomogeneous Markov chain
where 
for $1 \leq j \leq k$, the (forward smoothing) transition kernel $\text{F}_j$ from $X_{u_{j-1}}$ to $X_{u_j}$ 
is defined if $j \leq n$ as:
\begin{align*}
 \text{F}_j[y_{\T(u_j,  n-j)}](x_{u_{j-1}}; f) & = \Esp_\theta\bigl[ f(X_{u_j}) \,\bigm\vert\, Y_{\T(u_j,  n-j)} = y_{\T(u_j,  n-j)}, X_{u_{j-1}}=x_{u_{j-1}} \bigr] \\
& = \Esp_\theta\bigl[ f(X_{u_j}) \,\bigm\vert\, Y_{\T_n} = y_{\T_n}, X_{u_{j-1}}=x_{u_{j-1}} \bigr] \\
& = \frac{
	\int_{\SpaceX}   
		f(x_{u_{j}}) p_\theta( y_{\T(u_j,n-j)} \,\vert\, X_{u_j} = x_{u_j})
					q_\theta(x_{u_{j-1}}, x_{u_{j}})  \, \lambda(\drv x_{u_{j}}) 
}
{
	\int_{\SpaceX} p_\theta(  y_{\T(u_j,n-j)} \,\vert\, X_{u_j} = x_{u_j})
					q_\theta(x_{u_{j-1}}, x_{u_{j}})  \, \lambda(\drv x_{u_{j}}) 
} ,
\end{align*}
for any $x_{u_{j-1}}\in\SpaceX$ and any bounded Borel function $f$ on $\SpaceX$
(note that  in the second equality, we used the Markov property of the HMT process, see \eqref{eq_illustration_Markov_prop});
and is defined as $F_j = Q$ for $j>n$.
(Note that \Cref{assump_HMM_2}-\ref{assump_HMM_2:item2} is only used to insure that
$p_\theta( y_{\T(u_j,n-j)} \,\vert\, X_{u_j} = x_{u_j})$ is positive,
and thus the denominator in the last equality is also positive.)

Note that for all $1\leq j \leq k\land n$, using \Cref{assump_HMM_2}-\ref{assump_HMM_2:item1},
 the transition kernel $\text{F}_j$ satisfies the following Doeblin condition:
\begin{equation*}
\frac{\sigma^-}{\sigma^+}  \nu_{j}[y_{\T(u_j,  n-j)}] (f)
\leq \text{F}_{j}[y_{\T(u_j,  n-j)}] (x;f) ,
\end{equation*}
where for any bounded Borel function $f$ on $\SpaceX$, we have:
\begin{align*}
 \nu_{j}[y_{\T(u_j,  n-j)}] (f) 
 & = \Esp_\theta[ f(X_{u_j}) \,\vert\, Y_{\T(u_j,n-j)} = y_{\T(u_j,n-j)} ] \\
  & = \frac{
	\int_{\SpaceX} f(x_{u_{j}})  p_\theta(  y_{\T(u_j,n-j)} \,\vert\, X_{u_j} = x_{u_j})
		 \, \lambda(\drv x_{u_{j}}) 
}
{
	\int_{\SpaceX} p_\theta(  y_{\T(u_j,n-j)} \,\vert\, X_{u_j} = x_{u_j})
		 \, \lambda(\drv x_{u_{j}})
}  \cdot 
\end{align*}
Note that the difference between the definitions of $F_j$ and $\nu_j$ is that the term $q_\theta(x_{\parent(u_j)},x_{u_j})$
	has disappear from both the numerator and the denominator of $\nu_j$.
Remark that \eqref{eq_Doeblin_cond_Q_theta} 
also implies the Doeblin condition $\sigma^- \lambda(\cdot) \leq Q(x,\cdot)$ for the transition kernel $Q$.
Thus, \Cref{lemma_Doeblin_Dobrushin} 
	shows that the Dobrushin coefficient of each transition kernel $\text{F}_j$ for $1 \leq j \leq k$
	is upper bounded by $\rho = 1 - \sigma^- / \sigma^+$.
Therefore, as the Dobrushin coefficient is sub-multiplicative (see \Cref{lemma_Dobrushin_sub_multiplicative}), 
applying \Cref{lemma_coumpling_bound_Dobrushin},
we get that \eqref{eq_exp_coupling_HMT} holds.
This concludes the proof.
\end{proof}

To construct the limit of the functions $\h_{u,k,x}(\theta)$ we first prove the following
lemma which states some uniform bound about the asymptotic behavior of those
functions when $k \to \infty$.
For this lemma, we need the following assumption on the density function $g_\theta$
that strengthens \Cref{assump_HMM_2}-\ref{assump_HMM_2:item2}.
Remind that  $\Prb_\theta$ denotes the stationary probability distribution under the parameter $\theta\in\Theta$ of the HMT process $(X,Y)$,
and by $\Esp_\theta$ the corresponding expectation.

\begin{assumption}[$L^1$ regularity, {\cite[Assumption 12.3.1]{CappeHMM}}]\label{assump_HMM_3}
Assume that we have:
\begin{enumerate}[label=(\roman*)]
\item\label{assump_HMM_3:item1}
 $b^+ := 1\land  \sup_\theta \sup_{x,y} g_\theta(x,y) < \infty$.
 \item\label{assump_HMM_3:item2}
  $\Etrue \vert \log b^-(Y_{\rooot}) \vert < \infty$, 
where $b^-(y) := \inf_\theta \int_{\SpaceX} g_\theta(x,y) \lambda(\drv x)$.
\end{enumerate}
\end{assumption}
Note that $b^-(y) > 0$ for all $y\in\SpaceY$
by \Cref{assump_HMM_1}-\ref{assump_HMM_1:item2}.

\begin{lemme}[Uniform bounds for $\h_{u,k,x}(\theta)$]
	\label{lemma_h_ukx_unif_Cauchy_and_bounded}
Assume that Assumptions \ref{assump_HMM_1}--\ref{assump_HMM_2}
	and \ref{assump_HMM_3}-\ref{assump_HMM_3:item2} hold.
For all vertices $u\in\T$
and all integers $k,k' \in\N^*$,
the following assertions hold true:
\begin{equation}\label{eq_h_ukx_unif_Cauchy}
\sup_{\theta\in\Theta} \sup_{x,x'\in\SpaceX}
	\vert \h_{u,k,x}(\theta) - \h_{u,k',x'}(\theta) \vert
\leq \frac{\rho^{(k\land k') -1} }{1-\rho} ,
\end{equation}
\begin{equation}\label{eq_h_ukx_unif_bounded}
\sup_{\theta\in\Theta} \sup_{k\in\N^*} \sup_{x\in\SpaceX}
	\vert \h_{u,k,x}(\theta) \vert
\leq  \log b^+ 	\lor	\vert \log(\sigma^- b^-(Y_u) ) \vert .
\end{equation}
\end{lemme}

\begin{proof}
{[The proof is a straightforward adaptation of the proof of \cite[Lemma 12.3.2]{CappeHMM} 
	using \Cref{lemme_exp_coupling_HMT} for the coupling.]}
Remind the definition of $\Hterm_{u,k,x}(\theta)$ in \eqref{eq_def_H_ukx}.
Let $k' \geq k \geq 1$,
and write $w = \parent^k(u)$, $w' = \parent^{k'}(u)$.
Then, write:
\begin{align}
\Hterm_{u,k,x}(\theta)
 =  \int_{\SpaceX \times \SpaceX} & \left[ \int_{\SpaceX} g_\theta(x_u,Y_u) q_\theta(x_{\parent(u)}, x_u) \lambda(\drv x_u) \right] \nonumber\\
	&  \times \Prb_\theta( X_{\parent(u)} \in \drv x_{\parent(u)}
			\,\vert\,  Y_{\DeltaR^*(u,k)} , X_w =x_w )
	\times \delta_x(\drv x_w) , \label{eq_H_ukxpi_1}
\end{align}
and using the Markov property at $X_w$, write:
\begin{align}
\Hterm_{u,k',x'}(\theta)
=  \int_{\SpaceX \times \SpaceX} & \left[ \int_{\SpaceX} g_\theta(x_u,Y_u) q_\theta(x_{\parent(u)}, x_u) \lambda(\drv x_u) \right] \nonumber\\
	& \times \Prb_\theta( X_{\parent(u)} \in \drv x_{\parent(u)}
			\,\vert\,  Y_{\DeltaR^*(u,k)} , X_w =x_w ) \nonumber\\
	& \times \Prb_\theta( X_w \in \drv x_w
		\,\vert\,  Y_{\DeltaR^*(u,k')} , X_{w'} =x' ) . \label{eq_H_ukxpi_2}
\end{align}
Applying \Cref{lemme_exp_coupling_HMT}, we get 
(note that the integrands in \eqref{eq_H_ukxpi_1} and \eqref{eq_H_ukxpi_2} are non-negative):
\begin{align}
\vert \Hterm_{u,k,x}(\theta) - \Hterm_{u,k',x'}(\theta) \vert
& \leq  \rho^{k-1} \sup_{x_{\parent(u)}\in\SpaceX} 
		\int_{\SpaceX} g_\theta(x_u,Y_u) q_\theta(x_{\parent(u)}, x_u) \lambda(\drv x_u) \nonumber\\
& \leq \rho^{k-1} \sigma^+   \int_{\SpaceX} g_\theta(x_u,Y_u) \lambda(\drv x_u) . \label{eq_upper_bound_H_ukxpi}
\end{align}
The integral in \eqref{eq_H_ukxpi_1} can be lower bounded giving us:
\begin{equation}\label{eq_H_ukxpi_lower_bound}
\Hterm_{u,k,x}(\theta) \geq \sigma^- \int_{\SpaceX} g_\theta(x_u, Y_u) \lambda(\drv x_u) ,
\end{equation}
where the right hand side is positive by \Cref{assump_HMM_2}-\ref{assump_HMM_2:item2};
and similarly for \eqref{eq_H_ukxpi_2}.
Combining \eqref{eq_upper_bound_H_ukxpi} with \eqref{eq_H_ukxpi_lower_bound},
and with the inequality $\vert \log x - \log y \vert \leq \vert x-y \vert / (x\land y)$, we get the first assertion of the lemma:
\begin{equation*}
\vert \h_{u,k,x}(\theta) - \h_{u,k',x'}(\theta) \vert
\leq \frac{\sigma^+}{\sigma^-} \rho^{k-1} = \frac{\rho^{k-1}}{1-\rho} \cdot
\end{equation*}
Combining \eqref{eq_def_H_ukx} and \eqref{eq_H_ukxpi_lower_bound},
we get that
$
\sigma^- b^-(Y_u)  \leq  \Hterm_{u,k,x}(\theta)  \leq  b^+
$
(remind that $b^-(Y_u) > 0$ by \Cref{assump_HMM_2}-\ref{assump_HMM_2:item2}),
which yields the second assertion of the lemma.
\end{proof}

We are now ready to construct the limit of the functions $\h_{u,k,x}(\theta)$
and state some properties of this limit.
Note that this result is stated for every $u\in\T$, but we will only need it for $u=\rooot$.
Remind that we are in the stationary case, and that the HMT process $(X,Y)$ is defined on $\Tpast$.

\begin{prop}[Properties of the limit function $\h_{u,\infty}(\theta)$]
	\label{prop_existence_stationary_likelihood}
Assume that Assumptions \ref{assump_HMM_0}--\ref{assump_HMM_3} hold.
For every $u\in\T$ and $\theta\in\Theta$,
there exists $\h_{u,\infty}(\theta) \in L^1(\P_\cU \otimes \Ptrue)$
such that for all $x\in\SpaceX$, the sequence $(\h_{u,k,x}(\theta))_{k\in\N}$
converges $\P_\cU \otimes \Ptrue$-\as and in $L^1(\P_\cU \otimes \Ptrue)$ to $\h_{u,\infty}(\theta)$.

Furthermore, this convergence is uniform over $\theta\in\Theta$ and $x\in\SpaceX$, that is,
we have that $\lim_{k\to\infty} \sup_{\theta\in\Theta} \sup_{x\in\SpaceX} \vert \h_{u,k,x}(\theta) - \h_{u,\infty}(\theta)\vert = 0$
$\P_\cU \otimes \Ptrue$-\as and in $L^1(\P_\cU \otimes \Ptrue)$.
\end{prop}

The limit function $\h_{u,\infty}(\theta)$ can be interpreted as 
$\log p_\theta( Y_u \,\vert\, Y_{\DeltaR^*(u,\infty)} ) $,
where $\DeltaR^*(u,\infty) = \{ v\in\Tpast \,:\, v <_\cU u \}$ is a random subset of vertices.
Note that $\h_{u,\infty}(\theta)$ is a function 
of the random set of variables  $(Y_v, v\in \DeltaR(u,\infty))$,
where $\DeltaR(u,\infty) = \DeltaR^*(u,\infty) \cup \{u\}$,
and thus implicitly depend on $\cU$ trough $\DeltaR(u,\infty)$.

\begin{proof}
Fix some $u\in\T$. 
Note that \eqref{eq_h_ukx_unif_Cauchy} shows
that the sequence $( \h_{u,k,x}(\theta) )_{k\in\N}$ is Cauchy uniformly in $\theta$ and $x$,
and thus has $\P_\cU \otimes \Ptrue$-almost surely a limit when $k \to \infty$
which does not depend on $x$; we denote this limit by $\h_{u,\infty}(\theta)$.
Furthermore, we get from \eqref{eq_h_ukx_unif_bounded} that 
$( \h_{u,k,x}(\theta) )_{k\in\N}$ is uniformly bounded in $L^1(\P_\cU\otimes\Ptrue)$,
and thus $\h_{u,\infty}(\theta)$ is in $L^1(\P_\cU\otimes\Ptrue)$ and
the convergence also holds in $L^1(\P_\cU\otimes\Ptrue)$.
Finally, as the bound in \eqref{eq_h_ukx_unif_Cauchy} is uniform in $\theta$ and $x$,
we get that the convergence holds uniform over $\theta$ and $x$
both $\P_\cU \otimes \Ptrue$-almost surely and in $L^1(\P_\cU\otimes\Ptrue)$.
\end{proof}

\subsection{Properties of the contrast function}

As the functions $\h_{\rooot,\infty}(\theta)$ are in $L^1(\Prb_\cU\otimes\Ptrue)$ 
under the assumptions used in \Cref{prop_existence_stationary_likelihood},
we can now define the \emph{contrast function} $\ell$ 
(which is deterministic) 
as: 
\begin{equation}
	\label{eq_def_contrast_function}
\ell(\theta) 
= \Esp_\cU \otimes \Etrue \bigl[ \h_{\rooot,\infty}(\theta) \bigr] ,
\end{equation}
where remind  $\Esp_\cU \otimes \Etrue$ is the expectation corresponding to $\Prb_\cU\otimes\Ptrue$.

We prove under the following $L^2$ regularity assumption the
convergence of the normalized log-likelihood to the contrast function.
Remind that $b^-(y) = \inf_\theta \int_{\SpaceX} g_\theta(x,y) \lambda(\drv x)$.
Also remind that  $\Prb_\theta$ denotes the stationary probability distribution under the parameter $\theta\in\Theta$ of the HMT process $(X,Y)$,
and by $\Esp_\theta$ the corresponding expectation.

\begin{assumption}[$L^2$ regularity]
	\label{assump_HMM_3bis}
Assume that $\Etrue \bigl[ ( \log b^-(Y_{\rooot}) )^2 \bigr] < \infty$.
\end{assumption}

Remind that the log-likelihood $\ell_{n,x}$ is defined in \eqref{eq_def_l_nx_theta_2} on page~\pageref{eq_def_l_nx_theta_2}.

\begin{prop}[Ergodic convergence for the stationary log-likelihood]
	\label{prop_conv_likelihood_to_contrast_func}
Assume that Assumptions~\ref{assump_HMM_0}--\ref{assump_HMM_3bis} hold.
Then, for all $x\in\SpaceX$, the normalized log-likelihood $\vert \T_n \vert^{-1} \ell_{n,x}(\theta)$
converges $\Ptrue$-\as  to the contrast function $\ell(\theta)$ as $n\to \infty$.
\end{prop}

\begin{proof}
Let $\theta\in\Theta$ be some parameter.
Fix some $k\in\N^*$ and $x\in\SpaceX$.
Remind that $\ell_{n,x}(\theta) = \sum_{u\in\T_n} \h_{u,\height{u},x}(\theta)$.
Applying \eqref{eq_h_ukx_unif_Cauchy} for each vertex $u\in\T_n\setminus\T_{k-1}$, we get:
\begin{equation}\label{eq_bound_likelihood_and_approx}
\inv{\vert \T_n \vert}
	\left\vert
		\ell_{n,x}(\theta) - \sum_{u\in\T_n\setminus\T_{k-1}} \h_{u,k,x}(\theta)
	\right\vert
\leq  \frac{\rho^{k-1}}{1-\rho}
	+ \inv{\vert \T_n \vert} \sum_{u\in\T_{k-1}} \vert \h_{u,\height{u},x}(\theta) \vert
.
\end{equation}
Note that by \eqref{eq_h_ukx_unif_bounded}, we have that $\vert \h_{u,\height{u},x}(\theta) \vert < \infty$ $\Ptrue$-\as
	for all $u\in\T\setminus\{ \rooot\}$.
For $u=\rooot$, we have $\h_{\rooot,0,x}(\theta) = \log g_\theta(x,Y_{\rooot})$ which is finite $\Ptrue$-\as by \Cref{assump_HMM_2}-\ref{assump_HMM_2:item3}.

For a vertex $u$ in $\T\setminus \T_{k-1}$, 
let $v_u\in\G_k$ be the unique vertex that satisfies \eqref{eq_def_shape_subtree} (on page \pageref{eq_def_shape_subtree}),
then we have:
\begin{equation}\label{eq_equal_h_ukx_up_to_shape}
\h_{u,k,x}(\theta; Y_{\Delta(u,k)}=y_{\Delta(u,k)}) 
	=  \h_{v_u,k,x}(\theta; Y_{\Delta(v_u)}=y_{\Delta(u,k)}) .
\end{equation}
Moreover, using \eqref{eq_h_ukx_unif_bounded} together with \Cref{assump_HMM_3bis},
	we get for every $u\in\T\setminus \T_{k-1}$ that the random variable $\h_{u,k,x}(\theta; Y_{\Delta(u,k)})$ is in $L^2(\Ptrue)$.
Hence, applying \Cref{lemma_ergodic_convergence}
to the collection of neighborhood-shape-dependent functions $( \h_{v,k,x}(\theta; Y_{\Delta(v)}=\cdot) )_{v\in\G_k}$
(remind that indexing functions with $\G_k$ or with $\ShapeSetValues_k$ is equivalent by \eqref{eq_def_set_possible_shapes}),
and using \eqref{eq_equal_h_ukx_up_to_shape} and \eqref{eq_equality_expectation_root_and_U_k} (in \Cref{rem_rerooting_delta}),
we get:
\begin{equation}
	\label{eq_conv_ergo_likelihood}
\vert \T_n \vert^{-1} \sum_{u\in\T_n\setminus\T_{k-1}} \h_{u,k,x}(\theta)
	\underset{n\to\infty}{\longrightarrow}
	\Esp_\cU \otimes \Etrue \bigl[ \h_{\rooot,k,x}(\theta) \bigr]
\qquad \text{$\Ptrue$-\as (and in $L^2(\Ptrue))$}  .
\end{equation}

Using \eqref{eq_h_ukx_unif_Cauchy} with \Cref{prop_existence_stationary_likelihood},
	we get:
\begin{equation*}
\bigl\vert \Esp_\cU \otimes \Etrue \bigl[ \h_{\rooot,k,x}(\theta) \bigr] 
		- \Esp_\cU \otimes \Etrue \bigl[ \h_{\rooot,\infty}(\theta) \bigr] \bigr\vert
		\leq  \frac{\rho^{k-1}}{1-\rho} \cdot
\end{equation*}
Thus, combining this bound with \eqref{eq_bound_likelihood_and_approx} and \eqref{eq_conv_ergo_likelihood},
we get $\Ptrue$-\as that:
\begin{equation*}
\limsup_{n\to\infty} \Bigl\vert
		\vert \T_n \vert^{-1} \ell_{n,x}(\theta) 	- \Esp_\cU \otimes \Etrue \bigl[ \h_{\rooot,\infty}(\theta) \bigr]
	\Bigr\vert
\leq 2\, \frac{\rho^{k-1}}{1-\rho} \cdot
\end{equation*}
As the left hand side does not depend on $k$,
letting $k\to\infty$, we get that $\vert \T_n \vert^{-1} \ell_{n,x}(\theta)$ converges $\Ptrue$-\as
to $\ell(\theta)$ as $n\to\infty$.
This concludes the proof.
\end{proof}

We are going to prove that this convergence holds uniformly in $\theta$.
First, we need to prove that the contrast function is continuous
and has a unique global maximum at $\thetaTrue$.
In order to get those results,
we need a natural continuity assumption on the transition functions.

\begin{assumption}[Continuity, {\cite[Assumption 12.3.5]{CappeHMM}}]\label{assump_HMM_4}
For all $(x,x')\in \SpaceX\times\SpaceX$ and $y\in\SpaceY$,
the functions $\theta \mapsto q_\theta(x',x)$ and $\theta \mapsto g_\theta(x,y)$ defined on $\Theta\subset \R^\dimTheta$
are continuous.
\end{assumption}

We denote by $\norm{\cdot}$ the euclidean norm on $\R^{\dimTheta}$.

\begin{prop}[$\ell$ is continuous]\label{prop_l_continuous_complete}
Assume that Assumptions~\ref{assump_HMM_0}--\ref{assump_HMM_3}
and \ref{assump_HMM_4} hold.
Then, for any $n\in\N$ and $x\in\SpaceX$, the log-likelihood function $\theta \mapsto \ell_{n,x}(\theta)$ is $\Ptrue$-\as
continuous on $\Theta$.

Moreover, for any $\theta\in\Theta$, we have:
\begin{equation*} 
\Esp_\cU \otimes \Etrue \left[ 	\sup_{\theta' \in\Theta : \Vert \theta - \theta' \Vert \leq \delta}
		\bigl\vert \h_{\rooot,\infty}(\theta') - \h_{\rooot,\infty}(\theta) \bigr\vert 	\right]
	\to 0 \quad \text{ as } \delta \to 0 ,
\end{equation*}
and the contrast function $\theta \mapsto \ell(\theta)$ is continuous on $\Theta$.

\end{prop}

\begin{proof}
This proof is a straightforward adaptation from the proof of \cite[Proposition 12.3.6]{CappeHMM}.

Recall that $\h_{\rooot,\infty}(\theta)$ is the limit of $\h_{\rooot,k,x}(\theta)$ as $k\to\infty$.
We first prove that, for every $x\in\SpaceX$ and $k\geq 0$,
	$\h_{\rooot,k,x}(\theta)$ is a continuous function of $\theta$,
	and then use this to show continuity of the limit.
Recall from \eqref{eq_def_H_ukx} the second equality defining $\Hterm_{u,k,x}(\theta)$,
which we remind for convenience
for any $u\in\T$ and $x\in\SpaceX$:
\begin{equation*}
\Hterm_{u,k,x}(\theta)
= \frac{p_\theta(Y_{\DeltaR(u,k)} \,\vert\, X_{\parent^k(u)} = x)}
	{p_\theta(Y_{\DeltaR^*(u,k)} \,\vert\, X_{\parent^k(u)} = x)} \cdot 
\end{equation*}

Recall from \eqref{eq_def_p_theta_Y_Delta} the definition of $p_\theta(Y_{\Delta} \,\vert\, X_{\parent^k(u)} = x)$
where here the possibly random subtree $\Delta$ is either $\DeltaR(u,k)$ or $\DeltaR^*(u,k)$.
First note that the integrand in \eqref{eq_def_p_theta_Y_Delta} 
is by assumption continuous \wrt $\theta$
and upper bounded by $(1\lor \sigma^+ b^+)^{\vert \Delta \vert}$.
Thus, dominated convergence shows that $p_\theta(Y_{\Delta} \,\vert\, X_{\parent^k(u)} = x)$ 
is continuous
\wrt to $\theta$ (remind that $\lambda$, defined in \Cref{assump_HMM_1}, is finite).
Moreover, note that 
$p_\theta(Y_{\DeltaR^*(u,k)} \,\vert\, X_{\parent^k(u)} = x)$ is lower bounded by
$\prod_{v\in \DeltaR^*(u,k)\setminus\{\parent^k(u)\}} \sigma^- b^-(Y_v) $ 
which is positive
$\Prb_\cU\otimes \Ptrue$-\as (by \Cref{assump_HMM_2}).
Thus, $\Hterm_{u,k,x}(\theta)$ and $\h_{u,k,x}(\theta) = \log \Hterm_{u,k,x}(\theta)$ (remind \eqref{eq_def_h_ukx})
are continuous \wrt $\theta$ $\Prb_\cU\otimes \Ptrue$-\as as well.
Hence, using \eqref{eq_def_l_nx_theta}, for all $n\in\N$ and $x\in\SpaceX$,
we get that $\ell_{n,x}(\theta)$ is also continuous \wrt $\theta$ $\Ptrue$-\as

Remind from Proposition~\ref{prop_existence_stationary_likelihood} that
$(\h_{u,k,x}(\theta))_{k\in\N}$ converges to $\h_{u,\infty}(\theta)$ uniformly in $\theta$
$\Prb_\cU\otimes\Ptrue$-\as
Thus, the function $\theta \mapsto \h_{u,\infty}(\theta)$ is continuous $\Prb_\cU\otimes \Ptrue$-\as
Using the uniform bound~\eqref{eq_h_ukx_unif_bounded}, \Cref{assump_HMM_3}-\ref{assump_HMM_3:item2} 
and dominated convergence, we obtain the first part of the proposition.

We deduce the second part from the first one, as:
\begin{align*}
\sup_{\theta' \in\Theta  : \Vert \theta' - \theta \Vert \leq \delta}  \vert \ell(\theta') - \ell(\theta) \vert
&= \sup_{\theta' \in\Theta : \Vert \theta' - \theta \Vert \leq \delta}  
		\bigl\vert \Esp_\cU \otimes \Etrue[ \h_{\rooot,\infty}(\theta') - \h_{\rooot,\infty}(\theta) ] \bigr\vert \\
& \leq \Esp_\cU \otimes \Etrue \left[ 	\sup_{\theta' \in\Theta : \Vert \theta - \theta' \Vert \leq \delta}
		\vert \h_{\rooot,\infty}(\theta') - \h_{\rooot,\infty}(\theta) \vert 	\right] .
\end{align*}
This concludes the proof.
\end{proof}

We are now ready to state and prove that the convergence to the contrast function $\ell$ 
	holds uniformly in $\theta$.

\begin{prop}[Uniform convergence to $\ell$]
	\label{prop_unif_cv_l}
Assume that Assumptions~\ref{assump_HMM_0}--\ref{assump_HMM_4} hold and $\Theta$ is compact.
Then, we have:
\begin{equation*}
\lim_{n\to\infty} \sup_{\theta\in\Theta} \Bigl\vert \vert \T_n \vert^{-1} \ell_{n,x}(\theta) - \ell(\theta) \Bigr\vert = 0
\quad \Ptrue\text{-\as}
\end{equation*}
\end{prop}

\begin{proof}{[We mimic the proof of {\cite[Proposition 12.3.7]{CappeHMM}}.]}
As $\Theta$ is compact,
it is sufficient to prove that for every $\theta\in\Theta$:
\begin{equation}\label{eq_l_cv_unif_local_like_complete}
\limsup_{\delta\to 0} \limsup_{n\to \infty} \sup_{\theta' \in\Theta : \Vert \theta' - \theta \Vert \leq \delta}
	\Bigl\vert \vert  \T_n \vert^{-1} \ell_{n,x}(\theta') - \ell(\theta) \Bigr\vert 
	= 0 \quad \Ptrue\text{-\as}
\end{equation}
As this claim is not proven in the proof of {\cite[Proposition 12.3.7]{CappeHMM}},
we give a short proof.
Indeed, assume that \eqref{eq_l_cv_unif_local_like_complete} holds for all $\theta\in\Theta$.
Let $\eps>0$.
By \Cref{prop_l_continuous_complete}, the function $\ell$ is continuous, 
	and thus uniformly continuous as $\Theta$ is compact.
In particular, there exists $\delta>0$ such that for all $\theta,\theta'\in\Theta$,
	we have that $\Vert \theta - \theta' \Vert \leq \delta$ implies $\vert \ell(\theta) - \ell(\theta') \vert \leq \eps$.
For every $\theta\in\Theta$, let $\delta_\theta \in (0,\delta)$ be such that 
	$\limsup_{n\to \infty} \sup_{\theta' \in\Theta : \Vert \theta' - \theta \Vert \leq \delta_\theta}
	\Bigl\vert \vert  \T_n \vert^{-1} \ell_{n,x}(\theta') - \ell(\theta) \Bigr\vert 
	< \eps$.
As $\cup_{\theta\in\Theta} \{ \theta' : \Vert \theta' - \theta \Vert \leq \delta_\theta \}$ is an open cover of $\Theta$
and as $\Theta$ is compact, there exists a finite subset $\{ \theta_j : 1\leq j \leq m \}$ of $\Theta$ with $m\geq 1$
such that $\Theta = \cup_{j=1}^m \{ \theta' : \Vert \theta' - \theta_j \Vert \leq \delta_{\theta_j} \}$.
Note that for $n$ large enough, for all $1\leq j \leq m$, we have that
$\sup_{\theta' \in\Theta : \Vert \theta' - \theta_j \Vert \leq \delta_{\theta_j}}
	\Bigl\vert \vert  \T_n \vert^{-1} \ell_{n,x}(\theta') - \ell(\theta_j) \Bigr\vert 
	< \eps$.
Thus, for $n$ large enough, we have:
\begin{align*}
\sup_{\theta\in\Theta} \Bigl\vert \vert \T_n \vert^{-1} \ell_{n,x}(\theta) - \ell(\theta) \Bigr\vert
 \leq \eps + \max_{1\leq j \leq m} \sup_{\theta' \in\Theta : \Vert \theta' - \theta_j \Vert \leq \delta_{\theta_j}} 
	\Bigl\vert \vert  \T_n \vert^{-1} \ell_{n,x}(\theta') - \ell(\theta_j) \Bigr\vert 
\leq 2 \eps.
\end{align*}
This being true for all $\eps>0$, we get that the statement in the proposition holds.
\medskip

We now prove \eqref{eq_l_cv_unif_local_like_complete}.
Fix some $\theta\in\Theta$.
Remind that by \Cref{prop_conv_likelihood_to_contrast_func},
we have that $\lim_{n\to\infty} \vert  \T_n \vert^{-1} \ell_{n}(\theta) = \ell(\theta)$ $\Ptrue$-\as
Using this fact, we get:
\begin{align}
 \limsup_{n\to \infty}  & \sup_{\theta' \in\Theta :  \Vert \theta' - \theta \Vert \leq \delta}
	 \Bigl\vert \vert  \T_n \vert^{-1} \ell_{n,x}(\theta') - \ell(\theta) \Bigr\vert \nonumber\\
 & =  \limsup_{n\to \infty} \sup_{\theta' \in\Theta : \Vert \theta' - \theta \Vert \leq \delta}
	\Bigl\vert \vert  \T_n \vert^{-1} \ell_{n,x}(\theta') - \vert  \T_n \vert^{-1} \ell_{n,x}(\theta) \Bigr\vert  .
	\label{eq_local_uniform_conv_ell}
\end{align}
Using  \eqref{eq_bound_likelihood_and_approx}, for any $k\geq 1$,  we get that \eqref{eq_local_uniform_conv_ell} is $\Ptrue$-\as bounded by:
\begin{multline}
2 \limsup_{n\to \infty}  \sup_{\theta' \in\Theta : \Vert \theta' - \theta \Vert \leq \delta} \vert  \T_n \vert^{-1} 
	\Bigl\vert  \ell_{n,x}(\theta') - \sum_{u\in\T_n\setminus\T_{k-1}} \h_{u,k,x}(\theta') \Bigr\vert \\
+ \limsup_{n\to \infty}
	\vert  \T_n \vert^{-1} 	\sum_{u\in\T_n\setminus\T_{k-1}}
		\sup_{\theta' \in\Theta : \Vert \theta' - \theta \Vert \leq \delta}
		\Bigl\vert	\h_{u,k,x}(\theta') 
			-  \h_{u,k,x}(\theta)	\Bigr\vert \\
\begin{aligned}
& \leq \	2 \frac{\rho^{k-1}}{1-\rho}		+ 
	 \limsup_{n\to \infty}
	\vert  \T_n \vert^{-1} 	\sum_{u\in\T_n\setminus\T_{k-1}}
		\sup_{\theta' \in\Theta : \Vert \theta' - \theta \Vert \leq \delta}
		\Bigl\vert	\h_{u,k,x}(\theta') 
			-  \h_{u,k,x}(\theta)	\Bigr\vert \\
& = \  	2 \frac{\rho^{k-1}}{1-\rho}		+
	  \Esp_\cU \otimes \Etrue \left[ 
	\sup_{\theta' \in\Theta : \Vert \theta' - \theta \Vert \leq \delta}
		\Bigl\vert	\h_{\rooot,k,x}(\theta') 
			-  \h_{\rooot,k,x}(\theta)	\Bigr\vert	 \right] \\
& \leq \  	4 \frac{\rho^{k-1}}{1-\rho}		+
	  \Esp_\cU \otimes \Etrue \left[  
	\sup_{\theta' \in\Theta : \Vert \theta' - \theta \Vert \leq \delta}
		\Bigl\vert	\h_{\rooot,\infty}(\theta') 
			-  \h_{\rooot,\infty}(\theta)	\Bigr\vert	 \right] ,
			\label{eq_upper_bound_local_unif_cv_ell}
\end{aligned}
\end{multline}
where
we used \Cref{lemma_ergodic_convergence} for ergodic convergence 
	(with $L^2(\Ptrue)$ boundedness given by \eqref{eq_h_ukx_unif_bounded})
	in the equality,
and we used \eqref{eq_h_ukx_unif_Cauchy} (with \Cref{prop_existence_stationary_likelihood}) in the second inequality.
Then, letting $k\to\infty$ in the upper bound of \eqref{eq_upper_bound_local_unif_cv_ell}
	(note that \eqref{eq_local_uniform_conv_ell} does not depend on $k$),
	and then letting $\delta\to0$ with \Cref{prop_l_continuous_complete},
we get that \eqref{eq_l_cv_unif_local_like_complete} holds.
This concludes the proof.
\end{proof}

\begin{remark}[Uniform convergence for the log-likelihood with general initial condition]
Let $\nu$ be a probability distribution on $\SpaceX$ such that
	$\sup_\theta \vert \int g_\theta(x, Y_{\rooot}) \nu(\drv x) \vert$ is finite $\Ptrue$-\as
The uniform convergence of $\vert \T_n \vert^{-1} \ell_{n,x}(\theta)$ to $\ell(\theta)$
still holds when modifying the definition of the log-likelihood $\ell_{n,x}(\theta)$ of the HMT
to replace the Dirac mass $\delta_x$ by $\nu$ for the distribution
of the root hidden variable $X_{\rooot}$.
When $\nu$ is the stationary distribution $\pi_\theta$ associated to $q_\theta$,
uniform convergence holds without this extra regularity assumption
by conditioning on the state of the root's parent $X_{\parent(\rooot)}$ instead
(which allows to replace $\h_{\rooot,0,x}(\theta) = g_\theta(x,Y_\rooot)$ in \eqref{eq_bound_likelihood_and_approx}
	by $\h_{\rooot,1,\nu}(\theta) := \log \int_{\SpaceX} \Hterm_{\rooot,1,x}(\theta) \, \nu(\drv x)$
	for which $\sup_\theta \vert \h_{\rooot,1,\nu}(\theta) \vert$
	 is finite by an immediate adaptation of \eqref{eq_h_ukx_unif_bounded}).
\end{remark}

\subsection{Identifiability and strong consistency}

In this subsection, we prove the strong consistency of the MLE.
We must first study the identifiability of the parameter of the HMT model.
We start with a definition of equivalent parameters.

\begin{defin}[Equivalent parameters]
We say that two parameters $\theta, \theta'\in \Theta$ are \emph{equivalent}
if they define the same distribution for the process $(Y_u , u\in\T)$,
\ie $\Prb_\theta(Y\in \cdot) = \Prb_{\theta'}(Y\in \cdot)$.
\end{defin}

Note that by Kolmogorov's extension theorem,
$\theta$ and $\theta'$ are equivalent if and only if they define the same law
on every finite tree $\T_n$, \ie for $(Y_u , u\in\T_n)$.

The following proposition characterizes global maxima of the contrast function $\ell$.

\begin{prop}[Global maxima of the contrast function $\ell$]
\label{prop_global_max_l}
Assume that Assumptions~\ref{assump_HMM_0}--\ref{assump_HMM_3bis} hold.
Then a parameter $\theta\in\Theta$ is a global maximum of $\ell$
if and only if $\theta$ is equivalent to $\thetaTrue$.
\end{prop}

We get as an immediate corollary that $\thetaTrue$ is a global maximum of $\ell$.

The proof of \Cref{prop_global_max_l},
which is postponed to the end of this section,
is an adaptation of the proof of \cite[Theorem 12.4.2]{CappeHMM}.
This adaptation comes from the difference of topology between the tree and the line.

\medskip

Remind that the log-likelihood function $\theta \mapsto \ell_{n,x}(\theta)$ is continuous $\Ptrue$-\as
	under Assumptions~\ref{assump_HMM_0}-\ref{assump_HMM_3} and~\ref{assump_HMM_4}.
Thus, when we further assume that $\Theta$ is compact,
we get that the argmax set $\argmax_{\theta\in\Theta} \ell_{n,x}(\theta)$ is non-empty.
The maximum likelihood estimator (MLE) is then defined as
the maximizer over $\Theta$ of the log-likelihood $\ell_{n,x}$, 
that is as the following random variable (which depends on $Y_{\T_n}$):
\begin{equation}
	\label{eq_def_MLE_hat_theta_nx}
\hat \theta_{n,x} = \hat \theta_{n,x}(Y_{\T_n}) \in \argmax_{\theta\in\Theta} \ell_{n,x}(\theta).
\end{equation}
Note that the argmax set in \eqref{eq_def_MLE_hat_theta_nx} is not necessarily unique,
in which case we select one parameter $\theta$ from the argmax set in a measurable manner
(which is possible, see \cite[Proposition~7.33]{bertsekasStochasticOptimalControl1996}).

We are now ready to prove the following theorem that states the strong consistency of the MLE for the HMT model
in the stationary case.

\begin{theo}[Strong consistency of the MLE]
	\label{thm_Strong_consistency_MLE}
Assume that Assumptions~\ref{assump_HMM_0}--\ref{assump_HMM_4} hold,
the contrast function $\ell$ has a unique maximum
	(which is then located at $\thetaTrue\in\Theta$ by \Cref{prop_global_max_l}) 
	and $\Theta$ is compact.
Then, for any $x\in \SpaceX$, the MLE
	$\hat \theta_{n,x}$ (defined in \eqref{eq_def_MLE_hat_theta_nx})
	converges $\Ptrue$-\as as $n\to\infty$
	to the true parameter $\thetaTrue\in\Theta$,
\ie the MLE is strongly consistent.
\end{theo}

\begin{proof}
{[The proof is a straightforward adaptation of an argument for HMMs in \cite[Section~12.1]{CappeHMM},
	which itself adapts an argument that goes back to \cite{waldNoteConsistencyMaximum1949}.]}

By definition of $\hat \theta_n$, we have that 
	$\ell_{n,x}(\hat \theta_n) \geq \ell_{n,x}(\theta)$
	for every $\theta\in\Theta$.
As the contrast function $\ell$ has a unique maximum located at $\thetaTrue$,
we have that $\ell(\thetaTrue) \geq \ell(\theta)$ for every $\theta\in\Theta$,
and in particular, $\ell(\thetaTrue) \geq \ell(\hat\theta_n)$ for every $n\in\N$.
Combining those two bounds, we get that:
\begin{align*}
0 & \leq  \ell(\thetaTrue) - \ell(\hat\theta_n) \\
& \leq   \ell(\thetaTrue) - \vert \T_n\vert^{-1} \ell_{n,x}(\thetaTrue)
+ \vert \T_n\vert^{-1} \ell_{n,x}(\thetaTrue)  -  \vert \T_n\vert^{-1} \ell_{n,x}(\hat\theta_n)
+  \vert \T_n\vert^{-1} \ell_{n,x}(\hat\theta_n)  -  \ell(\hat\theta_n) \\
& \leq 2 \sup_{\theta\in\Theta} \Bigl\vert \ell(\theta) 
		-  \vert \T_n\vert^{-1} \ell_{n,x}(\theta) \Bigr\vert ,
\end{align*}
where the upper bound in the last line goes to zero $\Ptrue$-\as
as $n\to \infty$ by Proposition~\ref{prop_unif_cv_l} as $\Theta$ is compact.
Hence, we get that $\ell(\hat\theta_n) \to \ell(\thetaTrue)$ $\Ptrue$-\as as $n\to \infty$.
Consequently, as $\ell$ is continuous (by Proposition~\ref{prop_l_continuous_complete})
and has a unique global maximum located at $\thetaTrue$,
	and as $\Theta$ is compact,
we get that $\hat \theta_n$ converges $\Ptrue$-\as to $\thetaTrue$ as $n\to \infty$.
\end{proof}

We now prove \Cref{prop_global_max_l}.

\begin{proof}[Proof of \Cref{prop_global_max_l}]
Remind that $\h_{u,k,x}(\theta)$ is defined in \eqref{eq_def_h_ukx}.
By definition of $\ell(\theta)$ (see \eqref{eq_def_contrast_function}) 
and using the $L^1(\P_{\cU}\otimes\Ptrue)$ convergence of $(\h_{\rooot,k,x}(\theta))_{k\in\N}$ to $\h_{\rooot,\infty}(\theta)$
(remind Proposition~\ref{prop_existence_stationary_likelihood}),
we have:
\begin{align*}
\ell(\thetaTrue) - \ell(\theta)
& = \Esp_\cU \otimes \Etrue \bigl[  \h_{\rooot,\infty}(\thetaTrue) - \h_{\rooot,\infty}(\theta) \bigr] \\
& =  \Esp_\cU \otimes \Etrue \bigl[  \lim_{k\to\infty} \bigl( \h_{\rooot,k,x}(\thetaTrue) - \h_{\rooot,k,x}(\theta) \bigr)  \bigr] \\
& = \lim_{k\to\infty} \Esp_\cU \otimes \Etrue \bigl[    \h_{\rooot,k,x}(\thetaTrue) - \h_{\rooot,k,x}(\theta)   \bigr] .
\end{align*}
Remind that $\Hterm_{u,k,x}(\theta)$ is defined in \eqref{eq_def_H_ukx}.
Then, write:
\begin{equation}\label{eq_KL_expression_H_ukx}
\Etrue[ \h_{\rooot,k,x}(\thetaTrue) - \h_{\rooot,k,x}(\theta) ]
 = \Esp_\cU \left[ \Etrue \left[  \Etrue \left[
		\log \frac{\Hterm_{\rooot,k,x}(\thetaTrue)}{\Hterm_{\rooot,k,x}(\theta)}
 \,\middle\vert\, Y_{\DeltaR^*(\rooot,k)}, X_{\parent^k(\rooot)}=x \right]  \right] \right] ,
\end{equation}
where the inner expectation is on $Y_{\rooot}$
conditionally on $X_{\parent^k(\rooot)}=x$
and $Y_{\DeltaR^*(\rooot,k)}$
(and thus also implicitly on $\cU$ as $\DeltaR^*(\rooot,k) = \DeltaR^*_\cU(\rooot,k)$).
Recalling from \eqref{eq_def_H_ukx} that $\Hterm_{\rooot,k,x}(\theta)$ is the conditional density of
	$Y_{\rooot}$ given $Y_{\DeltaR^*(\rooot,k)}$ 
		and $X_{\parent^k(\rooot)}=x$,
we see that the inner (conditional) expectation in the right hand side is a
Kullback-Leibler divergence and thus is non-negative.
Hence, the two outer expectations and the limit $\ell(\thetaTrue) - \ell(\theta)$ as $k\to\infty$
	are non-negative as well, and thus $\thetaTrue$ is a global maximum of $\ell$.
	
Remark that if $\theta$ is equivalent to $\thetaTrue$, then as the process $(Y_u, u\in\T)$ is stationary
	and has same law under both parameters, 
	the roles of $\thetaTrue$ and $\theta$ can be swapped in the argument above,
	and thus we get $\ell(\theta)=\ell(\thetaTrue)$.
Hence, any $\theta$ equivalent to $\thetaTrue$ is a global maximum of $\ell$.
\medskip

We now turn to prove that any global maximum $\theta\in\Theta$ of $\ell$ is equivalent to $\thetaTrue$.

Remind that we use the letter $p$
to denote (possibly conditional) densities of random variables,
\eg $p_\theta(Y_u \,\vert\, Y_{\DeltaR^*(u,k)}, X_{\parent^k(u)}=x)$
denotes the \emph{conditional density} (\wrt the measure $\mu$ defined in \Cref{assump_HMM_1}-\ref{assump_HMM_1:item1}) 
under the parameter $\theta$ of $Y_u$
	conditionally on $Y_{\DeltaR^*(u,k)}$ and $X_{\parent^k(u)}=x$.
Note that $\Prb_\theta(Y_u \in \cdot \,\vert\, Y_{\DeltaR^*(u,k)}, X_{\parent^k(u)}=x)$
denotes the \emph{distribution}
under the parameter $\theta$  of $Y_u$
	conditionally on $Y_{\DeltaR^*(u,k)}$ and $X_{\parent^k(u)}=x$.

We first need a variant of the convergence in \Cref{prop_existence_stationary_likelihood}
where instead of considering one vertex $u$ as in $\h_{u,k,x}(\theta)$
we consider a whole subtree $\Tpast(u,m)$ for any $m\geq 1$
(this can be seen as a convergence by block).
To this end, we need to define an analogue of the breadth-first-search order relation $<$ on $\Tpast$ 
	for subtree blocks of the form $\Tpast(u,m)$.
Let $m\geq 1$ be fixed.
For $u,v\in\Tpast$ with $\height{u} \equiv \height{v}  \mod m+1$,
we write $\Tpast(u,m) < \Tpast(v,m)$ if $u < v$
(informally, ‘‘$\Tpast(u,m)$ is above or on the left of $\Tpast(v,m)$’’).
Note that the modulo congruence is there to insure the collection of block subtrees $\Tpast(u,m)$
	with $\height{u} \equiv \height{\rooot}  \mod m+1$
	form a partition (\ie a cover with non-overlapping subsets) of $\Tpast$
	(this still holds for any other class of congruence $\mod m+1$).
Also note that in this congruence we have $m+1$ and not $m$,
because any subtree $\Tpast(u,m)$ (\eg $\T_m = \Tpast(\rooot,m)$) 
spans over $m+1$ different generations (remind that $\height{\rooot}=0$).
We can then define the analogue of the subset $\DeltaR^*(u,k)$ for subtree blocks,
that is, for all $u\in\Tpast$ and $k\in\N$, we define: 

\begin{figure}[t]
\centering
\includegraphics[height=5cm]{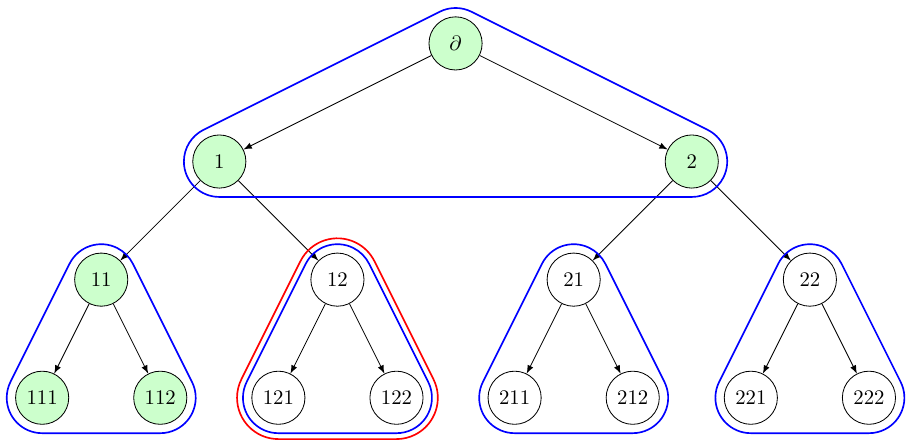}
\caption{Illustration of the ‘‘past’’ subtree $\DeltaR^*(\T(u,m),k)$
	of the block subtree $\T(u,m)$
	for $m=1$, $u=12$ and $k=1$.
	The block subtrees are circled with blue lines,
	and the block subtree $\T(12,1)$ is circled a
		second time with a red line.
	The vertices in green are those in $\DeltaR^*(\T(12,1),1)$.
	Note the difference with $\Delta(u',k')$, e.g.\
	vertex $111$ is in $\DeltaR^*(\T(12,1),1)$ 
		but not in $\DeltaR^*(12,2)$,
	and vertex $21$ is in $\DeltaR^*(121,3)$
		but not in $\DeltaR^*(\T(12,1),1)$.
	}
\label{fig_illustration_Delta_T_u_k}
\end{figure}

\begin{equation*}
\DeltaR^*(\T(u,m),k) = \bigcup \bigl\{ T(v,m) : v \in \DeltaR^*(u, k(m+1)) \text{ such that }  \height{v}\equiv \height{u} \mod m+1 \bigr\}.
\end{equation*}
See \Cref{fig_illustration_Delta_T_u_k} for an illustration
of the ‘‘past’’ subtree $\DeltaR^*(\T(u,m),k)$ of the block subtree $\T(u,m)$.
(Informally, ‘‘the subset $\DeltaR^*(\T(u,m),k)$ is the union of the subtree blocks (with height $m$)
		above and on the left of $\T(u,m)$ up to $k$ block generations’’.
Note that we will not need to understand in details the geometry of the subset $\DeltaR^*(\T(u,m),k)$,
	we only need to remember that all its vertices are upstream of the edge $(\parent(u),u)$,
	and we will then use the Markov property.)
Remind that $\Tpast(\rooot,m) = \T_m$.
Then, a straightforward adaptation of Lemma~\ref{lemma_h_ukx_unif_Cauchy_and_bounded}, 
	and Propositions~\ref{prop_existence_stationary_likelihood} and~\ref{prop_conv_likelihood_to_contrast_func}
	to a decomposition of the log-likelihood into non-overlapping subtrees of height $m$ instead of single vertices
	(see \Cref{section_convergence_Y_triangles} for detailed proofs of those adaptations)
	give us for all $\theta\in\Theta$, $x\in\SpaceX$ and $m\in\N^*$:
\begin{equation}
	\label{eq_lim_Y_Tm_cond_past}
\lim_{k\to\infty} \Esp_\cU \otimes \Etrue \left[
	\log p_\theta(Y_{\T_m} \,\vert\, Y_{\DeltaR^*(\T_m,k)} , X_{\parent^{k(m+1)}(\rooot)} =x )
	\right] 
= \vert \T_m \vert \, \ell(\theta) .
\end{equation}

Let $\cU^+ = (\cU_{(j)})_{1\leq j < \infty}$ be a sequence of independent random variables with Bernoulli distribution of parameter $1/2$
	(note that $\cU^+$ can be seen as a random forward spine),
	which is independent of $\cU$ and of the HMT process $(X,Y)$.
For all $n\in\N^*$, define the random vertex $U_n$ as the unique vertex in $\G_n$
	whose path from $\rooot$ is encoded by $\cU_{(1:n)}$ in Neveu's notation.
For all $n\in\N$, define the deterministic vertex $U_{-n} = \parent^n(\rooot)$.
Note that $\rooot = U_0$ and that $U_{n-1}$ is the parent vertex of $U_n$ for all $n\in\Z$.
Moreover, using a similar argument as in \Cref{rem_rerooting_delta}, 
note that for any $m,k\in\N$,
the sequence of random shapes $(\Shape(\DeltaR(\Tpast(U_n,m),k)))_{n\in\Z}$ is stationary.

Now, pick $\theta\in\Theta$ such that $\ell(\theta) = \ell(\thetaTrue)$.
Thus for any positive integer $n < m$, we have:
\begin{align}
0 & = \vert \T_m \vert \, (\ell(\thetaTrue) - \ell(\theta) ) \nonumber\\
& = \lim_{k\to\infty} \Esp_\cU \otimes \Etrue \left[   \log \frac
	{p_{\thetaTrue}(Y_{\T_m} \,\vert\, Y_{\DeltaR^*(\T(\rooot,m),k)} , X_{\parent^{k(m+1)}(\rooot)} =x)}
	{p_\theta(Y_{\T_m} \,\vert\, Y_{\DeltaR^*(\T(\rooot,m),k)} , X_{\parent^{k(m+1)}(\rooot)} =x)}
	\right]  \nonumber\\
& = \lim_{k\to\infty} \left\{ \Esp_\cU \otimes \Etrue \left[  \log \frac
{p_{\thetaTrue}(Y_{\T(U_{m-n},n)} \,\vert\, Y_{\DeltaR^*(\T(\rooot,m),k)} , X_{\parent^{k(m+1)}(\rooot)} =x)}
{p_\theta(Y_{\T(U_{m-n},n)} \,\vert\, Y_{\DeltaR^*(\T(\rooot,m),k)} , X_{\parent^{k(m+1)}(\rooot)} =x)}
	\right]  \right. \nonumber\\
& \qquad\qquad \left. + \Esp_\cU \otimes \Etrue \left[   \log \frac
{p_{\thetaTrue}(Y_{\T_m \setminus\T(U_{m-n},n)}
	\,\vert\, Y_{\DeltaR^*(\T(\rooot,m),k) \cup \T(U_{m-n},n)} , X_{\parent^{k(m+1)}(\rooot)} =x)}
{p_\theta(Y_{\T_m \setminus\T(U_{m-n},n)}
	\,\vert\, Y_{\DeltaR^*(\T(\rooot,m),k) \cup \T(U_{m-n},n)} , X_{\parent^{k(m+1)}(\rooot)} =x)}
	\right]  \right\} \nonumber\\
& \geq \limsup_{k\to\infty} \Esp_\cU \otimes \Etrue \left[  \log \frac
{p_{\thetaTrue}(Y_{\T(U_{m-n},n)} \,\vert\, Y_{\DeltaR^*(\T(\rooot,m),k)} , X_{\parent^{k(m+1)}(\rooot)} =x)}
{p_\theta(Y_{\T(U_{m-n},n)} \,\vert\, Y_{\DeltaR^*(\T(\rooot,m),k)} , X_{\parent^{k(m+1)}(\rooot)} =x)}
	\right] \nonumber\\
& = \limsup_{k\to\infty} \Esp_\cU \otimes \Etrue \left[  \log \frac
{p_{\thetaTrue}(Y_{\T_n}
	\,\vert\, Y_{\DeltaR^*(\T(U_{-m+n},m),k)} , X_{\parent^{k(m+1)}(U_{-m+n})} =x)}
{p_\theta(Y_{\T_n}
	\,\vert\, Y_{\DeltaR^*(\T(U_{-m+n},m),k)} , X_{\parent^{k(m+1)}(U_{-m+n})} =x)}
	 \right], 		\label{eq_bound_KL_div_p_theta}
\end{align}
where the inequality follows by 
	noting that the second term is non-negative as an expectation of a
	(conditional) Kullback-Leibler divergence
	(using an argument similar as for \eqref{eq_KL_expression_H_ukx} above),
and the last equality follows by using stationarity
	of the HMT process $(X,Y)$, of the spinal process $(U_n)_{n\in\Z}$,
	and of the shape process $(\Shape(\DeltaR(\Tpast(U_n,m),k)))_{n\in\Z}$.
Note that the term in the lower bound is also non-negative as an expectation of a
	(conditional) Kullback-Leibler divergence.

Let $n\in\N$ be fixed.	
Now, we define for all $\theta\in\Theta$ and $m,k\in\N^*$:
\begin{align}
 W_{m,k}(\theta) & = \log p_\theta(Y_{\T_n}
	\,\vert\, Y_{\DeltaR^*(\T(U_{-m},m+n),k)}, X_{\parent^{k(m+n+1)}(U_{-m})} = x), \nonumber\\
\text{ and } \qquad W(\theta) & = \log p_\theta(Y_{\T_n} ). \label{eq_def_W_mk_and_W}
\end{align}
Note that $\log p_\theta(Y_{\T_n} )$ is well defined
using an integral expression similar to \eqref{eq_def_exple_p_theta}
along Assumptions~\ref{assump_HMM_1} and~\ref{assump_HMM_2} 
and the comment on $\pi_\theta$ after \Cref{lemma_Q_unif_geom_ergodic}.
From \eqref{eq_bound_KL_div_p_theta}, we deduce that
(where $m$ in \eqref{eq_def_W_mk_and_W} and \eqref{eq_identifiability_tree_gap_variables} below
corresponds to $m-n$ in \eqref{eq_bound_KL_div_p_theta}):
\begin{equation}
\forall m\in\N^*,\qquad
\lim_{k\to\infty} \Esp_\cU \otimes \Etrue \bigl[   
		W_{m,k}(\thetaTrue) - W_{m,k}(\theta)
	\bigr] 
= 0
.	\label{eq_identifiability_tree_gap_variables}
\end{equation}
Hence, we have managed to insert a gap between the variables $(Y_v, v\in\T_n)$
whose density we examine and the variables $(Y_v, v\in \DeltaR^*(\T(U_{-m},m+n),k))$
and $X_{\parent^{k(m+n+1)}(U_{-m})}$ that appear in the conditioning.
Remark that the following fact illustrates the gap between the variables:
	if $u\in\T_n$ and $v\in\DeltaR^*(\T(U_{-m},m+n),k)$,
	then the most recent common ancestor $u\land v$ of $u$ and $v$
	has height $\height{u\land v} < -m$,
	that is $u\land v$ is an ancestor of $U_{-m}$.
	See \Cref{fig_tree_gap_illustration} for a graphical illustration of this gap.

\begin{figure}[t]
\centering
\includegraphics[width=8cm]{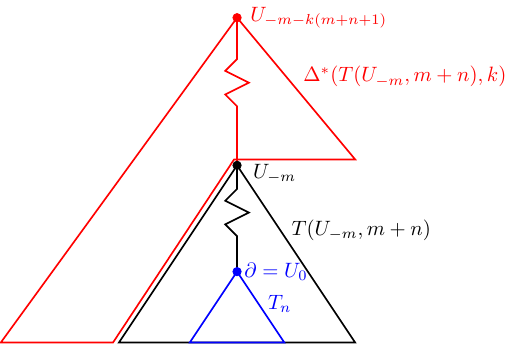}
\caption{Illustration of the gap in \eqref{eq_identifiability_tree_gap_variables} 
	between the variables $(Y_v, v\in\T_n)$
	(bottom triangle in blue)
	and the variables $(Y_v, v\in \DeltaR^*(\T(U_{-m},m+n),k))$
	and $X_{\parent^{k(m++n1)}(U_{-m})}$ that appear in the conditioning (top partial triangle in red).
	Note that the two groups of variables are separated by the path from $U_{-m-1} = \parent(U_{-m})$ to $\partial =U_0$,
	which is of length $m+1$.
	}
\label{fig_tree_gap_illustration}
\end{figure}

The idea is now to let this gap tend to infinity to show that in the limit the
	conditioning has no effect.
Our next goal is thus to prove that:
\begin{align}
\lim_{m\to\infty} \sup_{k\in\N} \Bigl\vert
	\Esp_\cU \otimes \Etrue \bigl[ 
		W_{m,k}(\thetaTrue) - W_{m,k}(\theta)
	\bigr] 
-   \Etrue \bigl[  
	W(\thetaTrue) - W(\theta) \bigr] 
\Bigr\vert  & = 0 . \label{eq_proof_identifiability_infinite_gap_mixing}
\end{align}
Combining \eqref{eq_proof_identifiability_infinite_gap_mixing}
with \eqref{eq_identifiability_tree_gap_variables}, it is clear that if $\theta\in\Theta$ is such that $\ell(\theta) = \ell(\thetaTrue)$,
then we have that $\Etrue[ \log[ p_{\thetaTrue}(Y_{\T_n}) / p_{\theta}(Y_{\T_n}) ]] = 0$,
that is, the Kullback-Leibler divergence between the $\vert \T_n \vert$-dimensional densities
$p_{\thetaTrue}(Y_{\T_n})$ and $p_{\theta}(Y_{\T_n})$ is null.
This implies, by the information inequality, that these densities coincide except
on a set with $\mu^{\otimes \vert \T_n \vert}$-measure zero,
so that the $\T_n$-marginal laws of $\Ptrue$ and $\Prb_\theta$ agree.
Because $n$ was arbitrary, we find that $\thetaTrue$ and $\theta$ are equivalent.
\medskip

What remains to do to complete the proof is thus to prove \eqref{eq_proof_identifiability_infinite_gap_mixing}.
Remind the definition of $W_{m,k}(\theta)$ and $W(\theta)$ in \eqref{eq_def_W_mk_and_W}.
Obviously, it is enough to prove that for all $\theta\in\Theta$, we have:
\begin{equation}\label{eq_proof_identifiability_uniform_limit_W_km}
\lim_{m\to\infty} \Esp_\cU \otimes \Etrue \left[  
	\sup_{k\in\N} \vert W_{m,k}(\theta) - W(\theta) \vert
\right]  = 0 .
\end{equation}
Let $\theta\in\Theta$ be fixed.
To prove that \eqref{eq_proof_identifiability_uniform_limit_W_km} holds for $\theta$, 
we write (remind the discussion above on the gap between variables):
\begin{align*}
\exp( W_{m,k}(\theta))  = & \
p_\theta(Y_{\T_n} 
	 \,\vert\, Y_{\DeltaR^*(\T(U_{-m},m+n),k)}, X_{\parent^{k(m+n+1)}(U_{-m})} = x) \\
 = & \int_{\SpaceX\times\SpaceX} p_\theta( Y_{\T_n} \,\vert\, X_{\parent(\rooot)} = x_{\parent(\rooot)}) 
 		\, Q_\theta^{m-1}(x_{U_{-m}}; \drv x_{\parent(\rooot)}) \\
	& \qquad \times \Prb_\theta( X_{U_{-m}} \in \drv x_{U_{-m}} 
		\,\vert\, Y_{\DeltaR^*(\T(U_{-m},m+n),k)}, X_{\parent^{k(m+n+1)}(U_{-m})} = x),
\end{align*}
and
\begin{align*}
\exp(W(\theta)) =
p_\theta(Y_{\T_n} )
= \int_{\SpaceX\times\SpaceX} p_\theta( Y_{\T_n} \,\vert \,X_{\parent(\rooot)} = x_{\parent(\rooot)}) 
		\, Q_\theta^{m-1}(x_{U_{-m}}; \drv x_{\parent(\rooot)})
	\pi_\theta(\drv x_{U_{-m}}) ,
\end{align*}
where remind from \Cref{lemma_Q_unif_geom_ergodic} that $\pi_\theta$ is the stationary distribution of the process $(X_u , u\in\Tpast)$
with transition kernel $Q_\theta$ (that is, under the distribution $\Prb_\theta$).
Note that we have the upper bound
(remind that $b^+$ is defined in \Cref{assump_HMM_3}):
\begin{align}
p_\theta( Y_{\T_n} \,\vert\, X_{\parent(\rooot)} = x_{\parent(\rooot)})
 = \int_{\SpaceX^{\T_n}} 
	\prod_{u\in\T_n} q_\theta(x_{\parent(u)},x_u) g_\theta(x_u, Y_u) \, \lambda(\drv x_u) 
 \leq  (b^+)^{\vert \T_n\vert} .
 \label{eq_p_Y_Tn_upper_bound}
\end{align}
Thus, as Assumptions~\ref{assump_HMM_1} and~\ref{assump_HMM_2} hold, 
applying the uniform geometric bound from \Cref{lemma_Q_unif_geom_ergodic}
to the Markov chain $(X_{U_{j}})_{j \in\Z}$ with transition kernel $Q_\theta$,
we obtain $\Ptrue$-\as:
\begin{equation}
\sup_{k\in\N} \Bigl\vert p_\theta(Y_{\T_n}
	\,\vert\, Y_{\DeltaR^*(\T(U_{-m},m+n),k)}, X_{\parent^{k(m+n+1)}(U_{-m})} = x)
	- p_\theta(Y_{\T_n} ) \Bigr\vert 
\leq (b^+)^{\vert \T_n\vert} (1 - \sigma^-)^{m-1}  .
\label{eq_proof_indentifiability_upper_bound_densities}
\end{equation}
Moreover, as we have the lower bound:
\begin{align}
p_\theta( Y_{\T_n} \,\vert\, X_{\parent(\rooot)} = x_{\parent(\rooot)})
& = \int_{\SpaceX^{\T_n}} 
	\prod_{u\in\T_n} q_\theta(x_{\parent(u)},x_u) g_\theta(x_u, Y_u) \lambda(\drv x_u) \nonumber\\
& \geq \prod_{u \in\T_n} \sigma^- b^-(Y_u),
	\label{eq_p_Y_Tn_lower_bound}
\end{align}
this implies that $p_\theta(Y_{\T_n} 
	\,\vert\, Y_{\DeltaR^*(\T(U_{-m},m+n),k)}, X_{\parent^{k(m+n+1)}(U_{-m})} = x)$
and $p_\theta(Y_{\T_n} )$ both obey the same lower bound.
This lower bound combined with the observation that 
$b^-(Y_u) > 0$ for all $u\in\T_n$ 
(which follows from Assumption~\ref{assump_HMM_2}-\ref{assump_HMM_2:item2}),
and the bound $\vert \log(x) - \log(y) \vert \leq \vert x-y \vert / x\land y$,
\eqref{eq_proof_indentifiability_upper_bound_densities} shows that:
\begin{equation*}
\Prb_\cU\otimes \Ptrue\text{-\as} \qquad
\lim_{m\to\infty}	\sup_{k\in\N} \vert W_{m,k}(\theta) - W(\theta) \vert
= 0 .
\end{equation*}
Using the bounds \eqref{eq_p_Y_Tn_upper_bound} and \eqref{eq_p_Y_Tn_lower_bound}
with Assumptions~\ref{assump_HMM_2} and ~\ref{assump_HMM_3}-\ref{assump_HMM_3:item2}, we get:
\begin{equation*}
\Esp_\cU \otimes \Etrue \left[  
	\sup_{m\in\N^*} \sup_{k\in\N} \vert W_{m,k}(\theta) \vert
 \right] < \infty .
\end{equation*}
Hence, as this expectation is finite,
\eqref{eq_proof_identifiability_uniform_limit_W_km} follows from dominated convergence.
This concludes the proof.
\end{proof}

\section{Asymptotic normality of the MLE}
	\label{section_asymptotic_normality}

In this section, we prove that the MLE for the HMT has asymptotic normal fluctuations.
We keep the assumptions used in \Cref{section_strong_consistency}.
This section is divided in two parts: we first prove the asymptotic normality of the score,
	and then we prove a strong law of large numbers for the observed information.
Together with the strong consistency, those two results imply the asymptotic normality of the MLE.
\medskip

We will need the following assumption for existence and regularity of the gradient and Hessian
of the transition kernels.
Remind that  $\Prb_\theta$ denotes the stationary probability distribution under the parameter $\theta\in\Theta$ of the HMT process $(X,Y)$,
and by $\Esp_\theta$ the corresponding expectation.
Also remind that the measures $\lambda$ and $\mu$ are defined in \Cref{assump_HMM_1}.
We denote by $\nabla_\theta$ and $\nabla_\theta^2$, respectively, the gradient and Hessian operator \wrt the parameter $\theta\in\Theta$.
With a slight abuse of notations, we denote by $\norm{\cdot}$ the euclidean norm 
on either $\R^{\dimTheta}$ or $\R^{\dimTheta\times \dimTheta}$.

\begin{assumption}[Regularity of the gradient, {\cite[Assumption 12.5.1]{CappeHMM}}]\label{assump_HMM_grad_1}
There exists an open (for the trace topology on $\Theta\subset\R^{\dimTheta}$)  neighborhood 
$\ThetaNeighborhood = \{ \theta\in\Theta : \Vert \theta - \thetaTrue \Vert < \delta_0 \}$
of $\thetaTrue$ such that the following hold.
\begin{enumerate}[label=(\roman*)]

\item\label{assump_HMM_grad_1:item1}
For all $(x,x')\in \SpaceX\times\SpaceX$ and all $y\in\SpaceY$,
the functions $\theta \mapsto q_\theta(x,x')$ and $\theta \mapsto g_\theta(x,y)$
are twice continuously differentiable on $\ThetaNeighborhood$.

\item\label{assump_HMM_grad_1:item2}
We have:
\begin{equation*}
\sup_{\theta\in\ThetaNeighborhood} \sup_{x,x'} \Vert \nabla_\theta \log q_\theta(x,x') \Vert < \infty,
\qquad \text{and} \qquad
\sup_{\theta\in\ThetaNeighborhood} \sup_{x,x'} \Vert \nabla_\theta^2 \log q_\theta(x,x') \Vert < \infty .
\end{equation*}

\item\label{assump_HMM_grad_1:item3}
We have:
\begin{equation*}
\Etrue \left[
	\sup_{\theta\in\ThetaNeighborhood} \sup_{x} \Vert \nabla_\theta \log g_\theta(x,Y_{\rooot}) \Vert^2 
	\right] < \infty,
\qquad\! \text{and} \qquad\!
\Etrue \left[
	\sup_{\theta\in\ThetaNeighborhood} \sup_{x} \Vert \nabla_\theta^2 \log g_\theta(x,Y_{\rooot}) \Vert 
	\right] < \infty .
\end{equation*}

\item\label{assump_HMM_grad_1:item4}
For $\mu$-almost all $y\in\SpaceY$, there exists a function $f_y : \SpaceX \to \R_+$
in $L^1(\lambda)$ such that we have $\sup_{\theta\in\ThetaNeighborhood} g_\theta(x,y) \leq f_y(x)$.

\item\label{assump_HMM_grad_1:item5}
For $\lambda$-almost all $x\in\SpaceX$, there exist functions $f_x^1 : \SpaceX \to \R_+$
and $f_x^2 : \SpaceX \to \R_+$ in $L^1(\mu)$ such that
$\sup_{\theta\in\ThetaNeighborhood} \Vert \nabla_\theta g_\theta(x,y) \Vert \leq f_x^1(y)$
and $\sup_{\theta\in\ThetaNeighborhood} \Vert \nabla_\theta^2 g_\theta(x,y) \Vert \leq f_x^2(y)$.

\end{enumerate}
\end{assumption}

These assumptions insures that the log-likelihood $\ell_{n,x}$
is twice continuously differentiable,
and that the \emph{score function} $\nabla_\theta \ell_{n,x}(\theta)$ 
and the \emph{observed information} $-\nabla_\theta^2 \ell_{n,x}(\theta)$ 
exist and are in $L^2(\Ptrue)$ and $L^1(\Ptrue)$, respectively.

\subsection{Asymptotic normality of the score}
	\label{subsection_asymptotic_normality_score}

In this subsection, we prove the asymptotic normality of the score under the true parameter $\thetaTrue$.
Note that the score function can be written for all $n\in\N$ and $x\in\SpaceX$ as:
\begin{align*}
\nabla_\theta \ell_{n,x}(\theta)
& = \sum_{u\in\T_n} \nabla_\theta \log \left[
		\int g_\theta(X_u,Y_u)\, \Prb_\theta(X_u\in\drv x_u 
			\,\vert\, Y_{\Delta^*(u,\height{u})}, X_{\rooot} = x)
	\right] ,
\end{align*}
and $\nabla_\theta \ell_{n,x}(\theta)$ is implicitly a function of $Y_{\T_n}$.

\subsubsection{Decomposition of the score as a sum of increments}

Remind that for $u\in\T$, the subtrees $\Delta^*(u,k)$ and $\Delta(u,k)$ are defined in \Cref{subsection_ergodic_theorem} for $k\leq \height{u}$
(with $\Delta^*(u) = \Delta^*(u,\height{u})$ and $\Delta(u) = \Delta(u,\height{u})$)
and the random subtrees $\Delta^*(u,k)$ and $\Delta(u,k)$ are defined in \Cref{subsection_decomp_likelihood_ordering} for $k > \height{u}$.
Also remind that we use the letter $p$ to denote (possibly conditional) probability density,
and in particular remind that $p_\theta( Y_{\Delta} \,\vert\, X_{\rooot} = x_{\rooot} ) $
for any subtree $\Delta\subset\T$ with root $\rooot$
is defined in \eqref{eq_def_p_theta_Y_Delta} in \Cref{subsubsection_log_likelihood_sum_increments}
(with the convention $p_\theta( Y_{\emptyset} \,\vert\, X_{\rooot} = x_{\rooot} ) = 1$). Using \eqref{eq_def_H_ukx} and \eqref{eq_def_h_ukx}  in \Cref{subsubsection_log_likelihood_sum_increments},
note that for any $u\in\T$  and $x\in\SpaceX$, we have:
\begin{equation}\label{eq_h_ukx_as_diff_log_p}
 \h_{u,\height{u},x}(\theta) =
	\log p_\theta( Y_{\Delta(u)} \,\vert\, X_{\rooot} = x ) 
		- \log p_\theta( Y_{\Delta^*(u)} \,\vert\, X_{\rooot} = x ) .
\end{equation}

Using elementary computation along with permutations of the integral and the gradient operator
which are valid under \Cref{assump_HMM_grad_1}
(note that this result is also known as \emph{Fisher identity}, see \cite[Proposition 10.1.6]{CappeHMM}),
we get:
\begin{align}
\nabla_\theta \log p_\theta&( Y_{\Delta(u)} \,\vert\, X_{\rooot} = x ) \nonumber\\
& = \nabla_\theta \log g_\theta(x, Y_{\rooot})
	+ \Esp_\theta \left[ 
		\sum_{v\in \Delta(u)\setminus\{\rooot\} }  
			\phi_\theta( X_{\parent(v)}, X_{v}, Y_{v} )
		\,\middle\vert\, Y_{\Delta(u)}, X_{\rooot}=x
	\right] ,  \label{eq_Fisher_indentity_log_p}
\end{align}
where 
\begin{equation}\label{eq_def_phi}
\phi_\theta(x',x,y) = \nabla_\theta \log\bigl[ q_\theta(x',x) g_\theta(x,y) \bigr] .
\end{equation}
Note that under \Cref{assump_HMM_grad_1}, 
	$\norm{ \phi_\theta( X_{\parent(v)}, X_{v}, Y_{v} ) }$ is upper bounded by 
		a square integrable function of $Y_v$
		(which does not depend on $\theta$), 	and $\phi_\theta( X_{\parent(v)}, X_{v}, Y_{v} ) $ is thus integrable conditionally on $Y_{\Delta(u)}$ and $X_{\rooot}=x$.
Also note that $\nabla_\theta \log g_\theta(x, Y_{\rooot})$ is $\Ptrue$-\as finite 
by \Cref{assump_HMM_grad_1}-\ref{assump_HMM_grad_1:item3}.
	
Combining those two equations with \eqref{eq_l_nx_as_sum_h_ukx} in \Cref{subsubsection_log_likelihood_sum_increments}, 
we can express the score function as:
\begin{align*}
\nabla_\theta \ell_{n,x} (\theta)
= \nabla_\theta \log g_\theta(x, Y_{\rooot})
	&+ \sum_{u\in\T_n^*} \Esp_\theta \left[ 
		\sum_{\Delta(u)\setminus\{\rooot\}}  
			\phi_\theta( X_{\parent(v)}, X_{v}, Y_{v} )
		\middle\vert Y_{\Delta(u)}, X_{\rooot}=x
		\right] \\
	&-  \sum_{u\in\T_n^*} \Esp_\theta \left[ 
		\sum_{\Delta^*(u)\setminus\{\rooot\}}  
			\phi_\theta( X_{\parent(v)}, X_{v}, Y_{v} )
		\middle\vert Y_{\Delta^*(u)}, X_{\rooot}=x
		\right] .
\end{align*}

We want to express the score function $\nabla_\theta \ell_{n,x} (\theta)$ as a sum of increments (conditional scores)
in order to apply a convergence result for the normalized score.
To this end, define for every $u\in\T$, $k\in\N$ and $x\in\SpaceX$,
the function $\dot \h_{u,k,x}(\theta)$
by $\dot \h_{u,0,x}(\theta) = \nabla_\theta \log g_\theta(x,Y_u)$ if $k=0$,
and otherwise by:
\begin{align*}
\dot \h_{u,k,x}(\theta)
& =  \Esp_\theta \left[ 
		\sum_{v\in\DeltaR(u,k)\setminus\{\parent^k(u)\}}  
			\phi_\theta( X_{\parent(v)}, X_{v}, Y_{v} )
		\,\middle\vert\, Y_{\DeltaR(u,k)}, X_{\parent^k(u)}=x
		\right] \\
& \quad - \Esp_\theta \left[ 
		\sum_{v\in\DeltaR^*(u,k)\setminus\{\parent^k(u)\}}  
			\phi_\theta( X_{\parent(v)}, X_{v}, Y_{v} )
		\,\middle\vert\, Y_{\DeltaR^*(u,k)}, X_{\parent^k(u)}=x
		\right] .
\end{align*}
Note that $\dot \h_{u,k,x}(\theta)$ 
	is well defined as $\DeltaR(u,k)$ is finite 
	and as 
	$\phi_\theta( X_{\parent(v)}, X_{v}, Y_{v} )$ 	is integrable conditionally on $Y_{\DeltaR^*(u,k)}$ and $X_{\parent^k(u)}=x$
	under \Cref{assump_HMM_grad_1}
	(see the comment after \eqref{eq_def_phi}).
Also note that $\dot \h_{u,k,x}(\theta)$ is the gradient \wrt $\theta$ of $\h_{u,k,x}(\theta)$ defined in \eqref{eq_def_h_ukx}
	(see \eqref{eq_h_ukx_as_diff_log_p} and \eqref{eq_Fisher_indentity_log_p} for the case $k=\height{u}$).
Furthermore, note that $\dot\h_{u,k,x}(\theta)$ is a function of $Y_{\DeltaR(u,k)}$ 
with an implicit dependence on $\cU$ through $\DeltaR(u,k)$,
and that $\dot\h_{u,k,x}(\theta)$ does not depend on $\cU$ is $k\leq \height{u}$.

Using the increment functions $\dot \h_{u,k,x}(\theta)$, we can rewrite the score function as:
\begin{equation}\label{eq_equal_score_sum_increments}
\nabla_\theta \ell_{n,x}(\theta)
= \sum_{u\in\T_n} \dot \h_{u,\height{u},x}(\theta).
\end{equation}

\subsubsection{Construction of score increments with infinite past}

Our goal is to let $k\to\infty$ as before to get a limit function $\dot \h_{u,\infty}$.
We now proceed to construct $\dot \h_{u,\infty}$.
First, we rewrite $\dot \h_{u,k,x}(\theta)$ (which is in $L^2(\Prb_\cU\otimes \Ptrue)$ by \Cref{assump_HMM_grad_1}), as:
\begin{align}
\dot \h_{u,k,x}(\theta)
= & \ \Esp_\theta [  \phi_\theta( X_{\parent(u)}, X_{u}, Y_{u} )
		\,\vert\, Y_{\DeltaR(u,k)}, X_{\parent^k(u)}=x  ]  
			\nonumber\\
  & + \sum_{v\in\DeltaR^*(u,k)\setminus\{\parent^k(u)\}}  \Bigl(
	\Esp_\theta [ \phi_\theta( X_{\parent(v)}, X_{v}, Y_{v} )
		\,\vert\, Y_{\DeltaR(u,k)}, X_{\parent^k(u)}=x ] 
		\Bigr.	\nonumber\\
& \qquad \qquad \qquad \qquad	 \qquad \Bigl.
	- \Esp_\theta [ \phi_\theta( X_{\parent(v)}, X_{v}, Y_{v} )
		\,\vert\, Y_{\DeltaR^*(u,k)}, X_{\parent^k(u)}=x ] \Bigr) .
			\label{eq_sum_incr_score_rearranged}
\end{align}

We will need
the following lemma that states a coupling bound that works ‘‘backwards in time’’, or rather
	along the path between a vertex $v$ and the newly observed vertex $u$.
Remind from \eqref{eq_inclusion_Delta_Tpast} on page~\pageref{eq_inclusion_Delta_Tpast}
that $\DeltaR(u,k)$ is a random subtree of the deterministic subtree $\Tpast(\parent^k(u),k)$.

\begin{lemme}[Total variation bound ‘‘backwards in time’’]
	\label{lemma_backward_mixing_bound}
Assume that Assumptions~\ref{assump_HMM_1}--\ref{assump_HMM_2} hold.
Let $k\in\N^*$, $x\in\SpaceX$ and $u\in \T$,
and let $v\in \Tpast(\parent^k(u),k) \setminus\{u\}$. Then, we have: 
\begin{align*}
\bigl\Vert  \Prb_\theta (  X_{v} \in \cdot
		\,\vert\, Y_{\DeltaR(u,k)}, X_{\parent^k(u)}=x ) 
 -  \Prb_\theta (  X_{v} \in \cdot
		\,\vert\, Y_{\DeltaR^*(u,k)}, X_{\parent^k(u)}=x )
	\bigr\Vert_{\text{TV}} 
 \leq  \rho^{d(u,v)-1}.
\end{align*}
\end{lemme}

The proof of \Cref{lemma_backward_mixing_bound}, which is postponed to \Cref{appendix_proof_backward_coupling_lemma},
relies on a ‘‘backward in time’’ bound from $\parent(u)$ to $u\land v$
and then a ‘‘forward in time’’ bound from $u\land v$ to $v$,
and using the initial distributions $\Prb_\theta(X_{\parent(u)}\in\cdot \,\vert\, Y_{\Delta}, X_{\parent^k(u)}=x )$
with $\Delta$ equal to $\DeltaR(u,k)$ and $\DeltaR^*(u,k)$, respectively.
Note that this proof is similar to the proofs for \Cref{lemme_exp_coupling_HMT} and \cite[Proposition~12.5.4]{CappeHMM}.

\medskip

The following lemma gives an $L^2$-bound on the difference
between $\dot \h_{u,k,x}(\theta)$ and $\dot \h_{u,k',x'}(\theta)$
with a geometric decay.
As we will reuse this result later with different functions,
we state a more general version.

Note that the condition $\rho < 1/\sqrt{2}$ on the mixing rate $\rho$ of the HMT process $(X,Y)$
is due to the coupling bounds and the grouping of terms used in the proof of \Cref{lemma_score_incr_L2_bound}
(the upper bounds at the end of the proof only add a constant multiplicative factor).
See the discussion in \Cref{rem_rho_smaller_than_2}.

\begin{lemme}\label{lemma_score_incr_L2_bound}
Assume that Assumptions~\ref{assump_HMM_0}--\ref{assump_HMM_3}.
Further assume that $\rho < 1/\sqrt{2}$.

Let $\Theta_0$ be a closed ball in $\Theta$, 
and let $\psi$ be a Borel function 
from $\Theta_0 \times \SpaceX^2 \times \SpaceY$ to $\R^{\dimTheta}$ for some $\dimTheta\in\N$
such that for all $x,x'\in\SpaceX$ and $y\in\SpaceY$, 
$\theta \mapsto \psi(\theta, x, x', y) = \psi_\theta(x,x',y)$ is a continuous function on $\Theta_0$.
Furthermore, assume that there exists $b\in[1,+\infty)$ such that:
\begin{equation*}
\Etrue\left[
		\sup_{\theta\in\Theta_0} \sup_{x,x'\in\SpaceX} \Vert \psi_\theta(x,x',Y_{\rooot}) \Vert^b
	\right] < \infty.
\end{equation*}
Let $\xi_{u,k,x}(\theta)$ be defined as in \eqref{eq_sum_incr_score_rearranged}
	(with $\dot\h_{u,k,x}(\theta)$ and $\phi_\theta$ replaced by $\xi_{u,k,x}(\theta)$ and $\psi_\theta$, respectively),
and note that it is in $L^b(\Prb_\cU \otimes \Ptrue)$.
Then, there exists a finite constant $C<\infty$ such that for all $u\in\T$ 
and $k' \geq k \geq 1$, we have:
\begin{multline*}
\left( \Esp_\cU \otimes \Etrue  \left[ \sup_{\theta\in\Theta_0}  \sup_{x,x' \in \SpaceX} \norm{ \xi_{u,k,x}(\theta)  - \xi_{u,k',x'}(\theta) }^b \right] \right)^{1/b} \\
 \leq C \left( \Etrue \left[ 
		\sup_{\theta\in\Theta_0} \sup_{x,x' \in \SpaceX} \norm{\psi_\theta(x,x',Y_{\rooot})}^b
	\right] \right)^{1/b}
	\, k \, \bigl(\max(\rho, 2\rho^2)\bigr)^{k/2} .
\end{multline*}
\end{lemme}

As a consequence of \Cref{lemma_score_incr_L2_bound},
for all $u\in\T$ and $x\in\SpaceX$, the sequence of function $(\xi_{u,k,x}(\theta))_{n\in\N}$
converges in $L^b(\Prb_\cU \otimes \Ptrue)$ to some limit function $\xi_{u,\infty}(\theta)$
which does not depend on $x$.
Moreover, the bound in \Cref{lemma_score_incr_L2_bound} still holds when
$\xi_{u,k',x'}(\theta)$ is replaced by $\xi_{u,\infty}(\theta)$.

For the particular choice of $\psi_\theta = \phi_\theta$, 
under Assumptions-\ref{assump_HMM_0}-\ref{assump_HMM_3} and~\ref{assump_HMM_grad_1},
for all $u\in\T$, we denote by $\dot\h_{u,\infty}(\theta)$ the limit function of the sequence $(\dot\h_{u,k,x}(\theta))_{n\in\N}$
(for all $x\in\SpaceX$) which is in $L^2(\Prb_\cU \otimes \Ptrue)$.

\medskip

As an immediate corollary of \Cref{lemma_score_incr_L2_bound},
there exists a finite constant $C'<\infty$ such that for all $x\in\SpaceX$, $u\in\T^*$ and $k \geq 1$, we have that
$( \Esp_\cU \otimes \Etrue[ \sup_{x\in\SpaceX} \norm{ \xi_{u,k,x}(\theta) }^b ] )^{1/b} \leq \Etrue[ \sup_{x\in\SpaceX} \norm{ \xi_{u,1,x}(\theta) }^b ] )^{1/b} + C' < \infty$,
(note that by stationarity, $\Etrue[ \sup_{x\in\SpaceX} \norm{ \xi_{u,1,x}(\theta) }^b ] )^{1/b} = \Etrue[ \sup_{x\in\SpaceX} \norm{ \xi_{v,1,x}(\theta) }^b ] )^{1/b}$
	for any other $v\in\T^*$).
Hence, we get:
\begin{equation}\label{eq_unif_bound_dot_hukx}
\sup_{\theta\in\ThetaNeighborhood}\sup_{u\in\T} \sup_{k\geq 1} \
\left( \Esp_\cU \otimes \Etrue \left[ \sup_{x\in\SpaceX} \norm{ \xi_{u,k,x}(\theta) }^b \right] \right)^{1/b}
< \infty.
\end{equation}

\begin{proof}
We mimic the scheme of the proof of \cite[Lemma 12.5.3]{CappeHMM}.

Let $u\in\T$ and $k'\geq k \geq 1$ be fixed.
The idea of the proof is to match, for each vertex index $v$ of the sums expressing
$\xi_{u,k,x}(\theta)$ and $\xi_{u,\infty}(\theta)$,
pairs of terms that are close.
To be more precise, we match:
\begin{enumerate}
\item
For $v$ close to $u$, \begin{equation*}
\Esp_\theta [ \psi_\theta( X_{\parent(v)}, X_{v}, Y_{v} )
		\,\vert\, Y_{\DeltaR(u,k)}, X_{\parent^k(u)}=x ] 
\end{equation*}
and
\begin{equation*}
\Esp_\theta [ \psi_\theta( X_{\parent(v)}, X_{v}, Y_{v} )
		\,\vert\, Y_{\DeltaR(u,k')}, X_{\parent^{k'}(u)}=x' ], \end{equation*}
and similarly for the corresponding terms with $\DeltaR(u,k)$ and $\DeltaR(u,k')$ replaced by
$\DeltaR^*(u,k)$ and $\DeltaR^*(u,k')$, respectively;
\item
For $v$ far from $u$,
\begin{equation*}
\Esp_\theta [ \psi_\theta( X_{\parent(v)}, X_{v}, Y_{v} )
		\,\vert\, Y_{\DeltaR(u,k)}, X_{\parent^k(u)}=x ] 
\end{equation*}
and
\begin{equation*}
\Esp_\theta [ \psi_\theta( X_{\parent(v)}, X_{v}, Y_{v} )
		\,\vert\, Y_{\DeltaR^*(u,k)}, X_{\parent^k(u)}=x ] ,
\end{equation*}
and similarly for the corresponding terms with $k$ and $x$ replaced by $k'$ and $x'$, respectively.
\end{enumerate}

Remind from \eqref{eq_inclusion_Delta_Tpast} on page~\pageref{eq_inclusion_Delta_Tpast} 
that $\DeltaR(u,k) \subset \Tpast(\parent^k(u),k)$
and that the subtree $\DeltaR(u,k)$ is random
while the subtree $\Tpast(\parent^k(u),k)$ is deterministic.
Let $(x,x')\in \SpaceX\times\SpaceX$ and let $v\in\Tpast(\parent^k(u),k)\setminus\{ \parent^k(u) \}$,
which implies that $\parent(v)\in\Delta(u,k)$.

We start with the first kind of matches.
Using the Markov property (remind \eqref{eq_illustration_Markov_prop}), we have:
\begin{align}
\Bigl\Vert \Esp_\theta [ \psi_\theta( X_{\parent(v)}, X_{v}, &Y_{v} )
		\,\vert\, Y_{\DeltaR(u,k)}, X_{\parent^k(u)}=x ] 
	- \Esp_\theta [ \psi_\theta( X_{\parent(v)}, X_{v}, Y_{v} )
		\,\vert\, Y_{\DeltaR(u,k')}, X_{\parent^{k'}(u)}=x' ] 
	\Bigr\Vert \nonumber\\
&= \biggl\Vert \int_{\SpaceX^3}   
	\psi_\theta( x_{\parent(v)}, x_{v}, Y_{v} ) \, 
		\Prb_\theta(X_v \in \drv x_v \,\vert\, Y_{\DeltaR(u,k)\cap\T(v)}, X_{\parent(v)}=x_{\parent(v)}) 
			\biggr. \nonumber\\
		& \qquad \qquad \qquad \left.
	\times \Prb_\theta( X_{\parent(v)} \in \drv x_{\parent(v)}
		\,\vert\, Y_{\DeltaR(u,k)}, 
					X_{\parent^k(u)}=x_{\parent^k(u)} )
			\right. \nonumber\\
		& \qquad \qquad \qquad \biggl.
	\times \bigl[ \delta_x(\drv x_{\parent^k(u)})
		-  \Prb_\theta( X_{\parent^k(u)} \in \drv x_{\parent^k(u)}
			\,\vert\, Y_{\DeltaR(u,k')}, 
					X_{\parent^{k'}(u)}=x'  ) \bigr]
	\biggr\Vert \nonumber\\
& \leq 2 \sup_{x_1,x_2\in\SpaceX} \norm{\psi_\theta(x_1,x_2,Y_v)} \,  \rho^{d(v,\parent^k(u)) -1} ,
			\label{eq_score_incr_forward_mixing}
\end{align}
where the inequality is obtained using Lemma~\ref{lemme_exp_coupling_HMT}
(note that $d(\parent(v),\parent^k(u)) = d(v,\parent^k(u)) -1$).
Note that this upper bound is \as finite as $\sup_{x_1,x_2\in\SpaceX} \norm{\phi_\theta(x_1,x_2,Y_v)}$
is in $L^b(\Prb_\cU \otimes \Ptrue)$ by assumption 
(remind that the HMT process $(X,Y)$ is stationary by \Cref{assump_HMM_0}).
For $v\neq u$, note that this bound remains valid if $\DeltaR(u,k)$ and $\DeltaR(u,k')$
are replaced by $\DeltaR^*(u,k)$ and $\DeltaR^*(u,k')$, respectively.
Obviously, this bound is small if $v$ is far away from $\parent^k(u)$
(remind that $k$ is fixed).

\medskip

We now give a bound for the second kind of matches.
Assume that $v\neq u$.
If $v$ is not an ancestor of $u$ (then $\dgr(u,v) = \dgr(u,\parent(v)) +1$), 
using the Markov property  (remind \eqref{eq_illustration_Markov_prop}) and \Cref{lemma_backward_mixing_bound},
we get:
\begin{align*}
\bigl\Vert & \Esp_\theta [  \psi_\theta( X_{\parent(v)}, X_{v}, Y_{v} )
		\,\vert\,  Y_{\DeltaR(u,k)},\ X_{\parent^k(u)}=x ]  
-   \Esp_\theta [ \psi_\theta( X_{\parent(v)}, X_{v}, Y_{v} )
		 \,\vert\, Y_{\DeltaR^*(u,k)}, X_{\parent^k(u)}=x ] \bigr\Vert \\
&= \biggl\Vert \int_{\SpaceX^3}   
	\psi_\theta( x_{\parent(v)}, x_{v}, Y_{v} ) \, 
		\Prb_\theta(X_v \in \drv x_v \,\vert\, Y_{\DeltaR(u,k)\cap\T(v)}, X_{\parent(v)}=x_{\parent(v)}) 
			\biggr. \nonumber\\
		& \qquad \quad  \biggl.
	\times \bigl[ \Prb_\theta( X_{\parent(v)} \in \drv x_{\parent(v)}
			\,\vert\, Y_{\DeltaR(u,k)}, 
					X_{\parent^{k}(u)}=x  )
		-  \Prb_\theta( X_{\parent(v)} \in \drv x_{\parent(v)}
			\,\vert\, Y_{\DeltaR^*(u,k)}, 
					X_{\parent^{k}(u)}=x  ) \bigr]
	\biggr\Vert \nonumber\\
& \leq 2 \sup_{\theta\in\Theta_0} \sup_{x_1,x_2\in\SpaceX} \norm{\psi_\theta(x_1,x_2,Y_v)} \, \rho^{d(u,v)-2} .
\end{align*}
If $v$ is an ancestor of $u$ (then $\dgr(u,\parent(v)) = \dgr(u,v) +1$), 
using the Markov property  (remind \eqref{eq_illustration_Markov_prop}) and \Cref{lemma_backward_mixing_bound},
we get:
\begin{align*}
\bigl\Vert \Esp_\theta [  \psi_\theta( &X_{\parent(v)}, X_{v}, Y_{v} )
		\,\vert\,  Y_{\DeltaR(u,k)},\ X_{\parent^k(u)}=x ]  
-   \Esp_\theta [ \psi_\theta( X_{\parent(v)}, X_{v}, Y_{v} )
		 \,\vert\, Y_{\DeltaR^*(u,k)}, X_{\parent^k(u)}=x ] \bigr\Vert \\
&= \biggl\Vert \int_{\SpaceX^3}   
	\psi_\theta( x_{\parent(v)}, x_{v}, Y_{v} ) \, 
		\Prb_\theta(X_{\parent(v)} \in \drv x_{\parent(v)} \,\vert\, Y_{\DeltaR(u,k)\setminus\T(v)}, X_{v}=x_{v}) 
			\biggr. \nonumber\\
		& \qquad \qquad  \biggl.
	\times \bigl[ \Prb_\theta( X_{v} \in \drv x_{v}
			\,\vert\, Y_{\DeltaR(u,k)}, 
					X_{\parent^{k}(u)}=x  )
		-  \Prb_\theta( X_{v} \in \drv x_{v}
			\,\vert\, Y_{\DeltaR^*(u,k)}, 
					X_{\parent^{k}(u)}=x  ) \bigr]
	\biggr\Vert \nonumber\\
& \leq 2 \sup_{\theta\in\Theta_0} \sup_{x_1,x_2\in\SpaceX} \norm{\psi_\theta(x_1,x_2,Y_v)} \, \rho^{d(u,v)-1} .
\end{align*}
In both cases, we get:
\begin{multline}
\bigl\Vert \Esp_\theta [  \psi_\theta( X_{\parent(v)}, X_{v}, Y_{v} )
		\,\vert\,  Y_{\DeltaR(u,k)},\ X_{\parent^k(u)}=x ]  
-   \Esp_\theta [ \psi_\theta( X_{\parent(v)}, X_{v}, Y_{v} )
		 \,\vert\, Y_{\DeltaR^*(u,k)}, X_{\parent^k(u)}=x ] \bigr\Vert \\
 \leq 2 \sup_{\theta\in\Theta_0} \sup_{x_1,x_2\in\SpaceX} \norm{\psi_\theta(x_1,x_2,Y_v)} \, \rho^{d(u,v)-2} .
\label{eq_score_incr_backward_mixing}
\end{multline}
Note that the same bound remain valid
for the corresponding terms with $k$ and $x$ replaced by $k'$ and $x'$, respectively,
and with $v\in\Delta(u,k')\setminus\{\parent^{k'}(u)\}$ instead of $v\in\Delta(u,k)\setminus\{\parent^{k}(u)\}$.
This bound is small if $v$ is far away from $u$.

\medskip

Remind from \eqref{eq_inclusion_Delta_Tpast} on page~\pageref{eq_inclusion_Delta_Tpast} that $\DeltaR(u,k) \subset \Tpast(\parent^k(u),k)$
and  as $k'\geq k$ note that:
\begin{equation*}
\DeltaR(u,k') \setminus \DeltaR(u,k) \subset \Tpast(\parent^{k'}(u),k') \setminus \Tpast(\parent^k(u),k) .
\end{equation*}
For a vertex $v\in\Tpast(\parent^k(u),k)\setminus\{\parent^k(u)\}$
(note that $u\land v \in \Tpast(\parent^k(u))$),
note that the term $\rho^{d(v,\parent^k(u))-1}$ is smaller than $\rho^{d(u,v) - 2}$
whenever $d(v,\parent^k(u)) > d(u,v) -1$, that is when 
$d(u\land v, \parent^k(u)) \geq d(u\land v, u)$,
that is when $d(u\land v, u) \leq k/2$.

\newcommand{\kHalf}{\floor{k/2}}

Combining those facts with the bounds \eqref{eq_score_incr_forward_mixing} and \eqref{eq_score_incr_backward_mixing}, 
and using Minkowski's inequality for the $L^b$-norm, we find that 
$( \Esp_\cU \otimes \Etrue [ \sup_{\theta\in\Theta_0} \sup _{x,x'\in\SpaceX} \norm{ \xi_{u,k,x}(\theta)  - \xi_{u,k',x'}(\theta) }^b ] )^{1/b}$
is upper bounded by:
\begin{align}
4 \sum_{v\in\Tpast(\parent^{\kHalf}(u),\kHalf)}
	\rho^{d(v,\parent^k(u))-1} 
+ 4 \sum_{v\in\Tpast(\parent^{k'}(u),k')\setminus\Tpast(\parent^{\kHalf}(u),\kHalf)}
	\rho^{d(u,v) -2} ,
		\label{eq_bound_score_increments}
\end{align}
up to the factor $\bigl( \Etrue \left[ \sup_{\theta\in\Theta_0} \sup_{x_1,x_2\in\SpaceX} \norm{\psi_\theta(x_1,x_2,Y_v)}^b  \right] \bigr)^{1/b}$
(remind that the process $(Y_u, u\in\Tpast)$ is stationary under \Cref{assump_HMM_0}).
Denote by $A_1$ and $A_2$ respectively the first and second terms in \eqref{eq_bound_score_increments}.
We are going to reindex those sums by $j := \dgr(u,u\land v)$ and $q := \dgr(u\land v, v)$ with $q \leq j$.
Note that if $q>0$, then the first vertex after $u\land v = \parent^j(u)$ on the path from $u$ to $v$
	cannot be $\parent^{j-1}(u)$ and must be the other children of $\parent^j(u)$.
Thus, there are $2^{q-1}$ choices of $v$ with the same coding $(j,q)$ with $0<q\leq j$.
Hence, we get:
\begin{equation*}
A_1 
= 4 \sum_{j=0}^{\kHalf} \rho^{k-j -1} 
	\left(
		1 + \sum_{q=1}^{j} 2^{q-1} \rho^q
	\right)
\quad\text{and}\quad
A_2
= 4 \sum_{j=\lfloor  k/2 \rfloor +1}^{k'} \rho^{j -2} 
	\left(
		1 + \sum_{q=1}^{j} 2^{q-1} \rho^q
	\right).
\end{equation*}
Remark that there exists a finite constant $C < \infty$ (which depends on the value of $\rho$) such that for all $j\in\N^*$ we have
$(1 + \sum_{q=1}^{j} 2^{q-1} \rho^q ) \leq C \max( j, (2\rho)^{j})$
and $\sum_{q=j}^{\infty} q \rho^q \leq C \rho^j$.
Hence, there exists a finite constant $C'<\infty$ (which depends only on the value of $\rho$) such that
(remind that $\rho< 1 / \sqrt{2}$):
\begin{equation}\label{eq_upper_bound_score_A1_A2}
A_1 \leq C' \left( k \rho^{k/2} + (2 \rho^2)^{k/2} \right)
\quad\text{and}\quad
A_2 \leq C' \left( \rho^{k/2} + (2 \rho^2)^{k/2} \right).
\end{equation}
Combining \eqref{eq_bound_score_increments} and \eqref{eq_upper_bound_score_A1_A2},
we get that the bound in the lemma holds.
This concludes the proof of the lemma.
\end{proof}

\subsubsection{Asymptotic normality of the score}

Define the limiting Fisher information as:
\begin{equation}\label{eq_def_Fisher_information}
\cI(\thetaTrue)
= \Esp_\cU \otimes \Etrue \Bigl[  \dot\h_{\rooot,\infty}(\thetaTrue) \dot\h_{\rooot,\infty}(\thetaTrue)^t \Bigr] ,
\end{equation}
where we see $\dot\h_{\rooot,\infty}(\thetaTrue)$ as a column vector.

For the asymptotic normality of the score, we need the following extra regularity assumption of the gradient.
\begin{assumption}[$L^4$ gradient regularity]	\label{assump_HMM_grad_2}
In addition to \Cref{assump_HMM_grad_1}, we have:
\begin{equation*}
\Etrue \left[
	\sup_{\theta\in\ThetaNeighborhood} \sup_{x\in\SpaceX} \Vert \nabla_\theta \log g_\theta(x,Y_{\rooot}) \Vert^4
	\right] < \infty.
\end{equation*}
\end{assumption}

We are now ready to prove the following theorem stating the asymptotic normality of the normalized score
towards a centered Gaussian random variable whose variance is the limiting Fisher information.
Note that the condition $\rho < 1/\sqrt{2}$ on the mixing rate $\rho$ of the HMT process $(X,Y)$
comes from the use of \Cref{lemma_score_incr_L2_bound} in the proof of this theorem.
See the discussion in \Cref{rem_rho_smaller_than_2}
for comments on this condition on $\rho$.

\begin{theo}[Asymptotic normality of the normalized score]
	\label{thm_convergence_score}
Assume that Assumptions~\ref{assump_HMM_0}--\ref{assump_HMM_3} and \ref{assump_HMM_grad_1}--\ref{assump_HMM_grad_2} hold
with $\thetaTrue\in\Theta$ given.
Further assume that $\rho < 1/\sqrt{2}$.
Then, for all $x\in\SpaceX$, we have:
\begin{equation*}
\vert \T_n \vert^{-1/2} \, \nabla_\theta \ell_{n,x}(\thetaTrue) 
	\underset{n\to\infty}{\overset{(d)}{\longrightarrow}}
\cN(0, \cI(\thetaTrue))
\quad \text{under $\Ptrue$,}
\end{equation*}
where $\cN(0,M)$ denotes the centered Gaussian distribution with covariance matrix $M$,
and $\cI(\thetaTrue)$ is the limiting Fisher information defined in \eqref{eq_def_Fisher_information}.
\end{theo}

\begin{proof}
\textbf{Step~1: Approximation of the score by the stationary score.}

Remind from \Cref{lemma_Q_unif_geom_ergodic}
that $\pi_{\thetaTrue}$ denotes the invariant distribution
for the hidden process $X$ associated with $Q_{\thetaTrue}$.
Define the stationary score $\nabla_{\theta} \ell_{n,\pi_{\thetaTrue}}(\theta)$ as:
\begin{equation*}
\nabla_{\theta} \ell_{n,\pi_{\thetaTrue}}(\theta)
:= \int_{\SpaceX} \nabla_{\theta} \ell_{n,x}(\theta) \, \pi_{\thetaTrue}(\drv x).
\end{equation*}

First, for all $x,x'\in\SpaceX$ and $\theta\in\ThetaNeighborhood$, write:
\begin{equation*}
\nabla_{\theta} \ell_{n,x}(\theta) 
	- \nabla_{\theta} \ell_{n,x'}(\theta) 
= \sum_{u\in\T_n^*} \Phi(\theta; u, x,x') ,
\end{equation*}
where:
\begin{equation*}
\Phi(\theta; u, x,x')
= \Esp_{\theta}[ \phi_{\theta}(X_{\parent(u)},X_u,Y_u)
			\,\vert\, Y_{\T_n}, X_{\rooot}=x ]
		- \Esp_{\theta}[ \phi_{\theta}(X_{\parent(u)},X_u,Y_u)
			\,\vert\, Y_{\T_n}, X_{\rooot}=x' ].
\end{equation*}
Using Minkowski's inequality 
and the upper bound \eqref{eq_score_incr_forward_mixing}
from the proof of \Cref{lemma_score_incr_L2_bound},
we get:
\begin{multline*}
\left( \Etrue \left[ \sup_{x,x' \in \SpaceX}   \inv{\vert \T_n\vert}
	\BigNorm{ \nabla_{\theta} \ell_{n,x}(\theta) 
	- \nabla_{\theta} \ell_{n,x'}(\theta) }^2 \right] \right)^{1/2} \\
\begin{aligned}
& \leq  \inv{\vert \T_n\vert^{1/2}}
	\sum_{u\in\T_n^*} \left(  \Etrue  \left[  \sup_{x,x' \in \SpaceX} 
	\BigNorm{  \Phi(\theta; u,x, x')  }^2 
\right] \right)^{1/2} \\ 
& \leq 2\, \left( \Etrue \left[ 
		\sup_{\theta\in\Theta_0} \sup_{x,x' \in \SpaceX} \norm{\phi_\theta(x,x',Y_{\rooot})}^2
	\right] \right)^{1/2}
	\inv{\vert \T_n\vert^{1/2}}
	\sum_{k=1}^n 2^k \rho^{k-1} \\
& \leq C \max(n2^{-n}, (2\rho^2)^{n/2}) ,
\end{aligned}
\end{multline*}
where $C<\infty$ is some finite constant.
Thus (remind that $\rho < 1/\sqrt{2}$), for any $x\in\SpaceX$, we have:
\begin{equation}\label{eq_L2_conv_diff_score_and_stationary_score}
\lim_{n\to\infty} \inv{\vert \T_n\vert^{1/2}}
	\Bigl( \nabla_{\theta} \ell_{n,x}(\thetaTrue) 
	- \nabla_{\theta} \ell_{n,\pi_{\thetaTrue}}(\thetaTrue)  \Bigr)
= 0
\qquad \text{in $L^2(\Ptrue)$.}
\end{equation}
In particular, 
to prove asymptotic normality for
the score $\nabla_{\theta} \ell_{n,x}(\thetaTrue)$
for any $x\in\SpaceX$,
it is enough to prove asymptotic normality for
the stationary score $\nabla_{\theta} \ell_{n,\pi_{\thetaTrue}}(\thetaTrue)$ 
(see for instance \cite[Theorem~3.1]{billingsleyConvergenceProbabilityMeasures1999}).

For any $u\in\T$ and $k\in\N$ and $\theta\in\ThetaNeighborhood$, define:
\begin{equation}\label{eq_def_score_stationary_increments}
\dot \h_{u,k,\pi_{\thetaTrue}}(\theta)
:= \int_{\SpaceX} \dot \h_{u,k,x}(\theta) \, \pi_{\thetaTrue}(\drv x) .
\end{equation}
In particular, note that, as the bound in \Cref{lemma_score_incr_L2_bound}
is uniform in $x\in\SpaceX$,
this bound still holds with $\dot \h_{u,k,x}(\theta)$ replaced by
$\dot \h_{u,k,\pi_{\thetaTrue}}(\theta)$.
Using \eqref{eq_equal_score_sum_increments}, note that we have:
\begin{equation*}
\nabla_\theta \ell_{n,\pi_{\thetaTrue}}(\theta)
= \sum_{u\in\T_n} \dot \h_{u,\height{u},\pi_{\thetaTrue}}(\theta).
\end{equation*}
Moreover, remark that for $\theta=\thetaTrue$ 
and for any $u\in\T$ and $k\in\N^*$,
we have:
\begin{align}
\dot \h_{u,k,\pi_{\thetaTrue}}(\thetaTrue)
= & \ \Etrue [  \phi_{\thetaTrue}( X_{\parent(u)}, X_{u}, Y_{u} )
		\,\vert\, Y_{\DeltaR(u,k)}  ]  
			\nonumber\\
  & + \sum_{v\in\DeltaR^*(u,k)\setminus\{\parent^k(u)\}}  \Bigl(
	\Etrue [ \phi_{\thetaTrue}( X_{\parent(v)}, X_{v}, Y_{v} )
		\,\vert\, Y_{\DeltaR(u,k)} ] 
		\Bigr.	\nonumber\\
& \qquad \qquad \qquad \qquad	 \qquad \Bigl.
	- \Etrue [ \phi_{\thetaTrue}( X_{\parent(v)}, X_{v}, Y_{v} )
		\,\vert\, Y_{\DeltaR^*(u,k)} ] \Bigr) .
			\label{eq_sum_incr_score_rearranged_stationary}
\end{align}

\medskip

\textbf{Step~2: The stationary score is a sum of martingale increments.}

As $\T$ is a plane rooted tree, we can enumerate its vertices in a breadth-first-search manner,
that is, as a sequence $(u_j)_{j\in\N}$ which is increasing for $<$.
(Note that $u_0 = \delta$.)
Remind that $\Delta(u_{j-1}) = \Delta^*(u_j)$ for all $j\geq 1$.
Define the filtration $\cF$ by $\cF_j = \sigma( Y_{v} : v\in\T,\, v \leq u_j) = \sigma( Y_{\Delta(u_j)} )$ for all $j\in\N$,
and note that $\cF_j \subset \sigma(Y_{\T})$.
Let $j\in\N^*$, $1\leq k \leq \height{u_j}$, $x\in\SpaceX$ and $v\in Y_{\Delta^*(u_j,k)}$.
Note that we have:
\begin{align*}
 \Etrue\Bigl[ 
	\Etrue[ \phi_{\thetaTrue}(X_{\parent(v)},X_v,Y_v) \,\vert\,  Y_{\Delta(u_j,k)} ]
	\,\Bigm\vert\,
	\cF_{j-1} \Bigr] 
 \ = \ \Etrue[ \phi_{\thetaTrue}(X_{\parent(v)},X_v,Y_v) \,\vert\,  Y_{\Delta^*(u_j,k)} ] .
\end{align*}
Also note that \Cref{assump_HMM_grad_1} (on page~\pageref{assump_HMM_grad_1}) implies that:
\begin{multline*}
\Etrue[ \phi_{\thetaTrue} (X_{\parent(u_j)}, X_{u_j}, Y_{u_j}) \,\vert\,  X_{\parent(u_j)} ] \\
\begin{aligned}
& = \int_{\SpaceX \times \SpaceY} 
	\nabla_\theta \log[ q_\theta(X_{\parent(u_j)},x) g_\theta(x,y) ]\,  q_\theta(X_{\parent(u_j)},x) g_\theta(x,y) \, \lambda(\drv x) \mu(\drv y) \\
& = \int_{\SpaceX \times \SpaceY} 
	\nabla_\theta [ q_\theta(X_{\parent(u_j)},x) g_\theta(x,y) ]  \, \lambda(\drv x) \mu(\drv y) \\
& = \nabla_\theta \left[ \int_{\SpaceX \times \SpaceY} 
	 q_\theta(X_{\parent(u_j)},x) g_\theta(x,y)  \, \lambda(\drv x) \mu(\drv y) \right] \\
& = 0 .
\end{aligned}
\end{multline*}
Thus, we have:
\begin{multline*}
 \Etrue\Bigl[
	\Etrue[ \phi_{\thetaTrue}(X_{\parent(u_j)}, X_{u_j}, Y_{u_j}) \,\vert\,  
					Y_{\Delta(u_j,k)} ]
	\, \Bigm\vert \,
	\cF_{j-1} \Bigr] \nonumber\\
\begin{aligned}
& = \Etrue[ \phi_{\thetaTrue}(X_{\parent(u_j)}, X_{u_j}, Y_{u_j}) \,\vert\,  
					Y_{\Delta^*(u_j,k)} ] 
	\nonumber\\
& = \Etrue\Bigl[
	\Etrue[ \phi_{\thetaTrue}(X_{\parent(u_j)}, X_{u_j}, Y_{u_j}) \,\vert\,  
					Y_{\Delta^*(u_j,k)}, X_{\parent(u_j)} ]
	\, \Bigm\vert \,
	Y_{\DeltaR^*(u_j,k)} \Bigr] \nonumber\\
& = \Etrue\Bigl[
	\Etrue[ \phi_{\thetaTrue}(X_{\parent(u_j)}, X_{u_j}, Y_{u_j}) \,\vert\,  X_{\parent(u_j)} ]
	 \Bigm\vert \,
	Y_{\Delta^*(u_j,k)} \Bigr] \nonumber\\
& = 0 ,
\end{aligned}
\end{multline*}
where we used the Markov property for the inner expectation in the third equality.
Moreover, it is immediate that $\dot\h_{u_j, k, \pi_{\thetaTrue}}(\thetaTrue)$ is $\cF_{j}$-measurable for all $j\in\N^*$ and $1\leq k\leq \height{j}$.
Hence, we get that the sequence $\bigl( \dot\h_{u_j, \height{u_j}, \pi_{\thetaTrue}}(\thetaTrue) \bigr)_{j\in\N^*}$
	is a $\Ptrue$-martingale increment sequence adapted to the filtration $\cF=(\cF_j)_{j\in\N}$
	in $L^2(\Ptrue)$ (thanks to \Cref{assump_HMM_grad_1}).
We are going to apply a central limit theorem for martingales (see \cite[Corollary~2.1.10]{dufloRandomIterativeModels2011}).
For all $n\in\N$, define $M_n = \sum_{j=0}^n \dot\h_{u_j, \height{u_j}, \pi_{\thetaTrue}}(\thetaTrue)$.
Note that $M_0 = \dot\h_{\rooot,0,\pi_{\thetaTrue}}(\thetaTrue) 
	= \int_{\SpaceX} \nabla_\theta \log g_{\thetaTrue}(x,Y_\rooot)
		\, \pi_{\thetaTrue}(\drv x)$ is in $L^2(\Ptrue)$
	by \Cref{assump_HMM_grad_1}.
Hence, the sequence $(M_n)_{n\in\N}$ is a $\Ptrue$-martingale sequence adapted to the filtration $\cF=(\cF_j)_{j\in\N}$ in $L^2(\Ptrue)$,
and whose quadratic variation is: \begin{equation*}
\langle M \rangle_n 
= \sum_{j=1}^n 
	\Etrue\Bigl[ \dot\h_{u_j, \height{u_j}, \pi_{\thetaTrue}}(\thetaTrue) \dot\h_{u_j, \height{u_j}, \pi_{\thetaTrue}}(\thetaTrue)^t
		\, \Bigm\vert \, \cF_{j-1} \Bigr] ,
\end{equation*}
where, as in \eqref{eq_def_Fisher_information}, 
we see $\dot\h_{u_j, \height{u_j}, \pi_{\thetaTrue}}(\thetaTrue)$ as a column vector.
Note that for all $n\in\N$, $M_n$ and $\langle M \rangle_n$ do not depend on $\cU$.

\textbf{Step~3: Convergence of the quadratic variation.}
Before applying the central limit theorem for martingales,
we first need to prove that $\lim_{n\to\infty} n^{-1} \langle M \rangle_n = \cI(\thetaTrue)$ in $\Ptrue$-probability.
Indeed, we will prove that this convergence holds in $L^2(\Ptrue)$.
Let $k\in\N^*$ and $x\in\SpaceX$.
Note that for $u_j \in \T\setminus \T_{k-1}$ is equivalent to $j \geq \vert \T_{k-1} \vert$
(remind that $u_0 = \rooot$).
Using \eqref{eq_sum_incr_score_rearranged} along with \Cref{assump_HMM_grad_2}, 
	we get that $\sup_{x\in\SpaceX} \dot\h_{u,k,x}(\thetaTrue, Y_{\Delta(u,k)})$ is in $L^4(\Ptrue)$,
	and thus the random variable $\sup_{x\in\SpaceX} \dot\h_{u,k,x}(\thetaTrue, Y_{\Delta(u,k)}) \dot\h_{u,k,x}(\thetaTrue, Y_{\Delta(u,k)})^t$
is in $L^2(\Ptrue)$ for every $u\in\T\setminus \T_{k-1}$.
Thus, using \eqref{eq_def_score_stationary_increments} and 
\Cref{lemma_score_incr_L2_bound} (remind that $\rho<1/\sqrt{2}$) for the first moment ($b=2$), 
there exists a finite constant $C>0$ and $\alpha\in (0,1)$ such that we have
(remind \eqref{eq_unif_bound_dot_hukx}):
\begin{equation}\label{eq_approx_hook_Mn}
\Etrue\left[
\left\Vert n^{-1} \langle M \rangle_n -
\inv{n} \sum_{j=\vert \T_{k-1} \vert}^n 
	\Etrue\Bigl[ \dot\h_{u_j,k, x}(\thetaTrue) \dot\h_{u_j, k, x}(\thetaTrue)^t
		\, \Bigm\vert \, \cF_{j-1} \Bigr]
\right\Vert \right]
\leq C  \alpha^{k} + \frac{C \vert \T_{k-1}\vert}{n} ,
\end{equation}
where remind that $\norm{\cdot}$ denotes the euclidean norm for $d\times d$ matrices
	(or any other norm as they are all equivalent in finite dimension).
To prove that the second term inside the expectation in the left hand side of \eqref{eq_approx_hook_Mn} 
	converges in $L^2(\Ptrue)$ as $n\to\infty$,
we are going to apply the ergodic convergence \Cref{lemma_ergodic_convergence_2}
	where the averages are done on the vertex subset $\{ u_j \,:\, \vert T_{k-1} \vert \leq j \leq n \}$. Note that this lemma is stated for scalar-valued functions, but we can apply it individually for each of the matrix coefficients
	to get the equivalent for matrix-valued functions.

For all $u\in\T\setminus \T_{k-1}$, define the function:
\begin{equation*}
\Psi_{u,k,x} : y_{\Delta^*(u,k)} \in \SpaceY^{\Delta^*(u,k)} \mapsto 
\Etrue\Bigl[
	\dot\h_{u,k,x}(\thetaTrue; Y_{\Delta(u,k)}) \dot\h_{u,k,x}(\thetaTrue; Y_{\Delta(u,k)})^t 
	\,\Bigm\vert\, Y_{\Delta^*(u,k)} = y_{\Delta^*(u,k)}
\Bigr] .
\end{equation*}
For a vertex $u$ in $\T\setminus \T_{k-1}$, 
let $v_u\in\G_k$ be the unique vertex that satisfies the shape equality constraint 
\eqref{eq_def_shape_subtree} (on page \pageref{eq_def_shape_subtree}),
then we have the equality between functions:
\begin{equation}\label{eq_equal_Psi_ukx_up_to_shape}
\Psi_{u,k,x} = \Psi_{v_u,k,x} .
\end{equation}
Moreover, using \eqref{eq_sum_incr_score_rearranged} along with \Cref{assump_HMM_grad_2}, 
	we get that $\dot\h_{u,k,x}(\thetaTrue, Y_{\Delta(u,k)})$ is in $L^4(\Ptrue)$,
	and thus the random variable $\Psi_{u,k,x}(Y_{\Delta^*(u,k)})$
is in $L^2(\Ptrue)$ for every $u\in\T\setminus \T_{k-1}$.
Hence, applying \Cref{lemma_ergodic_convergence_2}
to the collection of neighborhood-shape-dependent functions 
$( \Psi_{v,k,x} )_{v\in\G_k}$
(remind that indexing functions with $\G_k$ or with $\ShapeSetValues_k$ is equivalent by \eqref{eq_def_set_possible_shapes}),
and using \eqref{eq_equal_Psi_ukx_up_to_shape} and \eqref{eq_equality_expectation_root_and_U_k} (in \Cref{rem_rerooting_delta}),
we get that the second term inside the expectation in the left hand side of \eqref{eq_approx_hook_Mn}
converges in $L^2(\Ptrue)$ to 
	$\Esp_\cU \otimes \Etrue \bigl[   \dot\h_{\partial,k, x}(\thetaTrue) \dot\h_{\partial, k, x}(\thetaTrue)^t \bigr]$ as $n\to\infty$.
Using \Cref{lemma_score_incr_L2_bound}, we have that
$\lim_{k\to\infty} \Esp_\cU \otimes \Etrue \bigl[   \dot\h_{\partial,k, x}(\thetaTrue) \dot\h_{\partial, k, x}(\thetaTrue)^t \bigr]
	= \cI(\thetaTrue)$.
Combining those facts with \eqref{eq_approx_hook_Mn}, we get that 
	$\lim_{n\to\infty} n^{-1} \langle M \rangle_n = \cI(\thetaTrue)$ in $L^2(\Ptrue)$.

\medskip

\newcommand{\indLarger}[1]{\ind_{\{ #1 \}}}

\textbf{Step~4: Lindeberg's condition holds.}
We now need to verify that Lindeberg's condition holds (see \cite[Corollary~2.1.10]{dufloRandomIterativeModels2011}),
that is, to prove for all $\eps>0$ that $\lim_{n\to\infty} F_n(\eps \sqrt{n}) = 0$ in $\Ptrue$-probability where for all $n\in\N^*$ and $A\in\R_+$:
\begin{equation}
F_n(A) = \inv{n} \sum_{j=1}^n 
	\Etrue\Bigl[ \BigNorm{ \dot\h_{u_j,\height{u_j}, \pi_{\thetaTrue}}(\thetaTrue) }^2 
			\indLarger{ \norm{ \dot\h_{u_j,\height{u_j}, \pi_{\thetaTrue}}(\thetaTrue) } \geq A }
		\, \Bigm\vert \, \cF_{j-1} \Bigr] .
\end{equation}

Remind that 
by \Cref{assump_HMM_grad_2} and \Cref{lemma_score_incr_L2_bound} (remind that $\rho<1/\sqrt{2}$) for the fourth moment ($b=4$),
we have:
\begin{equation*}
C := \sup_{u\in\T} \sup_{k\in\N^*} \ \Esp_\cU \otimes \Etrue \left[ \sup_{x\in\SpaceX}
 			\BigNorm{ \dot\h_{u,k, x}(\thetaTrue) }^4 \right] < \infty  .
\end{equation*}
Using Cauchy-Schwarz inequality and Markov inequality, we get:
\begin{align*}
\Etrue [ F_n(A) ]
 \leq \inv{n} \sum_{j=1}^n  \frac{\Etrue\Bigl[ \sup_{x\in\SpaceX} \BigNorm{ \dot\h_{u_j,\height{u_j}, x}(\thetaTrue) }^4 \Bigr]}{A^2} 
 \leq \frac{C}{A^2}.
\end{align*}
Let $\eps>0$. Then, setting $A_n = \eps \sqrt{n}$ for all $n\in\N^*$, we get
that $\lim_{n\to\infty} F_n(\eps \sqrt{n}) =0$ in $L^1(\Ptrue)$, and thus in $\Ptrue$-probability.
Hence, we get that Lindeberg's condition holds.

\medskip

\textbf{Step~5: Applying the central limit theorem for martingales.}
Hence, we can apply the central limit theorem for martingales (see \cite[Corollary~2.1.10]{dufloRandomIterativeModels2011}),
which gives us that $\Ptrue$-\as $\lim_{n\to\infty} n^{-1} M_n = 0$
and that the sequence $(n^{-1/2} M_n)_{n\in\N^*}$ converges in $\Ptrue$-distribution towards a centered Gaussian distribution $\cN(0, \cI(\thetaTrue))$
whose covariance matrix is $\cI(\thetaTrue)$.
In particular, using \eqref{eq_L2_conv_diff_score_and_stationary_score}, 
we get that $\Ptrue$-\as $\lim_{n\to\infty} \vert \T_n \vert^{-1} \nabla_\theta \ell_{n,x}(\thetaTrue) = 0$
and that:
\begin{equation*}
\vert \T_n \vert^{-1/2} \, \nabla_\theta \ell_{n,x}(\thetaTrue) 
	\underset{n\to\infty}{\overset{(d)}{\longrightarrow}}
\cN(0, \cI(\thetaTrue))
\quad \text{under $\Ptrue$.}
\end{equation*}
This concludes the proof of the theorem.
\end{proof}

\subsection{Law of large number for the normalized observed information}

In this subsection, we prove that for all possibly random sequence $(\theta_n)_{n\in\N}$
such that $\lim_{n\to\infty} \theta_n = \thetaTrue$ $\Ptrue$-\as,
then the normalized observed information $-n^{-1} \nabla_\theta^2 \ell_{n,x}(\theta_n)$ converges $\Ptrue$-\as as $n\to\infty$
to the limiting Fisher information matrix $\cI(\thetaTrue)$ which is defined in \eqref{eq_def_Fisher_information}.

Remind the definition of the log-likelihood $\ell_{n,x}(\theta)$ in \eqref{eq_def_l_nx_theta_2} on page~\pageref{eq_def_l_nx_theta_2}.
We start by decomposing the Hessian of the log-likelihood $\ell_{n,x}(\theta)$ as a sum of increment indexed by the tree $\T$.
Using elementary computation along with permutations of the integral and the gradient operator
which are valid under \Cref{assump_HMM_grad_1}
(note that this result is also known as \emph{Louis missing information principle}, see \cite[Proposition~10.1.6]{CappeHMM}),
we get for all $\theta\in\ThetaNeighborhood$ and $x\in\SpaceX$:
\begin{align*}
\nabla_\theta^2 \ell_{n,x}(\theta)
 = \nabla_\theta^2 \log( g_\theta(X_{\rooot}, Y_{\rooot})) 
	& + \Esp_\theta \left[  
		\sum_{u\in\T_n^* } \varphi_\theta(X_{\parent(u)}, X_u, Y_u)
		\,\middle\vert\,  Y_{\T_n}, X_{\rooot}=x
	\right] \nonumber\\
&  + \Var_\theta \left[
		\sum_{u\in\T_n^* } \phi_\theta(X_{\parent(u)}, X_u, Y_u)
		\,\middle\vert\,  Y_{\T_n}, X_{\rooot}=x
	\right] , \end{align*}
where remind that $\phi_\theta$ is defined in \eqref{eq_def_phi} on page~\pageref{eq_def_phi},
and $\varphi_\theta$ is defined as:
\begin{equation}
\varphi_\theta(x', x, y) = \nabla_\theta^2 \log( q_\theta(x',x) g_\theta(x,y) ) .
\end{equation}
Note that similarly to the case of $\phi_\theta$,
the random variale $\varphi_\theta(X_{\parent(u)}, X_u, Y_u)$ is 
integrable conditionally on $Y_{\Delta(u)}$ and $X_{\rooot}=x$
(see the discussion after \eqref{eq_def_phi}).
Also note that $\nabla_\theta \log g_\theta(x, Y_{\rooot})$ is $\Ptrue$-\as finite 
by \Cref{assump_HMM_grad_1}-\ref{assump_HMM_grad_1:item3}.

For all $u\in\T$, $k\in\N^*$ and $x\in\SpaceX$, we define:
\begin{align}
\Lambda_{u,k,x}(\theta) & = 
\Esp_\theta\left[
				\sum_{v\in\DeltaR(u,k)\setminus\{\parent^k(u)\}} \varphi_\theta(X_{\parent(v)}, X_v, Y_v)
			\,\middle\vert\,  Y_{\DeltaR(u,k)}, X_{\rooot} = x_{\rooot}
			\right]
	 \nonumber\\*
&  \qquad \qquad
		- \Esp_\theta\left[
				\sum_{v\in\DeltaR^*(u,k)\setminus\{\parent^k(u)\}} \varphi_\theta(X_{\parent(v)}, X_v, Y_v)
			\,\middle\vert\,  Y_{\DeltaR^*(u,k)}, X_{\rooot} = x_{\rooot}
			\right], \label{eq_def_Lambda_ukx}
\end{align}
and:
\begin{align}
\Gamma_{u,k,x}(\theta) & = 
\Var_\theta\left[
				\sum_{v\in\DeltaR(u,k)\setminus\{\parent^k(u)\}} \phi_\theta(X_{\parent(v)}, X_v, Y_v)
			\,\middle\vert\,  Y_{\DeltaR(u,k)}, X_{\rooot} = x_{\rooot}
			\right]
	 \nonumber\\*
&  \qquad \qquad 
		- \Var_\theta\left[
				\sum_{v\in\DeltaR^*(u,k)\setminus\{\parent^k(u)\}} \phi_\theta(X_{\parent(v)}, X_v, Y_v)
			\,\middle\vert\,  Y_{\DeltaR^*(u,k)}, X_{\rooot} = x_{\rooot}
			\right] , \label{eq_def_Gamma_ukx}
\end{align}
where $\Var_\theta$ (resp. $\Cov_\theta$) denotes the (possibly conditional) variance (resp. covariance) 
corresponding to $\Prb_\cU \otimes \Prb_\theta$.
Note that $\Lambda_{u,k,x}(\theta)$ and $\Gamma_{u,k,x}(\theta)$ are random variables
which depend on $Y_{\DeltaR(u,k)}$ with an implicit dependence on $\cU$,
and that they do not depend on $\cU$ if $k\leq\height{u}$.

Then, using telescopic sums involving the quantities
defined in \eqref{eq_def_Lambda_ukx} and \eqref{eq_def_Gamma_ukx},
the Hessian of the log-likelihood $\ell_{n,x}(\theta)$ can be rewritten 
for all $\theta\in\ThetaNeighborhood$ and $x\in\SpaceX$ as:
\begin{equation}\label{eq_Hessian_likelihood}
\nabla_\theta^2 \ell_{n,x}(\theta)
 = \nabla_\theta^2 \log( g_\theta(X_{\rooot}, Y_{\rooot})) 
 + \sum_{u\in\T_n^*} \Lambda_{u,\height{u},x}(\theta) 
 + \sum_{u\in\T_n^*} \Gamma_{u,\height{u},x}(\theta) . 
\end{equation}

To prove the convergence of the two sums in the right hand side of \eqref{eq_Hessian_likelihood}, 
and thus get the convergence of the normalized observed information
$- n^{-1}\nabla_\theta^2 \ell_{n,x}(\theta)$,
we will need the following $L^2$ regularity assumption on the Hessian of the transition kernel $g_\theta$ of the HMT.

\begin{assumption}[$L^2$ Hessian regularity]
	\label{assump_HMM_grad_3}
In addition to \Cref{assump_HMM_grad_1}, assume that we have:
\begin{equation*}
\Etrue \left[
	\sup_{\theta\in\ThetaNeighborhood} \sup_{x} \Vert \nabla_\theta^2 \log g_\theta(x,Y_{\rooot}) \Vert^2
	\right] < \infty .
\end{equation*}
\end{assumption}

Propositions~\ref{prop_conv_Lambda_term} and~\ref{prop_conv_Gamma_term} below 
	(whose proofs are given in Sections~\ref{section_proof_prop_conv_Lambda_term} 
	and \ref{section_proof_prop_conv_Gamma_term}, respectively)
	state that $\Lambda_{u,k,x}(\theta)$ and $\Gamma_{u,k,x}(\theta)$ 
	both have limits $\Ptrue$-\as and in $L^2(\Ptrue)$ when $k\to\infty$.
Denote those limits by $\Lambda_{u,\infty}(\theta)$ and $\Gamma_{u,\infty}(\theta)$, respectively.
Furthermore, Propositions~\ref{prop_conv_Lambda_term} and~\ref{prop_conv_Gamma_term} also state that
	the two sums in the right hand side of \eqref{eq_Hessian_likelihood} converge
	to $\Esp_\cU \otimes \Etrue[  \Lambda_{\rooot,\infty}(\theta) ]$ 
	and $\Esp_\cU \otimes \Etrue[  \Gamma_{\rooot,\infty}(\theta) ]$, respectively,
	with some uniformity in $\theta$ near $\thetaTrue$.

We start with the proposition for the terms $\Lambda_{u,k,x}(\theta)$.
Note that the condition $\rho<1/\sqrt{2}$ on the mixing rate $\rho$ of the HMT process $(X,Y)$ is due
to the use of \Cref{lemma_score_incr_L2_bound} in the proof of \Cref{prop_conv_Lambda_term}.
See the discussion in \Cref{rem_rho_smaller_than_2}
for comments on this condition on $\rho$.

\begin{prop}[Convergence for averages of $\Lambda_{u,k,x}(\theta)$]
	\label{prop_conv_Lambda_term}
Assume that Assumptions~\ref{assump_HMM_0}--\ref{assump_HMM_3}, \ref{assump_HMM_4}-\ref{assump_HMM_grad_1}
and~\ref{assump_HMM_grad_3} hold.
Assume that $\rho<1/\sqrt{2}$.
Then, for each $\theta\in\ThetaNeighborhood$, we have that $\Lambda_{u,k,x}(\theta)$ converges in $L^2(\Prb_\cU\otimes\Ptrue)$ to some limit $\Lambda_{u,\infty}(\theta)$ (that does not depend on $x$)
as $k\to\infty$.
Moreover, we have:
\begin{equation}\label{eq_conv_as_sum_Lambda_ukx}
\lim_{n\to\infty}
\Etrue\left[  \sup_{x\in\SpaceX} \left\vert
	\inv{\vert \T_n \vert} \sum_{u\in\T_n^*} \Lambda_{u,\height{u},x}(\thetaTrue) 
	- \Esp_\cU \otimes \Etrue \bigl[  \Lambda_{\rooot,\infty}(\thetaTrue) \bigr]
\right\vert \right]
= 0.
\end{equation}
Furthermore, the function $\theta \mapsto \Esp_\cU  \otimes \Etrue [ \Lambda_{\rooot,\infty}(\theta) ]$ is continuous on $\ThetaNeighborhood$,
and for all $x\in\SpaceX$ and $\theta\in\ThetaNeighborhood$, we have:
\begin{align*}
\lim_{\delta\to 0} \lim_{n\to\infty}  \sup_{\theta' \in\ThetaNeighborhood : \Vert \theta' - \theta \Vert \leq \delta} 
	\left\vert
		\vert\T_n\vert^{-1}   \sum_{u\in\T_n^*} \Lambda_{u,\height{u},x}(\theta')		
	 - \Esp_\cU \otimes \Etrue [ \Lambda_{\rooot,\infty}(\theta) ]
	\right\vert 
 = 0, \quad \Ptrue\text{-\as}
\end{align*} 
\end{prop}

The following proposition is the equivalent of \Cref{prop_conv_Lambda_term} for the terms $\Gamma_{u,k,x}(\theta)$.
Note that the condition $\rho<1/2$ on the mixing rate $\rho$
of the HMT process $(X,Y)$
is due to the use of \Cref{lemma_Gamma_incr_L2_bound}
in the proof of \Cref{prop_conv_Gamma_term}.
See the discussion in \Cref{rem_rho_smaller_than_2}
for comments on this condition on $\rho$.

\begin{prop}[Convergence for the averages of $\Gamma_{u,k,x}(\theta)$]
	\label{prop_conv_Gamma_term}
Assume that Assumptions~\ref{assump_HMM_0}--\ref{assump_HMM_3} and \ref{assump_HMM_4}-\ref{assump_HMM_grad_2} hold.
Assume that $\rho<1/2$.
Then, for each $\theta\in\ThetaNeighborhood$, we have that $\Gamma_{u,k,x}(\theta)$ converges in $L^2(\Prb_\cU\otimes\Ptrue)$ to some limit $\Gamma_{u,\infty}(\theta)$ (that does not depend on $x$)
as $k\to\infty$.
Moreover, we have:
\begin{equation}\label{eq_conv_as_sum_Gamma_ukx}
\lim_{n\to\infty}
\Etrue\left[  \sup_{x\in\SpaceX} \left\vert
	\inv{\vert \T_n \vert} \sum_{u\in\T_n^*} \Gamma_{u,\height{u},x}(\thetaTrue) 
	- \Esp_\cU \otimes \Etrue \bigl[ \Gamma_{\rooot,\infty}(\thetaTrue) \bigr]
\right\vert \right]
= 0.
\end{equation}
Furthermore, the function $\theta \mapsto \Esp_\cU \otimes \Etrue [  \Gamma_{\rooot,\infty}(\theta) ]$ is continuous on $\ThetaNeighborhood$,
and for all $x\in\SpaceX$ and $\theta\in\ThetaNeighborhood$, we have:
\begin{align*}
\lim_{\delta\to 0} \lim_{n\to\infty} \sup_{\theta' \in\ThetaNeighborhood : \Vert \theta' - \theta \Vert \leq \delta}
	\left\vert
		\vert\T_n\vert^{-1} \sum_{u\in\T_n^*} \Gamma_{u,\height{u},x}(\theta')
		- \Esp_\cU \otimes \Etrue [  \Gamma_{\rooot,\infty}(\theta) ]
	\right\vert 
= 0, \quad \Ptrue\text{-\as}
\end{align*} 
\end{prop}

With Propositions~\ref{prop_conv_Lambda_term} and~\ref{prop_conv_Gamma_term},
we are now ready to prove the following theorem which states that
the normalized observed information $-\vert\T_n\vert^{-1} \nabla_\theta^2 \ell_{n,x}(\theta_n)$
converges $\Ptrue$-\as locally uniformly to the limiting Fisher information $\cI(\thetaTrue)$ 
(which is defined in \eqref{eq_def_Fisher_information}).
Note that the condition $\rho<1/2$ on the mixing rate $\rho$
of the HMT process $(X,Y)$
is inherited from \Cref{prop_conv_Gamma_term}.
See the discussion in \Cref{rem_rho_smaller_than_2}
for comments on this condition on $\rho$.

\begin{theo}[Convergence of the normalized observed information]
	\label{thm_LLN_observed_information}
Assume that Assumptions~\ref{assump_HMM_0}--\ref{assump_HMM_3}
	and \ref{assump_HMM_4}--\ref{assump_HMM_grad_3} hold.
Assume that $\rho<1/2$.
Assume that $\Theta$ is compact.
Then, for all $x\in\SpaceX$, we have:
\begin{equation}\label{eq_unif_conv_observed_information}
\lim_{\delta\to 0} \lim_{n\to\infty} \sup_{\theta\in\ThetaNeighborhood \,:\, \Vert \theta - \thetaTrue \Vert \leq \delta}
	\ \Bigl\Vert{-\vert\T_n\vert^{-1} \nabla_\theta^2 \ell_{n,x}(\theta) - \cI(\thetaTrue)}\Bigr\Vert = 0
\quad \text{$\Ptrue$-\as}
\end{equation}
\end{theo}

As an immediate corollary, for any possibly random sequence  $(\theta_n)_{n\in\N}$ 
	such that $\lim_{n\to\infty} \theta_n = \thetaTrue$ $\Ptrue$-\as
and any $x\in\SpaceX$, we get that $\Ptrue$-\as
	$\lim_{n\to\infty} -\vert\T_n\vert^{-1} \nabla_\theta^2 \ell_{n,x}(\theta_n) = \cI(\thetaTrue)$.
In particular, choosing $\theta_n = \hat{\theta}_{n,x}$ for all $n\in\N$
	(remind that the MLE $\hat{\theta}_{n,x}$ is defined in \eqref{eq_def_MLE_hat_theta_nx} on page~\pageref{eq_def_MLE_hat_theta_nx}),
and combining Theorems~\ref{thm_Strong_consistency_MLE} 
	and~\ref{thm_LLN_observed_information},
we get that the normalized observed information
$-\vert\T_n\vert^{-1} \nabla_\theta^2 \ell_{n,x}(\hat\theta_{n,x})$
at the MLE $\hat\theta_{n,x}$ is a strongly consistent estimator
of the Fisher information matrix $\cI(\thetaTrue)$.

\begin{proof}
Using \eqref{eq_Hessian_likelihood} and Propositions~\ref{prop_conv_Lambda_term} and~\ref{prop_conv_Gamma_term},
we get that \eqref{eq_unif_conv_observed_information} holds
with $\cI(\thetaTrue)$ replaced by $- \Esp_\cU \otimes \Etrue [ \Lambda_{\rooot,\infty}(\thetaTrue) + \Gamma_{\rooot,\infty}(\thetaTrue) ]$.
Thus, it remains to prove that this latter quantity is equal to $\cI(\thetaTrue)$.

Using elementary computation along with permutations of the integral and the gradient operator
which are valid under \Cref{assump_HMM_grad_1}
(note that this result is also known as \emph{Fisher information matrix identity}, 
see \cite[p.21]{raschMathematicalStatistics2018} or \cite[p.355]{CappeHMM}),
we get for all $\theta\in\ThetaNeighborhood$ and $x\in\SpaceX$:
\begin{equation*}
\vert \T_n \vert^{-1}\, \Esp_{\theta}\bigl[ \nabla_{\theta} \ell_{n,x}(\theta)  \nabla_{\theta} \ell_{n,x}(\theta)^t  \,\bigm\vert\, X_{\rooot} =x \bigr]
= - \vert \T_n \vert^{-1}\, \Esp_{\theta}\bigl[ \nabla_{\theta}^2 \ell_{n,x}(\theta)  \,\bigm\vert\, X_{\rooot} =x \bigr].
\end{equation*}
Setting $\theta=\thetaTrue$ and taking the expectation over $X_\rooot$, we get:
\begin{equation}\label{eq_Fisher_information_principle}
\vert \T_n \vert^{-1}\, \Etrue\bigl[ \nabla_{\theta} \ell_{n,X_\rooot}(\thetaTrue)  \nabla_{\theta} \ell_{n,X_\rooot}(\thetaTrue)^t  \bigr]
= - \vert \T_n \vert^{-1}\, \Etrue\bigl[ \nabla_{\theta}^2 \ell_{n,X_\rooot}(\thetaTrue)  \bigr].
\end{equation}

Using \eqref{eq_Hessian_likelihood} on page~\pageref{eq_Hessian_likelihood},
Propositions~\ref{prop_conv_Lambda_term} and~\ref{prop_conv_Gamma_term}
give us that the right hand side of \eqref{eq_Fisher_information_principle}
converges as $n\to\infty$ to $-\Esp_\cU \otimes \Etrue [ \Lambda_{\rooot,\infty}(\thetaTrue) + \Gamma_{\rooot,\infty}(\thetaTrue) ]$.

Remind that using \eqref{eq_sum_incr_score_rearranged} along with \Cref{assump_HMM_grad_2}, 
	we get that $\dot\h_{u,k,x}(\thetaTrue, Y_{\Delta(u,k)})$ is in $L^4(\Ptrue)$,
	and thus the random variable $\dot\h_{u,k,x}(\thetaTrue, Y_{\Delta(u,k)}) \dot\h_{u,k,x}(\thetaTrue, Y_{\Delta(u,k)})^t$
is in $L^2(\Ptrue)$ for every $u\in\T\setminus \T_{k-1}$.
Thus, using \Cref{lemma_score_incr_L2_bound} for the first moment ($b=1$), 
there exists a finite constant $C>0$ and $\alpha\in (0,1)$ such that
for any $k\in\N^*$ and $x\in\SpaceX$, we have:
\begin{align}
 \Etrue\left[ \inv{\vert \T_n\vert} \left\Vert 
	\nabla_{\theta} \ell_{n,X_\rooot}(\thetaTrue)  \nabla_{\theta} \ell_{n,X_\rooot}(\thetaTrue)^t
	- \!\!\! \sum_{u\in\T_n\setminus\T_{k-1}} \!\!\! \dot\h_{u,k,x}(\thetaTrue)  \dot\h_{u,k,x}(\thetaTrue)^t
\right\Vert \right] 
 \leq C \alpha^{k} + \frac{C \vert \T_{k-1} \vert}{\vert \T_n\vert} ,
	\label{eq_upper_bound_approx_score_matrix}
\end{align}
where remind that we see $\dot\h_{u,k,x}(\thetaTrue)$ as a column vector.
Then, using an ergodic convergence argument similar to the one used in Step~3 in the proof of \Cref{thm_convergence_score},
we get:
\begin{equation*}
\lim_{n\to\infty}
\inv{\vert \T_n\vert} \sum_{u\in\T_n\setminus\T_{k-1}} \!\!\! \dot\h_{u,k,x}(\thetaTrue)  \dot\h_{u,k,x}(\thetaTrue)^t
= \Esp_\cU \otimes \Etrue \bigl[   \dot\h_{\partial,k, x}(\thetaTrue) \dot\h_{\partial, k, x}(\thetaTrue)^t \bigr]
\qquad \text{in $L^2(\Ptrue)$}.
\end{equation*}
Using \Cref{lemma_score_incr_L2_bound}, we have that
$\lim_{k\to\infty} \Esp_\cU \otimes \Etrue \bigl[   \dot\h_{\partial,k, x}(\thetaTrue) \dot\h_{\partial, k, x}(\thetaTrue)^t \bigr]
	= \cI(\thetaTrue)$.
Combining those facts with \eqref{eq_upper_bound_approx_score_matrix}, 
we get that the left hand side in \eqref{eq_Fisher_information_principle} converges to $\cI(\thetaTrue)$ as $n\to\infty$.

Hence, we get $\cI(\thetaTrue) = - \Esp_\cU \otimes \Etrue [ \Lambda_{\rooot,\infty}(\thetaTrue) + \Gamma_{\rooot,\infty}(\thetaTrue) ] $,
which concludes the proof.
\end{proof}

Using Theorems~\ref{thm_Strong_consistency_MLE}, \ref{thm_convergence_score} and \ref{thm_LLN_observed_information},
we can prove the following theorem which states that the MLE has asymptotic normal fluctuations
with covariance matrix $\cI(\thetaTrue)^{-1}$ where the Fisher information matrix $\cI(\thetaTrue)$ is defined in
\eqref{eq_def_Fisher_information} on page~\pageref{eq_def_Fisher_information}.
Recall that the contrast function $\ell$ is defined in \eqref{eq_def_contrast_function} on page~\pageref{eq_def_contrast_function},
that the MLE $\hat{\theta}_{n,x}$ is defined in \eqref{eq_def_MLE_hat_theta_nx}  on page~\pageref{eq_def_MLE_hat_theta_nx},
and that the mixing rate $\rho$ of the HMT process $(X,Y)$ is defined after \Cref{assump_HMM_2}
on page~\pageref{assump_HMM_2}.

Note that the condition $\rho<1/2$ on the mixing rate $\rho$
of the HMT process $(X,Y)$
is inherited from \Cref{thm_LLN_observed_information},
and thus from \Cref{prop_conv_Gamma_term}.
See the discussion in \Cref{rem_rho_smaller_than_2}
for comments on this condition on $\rho$.

\begin{theo}[Asymptotic normality of the MLE]
	\label{thm_asymptotic_normality}
Assume that Assumptions~\ref{assump_HMM_0}--\ref{assump_HMM_grad_3} hold.
Assume that $\rho<1/2$.
Further assume that the contrast function $\ell$ has a unique maximum
	(which is then located at $\thetaTrue\in\Theta$ by \Cref{prop_global_max_l}) 
	and that $\Theta$ is compact, $\thetaTrue$ is an interior point of $\Theta$,
	and the limiting Fisher information matrix $\cI(\thetaTrue)$ (which is defined in \eqref{eq_def_Fisher_information}) is non-singular.
Then, for all $x\in\SpaceX$, we have:
\begin{align*}
\vert \T_n \vert^{1/2} \bigl( \hat{\theta}_{n,x} - \thetaTrue \bigr)
	\underset{n\to\infty}{\overset{(d)}{\longrightarrow}}
\cN(0, \cI(\thetaTrue)^{-1})
\quad \text{under $\Ptrue$,}
\end{align*}
where $\cN(0,M)$ denotes the centered Gaussian distribution with covariance matrix $M$.
\end{theo}

\begin{proof}
The proof is a standard argument and is similar to the proof of \cite[Theorem 1]{bickelAsymptoticNormalityMaximumlikelihood1998}.
Remind that the gradient of $\ell_{n,x}$ vanishes at the MLE $\hat{\theta}_{n,x}$ by definition.
Thus, using a Taylor expansion for $\nabla_\theta \ell_{n,x}$ around $\thetaTrue$, we get:
\begin{equation*}
0 = \nabla_\theta \ell_{n,x}(\hat{\theta}_{n,x})
= \nabla_\theta \ell_{n,x}(\thetaTrue) 
	+ \left( \int_0^1 \nabla_\theta^2 \ell_{n,x}( \thetaTrue + t (\hat{\theta}_{n,x} - \thetaTrue))\, \drv t \right)
		(\hat{\theta}_{n,x} - \thetaTrue) ,
\end{equation*}
where we see $\hat{\theta}_{n,x}$ and $\thetaTrue$ as column vectors.
As $\cI(\thetaTrue)$ is non-singular (indeed definite positive),
remark that Theorems~\ref{thm_Strong_consistency_MLE} and~\ref{thm_LLN_observed_information}
insure that $\Ptrue$-\as for $n$ large enough the integrand in the integral of the above formula is 
	non-singular (indeed definite positive) for all values of $t$,
	and thus the matrix-valued integral is non-singular.
Thus, from the above equation, we obtain $\Ptrue$-\as for $n$ large enough:
\begin{equation*}
\vert \T_n \vert^{1/2} \bigl( \hat{\theta}_{n,x} - \thetaTrue \bigr)
= \left( - \vert \T_n\vert^{-1} \int_0^1 \nabla_\theta^2 \ell_{n,x}( \thetaTrue + t (\hat{\theta}_{n,x} - \thetaTrue))\, \drv t \right)^{-1}
	\vert \T_n \vert^{-1/2} \nabla_\theta \ell_{n,x}(\thetaTrue).
\end{equation*}
As by \Cref{thm_Strong_consistency_MLE}, we have that $\Ptrue$-\as $\lim_{n\to\infty} \hat{\theta}_{n,x} = \thetaTrue$,
	using \Cref{thm_LLN_observed_information}, we get that the first factor in the right hand side
	$\Ptrue$-\as converges to $\cI(\thetaTrue)$ as $n\to\infty$.
Using \Cref{thm_convergence_score}, we get that the second factor in the right hand side
	converges $\Ptrue$-weakly as $n\to\infty$ to the Gaussian random distribution $\cN(0, \cI(\thetaTrue))$.
Hence, using Cram\'er-Slutsky’s theorem, we get that $\vert \T_n \vert^{1/2} \bigl( \hat{\theta}_{n,x} - \thetaTrue \bigr)$
converges $\Ptrue$-weakly as $n\to\infty$ to the Gaussian random distribution $\cN(0, \cI(\thetaTrue)^{-1})$.
This concludes the proof.
\end{proof}

\subsubsection{Proof of Proposition~\ref{prop_conv_Lambda_term}}
	\label{section_proof_prop_conv_Lambda_term}

We are going to prove a version of Proposition~\ref{prop_conv_Lambda_term}
where the functions $\varphi_\theta$ used in \eqref{eq_def_Lambda_ukx} to define $\Lambda_{u,k,x}(\theta)$
are replaced by scalar-valued functions, still denoted by $\varphi_\theta$, under more general assumptions.
The extension to matrix-valued functions is then straightforward
	by applying the result coordinate-wise.

Let $\Theta_0$ be a compact subset of $\Theta$, 
Let $\Theta_0$ be a closed ball in $\Theta$, 
and let $\varphi : \Theta_0 \times \SpaceX^2 \times \SpaceY \to \R$
be a Borel function such that for all $x',x\in\SpaceX$ and $y\in\SpaceY$, 
$\theta \mapsto \varphi(\theta, x', x, y) = \varphi_\theta(x',x,y)$ is a continuous function on $\Theta_0$,
and such that:
\begin{equation*}
\Etrue\left[
		\sup_{\theta\in\Theta_0} \sup_{x,x'\in\SpaceX} \vert \varphi_\theta(x,x',Y_{\rooot}) \vert^2
	\right] < \infty.
\end{equation*}
Let $\Lambda_{u,k,x}(\theta)$ be defined as in \eqref{eq_def_Lambda_ukx} on page~\pageref{eq_def_Lambda_ukx}
and note that it is in $L^2(\Prb_\cU \otimes \Ptrue)$.

The proof of \Cref{prop_conv_Lambda_term} is decomposed into several lemmas.

\medskip

We start with the following lemma, 
stating a uniform $L^2(\Prb_\cU\otimes\Ptrue)$ approximation bound on the quantities $\Lambda_{u,k,x}(\theta)$, 
and the existence of a limit function $\Lambda_{u,\infty}(\theta)$ which does not depend on $x$.
This lemma is an immediate consequence of \Cref{lemma_score_incr_L2_bound}
(remind that $\rho < 1/\sqrt{2}$ under the assumptions of \Cref{prop_conv_Lambda_term})
for the second moment ($b=2$)
with $\psi_\theta = \varphi_\theta$,
see also the discussion after \Cref{lemma_score_incr_L2_bound}
for the existence of the limit function.

\begin{lemme}\label{lemma_Lambda_incr_L2_bound}
Under the assumptions of \Cref{prop_conv_Lambda_term},
there exist finite constants $C < \infty$ and $\alpha\in(0,1)$ such that
for all $u\in\T$ there exists some function $\Lambda_{u,\infty}(\theta)$ in $L^2(\Prb_\cU\otimes\Ptrue)$
such that for all $k\in\N^*$, we have:
\begin{align*}
 \left( \Esp_\cU \otimes \Etrue \left[ \sup_{\theta\in\Theta_0} \sup_{x,\in\SpaceX}  \vert \Lambda_{u,k,x}(\theta) - \Lambda_{u,\infty}(\theta) \vert^2 \right] \right)^{1/2} \leq C \alpha^k .
\end{align*}
\end{lemme}

In particular, for all $u\in\T$, $\theta\in\Theta_0$ and $x\in\SpaceX$, 
the sequence $(\Lambda_{u,k,x}(\theta))_{k\in\N^*}$
converges in $L^2(\Prb_\cU\otimes\Ptrue)$ to $\Lambda_{u,\infty}(\theta)$ which does not depend on $x$.

\medskip

The following lemma gives an exponential bound on the $L^2(\Ptrue)$ norm
uniformly in $x\in\SpaceX$ for the the average of the quantities $\Lambda_{u,\height{u},x}(\thetaTrue)$
over $u\in\T_n^*$.

\begin{lemme}
	\label{lemma_Lambda_conv_unif_x}	
Under the assumptions of \Cref{prop_conv_Lambda_term}, for all $x\in\SpaceX$ and $\theta\in\Theta_0$,
	there exist finite constants $C<\infty$ and $\alpha \in (0,1)$ such that for all $n\in\N^*$
	we have:
\begin{equation*}\label{eq_Lambda_exp_decay_unif_x}
\Etrue\left[  \sup_{x\in\SpaceX} \left\vert
	\inv{\vert \T_n \vert} \sum_{u\in\T_n^*} \Lambda_{u,\height{u},x}(\theta) 
	- \Esp_\cU \otimes \Etrue\bigl[  \Lambda_{\rooot,\infty}(\theta) \bigr]
\right\vert^2 \right]^{1/2}
\leq C \alpha^n .
\end{equation*}
\end{lemme}

\begin{proof}
Let $x'\in\SpaceX$ and $\theta\in\Theta_0$.
Using Minkowski's inequality and Jensen's inequality, for all $n,k\in\N^*$, we get:
\begin{multline}
\Etrue\left[  \sup_{x\in\SpaceX} \left\vert
		\inv{\vert \T_n \vert} \sum_{u\in\T_n^*} \Lambda_{u,\height{u},x}(\theta) 
		- \Esp_\cU \otimes \Etrue \bigl[  \Lambda_{\rooot,\infty}(\theta) \bigr]
	\right\vert^2 \right]^{1/2} \displaybreak[3] \\
\begin{aligned} 
& \leq  \Etrue\left[  \sup_{x,x'\in\SpaceX} \left\vert  \inv{\vert \T_n \vert}
	 		 \sum_{u\in\T_{k-1}^*} \Lambda_{u,\height{u},x}(\theta) 
		\right\vert^2 \right]^{1/2} \\*
& \qquad +  \Etrue\left[  \sup_{x,x'\in\SpaceX}  \left\vert \inv{\vert \T_n \vert}
	 		 \sum_{u\in\T_n\setminus\T_{k-1}} \Lambda_{u,\height{u},x}(\theta)  - \Lambda_{u,k,x'}(\theta) 
		\right\vert^2 \right]^{1/2} \\*
& \qquad + \Etrue\left[   \left\vert
		\inv{\vert \T_n \vert} \sum_{u\in\T_n\setminus\T_{k-1}} \Lambda_{u,k,x'}(\theta) 
		- \Esp_\cU \otimes \Etrue \bigl[  \Lambda_{\rooot,k,x'}(\theta) \bigr]
	\right\vert^2 \right]^{1/2}  \\* 
& \qquad + \Esp_\cU \otimes \Etrue \bigl[  
		\vert \Lambda_{\rooot,k,x'}(\theta) - \Lambda_{\rooot,\infty}(\theta) \vert^2
	\bigr]^{1/2} 
.	
\end{aligned}\label{eq_sum_Lambda_conv_L2_approx}
\end{multline}
Using \Cref{lemma_Lambda_incr_L2_bound} 
together with \eqref{eq_unif_bound_dot_hukx} on page~\pageref{eq_unif_bound_dot_hukx}
(which, remind, are both immediate consequences of \Cref{lemma_score_incr_L2_bound}), 
there exists a finite constant $C<\infty$  and $\beta\in (0,1)$ such that
the first term in the right hand side of \eqref{eq_sum_Lambda_conv_L2_approx}
is upper bounded by $C 2^{-(n-k)}$
(note that $\frac{\vert T_{k-1} \vert}{\vert \T_n \vert} \leq 2^{-(n-k)}$),
and the second and fourth terms in the right hand side of \eqref{eq_sum_Lambda_conv_L2_approx}
are both upper bounded by $C \beta^{k/2}$.

We now give an upper bound for the second term in the right hand side of \eqref{eq_sum_Lambda_conv_L2_approx}.
For a vertex $u$ in $\T\setminus \T_{k-1}$, 
let $v_u\in\G_k$ be the unique vertex that satisfies the shape equality constraint \eqref{eq_def_shape_subtree} (on page \pageref{eq_def_shape_subtree}),
then we have:
\begin{equation}\label{eq_equal_Lambda_ukx_up_to_shape}
\Lambda_{u,k,x'}(\theta; Y_{\Delta(u,k)}=y_{\Delta(u,k)}) 
	=  \Lambda_{v_u,k,x'}(\theta; Y_{\Delta(v_u)}=y_{\Delta(u,k)}) .
\end{equation}
Moreover, using the definition of $\Lambda_{u,k,x}(\theta)$ in \eqref{eq_def_Lambda_ukx} 
	together with the assumption on $\varphi_\theta$ in \Cref{prop_conv_Lambda_term},
	we get that the random variable $\Lambda_{u,k,x'}(\theta; Y_{\Delta(u,k)}=y_{\Delta(u,k)})$ 
	is in $L^2(\Ptrue)$ for every $u\in\T\setminus \T_{k-1}$.
Thus, we can apply \Cref{lemma_ergodic_convergence} (see in particular \eqref{eq_LLN_upper_bound_Esp_M_Gn_main_body})
to the collection of neighborhood-shape-dependent functions $( \Lambda_{v_u,k,x'}(\theta; Y_{\Delta(v)}=\cdot) )_{v\in\G_k}$
(remind that indexing functions with $\G_k$ or with $\ShapeSetValues_k$ is equivalent by \eqref{eq_def_set_possible_shapes}).
Using \eqref{eq_LLN_upper_bound_Esp_M_Gn_main_body} in \Cref{lemma_ergodic_convergence}
together with \eqref{eq_equal_h_ukx_up_to_shape} and \eqref{eq_equality_expectation_root_and_U_k} in \Cref{rem_rerooting_delta},
	we get that
	there exist $\gamma\in (0,1)$ and a finite constant $C'<\infty$
	(note that they both do not depend on $k$ and $n$) such that for all $n,k\in\N^*$ with $n\geq k$,
the second term in the right hand side of \eqref{eq_sum_Lambda_conv_L2_approx} 
	is upper bounded by $C' \gamma^{n-k}$.

Hence, taking $k = \ceil{n / 2}$, we get that the left hand side of \eqref{eq_sum_Lambda_conv_L2_approx} 
is upper bounded by $2 C \beta^{n/4}  + C' \alpha^{n/2} + C 2^{-n/2+1}$,
and thus decays at exponential rate as desired.
This concludes the proof.
\end{proof}

\Cref{lemma_Lambda_conv_unif_x} implies as a corollary
the convergence $\Ptrue$-\as and in $L^2(\Ptrue)$ uniformly in $x\in\SpaceX$
for the the sum of the quantities $\Lambda_{u,\height{u},x}(\thetaTrue)$ over $u\in\T_n^*$.

\begin{corol}
	\label{corol_Lambda_conv_unif_x}	
Under the assumptions of \Cref{prop_conv_Lambda_term}, for all $x\in\SpaceX$ and $\theta\in\Theta_0$, 
	we have:
\begin{equation*}\lim_{n\to\infty}
  \sup_{x\in\SpaceX} \left\vert
	\inv{\vert \T_n \vert} \sum_{u\in\T_n^*} \Lambda_{u,\height{u},x}(\theta) 
	- \Esp_\cU \otimes \Etrue \bigl[  \Lambda_{\rooot,\infty}(\theta) \bigr]
\right\vert
= 0
\quad \text{$\Ptrue$-\as and in $L^2(\Ptrue)$.}
\end{equation*}
\end{corol}

\begin{proof}
The convergence in $L^2(\Ptrue)$ follows immediately from \Cref{lemma_Lambda_conv_unif_x}.
Moreover, using again \Cref{lemma_Lambda_conv_unif_x}, we have:
\begin{equation*}
\sum_{n\in\N^*} \Etrue\left[ \sup_{x\in\SpaceX} \left\vert
		\inv{\vert\T_n\vert} \sum_{u\in\T_n^*} \Lambda_{u,\infty}(\theta)
		- \Esp_\cU \otimes \Etrue [  \Lambda_{\rooot,\infty}(\theta) ]
	\right\vert^2 \right] < \infty .
\end{equation*}
Hence, Borel-Cantelli lemma and Markov's inequality imply that the convergence 
in the lemma also holds $\Ptrue$-\as
\end{proof}

The following lemma gives some continuity properties of the function $\theta \mapsto \Lambda_{\rooot,k,x}(\theta)$.

\begin{lemme}\label{lemma_Lambda_root_continuous}
Under the assumptions of \Cref{prop_conv_Lambda_term}, for all $x\in\SpaceX$ and $k\in\N$,
the random function $\theta \mapsto \Lambda_{\rooot,k,x}(\theta)$ is $\Prb_\cU \otimes \Ptrue$-\as continuous on $\Theta_0$.
Moreover, for all $\theta\in\Theta_0$, we have:
\begin{equation*}
\lim_{\delta\to 0} \Esp_\cU \otimes \Etrue\left[
		\sup_{\theta' \in\Theta_0 : \Vert \theta' - \theta \Vert \leq \delta} 
			\vert \Lambda_{\rooot,k,x}(\theta') - \Lambda_{\rooot,k,x}(\theta) \vert^2
	\right] = 0.
\end{equation*}
\end{lemme}

\begin{proof}
We mimic the proof of \cite[Lemma 14]{doucAsymptoticPropertiesMaximum2004}.

For all $v\in\Tpast$, define the random variable 
$\norm{\varphi^v}_\infty = \sup_{\theta'\in\Theta_0} \sup_{x,x'\in\SpaceX} \vert \varphi_{\theta'}(x',x,Y_v) \vert$.
Remind that under the assumptions of \Cref{prop_conv_Lambda_term},
the HMT process $(X,Y)$ is stationary and the random variable $\norm{\varphi^{\rooot}}_\infty$ is in $L^2(\Ptrue)$.
Thus, for all $v\in\Tpast$, the random variable $\norm{\varphi^v}_\infty$ is in $L^2(\Ptrue)$.
Remind from \eqref{eq_inclusion_Delta_Tpast} on page~\pageref{eq_inclusion_Delta_Tpast}
that $\DeltaR(\rooot,k)$ is a random subtree of the deterministic subtree $\Tpast(\parent^k(u),k)$.
Then, note that we have: 
\begin{equation*}
\sup_{\theta\in\Theta_0} \vert\Lambda_{\rooot,k,x}(\theta)\vert \leq 2 \sum_{v\in\Tpast(\parent^k(\rooot),k)} \norm{\varphi^v}_\infty ,
\end{equation*}
where the upper bound is a random variable in $L^2(\Ptrue)$ (and thus in $L^2(\Prb_\cU\otimes\Ptrue)$)
which depends only on $Y_{\Tpast(\parent^k(u),k)}$ but not on $\cU$.
Hence, to prove the lemma, it suffices to prove that for all $v\in \Tpast(\parent^k(u),k)\setminus\{\parent^k(\rooot)\}$ we have:
\begin{align*}
\lim_{\delta\to 0} \sup_{\theta' \in\Theta_0 : \Vert \theta' - \theta \Vert \leq \delta} \
	\Bigl\vert \Bigr.& \Esp_{\theta'}[ \varphi_{\theta'}(X_{\parent(v)},X_v,Y_v) \,\vert\, Y_{\DeltaR(\rooot,k)}, X_{\parent^k(\rooot)}=x ] \\
	& \Bigl. - \Esp_{\theta}[ \varphi_{\theta}(X_{\parent(v)},X_v,Y_v) \,\vert\, Y_{\DeltaR(\rooot,k)}, X_{\parent^k(\rooot)}=x ]
	\Bigr\vert
= 0, \qquad \Prb_\cU \otimes \Ptrue\text{-\as}
\end{align*}

Denote $x_{\parent^k(\rooot)}=x$, and write:
\begin{multline}
\Esp_{\theta}[ \varphi_{\theta}(X_{\parent(v)},X_v,Y_v) \,\vert\, Y_{\DeltaR(\rooot,k)}, X_{\parent^k(\rooot)}=x ] \\
 = \frac{
		\int_{\SpaceX^{\DeltaR(\rooot,k)\setminus\{\parent^k(\rooot)\}}} \varphi_\theta(x_{\parent(v)},x_v,Y_v)
		\prod_{w\in\DeltaR(\rooot,k)\setminus\{\parent^k(\rooot)\}} q_\theta(x_{\parent(w)},x_w) g_\theta(x_w,Y_w)
			\lambda(\drv x_w)
	}
	{
		\int_{\SpaceX^{\DeltaR(\rooot,k)\setminus\{\parent^k(\rooot)\}}}
		\prod_{w\in\DeltaR(\rooot,k)\setminus\{\parent^k(\rooot)\}} q_\theta(x_{\parent(w)},x_w) g_\theta(x_w,Y_w)
			\lambda(\drv x_w)
	}\cdot 	\label{eq_control_varphi_frac}
\end{multline}
Using Assumptions~\ref{assump_HMM_1}-\ref{assump_HMM_3} 
(which are part of the assumptions in \Cref{prop_conv_Lambda_term}),
we know that the integrand in the numerator of the right hand side of \eqref{eq_control_varphi_frac}
is continuous \wrt $\theta$ and is upper bounded by the random variable 
$\norm{\varphi^v}_\infty (\sigma^+ b^+)^{\vert \Tpast(\parent^k(u),k)\vert-1}$ (remind that $\sigma^+\geq 1$ and $b^+\geq 1$).
And similarly, the denominator is continuous \wrt $\theta$, 
and, using \Cref{assump_HMM_2}-\ref{assump_HMM_2:item2}, is lower bounded by the random variable:
\begin{equation*}
\prod_{w\in\DeltaR(\rooot,k)\setminus\{\parent^k(\rooot)\}} \sigma^- \inf_{\theta'\in\Theta}\int g_{\theta'}(x_w,Y_w) \lambda(\drv x_w) > 0.
\end{equation*}
Hence, using dominated convergence, we conclude that $\Prb_\cU\otimes\Ptrue$-\as the left hand side of \eqref{eq_control_varphi_frac}
is continuous \wrt $\theta$.
This concludes the proof.
\end{proof}

As a corollary of \Cref{lemma_Lambda_root_continuous},
we get that the function $\theta \mapsto \Lambda_{\rooot,\infty}(\theta)$ is continuous in $L^2(\Ptrue)$.

\begin{corol}\label{corol_Lambda_root_infinite_continuous}
Under the assumptions of \Cref{prop_conv_Lambda_term}, for all $\theta\in\Theta_0$, we have:
\begin{equation*}
\lim_{\delta\to 0} \Esp_\cU \otimes \Etrue \left[ 
		\sup_{\theta' \in\Theta_0 : \Vert \theta' - \theta \Vert \leq \delta} 
			\vert \Lambda_{\rooot,\infty}(\theta') - \Lambda_{\rooot,\infty}(\theta) \vert^2
	\right]  = 0.
\end{equation*}
In particular, the function $\theta \mapsto \Esp_\cU  \otimes \Etrue [ \Lambda_{\rooot,\infty}(\theta) ]$ is continuous on $\Theta_0$.
\end{corol}

\begin{proof}
Using Minkowski's inequality and \Cref{lemma_Lambda_incr_L2_bound}, 
there exist a finite constant $C<\infty$ and $\beta\in (0,1)$ such that for all $x\in\SpaceX$ and $k\in\N^*$, we have:
\begin{multline}
 \Esp_\cU \otimes \Etrue \left[ 
		\sup_{\theta' \in\Theta_0 : \Vert \theta' - \theta \Vert \leq \delta} 
			\vert \Lambda_{\rooot,\infty}(\theta') - \Lambda_{\rooot,\infty}(\theta) \vert^2
	 \right]^{1/2} \\
 \leq 	2 C \beta^{k/2}
	+ \Esp_\cU \otimes \Etrue \left[ 
		\sup_{\theta' \in\Theta_0 : \Vert \theta' - \theta \Vert \leq \delta} 
			\vert \Lambda_{\rooot,k,x}(\theta') - \Lambda_{\rooot,k,x}(\theta) \vert^2
	\right]^{1/2} . \label{eq_bound_cont_Lambda_root}
\end{multline}
Using \Cref{lemma_Lambda_root_continuous}, we get:
\begin{equation*}
\limsup_{\delta\to 0}
\Esp_\cU \otimes \Etrue \left[ 
		\sup_{\theta' \in\Theta_0 : \Vert \theta' - \theta \Vert \leq \delta} 
			\vert \Lambda_{\rooot,\infty}(\theta') - \Lambda_{\rooot,\infty}(\theta) \vert^2
	 \right]^{1/2}
\leq 2 C \beta^{k/2},
\end{equation*}
and taking $k\to\infty$, the upper bound vanishes.
This concludes the proof.
\end{proof}

We now prove a locally uniform law of large numbers for the quantities $\Lambda_{u,k,x}(\theta)$.

\begin{lemme}\label{lemma_conv_unif_sum_Lambda_k}
Under the assumptions of \Cref{prop_conv_Lambda_term}, for all $x\in\SpaceX$, we have:
\begin{align*}
\lim_{\delta\to 0} \lim_{n\to\infty} \sup_{\theta'\in\Theta_0 \,:\, \Vert\theta'-\theta\Vert \leq \delta}
	\left\vert
		\inv{\vert\T_n\vert} \sum_{u\in\T_n^*} \Lambda_{u,\height{u},x}(\theta')
		- \Esp_\cU \otimes \Etrue [  \Lambda_{\rooot,\infty}(\theta) ] 
	\right\vert
= 0,
\quad \Ptrue\text{-\as}
\end{align*}
\end{lemme}

\begin{proof}
First, write:
\allowdisplaybreaks[4]
\begin{multline}
 \sup_{\theta'\in\Theta_0 \,:\, \Vert\theta'-\theta\Vert \leq \delta} 
	\left\vert
		\inv{\vert\T_n\vert} \sum_{u\in\T_n^*} \Lambda_{u,\height{u},x}(\theta')
		- \Esp_\cU \otimes \Etrue [  \Lambda_{\rooot,\infty}(\theta) ]
	\right\vert \\
\begin{aligned}
& \leq 
		\inv{\vert\T_n\vert} \sum_{u\in\T_n^*} 
			\sup_{\theta'\in\Theta_0 \,:\, \Vert\theta'-\theta\Vert \leq \delta}
		\bigl\vert \Lambda_{u,\height{u},x}(\theta') - \Lambda_{u,\height{u},x}(\theta) \bigr\vert 
	 \\*
& \qquad\qquad +
	\left\vert
		\inv{\vert\T_n\vert} \sum_{u\in\T_n^*} \Lambda_{u,\height{u},x}(\theta)
		- \Esp_\cU \otimes \Etrue [  \Lambda_{\rooot,\infty}(\theta) ] 
	\right\vert . \label{eq_upper_bound_unif_sup_Lambda_k}
\end{aligned}
\end{multline}
\allowdisplaybreaks[1]

Then, we use the exact same argument as in the proofs of \Cref{lemma_Lambda_conv_unif_x} and \Cref{corol_Lambda_conv_unif_x}
where for all $u\in\T^*$, the random variable $\Lambda_{u,\height{u},x}(\theta)$ is replaced by the random variable:
\begin{equation*}
\sup_{\theta'\in\Theta_0 \,:\, \Vert\theta'-\theta\Vert \leq \delta}
		\bigl\vert \Lambda_{u,\height{u},x}(\theta') - \Lambda_{u,\height{u},x}(\theta) \bigr\vert ,
\end{equation*}
	which are in $L^2(\Ptrue)$ using the assumptions of \Cref{prop_conv_Lambda_term}.
This gives us
that the first term in the upper bound of \eqref{eq_upper_bound_unif_sup_Lambda_k}
converges $\Ptrue$-\as as $n\to\infty$ to:
\begin{equation*}
\Esp_\cU \otimes \Etrue \left[
		\sup_{\theta' : \Vert \theta' - \theta \Vert \leq \delta} 
			\vert \Lambda_{\rooot,\infty}(\theta') - \Lambda_{\rooot,\infty}(\theta) \vert
	 \right],
\end{equation*}
which, by \Cref{corol_Lambda_root_infinite_continuous}, vanishes when $\delta\to 0$.
\Cref{corol_Lambda_conv_unif_x} implies that 
the second term in the upper bound of \eqref{eq_upper_bound_unif_sup_Lambda_k}
vanishes $\Ptrue$-\as when $n\to\infty$.
This concludes the proof. \end{proof}

Combining the previous lemmas in this subsection, we are now ready to prove \Cref{prop_conv_Lambda_term}.

\begin{proof}[Proof of \Cref{prop_conv_Lambda_term}]
By \Cref{lemma_Lambda_incr_L2_bound}, for all $u\in\T$, we have that $(\Lambda_{u,k,x}(\theta))_{k\in\N^*}$
is a Cauchy sequence uniformly \wrt $\theta\in\Theta_0$ in $L^2(\Prb_\cU\otimes\Ptrue)$
that converges to some limit $\Lambda_{u,\infty}(\theta)$
	(that does not depend on $x$).
By \Cref{corol_Lambda_conv_unif_x}, we have that $\Ptrue$-\as
the convergence for the the average of the quantities $\Lambda_{u,\height{u},x}(\thetaTrue)$ over $u\in\T_n^*$ holds uniformly in $x\in\SpaceX$,
that is, \eqref{eq_conv_as_sum_Lambda_ukx} in \Cref{prop_conv_Lambda_term} holds.
By \Cref{corol_Lambda_root_infinite_continuous},
we have that the function $\theta \mapsto \Esp_\cU  \otimes \Etrue [ \Lambda_{\rooot,\infty}(\theta) ]$ is continuous on $\Theta_0$.
Finally, the last part of the proposition is given by \Cref{lemma_conv_unif_sum_Lambda_k}.
\end{proof}

\subsubsection{Proof of Proposition \ref{prop_conv_Gamma_term}}
	\label{section_proof_prop_conv_Gamma_term}

Similarly to what we have done for Proposition~\ref{prop_conv_Lambda_term},
we are going to prove a version of Proposition~\ref{prop_conv_Gamma_term}
where the functions $\phi_\theta$ used in \eqref{eq_def_Gamma_ukx} to define $\Gamma_{u,k,x}(\theta)$
are replaced by scalar-valued functions, still denoted by $\varphi_\theta$, under more general assumptions.
The extension to matrix-valued functions is then straightforward
	by applying the result coordinate-wise.

Let $\Theta_0$ be a compact subset of $\Theta$, 
Let $\Theta_0$ be a closed ball in $\Theta$, 
and let $\phi : \Theta_0 \times \SpaceX^2 \times \SpaceY \to \R$
be a Borel function such that for all $x',x\in\SpaceX$ and $y\in\SpaceY$, 
$\theta \mapsto \phi(\theta, x', x, y) = \phi_\theta(x',x,y)$ is a continuous function on $\Theta_0$,
and such that:
\begin{equation*}
\Etrue\left[
		\sup_{\theta\in\Theta_0} \sup_{x,x'\in\SpaceX} \vert \phi_\theta(x,x',Y_{\rooot}) \vert^4
	\right] < \infty.
\end{equation*}
Let $\Gamma_{u,k,x}(\theta)$ be defined as in \eqref{eq_def_Gamma_ukx} on page~\pageref{eq_def_Gamma_ukx}
and note that it is in $L^2(\Prb_\cU \otimes \Ptrue)$.

\medskip

The proof of Proposition \ref{prop_conv_Lambda_term} can be straightforwardly adapted to  \Cref{prop_conv_Gamma_term}
	except for Lemma~\ref{lemma_Lambda_incr_L2_bound}.
Thus, for brevity, we only present the adaptation of Lemma~\ref{lemma_Lambda_incr_L2_bound} 
	to the terms $\Gamma_{u,k,x}(\theta)$.
(The details of the adaptation for the rest of the proof of Proposition \ref{prop_conv_Lambda_term}
	to the terms $\Gamma_{u,k,x}(\theta)$ can be found in \Cref{appendix_proof_prop_conv_Gamma_term}.)

We start with two lemmas giving coupling bounds that will be used to control the covariance terms 
that appear in the definition of the terms $\Gamma_{u,k,x}(\theta)$.
The following lemma is a variant of \Cref{lemme_exp_coupling_HMT} 
for two vertices of $\T$.

\begin{lemme}[Forward coupling bound for two vertices]
	\label{lemma_forward_coupling_two_vertices}
Assume that Assumptions \ref{assump_HMM_1} and \ref{assump_HMM_2} hold.
Then, for all $u,v\in\T$, all $y_{\T_n}\in \SpaceY^{\vert \T_n \vert}$ (and $n\in\N$)
and all initials distributions $\nu$ and $\nu'$ on $\SpaceX$,
we have:
\begin{equation*}
\Bignorm{ \int_{\SpaceX} 
		\Prb_\theta\Bigl( X_u \in \cdot, X_v \in \cdot \, \Bigm\vert \,  Y_{\T_n}=y_{\T_n}, X_{\rooot} = x  \Bigr) 
		[ \nu(\drv x) - \nu'(\drv x) ]
	}_{\mathrm{TV}}
\leq 2 \, \rho^{\min(\height{u}, \height{v})} .
\end{equation*}
\end{lemme}

For simplicity, \Cref{lemma_forward_coupling_two_vertices} is stated with $\rooot$ as the initial vertex,
but note that the results still holds when replacing $\rooot$ and $\T_n$
by $v$ and $\T(v,n)$ for any $v\in\Tpast$.
We shall reuse this fact later.

\begin{proof}
We are going to construct a coupling that achieves this minimum.
Denote by $((X_w', Y_w'), \linebreak w\in\T)$ the process started from the distribution $\nu'$
	(and similarly without the $'$).
Remark that we only need to define the (joint) coupling for the variables 
	$X_w$ and $X_w'$ for $w$ on the paths between the vertices $\rooot$, $u$ and $v$.

\Cref{lemme_exp_coupling_HMT} applied to the vertex $u\land v$ gives us a coupling
	for the variables $X_w$ and $X_w'$ for $w$ on the path between $\rooot$ and $u\land v$
	with successful coupling probability upper bounded by $1 - \rho^{\height{u\land v}}$.
On this successful coupling event before or on $u\land v$, the two processes are still defined to be equal
	after the fork on both branches leading to $u$ and $v$.

On the complementary event (no successful coupling before or on $u\land v$),
	we get two new distributions $\nu_{u\land v}$ and $\nu_{u\land v}'$ 
	for the variables $X_{u\land v}$ and $X_{u\land v}'$, respectively.
Note that conditioned on the value of $X_{u\land v}$, the two branches leading to $u$ and $v$ are independent.
Thus, applying \Cref{lemme_exp_coupling_HMT} to $u$ (resp. $v$) with the initial distributions $\nu_{u\land v}$ and $\nu_{u\land v}'$,
	we construct a coupling of the processes $X$ and $X'$ on the branch from $u\land v$ to $u$ (resp. $v$)
	with successful coupling probability upper bounded by $1 - \rho^{\height{u} - \height{u\land v}}$ 
		(resp. $1 - \rho^{\height{v} - \height{u\land v}}$).

Hence, the probability that we do not get a successful coupling on at least one of the two variables $X_u$ and $X_v$ 
	is upper bounded by $\rho^{\height{u\land v}} ( \rho^{\height{u} - \height{u\land v}} + \rho^{\height{v} - \height{u\land v}} )
	\leq 2 \rho^{\min(\height{u}, \height{v})}$.
This concludes the proof.
\end{proof}

The following lemma is a variant of \Cref{lemma_backward_mixing_bound} 
giving a ‘‘backward in time’’ coupling bound
for two vertices of $\T$.

\begin{lemme}[Backward coupling bound for two vertices]
	\label{lemma_backward_coupling_two_vertices}
Assume that Assumptions~\ref{assump_HMM_1}--\ref{assump_HMM_2} hold.
Let $k\in\N^*$, $x\in\SpaceX$ and $u\in \T$,
and let $v,w\in \Tpast(\parent^k(u),k) \setminus\{u\}$.
Then, we have:
\begin{align*}
\bigl\Vert  \Prb_\theta (  X_{v} \in \cdot, X_{w} \in \cdot
		\,\vert\, Y_{\DeltaR(u,k)}, X_{\parent^k(u)}=x ) 
 -  \Prb_\theta (  X_{v} \in \cdot, X_{w} \in \cdot
		\,\vert\, Y_{\DeltaR^*(u,k)}, X_{\parent^k(u)}=x )
	\bigr\Vert_{\text{TV}} \\
 \leq  2 \, \rho^{\min(d(u,v), d(u,w)) -1}.
\end{align*}
\end{lemme}

\begin{proof}
The idea of the proof is similar to that of \Cref{lemma_forward_coupling_two_vertices}.
We explicitly construct a coupling with coupling failure probability upper bounded by $2 \, \rho^{\min(d(u,v), d(u,w)) -1}$.
Denote by $w_0$ the vertex on the path between $v$ and $w$ that is the closest to $u$,
	and note that we have $w_0 \in \Tpast(\parent^k(u),k) \setminus\{u\}$.
Note that $w_0$ is on the path from $\parent(u)$ to $v$, and thus
	$\dgr(\parent(u), v) = \dgr(\parent(u), w_0) + \dgr(w_0, v)$, and similarly when replacing $v$ by $w$.
On the path from $\parent(u)$ to $w_0$, we use the coupling provided 
	by the ‘‘backward in time’’ coupling bound from \Cref{lemma_backward_mixing_bound}
	with successful probability $1 -  \rho^{\dgr(\parent(u),w_0)}$.
On the path from $w_0$ to $v$ and the other path from $w_0$ to $w$,
which are independent using the Markov property,
	we use two independent couplings that are constructed 
		using a similar coupling argument as in \Cref{lemma_backward_mixing_bound} 
		with $\parent(u)$ replaced by $w_0$.
Those independent couplings have successful probabilities $1 - \rho^{\dgr(w_0,v)}$ and 
		$1 - \rho^{\dgr(w_0,w)}$, respectively.
Note that the coupling we have constructed has a coupling failure probability upper bounded by $2 \, \rho^{\min(d(u,v), d(u,w)) -1}$.
\end{proof}

For brevity, for all $u\in\T$ we will denote $\phi_{\theta,u} = \phi_\theta(X_{\parent(u)}, X_u, Y_u)$ 
	and $\norm{\phi_u}_\infty = \sup_{\theta\in\Theta_0} \linebreak \sup_{x,x'\in\SpaceX}  \vert \phi_\theta(x',x, Y_u) \vert$.
The following lemma gives several upper bounds 
on the covariance terms that appear in the definition of the terms $\Gamma_{u,k,x}(\theta)$.
Remind from \eqref{eq_inclusion_Delta_Tpast} on page~\pageref{eq_inclusion_Delta_Tpast}
that $\DeltaR(\rooot,k)$ is a random subtree of the deterministic subtree $\Tpast(\parent^k(u),k)$.

Note that this lemma is stated under the assumptions of \Cref{prop_conv_Gamma_term},
but here we do not need the assumption that $\rho<1/2$ for the mixing rate $\rho$ of the HMT process $(X,Y)$.

\begin{lemme}
	\label{lemma_bound_covar_terms_Gamma}
Under the assumptions of \Cref{prop_conv_Gamma_term}
(without the need for the assumption that $\rho<1/2$),
for all $x,x'\in\SpaceX$, $\theta\in\Theta_0$, $k'\geq k>0$ and $u\in\T$, 
and for all $v,w\in \Tpast(\parent^k(u),k) \setminus\{\parent^k(u)\}$, we have: 
\begin{equation}\label{eq_lemma_covar_bound_1}
\vert \Cov_\theta[ \phi_{\theta, v}, \phi_{\theta, w} \,\vert\, Y_{\DeltaR(u,k)}, X_{\parent^k(u)}=x ] \vert
	\leq 2 \, \norm{\phi_v}_\infty \norm{\phi_w}_\infty \,  \rho^{d(v,w)-2} , 
\end{equation}
and,
\begin{multline}\label{eq_lemma_covar_bound_2}
\vert \Cov_\theta[ \phi_{\theta, v}, \phi_{\theta, w} \,\vert\, Y_{\DeltaR(u,k)}, X_{\parent^k(u)}=x ]
	- \Cov_\theta[ \phi_{\theta, v}, \phi_{\theta, w} \,\vert\, Y_{\DeltaR(u,k')}, X_{\parent^{k'}(u)}=x' ] \vert \\
	\leq 8 \, \norm{\phi_v}_\infty \norm{\phi_w}_\infty  \,  \rho^{\min( d(\parent^k(u),v), d(\parent^k(u),w) ) -2} . 
\end{multline}
Moreover, if  $v,w\in \DeltaR^*(u,k)$, then we have:
\begin{multline}\label{eq_lemma_covar_bound_3}
\vert \Cov_\theta[ \phi_{\theta, v}, \phi_{\theta, w} \,\vert\, Y_{\DeltaR(u,k)}, X_{\parent^k(u)}=x ] 	
	- \Cov_\theta[ \phi_{\theta, v}, \phi_{\theta, w} \,\vert\, Y_{\DeltaR^*(u,k)}, X_{\parent^k(u)}=x ] \vert \\
	\leq 8 \, \norm{\phi_v}_\infty \norm{\phi_w}_\infty \, \rho^{\min( d(u,v) , d(u,w) )-2} .
\end{multline}
\end{lemme}

\begin{proof}
We start by proving \eqref{eq_lemma_covar_bound_1}, that is, the first inequality in the lemma.
Let $A_1, A_2, B_1, B_2$ be measurable subsets of $\SpaceX$,
and we write $A = A_1 \times A_2$ and $B = B_1 \times B_2$.
If one of the two vertices $v$ and $w$ is an ancestor of the other,
say $w$ is an ancestor of $v$ (which implies that $w$ is also an ancestor of $\parent(v)$),
then using the Markov property of the HMT process $(X,Y)$, we get:
\begin{multline*}
\bigl\vert \Prb_\theta( X_{\{\parent(v),v\}}\in A, X_{\{\parent(w),w\}}\in B \,\vert\, Y_{\DeltaR(u,k)}, X_{\parent^k(u)}=x) \bigr. \\
	\ \bigl. - \Prb_\theta( X_{\{\parent(v),v\}}\in A \,\vert\, Y_{\DeltaR(u,k)}, X_{\parent^k(u)}=x)
		\Prb_\theta( X_{\{\parent(w),w\}}\in B \,\vert\, Y_{\DeltaR(u,k)}, X_{\parent^k(u)}=x) \bigr\vert \\
\begin{aligned}
&  = \Prb_\theta( X_{\{\parent(w),w\}}\in B \,\vert\, Y_{\DeltaR(u,k)}, X_{\parent^k(u)}=x) \\
& \qquad \times \Bigl\vert \int_{\SpaceX^2} \ind_{\{ (x_{\parent(v)},x_v) \in A \}} \, \Prb_\theta( X_v \in \drv x_v \,\vert\, Y_{\DeltaR(u,k)}, X_{\parent(v)}=x_{\parent(v)}) \Bigr. \\
& \quad \qquad \qquad  \times 
	[ \Prb_\theta( X_{\parent(v)}\in\drv x_{\parent(v)} \,\vert\, Y_{\DeltaR(u,k)}, X_w\in B_2, X_{\parent^k(u)}=x) \\
& \quad \qquad \qquad	 \qquad \qquad\qquad
		\Bigl. - \Prb_\theta( X_{\parent(v)}\in\drv x_{\parent(v)} \,\vert\, Y_{\DeltaR(u,k)}, X_{\parent^k(u)}=x) ] \Bigr\vert \\
&  \leq \rho^{\dgr(\parent(v),w)} \\
&  = \rho^{\dgr(v,w)-1} ,
\end{aligned}
\end{multline*}
where the inequality follows
using the same argument as in the proof of the ‘‘backward in time’’ coupling \Cref{lemma_backward_mixing_bound}
(with the role of $\parent(u)$ replaced by $w$ 
	and using the initial distributions $\Prb((X_{\parent(w)},X_w)\in\cdot \,\vert\, Y_{\DeltaR(u,k)}, X_w\in B_2, X_{\parent^k(u)}=x)$
	with $B' = B$ and $\SpaceX$ respectively).

Otherwise, we have that both $\parent(v)$ and $\parent(w)$ are on the path between $v$ and $w$,
and similarly to the first case, we get:
\begin{multline*}
\bigl\vert \Prb_\theta( X_{\{\parent(v),v\}}\in A, X_{\{\parent(w),w\}}\in B \,\vert\, Y_{\DeltaR(u,k)}, X_{\parent^k(u)}=x) \bigr. \\
	\ \bigl. - \Prb_\theta( X_{\{\parent(v),v\}}\in A \,\vert\, Y_{\DeltaR(u,k)}, X_{\parent^k(u)}=x)
		\Prb_\theta( X_{\{\parent(w),w\}}\in B \,\vert\, Y_{\DeltaR(u,k)}, X_{\parent^k(u)}=x) \bigr\vert \\
\begin{aligned}
&  = \Prb_\theta( X_{\{\parent(w),w\}}\in B \,\vert\, Y_{\DeltaR(u,k)}, X_{\parent^k(u)}=x) \\
& \qquad \times \Bigl\vert \int_{\SpaceX^2} \ind_{\{ (x_{\parent(v)},x_v) \in A \}} \, \Prb_\theta( X_v \in \drv x_v \,\vert\, Y_{\DeltaR(u,k)}, X_{\parent(v)}=x_{\parent(v)}) \Bigr. \\
& \quad \qquad \qquad  \times 
	[ \Prb_\theta( X_{\parent(v)}\in\drv x_{\parent(v)} \,\vert\, Y_{\DeltaR(u,k)}, X_{\parent(w)}\in B_1, X_{\parent^k(u)}=x) \\
& \quad \qquad \qquad	 \qquad \qquad\qquad
		\Bigl. - \Prb_\theta( X_{\parent(v)}\in\drv x_{\parent(v)} \,\vert\, Y_{\DeltaR(u,k)}, X_{\parent^k(u)}=x) ] \Bigr\vert \\
&  \leq \rho^{\dgr(\parent(v),\parent(w))} \\
&  = \rho^{\dgr(v,w)-2} . 
\end{aligned}
\end{multline*}
Thus, in both case, we get:
\begin{multline*}
\bigl\vert \Prb_\theta( X_{\{\parent(v),v\}}\in A, X_{\{\parent(w),w\}}\in B \,\vert\, Y_{\DeltaR(u,k)}, X_{\parent^k(u)}=x) \bigr. \\
	\ \bigl. - \Prb_\theta( X_{\{\parent(v),v\}}\in A \,\vert\, Y_{\DeltaR(u,k)}, X_{\parent^k(u)}=x)
		\Prb_\theta( X_{\{\parent(w),w\}}\in B \,\vert\, Y_{\DeltaR(u,k)}, X_{\parent^k(u)}=x) \bigr\vert \\
\leq \rho^{\dgr(v,w)-2} . 
\end{multline*}
This gives that \eqref{eq_lemma_covar_bound_1} holds. (Note that the functions $\phi_{\theta, v}$ and $\phi_{\theta, w}$ can take positive, null or negative values.)

For \eqref{eq_lemma_covar_bound_2}, that is, the second inequality in the lemma, use the decomposition:
\begin{multline*}
\vert \Cov_\theta[ \phi_{\theta, v}, \phi_{\theta, w} \,\vert\, Y_{\DeltaR(u,k)}, X_{\parent^k(u)}=x ] 
	- \Cov_\theta[ \phi_{\theta, v}, \phi_{\theta, w} \,\vert\, Y_{\DeltaR(u,k')}, X_{\parent^{k'}(u)}=x' ] \vert \\
\begin{aligned}
&\leq \vert \Esp_\theta[ \phi_{\theta, v} \, \phi_{\theta, w} \,\vert\, Y_{\DeltaR(u,k)}, X_{\parent^k(u)}=x ] 
	- \Esp_\theta[ \phi_{\theta, v} \, \phi_{\theta, w} \,\vert\, Y_{\DeltaR(u,k')}, X_{\parent^{k'}(u)}=x' ] \vert \\
& \quad + \vert \Esp_\theta[ \phi_{\theta, v} \,\vert\, Y_{\DeltaR(u,k)}, X_{\parent^k(u)}=x ] 
	- \Esp_\theta[ \phi_{\theta, v} \,\vert\, Y_{\DeltaR(u,k')}, X_{\parent^{k'}(u)}=x' ] \vert \\
& \qquad \qquad \qquad \qquad
	\times \vert \Esp_\theta[ \phi_{\theta, w} \,\vert\, Y_{\DeltaR(u,k)}, X_{\parent^k(u)}=x ]  \vert \\
& \quad + \vert \Esp_\theta[ \phi_{\theta, w} \,\vert\, Y_{\DeltaR(u,k)}, X_{\parent^k(u)}=x ] 
	- \Esp_\theta[ \phi_{\theta, w} \,\vert\, Y_{\DeltaR(u,k')}, X_{\parent^{k'}(u)}=x' ] \vert \\
& \qquad \qquad \qquad \qquad
	\times \vert \Esp_\theta[ \phi_{\theta, v} \,\vert\, Y_{\DeltaR(u,k')}, X_{\parent^{k'}(u)}=x' ]  \vert ,
\end{aligned}
\end{multline*}
and then use the joint coupling \Cref{lemma_forward_coupling_two_vertices} with $v'=\parent(v)$ and $w'=\parent(w)$
	for the first term in the upper bound,
and use the coupling \Cref{lemme_exp_coupling_HMT} for the other two terms
	with $\parent(v)$ and $\parent(w)$, respectively.

For \eqref{eq_lemma_covar_bound_3}, that is, the third inequality in the lemma, use a similar decomposition as for \eqref{eq_lemma_covar_bound_2},
and then use the ‘‘backward in time’’ coupling for two vertices from \Cref{lemma_backward_coupling_two_vertices} for the first term in the upper bound,
	and use the ‘‘backward in time’’ coupling \Cref{lemma_backward_mixing_bound} for the other two terms.
This gives an upper bound of $8 \, \norm{\phi_v}_\infty \norm{\phi_w}_\infty \, \rho^{m}$
with $m=\min\{ \dgr(\parent(u),w_0) \,: \, w_0\in\{\parent(v),v,\parent(w),w\} \}$.
Noting $m\leq \min( \dgr(u,v), \dgr(u,w) ) - 2$, we get that \eqref{eq_lemma_covar_bound_3} holds.
This concludes the proof of the lemma.
\end{proof}

We are now ready to prove the following lemma which is the adaptation of Lemma~\ref{lemma_Lambda_incr_L2_bound} 
	to the terms $\Gamma_{u,k,x}(\theta)$,
	and which gives us a uniform $L^2(\Ptrue)$ approximation bound.

Note that the condition $\rho < 1/2$ 
on the mixing rate $\rho$ of the HMT process $(X,Y)$
is due to the coupling bounds from \Cref{lemma_bound_covar_terms_Gamma}
and the grouping of terms used in the proof of \Cref{lemma_Gamma_incr_L2_bound}
(the upper bounds at the end of the proof only add a constant multiplicative factor).
See the discussion in \Cref{rem_rho_smaller_than_2}.

\begin{lemme}\label{lemma_Gamma_incr_L2_bound}
Under the assumptions of \Cref{prop_conv_Gamma_term},
there exists a positive constant $C<\infty$ such that
for all $u\in\T$ and $0 < k \leq k'$, we have:
\begin{align*}
&  \Esp_\cU \otimes \Etrue \left[ \sup_{\theta\in\Theta_0} \sup_{x,x'\in\SpaceX}  \vert \Gamma_{u,k,x}(\theta) - \Gamma_{u,k',x'}(\theta) \vert^2 \right]^{1/2}  \\
& \qquad\qquad\qquad\qquad\qquad\qquad \leq C \, \Etrue\left[
		\sup_{\theta\in\Theta_0} \sup_{x,x'\in\SpaceX} \vert \varphi_\theta(x,x',Y_{\rooot}) \vert^4
	\right]^{1/2}
	k^2 \, (2 \rho)^{  k/3  } .
\end{align*}
\end{lemme}

\begin{proof}\allowdisplaybreaks[3]
Let $u$, $k$ and $k'$ be as in the lemma.
Similarly to the proof of \Cref{lemma_score_incr_L2_bound}, we use the bounds from \Cref{lemma_bound_covar_terms_Gamma}
	and Minkowski's inequality to bound the left hand side of the inequality in the lemma.
For a finite subset $I\subset \Tpast$, we write $S_I = \sum_{v\in I} \phi_{\theta,v}$ (the dependence on $\theta$ is implicit).
Similarly to the proof of \cite[Lemma~17]{doucAsymptoticPropertiesMaximum2004},
the difference $\Gamma_{u,k,x}(\theta) - \Gamma_{u,k',x'}(\theta)$ 
may be rewritten as $A + 2B +C + D + 2E + 2F$, 
where all those terms are random variables which depend on $Y_{\DeltaR(u,k')}$
and implicitly on $\cU$, and are define as:
\begin{align*}
A  = & \ \Var_\theta[ S_{\DeltaR^*(u,k)} \,\vert\, Y_{\DeltaR(u,k)}, X_{\parent^k(u)}=x ]
	- \Var_\theta[ S_{\DeltaR^*(u,k)} \,\vert\, Y_{\DeltaR^*(u,k)}, X_{\parent^k(u)}=x ] \\*
& \	- \Var_\theta[ S_{\DeltaR^*(u,k)} \,\vert\, Y_{\DeltaR(u,k')}, X_{\parent^{k'}(u)}=x' ]
	+ \Var_\theta[ S_{\DeltaR^*(u,k)} \,\vert\, Y_{\DeltaR^*(u,k')}, X_{\parent^{k'}(u)}=x' ] , \\
B = & \ \Cov_\theta[ S_{\DeltaR^*(u,k)}, \phi_{\theta,u} \,\vert\, Y_{\DeltaR(u,k)}, X_{\parent^k(u)}=x ] \\*
& \	- \Cov_\theta[ S_{\DeltaR^*(u,k)}, \phi_{\theta,u} \,\vert\, Y_{\DeltaR(u,k')}, X_{\parent^{k'}(u)}=x' ] , \\
C = & \ \Var_\theta[ \phi_{\theta,u} \,\vert\, Y_{\DeltaR(u,k)}, X_{\parent^k(u)}=x ]
	- \Var_\theta[ \phi_{\theta,u} \,\vert\, Y_{\DeltaR(u,k')}, X_{\parent^{k'}(u)}=x' ], \\
D = & \ \Var_\theta[ S_{\DeltaR^*(u,k') \setminus \DeltaR^*(u,k)} \,\vert\, Y_{\DeltaR(u,k')}, X_{\parent^{k'}(u)}=x' ] \\*
& \	- \Var_\theta[ S_{\DeltaR^*(u,k') \setminus \DeltaR^*(u,k)} \,\vert\, Y_{\DeltaR^*(u,k')}, X_{\parent^{k'}(u)}=x' ] , \\
E = & \ \Cov_\theta[ S_{\DeltaR^*(u,k') \setminus \DeltaR^*(u,k)},  S_{\DeltaR^*(u,k)}
				\,\vert\, Y_{\DeltaR(u,k')}, X_{\parent^{k'}(u)}=x' ] \\*
& \	- \Cov_\theta[ S_{\DeltaR^*(u,k') \setminus \DeltaR^*(u,k)} , S_{\DeltaR^*(u,k)}
				\,\vert\, Y_{\DeltaR^*(u,k')}, X_{\parent^{k'}(u)}=x' ] , \\
F = & \ \Cov_\theta[ S_{\DeltaR^*(u,k') \setminus \DeltaR^*(u,k)},  \phi_{\theta,u}
				\,\vert\, Y_{\DeltaR(u,k')}, X_{\parent^{k'}(u)}=x' ] \\*
& \	- \Cov_\theta[ S_{\DeltaR^*(u,k') \setminus \DeltaR^*(u,k)} , \phi_{\theta,u}
				\,\vert\, Y_{\DeltaR^*(u,k')}, X_{\parent^{k'}(u)}=x' ] .
\end{align*}

Using Minkowski's inequality, we will upper bound each of those six terms separately.
First remark using Cauchy-Schwarz inequality, the stationarity of the process $((X_u,Y_u), u\in\Tpast)$
	and the assumptions in the proposition, that we have
	$\Etrue[ \norm{\phi_v}_\infty^2 \norm{\phi_w}_\infty^2 ] \leq \Etrue[ \norm{\phi_\rooot}_\infty^4] < \infty$
	for all $v,w\in\Tpast$.
	
Remind from \eqref{eq_inclusion_Delta_Tpast} on page~\pageref{eq_inclusion_Delta_Tpast}
that $\Delta(u,k)$ is a random subtree of the deterministic subtree $\Tpast(\parent^k(u),k)$.
\medskip

\textbf{Upper bound for A:}
Applying the three inequalities in \Cref{lemma_bound_covar_terms_Gamma} and Minkowski's inequality, 
	we get that $\Etrue[ \vert A \vert^2 ]^{1/2}$
 	is upper bounded (up to the factor $\Etrue[ \norm{\phi_\rooot}_\infty^4]$) by:
\begin{align}
 &2   \sum_{v,w\in \Tpast(\parent^k(u),k)} \bigl( 2 \times 8 \rho^{\min(d(v,u), d(w,u)) -2} 
 					\land 2 \times 8 \rho^{\min( d(v, \parent^k(u)) , d(w, \parent^k(u)) ) -2} 
 					\land 4 \times 2 \rho^{d(v,w) -2 } \bigr) \nonumber\\
\leq & \ \frac{32}{\rho^2}   \sum_{v,w\in \Tpast(\parent^k(u),k)} \bigl(  \rho^{\min(d(v,u), d(w,u)) } 
 					\land  \rho^{\min( d(v, \parent^k(u)) , d(w, \parent^k(u)) ) } 
 					\land  \rho^{d(v,w) } \bigr) .
			\label{eq_Gamma_bound_A}
\end{align}
Note that the value of this sum does not depend on the choice of $u\in\T$.

For all $j\in\N$, denote $u_j = \parent^j(u)$.
We will divide the sum in the upper bound of \eqref{eq_Gamma_bound_A}
	according to four cases: $v,w \in \Tpast(u_k,k) \setminus \Tpast(u_{\floor{k/3}},\floor{k/3})$,
		or $v,w \in \Tpast(u_{\floor{2k/3}},\floor{2k/3})$,
		or $v\in \Tpast(u_k,k) \setminus \Tpast(u_{\floor{2k/3}},\floor{2k/3})$ and $w \in \Tpast(u_{\floor{k/3}},\floor{k/3})$
		or similarly exchanging the roles of $v$ and $w$.
Note that those conditions are non-exclusive and we will count some vertices several times, but this is not a problem.

Let $i,j\in\N$ be such that $u \land v = u_i$ and $u \land w = u_j$,
and let $a,b\in\N$ be such that $a = d(u_i, v)$ and $b=d(u_j,w)$.
Note that for $v,w$ in the first case, either $\min(d(v,u), d(w,u))$ or $d(v,w)$ is large,
and thus using elementary computation we upper bound the sum for $v,w$ in the first case by:
\begin{align*}
& \sum_{i=\floor{k/3}}^k \sum_{j=\floor{k/3}}^k \sum_{a=0}^i \sum_{b=0}^j 2^{a+b} 
					\bigl( \rho^{\min(i+a, j+b)} \land \rho^{a +b +\vert j-i \vert} \bigr) \\
\leq & \ 2  \sum_{i=\floor{k/3}}^k \sum_{j=i}^k \sum_{a=0}^i \sum_{b=0}^j 2^{a+b} 
					\bigl( \rho^{i+\min(a,b)} \land \rho^{a +b } \bigr) \\
\leq & \ \frac{8 (1-\rho) }{(1-2\rho)^3}\, k\, (2\rho)^{\floor{k/3}+1}.
\end{align*}
Note that for $v,w$ in the second case, either $\min( d(v, \parent^k(u)) , d(w, \parent^k(u))$ or $d(v,w)$ is large,
and thus using elementary computation we upper bound the sum for $v,w$ in the second case by:
\begin{align*}
 \sum_{i=0}^{\floor{2k/3}} \sum_{j=0}^{\floor{2k/3}} \sum_{a=0}^i \sum_{b=0}^j 2^{a+b} 
				\bigl( \rho^{\min( k-i + a , k-j +b)} \land \rho^{a +b +\vert j-i \vert} \bigr) 
\leq  \ \frac{8 (1-\rho) }{(1-2\rho)^3}\, k\, (2\rho)^{\floor{k/3}+1}.
\end{align*}
Note that for $v,w$ in the third and fourth case, either $\min( d(v, \parent^k(u)) , d(w, \parent^k(u))$ or $d(v,w)$ is large,
and thus we upper bound the sum for $v,w$ in the third and fourth case by:
\begin{equation*}
 2 \sum_{i=0}^{\floor{k/3}} \sum_{j=\floor{2k/3}}^{k} \sum_{a=0}^i \sum_{b=0}^j 2^{a+b} 
				 \rho^{a +b +\vert j-i \vert}  
\leq  \   \frac{2 }{1-2\rho} \, k^2 \, \rho^{ k/3 -1} .
\end{equation*}
Putting those three upper bounds together,
we get that $\Etrue[\vert A\vert^2]^{1/2}$ is upper bounded by an expression as in the lemma.
\medskip

\textbf{Upper bound for B:}
Using the first and second inequalities in \Cref{lemma_bound_covar_terms_Gamma} and Minkowski's inequality, 
we get that $\Etrue[ \vert B \vert^2 ]^{1/2}$
 	is upper bounded (up to the factor $\Etrue[ \norm{\phi_\rooot}_\infty^4]$) by:
\begin{align*}
 8   \sum_{v\in \DeltaR^*(u,k)} ( \rho^{d(v,u)-2} \land \rho^{d(v, \parent^k(u))-2} ) 
  \leq C \,  k \, \bigl( \max(\rho, 2\rho^2) \bigr)^{k/2} ,
\end{align*}
where we used the same computation as in the proof of \Cref{lemma_score_incr_L2_bound}
and $C<\infty$ is some finite constant (which depends only on $\rho$).
\medskip

\textbf{Upper bound for C:}
Using the second equation in \Cref{lemma_bound_covar_terms_Gamma},
	we get that $\Etrue[\vert C \vert^2 ]^{1/2} \leq 8 \rho^{k-2}  \Etrue[ \norm{\phi_\rooot}_\infty^2]$.
\medskip

\textbf{Upper bound for D:}
Using the first and third equation in \Cref{lemma_bound_covar_terms_Gamma}, Minkowski's inequality
	and elementary computation,
	we get that $\Etrue[ \vert D \vert^2 ]^{1/2}$
 	is upper bounded (up to the factor $\Etrue[ \norm{\phi_\rooot}_\infty^4]$) by:
\begin{align*}
   \sum_{v,w\in \Tpast(u_{k'},k')\setminus\Tpast(u_k,k)}
 					( 2\times 2\rho^{d(v,w)-2} \land 8 \rho^{\min(d(v,u), d(w,u))-2} ) 
\leq  \  \frac{96}{\rho^2(1-2\rho)^4} \, k \, (2\rho)^k .
\end{align*}
\medskip

\textbf{Upper bound for E:}
Using the first and third equation in  \Cref{lemma_bound_covar_terms_Gamma}, Minkowski's inequality
	and elementary computation,
	we get that $\Etrue[ \vert E \vert^2 ]^{1/2}$
 	is upper bounded (up to the factor $\Etrue[ \norm{\phi_\rooot}_\infty^4]$) by:
\begin{multline*}
  \sum_{v\in \Tpast(u_{k'},k') \setminus \Tpast(u_k,k)} \sum_{w\in \Tpast(u_k,k)} 
 						( 2\times 2\rho^{d(v,w)-2} \land 8 \rho^{\min(d(v,u), d(w,u))-2} )  \\
\leq  \  \frac{64 }{\rho^2(1-\rho)(1-2\rho)^2} \, k^2 \, (2\rho)^{\floor{k/2}} .
\end{multline*}

\textbf{Upper bound for F:}
Using the first equation in \Cref{lemma_bound_covar_terms_Gamma} and Minkowski's inequality, 
	we get that $\Etrue[ \vert F \vert^2 ]^{1/2}$
 	is upper bounded (up to the factor $\Etrue[ \norm{\phi_\rooot}_\infty^4]$) by:
\begin{align*}
   2 \times 2  \sum_{v\in \Tpast(u_{k'},k') \setminus \Tpast(u_k,k)} \rho^{d(v,u)-2} 
\leq  \ \frac{4 }{1- 2\rho} \,  \frac{\rho^{k-1}}{1-\rho} \cdot
\end{align*}

Hence, as the $L^2(\Ptrue)$ norm for the six terms  $A$, $B$, $C$, $D$, $E$ and $F$ are all upper bounded
by expressions as in the lemma, we get that the upper bound in the lemma holds.
This concludes the proof.
\allowdisplaybreaks[1]
\end{proof}

As annonce at the beginning of this subsection,
the rest of the proof of \Cref{prop_conv_Gamma_term} closely follows 
the lines of the proof of \Cref{prop_conv_Lambda_term}.

\section{Extension to the non-stationary case}
	\label{section_non_stationary}

In Sections~\ref{section_strong_consistency} and~\ref{section_asymptotic_normality}, 
	the stationarity assumption of the process $(Y_u : u\in\T)$ played a crucial role.
In this section, we extend the strong consistency and the asymptotic normality of the MLE for the HMT
	to the case where this process is not stationary.
	
Hence, we assume that the HMT process $(X',Y') = ((X_u',Y_u') : u\in\T)$ has the same transition kernel $Q_{\thetaTrue}$ and $G_{\thetaTrue}$
	that are parametrized by some $\thetaTrue\in\Theta$ as before,
	and the hidden variable $X_{\rooot}'$ of the root vertex $\rooot$ has distribution $\zeta$.
This initial distribution $\zeta$ is unknown to us, may depend on $\thetaTrue$,
	and in general is different from the invariant distribution $\pi_{\thetaTrue}$.
As before, we will denote by $(X,Y) = ((X_u,Y_u) : u\in\T)$ a stationary process 
	distributed according to the same parameter $\thetaTrue$.
Note that, in this section, we will use the convention that objects with an added $'$ symbol 
	are related to the non-stationary process $(X',Y')$,
	while those without the $'$ symbol are their counterpart for the stationary process $(X,Y)$.
Also note that due to the non-stationarity assumption, in this section, we will only consider the HMT process
	on the original tree $\Tpast = \T(\rooot)$.

For the non-stationary process $(X',Y')$, 
similarly to the stationary case in \eqref{eq_def_l_nx_theta_2} on page~\pageref{eq_def_l_nx_theta_2},
define its log-likelihood 
for all $n\in\N$ and $x\in\SpaceX$ as:
\begin{equation}\label{eq_def_l_nx_theta_non_stationary}
\ell_{n,x}'(\theta) := \ell_{n,x}(\theta;Y_{\T_n}') ,
\end{equation}
where $\ell_{n,x}(\theta;\cdot)$ is defined in \eqref{eq_def_l_nx_theta} on page~\pageref{eq_def_l_nx_theta}.
Moreover, 
when Assumptions~\ref{assump_HMM_0}-\ref{assump_HMM_3} and~\ref{assump_HMM_4} hold and $\Theta$ is compact,
similarly to the stationary case in \eqref{eq_def_MLE_hat_theta_nx} on page~\pageref{eq_def_MLE_hat_theta_nx},
we define the MLE $\hat{\theta}_{n,x}'$ for the non-stationary process as:
\begin{equation}
	\label{eq_def_MLE_hat_theta_nx_prime}
\hat \theta_{n,x}' = \hat \theta_{n,x}'(Y_{\T_n}') \in \argmax_{\theta\in\Theta} \ell_{n,x}'(\theta).
\end{equation}
Denote by $\Pzeta$ the probability distribution of the non-stationary HMT process $(X',Y')$,
and by $\Ezeta$ the corresponding expectation.

We can now prove the strong consistency of the MLE for a non-stationary HMT process.

\begin{theo}[Strong consistency of the MLE, non-stationary case]
	\label{thm_strong_constistency_non_stationary}
Assume that Assumptions~\ref{assump_HMM_1}--\ref{assump_HMM_4} hold.
the contrast function $\ell$ has a unique maximum
	(which is then located at $\thetaTrue\in\Theta$ by \Cref{prop_global_max_l}) 
	and  $\Theta$ is compact.
Then, the MLE is strongly consistent, that is,
for all initial distributions $\zeta$ and all $x\in\SpaceX$, the MLE
	$\hat \theta_{n,x}'$
	converges $\Pzeta$-\as as $n\to\infty$
	to the true parameter $\thetaTrue\in\Theta$.
\end{theo}

\begin{proof}We start by proving that for any $n\in\N^*$, the distribution of the non-stationary HMT process $(X',Y')$ on $\T^*$
is absolutely continuous \wrt the distribution of the stationary HMT process $(X,Y)$ on $\T^*$,
that is:
\begin{equation}\label{eq_absolute_continuous_HMT_non_stationary}
\Pzeta( X_{\T^*}' \in \cdot, Y_{\T^*}' \in\cdot)
\ll \Prb_{\thetaTrue, \pi_{\thetaTrue}}( X_{\T^*}' \in \cdot,  Y_{\T^*}' \in\cdot)
= \Ptrue( X_{\T^*} \in \cdot, Y_{\T^*} \in\cdot).
\end{equation}
Remind that \Cref{assump_HMM_2}-\ref{assump_HMM_2:item1} implies that
$\pi_{\thetaTrue} \ll \lambda$ with density $\frac{\drv \pi_{\thetaTrue}}{\drv \lambda}$ 
taking value in $[ \sigma^-, \sigma^+]$.
Denote by $u_1$ and $u_2$ the two children vertices of $\rooot$.
Using \Cref{assump_HMM_2},
for any non-negative measurable function $f$ from $\SpaceX^2$ to $\R_+$,
we get:
\begin{multline*}
\int_{\SpaceX^2} f(x_{u_1}, x_{u_2}) \, \Pzeta( X_{u_1}' \in \drv x_{u_1}, X_{u_2}' \in \drv x_{u_2}) \\
\begin{aligned}
& = \int_{\SpaceX^3} f(x_{u_1}, x_{u_2}) \, q_{\thetaTrue}(x_{\rooot}, x_{u_1}) q_{\thetaTrue}(x_{\rooot}, x_{u_2})
			\, \lambda(\drv x_{u_1}) \lambda(\drv x_{u_2}) \zeta(\drv x_{\rooot}) \\
& \leq \left( \frac{\sigma^+}{\sigma^-} \right)^2
	\int_{\SpaceX^2} f(x_{u_1}, x_{u_2}) \, \pi_{\thetaTrue}(\drv x_{u_1}) \pi_{\thetaTrue}(\drv x_{u_2}) \\
& = \left( \frac{\sigma^+}{\sigma^-} \right)^2
	\int_{\SpaceX^2} f(x_{u_1}, x_{u_2}) \, \Ptrue( X_{u_1} \in \drv x_{u_1}, X_{u_2} \in \drv x_{u_2}) .
\end{aligned}
\end{multline*}
In particular, 
for any measurable subset $A$ of $\SpaceX^{\T^*} \times \SpaceY^{\T^*}$, we can choose $f$ to be define as:
\begin{align*}
f(x_{u_1}, x_{u_2}) 
& = \Ezeta[ \ind_A(X_{\T^*}', Y_{\T^*}') \,\vert\, X_{u_1}'=x_{u_1}, X_{u_2}'=x_{u_2} ] \\
& = \Etrue[ \ind_A(X_{\T^*}, Y_{\T^*}) \,\vert\, X_{u_1}=x_{u_1}, X_{u_2}=x_{u_2} ] 
\end{align*}
Hence, we get that \eqref{eq_absolute_continuous_HMT_non_stationary} holds.

Using \eqref{eq_absolute_continuous_HMT_non_stationary}, 
we get that \Cref{prop_unif_cv_l} also holds $\Pzeta$-\as  with $\ell_{n,x}(\theta)$ replaced by $\ell_{n,x}'(\theta)$,
that is, in the non-stationary case.
Thus, the proof of \Cref{thm_Strong_consistency_MLE} can be immediately adapted to the non-stationary case
(note that Propositions~\ref{prop_l_continuous_complete} and~\ref{prop_global_max_l} state properties of the contrast function $\ell$,
		which is the same in the stationary and non-stationary cases).
This concludes the proof of the Theorem.
\end{proof}

Using a similar argument as for \Cref{thm_strong_constistency_non_stationary},
we get that in the non-stationary case,
the normalized observed information $-\vert\T_n\vert^{-1} \nabla_\theta^2 \ell_{n,x}'(\theta_n)$
converges $\Pzeta$-\as locally uniformly to the limiting Fisher information $\cI(\thetaTrue)$ 
(which is defined in \eqref{eq_def_Fisher_information}).
Note that the condition $\rho<1/2$ on the mixing rate $\rho$
of the HMT process $(X,Y)$
is inherited from \Cref{thm_LLN_observed_information}
in the stationary case.
See the discussion in \Cref{rem_rho_smaller_than_2}
for comments on this condition on $\rho$.

\begin{theo}[Convergence of the normalized observed information, non-stationary case]
	\label{thm_LLN_observed_information_non_stationary}
Assume that Assumptions~\ref{assump_HMM_1}--\ref{assump_HMM_3}
	and \ref{assump_HMM_4}--\ref{assump_HMM_grad_3} hold.
Assume that $\rho<1/2$.
Assume that $\Theta$ is compact.
Then, for all initial distributions $\zeta$ and all $x\in\SpaceX$, we have:
\begin{equation*}\lim_{\delta\to 0} \lim_{n\to\infty} \sup_{\theta\in\ThetaNeighborhood \,:\, \Vert \theta - \thetaTrue \Vert \leq \delta}
	\ \Bigl\Vert{-\vert\T_n\vert^{-1} \nabla_\theta^2 \ell_{n,x}'(\theta) - \cI(\thetaTrue)}\Bigr\Vert = 0
\quad \text{$\Pzeta$-\as}
\end{equation*}
\end{theo}

In particular, 
combining Theorems~\ref{thm_strong_constistency_non_stationary} 
	and~\ref{thm_LLN_observed_information_non_stationary},
we get that the normalized observed information
$-\vert\T_n\vert^{-1} \nabla_\theta^2 \ell_{n,x}(\hat\theta_{n,x})$
at the MLE $\hat\theta_{n,x}$ is a strongly consistent estimator
of the Fisher information matrix $\cI(\thetaTrue)$.

\medskip

Before proving the asymptotic normality of the MLE in the non-stationary case,
we start with the following lemma which present a coupling construction for the two processes
	$(X,Y)$ and $(X',Y')$.

\begin{lemme}[Coupling construction of two HMTs]
	\label{lemma_coupling_time}
Assume that Assumptions~\ref{assump_HMM_1} and \ref{assump_HMM_2} hold.
Further assume that $\sigma^- \geq 1/2$.
Then, it is possible to construct the two processes $(X,Y)$ and $(X',Y')$ on a common probability space such that
there exists an \as finite random time $N$, which we call the \emph{coupling time},
such that $(X_u,Y_u) = (X_u', Y_u')$ for all $u\in \T_n$ with $n\geq N$.
\end{lemme}

We will denote by $\PtrueZeta$ the probability distribution that realizes this coupling.

Note that $\rho \leq 1/2$ implies that $\sigma^- \geq \sigma^+ / 2 \geq 1/2$
	(see \Cref{assump_HMM_2}).

\begin{proof}
We first construct the coupling only for the process $X$ and $X'$.
We define the coupling construction inductively on the height of the tree.
For the root vertex, we use an independent coupling construction for $X_{\rooot}$ and $X_{\rooot}'$,
	which are distributed according to $\pi_{\thetaTrue}$ and $\zeta$ respectively
	(note that it is also possible to use a perfect coupling with probability error $\normTV{\pi_{\thetaTrue} - \zeta}$).
Then, if the coupling has been constructed up to generation $n\in\N$, using the Markov property, 
	we proceed to construct independently the coupling for each vertices in generation $n+1$.
Let $u\in\G_{n+1}$. 
If the variables were already coupled for the parent vertex $\parent(u)$,
	that is $X_{\parent(u)} = X_{\parent(u)}'$, then we choose the new value $X_u = X_u'$ 
	according to the transition kernel $Q_{\thetaTrue}$.
Otherwise, if $X_{\parent(u)} \neq X_{\parent(u)}'$, then using the uniform geometric ergodicity (remind \Cref{assump_HMM_2})
	of the transition kernel $Q_{\thetaTrue}$, we know that 
		$\sup_{x,x'\in\SpaceX} \normTV{Q_{\thetaTrue}(x;\cdot) - Q_{\thetaTrue}(x';\cdot)} \leq 1- \sigma^-$,
	and thus we can construct a coupling of $X_u$ and $X_u'$ conditionally on $X_{\parent(u)} \neq X_{\parent(u)}'$
		with exact matching probability at least $1- \sigma^-$.
We have constructing the matching for $u$, and thus for the whole generation $n+1$.
Using Kolmogorov extension theorem, there exists a coupling measure for the whole tree $\T$
	whose finite dimensional marginals are the ones given above.

Remark that the joint process $(X,Y)$ satisfies a uniform geometric ergodicity bound
	with the same constant $1- \sigma^-$.
Thus, the construction above can be extended to the joint process $(X,Y)$.
Denote by $\PtrueZeta$ the probability distribution of the coupling
	we have constructed for the joint process $(X,Y)$.

Define the random coupling time $N = \inf \{ n\in\N \,\vert\, \forall u\in\G_n,\, X_u = X_u' \}$,
	which is the first generation for which the exact coupling occurs for all vertices
	(and $N = \infty$ is this never happens).
We are left to prove that $\PtrueZeta$-\as $N<\infty$.
We say that a vertex $u$ is a special vertex if $X_u \neq X_u'$.
Note that if $u$ is not special, then all its descendants are also not special.
Also note that special vertices form a Bienaymé-Galton-Watson tree 
	whose (homogeneous) offspring distribution takes the values:
	$0$ with probability $(\sigma^-)^2$; $1$ with probability $2 \sigma^- (1-\sigma^-)$;
	and $2$ with probability $(1-\sigma^-)^2$.
The average of this offspring distribution is $2 (1-\sigma^-)$.
Hence, the number of special vertices if finite $\PtrueZeta$-\as, 
	that is, $N$ is finite $\PtrueZeta$-\as, if and only if $2 (1-\sigma^-) \leq 1$,
	that is, $\sigma^- \geq 1/2$.
This concludes the proof.
\end{proof}

Remind that the log-likelihood function $\ell_{n,x}$ (resp. $\ell_{n,x}'$), which is a random function depending 
on $Y_{\T_n}$ from the stationary HMT process
(resp. on $Y_{\T_n}'$ from the non-stationary HMT process),
is defined in \eqref{eq_def_l_nx_theta_2} on page~\pageref{eq_def_l_nx_theta_2}
(resp. just before \eqref{eq_def_l_nx_theta_non_stationary} on page~\pageref{eq_def_l_nx_theta_non_stationary}).
For all $\theta\in\Theta$, define:
\begin{align*}
D_{n,x}(\theta) & = \ell_{n,x}'(\theta) - \ell_{n,x}(\theta) \\
& = \sum_{u\in\T_n} \log p_\theta (Y_u' \,\vert\, Y'_{\Delta^*(u,\height{u})}, X_{\rooot}'=x)
			- \log p_\theta (Y_u \,\vert\, Y_{\Delta^*(u,\height{u})}, X_{\rooot}=x) ,
\end{align*}
where remind that $p_\theta$ denotes possibly conditional density
(see \eqref{eq_def_exple_p_theta} on page~\pageref{eq_def_exple_p_theta}).

Remind that when Assumptions~\ref{assump_HMM_0}-\ref{assump_HMM_3} and~\ref{assump_HMM_4} hold and $\Theta$ is compact,
for all $x\in\SpaceX$, the MLE $\hat{\theta}_{n,x}$ (resp. $\hat{\theta}'_{n,x}$)
is a random variable which depends on $Y_{\T_n}$ from the stationary HMT process
(resp. on $Y_{\T_n}'$ from the non-stationary HMT process)
and is defined in \eqref{eq_def_MLE_hat_theta_nx} on page~\pageref{eq_def_MLE_hat_theta_nx}
(resp. in \eqref{eq_def_MLE_hat_theta_nx_prime} on page~\pageref{eq_def_MLE_hat_theta_nx_prime}).

To prove that  $\lim_{n\to\infty} \vert \T_n \vert^{1/2} ( \hat{\theta}'_{n,x} - \hat{\theta}_{n,x} ) = 0$ $\PtrueZeta$-\as,
and thus that $\hat{\theta}'_{n,x}$ and $\hat{\theta}_{n,x}$ are asymptotically normal with the same covariance matrix
(remind \Cref{thm_asymptotic_normality}),
we must first prove that the function $\theta \mapsto D_{n,x}(\theta)$ satisfies some kind of continuity property.
Note that we proved in the proof of \Cref{thm_strong_constistency_non_stationary} that 
\Cref{prop_unif_cv_l} holds both in the stationary and the non-stationary cases,
and thus we have: 
\begin{equation*}
\lim_{n\to\infty} \sup_{\theta\in\Theta} \Bigl \vert \vert \T_n \vert^{-1} D_{n,x}(\theta)  \Bigr\vert = 0
\quad \PtrueZeta\text{-\as}
\end{equation*}
However, we need some kind of continuity property without the normalizing term $\vert \vert \T_n \vert^{-1}$,
which is given by the following lemma. 

\begin{lemma}
	\label{lemma_coupling_continuity_Dnx}
Assume that Assumptions~\ref{assump_HMM_1}--\ref{assump_HMM_4} hold.
Further assume that $\rho < 1/2$.
Then, for all initial distributions $\zeta$ and all $x\in\SpaceX$, we have:
\begin{equation*}
\lim_{n\to\infty} \vert D_{n,x}(\hat{\theta}'_{n,x}) - D_{n,x}(\hat{\theta}_{n,x}) \vert = 0,
\quad \text{$\PtrueZeta$-\as}
\end{equation*}
\end{lemma}

\begin{proof}
{[The proof is a straightforward adaptation of the proof of \cite[Lemmas~11 and~12]{doucAsymptoticPropertiesMaximum2004}.]}

Let $N$ be the random time provided by \Cref{lemma_coupling_time}.
We first prove that $\PtrueZeta$-\as, we have:
\begin{equation}\label{eq_sum_remainder_non_stationary}
\sum_{u\in \T \setminus T_{N}} \sup_{\theta\in\Theta} \vert \log p_\theta (Y_u' \,\vert\, Y'_{\Delta^*(u,\height{u})}, X_{\rooot}'=x)
			- \log p_\theta (Y_u \,\vert\, Y_{\Delta^*(u,\height{u})}, X_{\rooot}=x) \vert
< \infty .
\end{equation}
Note that for all $u\in\T_n$ and $v$ an ancestor of $u$ (distinct of $u$), we have:
\begin{align*}
p_\theta (Y_u & \,\vert\, Y_{\Delta^*(u,\height{u})}, X_{\rooot}=x) \\
& = \int_{\SpaceX^3} g_\theta(x_u, Y_u) q_\theta(x_{\parent(u)},x_u) \, \lambda(\drv x_u)
		\Prb_\theta(X_{\parent(u)} \in \drv x_{\parent(u)} \,\vert\, X_v =x_v, Y_{\Delta^*(u, \height{u}-\height{v})} ) \\
& \qquad\qquad\qquad \times  \Prb_\theta( X_v \in \drv x_v \,\vert\, Y_{\Delta^*(u,\height{u})}, X_{\rooot}=x) ,
\end{align*}
and similarly for $p_\theta (Y'_u  \,\vert\, Y'_{\Delta^*(u,\height{u})}, X'_{\rooot}=x)$.
Using the fact that for $v\in\T$ with height $\height{v} \geq N$, we have $Y_v = Y'_v$,
and using \Cref{lemme_exp_coupling_HMT}, we have for all $u\in\T$ with height $\height{u}>N$:
\begin{align*}
\vert  p_\theta ( Y_u' & \,\vert\, Y'_{\Delta^*(u,\height{u})}, X_{\rooot}'=x)
			-  p_\theta (Y_u \,\vert\, Y_{\Delta^*(u,\height{u})}, X_{\rooot}=x) \vert \\
& \leq  2\, \rho^{\height{u}-N-1} \sigma^+ \int g_\theta(x,Y_u) \, \lambda(\drv x) .
\end{align*}
Thus, using a similar argument as in the proof of \Cref{lemma_h_ukx_unif_Cauchy_and_bounded}, we get:
\begin{align*}
\vert \log p_\theta (Y_u' \,\vert\, Y'_{\Delta^*(u,\height{u})}, X_{\rooot}'=x)
			- \log p_\theta (Y_u \,\vert\, Y_{\Delta^*(u,\height{u})}, X_{\rooot}=x) \vert
 \leq \frac{ \rho^{\height{u}-N-1} }{1-\rho} \cdot
\end{align*}
Hence, the sum in \eqref{eq_sum_remainder_non_stationary}
is $\PtrueZeta$-\as upper bounded by a constant times $\sum_{k=N+1}^\infty 2^k \rho^{k} \leq (2\rho)^{N+1}/(1-2\rho)$,
and is thus finite $\PtrueZeta$-\as

Let $\eps>0$.
Using \eqref{eq_sum_remainder_non_stationary},
there exists a random integer $N_\eps$
which $\PtrueZeta$-\as is finite and satifies:
\begin{equation*}
\sum_{u\in \T \setminus T_{N_\eps}} \sup_{\theta\in\Theta} \vert \log p_\theta (Y_u' \,\vert\, Y'_{\Delta^*(u,\height{u})}, X_{\rooot}'=x)
			- \log p_\theta (Y_u \,\vert\, Y_{\Delta^*(u,\height{u})}, X_{\rooot}=x) \vert
\leq \eps .
\end{equation*}
Thus, $\PtrueZeta$-\as, for all $n\geq N_\eps$, we have:
\begin{align*}
\vert D_{n,x}(\hat{\theta}'_{n,x}) - D_{n,x}(\hat{\theta}_{n,x}) \vert
\leq 2 \eps + \vert \ell'_{N_\eps,x}(\hat{\theta}_{n,x}') - \ell'_{N_\eps,x}(\hat{\theta}_{n,x}) \vert
		+ \vert \ell_{N_\eps,x}(\hat{\theta}_{n,x}') - \ell_{N_\eps,x}(\hat{\theta}_{n,x}) \vert.
\end{align*}
Note that under the given assumptions, the functions $\theta \mapsto \ell'_{N_\eps,x}(\theta)$
	and $\theta \mapsto \ell_{N_\eps,x}(\theta)$ are continuous $\PtrueZeta$-\as
	(see the proof of \Cref{prop_l_continuous_complete}).
Hence, the proof is complete upon observing that $\hat{\theta}_{n,x}'$ and $\hat{\theta}_{n,x}$
	both converge $\PtrueZeta$-\as to $\thetaTrue$ (see \Cref{thm_strong_constistency_non_stationary}),
	and that $\eps$ was arbitrary.
\end{proof}

We can now prove the asymptotic normality of the MLE $\hat{\theta}_{n,x}'$ in the non-stationary case.
Remind that the contrast function $\ell$ is defined in \eqref{eq_def_contrast_function} on page~\pageref{eq_def_contrast_function}.

\begin{theo}[Asymptotic normality of the MLE, non-stationary case]
	\label{thm_asymptotic_normality_non_stationary}
Assume that Assumptions~\ref{assump_HMM_1}--\ref{assump_HMM_grad_3} hold.
Assume that $\rho<1/2$.
Further assume that the contrast function $\ell$ has a unique maximum
	(which is then located at $\thetaTrue\in\Theta$ by \Cref{prop_global_max_l}) 
	and that $\Theta$ is compact, $\thetaTrue$ is an interior point of $\Theta$,
	and the limiting Fisher information matrix $\cI(\thetaTrue)$ (which is defined in \eqref{eq_def_Fisher_information}) is non-singular.
Then, for all initial distributions $\zeta$ and for all $x\in\SpaceX$, we have:
\begin{align*}
\vert \T_n \vert^{1/2} \bigl( \hat{\theta}_{n,x}' - \thetaTrue \bigr)
	\underset{n\to\infty}{\overset{(d)}{\longrightarrow}}
\cN(0, \cI(\thetaTrue)^{-1})
\quad \text{under $\Pzeta$,}
\end{align*}
where $\cN(0,M)$ denotes the centered Gaussian distribution with covariance matrix $M$.
\end{theo}

\begin{proof}
{[The proof is a straightforward adaptation of the proof of \cite[Theorem~6]{doucAsymptoticPropertiesMaximum2004}.]}

Define $\eps_n = \vert \T_n \vert^{1/2} ( \hat{\theta}_{n,x} - \hat{\theta}'_{n,x} )$ for all $n\in\N$,
and remark that it is sufficient to prove that $\lim_{n\to\infty} \eps_n = 0$ $\PtrueZeta$-\as
Since $\hat{\theta}'_{n,x}$ is the maximizer of the function $\theta \mapsto \ell'_{n,x}(\theta)$,
we have that $\ell'_{n,x}(\hat{\theta}'_{n,x}) \geq \ell'_{n,x}(\hat{\theta}_{n,x})$.
Thus, using a Taylor expansion of $\ell_{n,x}$ around its maximizer $\hat{\theta}_{n,x}$
	(for which we have $\nabla_\theta \ell_{n,x}(\hat{\theta}_{n,x})=0$),
we get that there exists $t_n\in [0,1]$ such that:
\begin{align*}
D_{n,x}(\hat{\theta}'_{n,x}) - D_{n,x}(\hat{\theta}_{n,x})
& \geq \ell_{n,x}(\hat{\theta}_{n,x}) - \ell_{n,x}(\hat{\theta}_{n,x}') \\
& = - \inv{2} \vert \T_n \vert^{-1} \eps_n^t \nabla_\theta^2 \ell_{n,x}(t_n  \hat{\theta}_{n,x}' + (1-t_n) \hat{\theta}_{n,x}) \eps_n.
\end{align*}
Note that we have $\lim_{n\to\infty} t_n  \hat{\theta}_{n,x}' + (1-t_n) \hat{\theta}_{n,x} = \thetaTrue$
	$\PtrueZeta$-\as
	by \Cref{thm_strong_constistency_non_stationary}.
Thus, applying \Cref{thm_LLN_observed_information}, we have:
\begin{equation*}
\lim_{n\to\infty} -  \vert \T_n \vert^{-1}  \nabla_\theta^2 \ell_{n,x}(t_n  \hat{\theta}_{n,x}' + (1-t_n) \hat{\theta}_{n,x})
= \cI(\thetaTrue), \quad \text{$\PtrueZeta$-\as}
\end{equation*}
As $\cI(\thetaTrue)$ is positive definite, there exits $M>0$ such that on a set with $\PtrueZeta$-probability one
 	and for $n$ sufficiently large, we have:
\begin{equation*}
D_{n,x}(\hat{\theta}'_{n,x}) - D_{n,x}(\hat{\theta}_{n,x})
\geq M \vert \eps_n \vert^2 .
\end{equation*}
Then, the proof is complete by applying \Cref{lemma_coupling_continuity_Dnx}.
\end{proof}



\newcommand{\etalchar}[1]{$^{#1}$}



\appendix

\section{Ergodic theorem for Markov processes indexed by trees with neighborhood-dependent functions}
	\label{appendix_ergodic_theorems}

In this appendix, we prove generalization of the ergodic theorems in \cite{GuyonLimitTheorem}
	and in \cite{weibelErgodicTheoremBranching2024},
which give \as and $L^2$ convergences for branching Markov chains,
to allow for neighborhood-dependent functions.
Those ergodic theorems are used to prove the ergodic convergence lemmas in \Cref{subsection_ergodic_theorem}
	which are used in the main body of this article.
Remind that we need those generalization as in the study of asymptotic property of the MLE for the HMT
relies on the study of the likelihood contribution functions $\h_{u,k,x}(\theta; Y_{\Delta(u,k)})$ 
	(defined in \eqref{eq_def_h_ukx} on page \pageref{eq_def_h_ukx})
	which are neighborhood-dependent.

Remind from \Cref{section_notations} that if $(X,Y)$ is a HMT process, 
then the joint process $((X_u,Y_u) , u\in\T)$ is a branching Markov chain.
Thus, it is enough to prove those ergodic theorems for branching Markov chains instead of HMT processes.

\medskip

Let $Q$ be a transition kernel on $(\SpaceX, \BorelX)$ where $\SpaceX$ is a metric space.
We assume that $Q$ has a unique invariant probability distribution $\pi$ and 
is uniformly geometrically ergodic,
that is, there exists $\rho\in(0,1)$ and a finite positive constant $C$ 
	such that for all $x\in\SpaceX$, we have $\normTV{Q^n(x;\cdot) - \pi} \leq C \rho^n$.
Remind from \Cref{lemma_Q_unif_geom_ergodic} that this covers the case
$Q = Q_\theta$ for any $\theta\in\Theta$.
Let $X = (X_u , u\in\T)$ be a branching Markov chain
with transition kernel $Q$ and initial distribution $\pi$.
Denote by $\Prb_Q$ the probability distribution of the process $X$,
and by $\Esp_Q$ the corresponding expectation.

In this section, for a probability measure $\nu$ on $\SpaceX$, a transition kernel $Q$ on $(\SpaceX, \BorelZ{\SpaceZ})$
and a Borel integrable function $f$ on $\BorelZ{\SpaceZ}$ where $\SpaceZ = \SpaceX^{A}$ for some finite subset $A\subset\T$,
we will write $\nu Q$ for the image probability measure $(\nu Q) (\cdot) = \int_{\SpaceX} Q(x;\cdot) \, \nu(\drv x)$,
and $Q f$ for the Borel function $(Q f) (x) = \int_{\SpaceZ} f(z) \, Q(x;\drv z)$.
For a probability measure $\nu$ on $\SpaceX$ and a Borel integrable function $f$ on $\SpaceX$,
we will write $\scalProd{\nu}{f} = \nu f= \int_{\SpaceX} f\, \drv \nu$.

We will need the following lemma which states geometric convergence bounds for functions in $L^2(\pi)$.

\begin{lemme}[Convergence bounds when $Q$ is uniformly geometrically ergodic]
	\label{lemma_assump_Thm_ergo_Q_unif_geom_ergo}
Assume that the transition kernel $Q$ has a unique invariant measure $\pi$,
and that $Q$ is uniformly geometrically ergodic.
Then, there exists finite positive constants $\alpha\in(0,1)$ and  $M<\infty$ such that
for all functions $f\in L^2(\pi)$, we have:
\begin{equation*}
\forall n\in\N,\qquad
\sup_{k\in\N} \pi Q^k (Q^n f - \scalProd{\pi}{f})^2 = \pi (Q^n f - \scalProd{\pi}{f})^2  \leq M \alpha^{2n} \norm{f- \scalProd{\pi}{f}}_{L^2(\pi)}^2.
\end{equation*}
In particular, the function $f$ satisfies 	$\sup_{n\in\N} \pi Q^n f^2 < \infty$.
\end{lemme}

Note, using Cauchy-Schwarz and Jensen's inequalities, 
that $\sup_{n\in\N} \pi Q^n f^2 < \infty$ implies that
$Q^n f$, $Q^n f^2$, and $Q^k (Q^n f \times Q^m f)$ (with $n,m,k\in\N$) 
are well-defined and finite $\pi$-almost everywhere and are $\pi$-integrable.

\begin{proof}
Using \cite[Proposition 22.3.5 and Definition 22.3.1]{DoucMC}, 
we get that there exists finite positive constants $\alpha\in(0,1)$ and $M<\infty$ such that
for all functions $f\in L^2(\pi)$, we have
	$\pi (Q^n f - \scalProd{\pi}{f})^2 
		= \norm{Q^n ( f - \scalProd{\pi}{f})}_{L^2(\pi)}^2 
		\leq M \alpha^{2n} \norm{f- \scalProd{\pi}{f}}_{L^2(\pi)}^2$ for all $n\in\N$.
In particular, we get that $\sup_{n\in\N} \pi Q^n f^2 < \infty$.
\end{proof}

Let $k\in\N$ be fixed.
Remind from \Cref{subsection_ergodic_theorem} the definitions of the subtrees $\Delta(u,k)$,
of their shapes $\Shape\bigl( \Delta(u,k) \bigr)$ (defined in \eqref{eq_def_shape_subtree}),
and of the finite set of possible shapes $\ShapeSetValues_k$ (defined in \eqref{eq_def_set_possible_shapes}).
For simplicity, in this appendix we will write $\ShapeValue_u$ instead of $\Shape\bigl( \Delta(u,k) \bigr)$.
Also remind from \Cref{subsection_ergodic_theorem} the definition of a collection of neighborhood-shape-dependent
	functions $(f_{\ShapeValue} : \SpaceX^{\ShapeValue} \to \R)_{\ShapeValue\in\ShapeSetValues_k}$.
Remind that for such a collection of functions, we simply write $f_{\Delta(u,k)}$ or $f_{\ShapeValue_u}$ 	instead of $f_{\Shape(\Delta(u,k))}$.  And also remind that we write $f_{\ShapeValue_u}(X_{\Delta(u,k)})$ for the evaluation of $f_{\ShapeValue_u} = f_{\Delta(u,k)}$ on $X_{\Delta(u,k)}$.
Note that up to translation, we may identify $\SpaceX^{\ShapeValue}$ and $\SpaceX^{\Delta(u,k)}$
for any $u\in\T\setminus\T_{k-1}$ such that $\ShapeValue_u = \ShapeValue$.

Remind that any subset $A\subset \T$, we denote by $X_A$ the gathered variables
	$(X_v \,:\, v\in A)$.
For a collection of neighborhood-shape-dependent
	functions $\f = (f_{\ShapeValue} : \SpaceX^{\ShapeValue} \to \R)_{\ShapeValue\in\ShapeSetValues_k}$,
define the empirical average of $\f$ over a finite subset $A\subset\T\setminus\T_{k-1}$ as:
\begin{equation}\label{eq_def_M_A_f}
\bar M_{A}(\f) = 
	\vert A \vert^{-1} \sum_{u\in A} f_{\ShapeValue_u}(X_{\Delta(u,k)}) .
\end{equation}

For a neighborhood shape $\ShapeValue\in\ShapeSetValues_k$,
let $u\in\G_k$ be the unique vertex such that $\ShapeValue_u = \ShapeValue$,
and define the transition kernel $Q^{\ShapeValue}$ on $(\SpaceX, \BorelZ{\SpaceX^{\ShapeValue}})$
for any $x\in\SpaceX$ and any Borel function $f$ on $\SpaceX^{\ShapeValue} = \SpaceX^{\Delta(u)}$ 
which is in $L^1(X_{\Delta(u)}) = L^1(X_{\ShapeValue})$
by:
\begin{equation}\label{eq_def_Q_S}
Q^{\ShapeValue} f (x) = \Esp_Q \Bigl[ f(X_{\Delta(u)}) \Bigm\vert X_{\rooot} = x \Bigr] .
\end{equation}
That is, from the value $x\in\SpaceX$ of the root vertex $v$ in $\ShapeValue$,
the transition kernel $Q^{\ShapeValue}$
returns the distribution of the Markov process $X$ on $\ShapeValue$ with transition kernel $Q$
	conditioned on the value $X_v$ of the vertex $v$ being $x$.
Note that \eqref{eq_def_Q_S} also extends to any vertex $u\in\T\setminus\T_{k-1}$ such that $\ShapeValue_u = \ShapeValue$,
which gives us:
\begin{equation}\label{eq_def_Q_S_u}
Q^{\ShapeValue_u} f (x) = \Esp_Q \Bigl[ f(X_{\Delta(u,k)}) \Bigm\vert X_{\parent^k(u)} = x \Bigr] .
\end{equation}
Moreover, using Jensen's inequality, note that if $f \in L^2(X_{\ShapeValue})$, 
then $Q^{\ShapeValue} f$ is in $L^2(\pi) = L^2(X_{\rooot})$.

\medskip

Remind that as $\T$ is a plane rooted tree, we can enumerate its vertices as a sequence $(v_j)_{j\in\N}$
in a breadth-first-search manner, that is, which is increasing for $<$
(note that $u_0 = \rooot$).
Also remind that, for $n\geq \vert \T_{k-1}\vert$, if $V_n$ is uniformly distributed over 
$A_n := \{ v_j \,:\, \vert \T_{k-1} \vert < j \leq n \}
= \Delta(v_n) \setminus \T_{k-1}$,
then the distribution of $\ShapeValue_{V_n}$ converges to the uniform distribution over $\ShapeSetValues_k$
as $n\to\infty$.

We are now ready to state the ergodic theorem with neighborhood-shape-dependent functions
for branching Markov chains indexed by the infinite complete binary tree $\T$.
Remind that $\bar M_{A_n}(\f)$ is defined in \eqref{eq_def_M_A_f}.

\begin{theo}[Ergodic theorem with neighborhood-dependent functions]
	\label{Ergodic_theorem_with_neighborhood_functions}
Let $k\in\N$ be fixed.
Let $(v_j)_{j\in\N}$ be the sequence enumerating the vertices in $\T$ in a breadth-first-search manner.
For all $n> \vert \T_{k-1}\vert$, define $A_n = \Delta(v_n) \setminus \T_{k-1}$.

Let $Q$ be a transition kernel on $(\SpaceX, \BorelX)$
which is uniformly geometrically ergodic and has a unique invariant probability measure $\pi$.
Let $X = (X_u, u\in\T)$ be a branching Markov chain 	with transition kernel $Q$ and initial distribution $\pi$.

Let $\f = (f_{\ShapeValue} : \SpaceX^{\ShapeValue} \to \R)_{\ShapeValue\in\ShapeSetValues_k}$ 
be a collection of neighborhood-shape-dependent Borel functions that are in $L^2(X)$.
Then, we have:
\begin{equation}
	\label{eq_conv_thm_ergo_neigh_funct}
\bar M_{A_n}(\f) 
	{\underset{n\to\infty}{\longrightarrow}}
\Esp_{U_k} \otimes \Esp_Q \bigl[ f_{\ShapeValue_{U_k}}(X_{\Delta(U_k)}) \bigr]
\qquad \text{ in $L^2(\pi) = L^2(X)$},
\end{equation}
where $U_k$ is uniformly distributed over $\G_k$ and independent of the process $X$,
and $\Esp_{U_k} \otimes \Esp_Q$ denotes the joint expectation over $U_k$ and $X$.
\end{theo}

As an immediate corollary, we get that the result still holds if $A_n$ is replaced by $\T_n\setminus\T_{k-1}$.

\begin{remark}[More general assumptions]
Note that without changing the proof, we could replace the subtrees $\Delta(u,k)$
	by general subtrees (\ie a connected subsets) $\Neighbor_u$
	of the the $k$-neighborhood $B_{\T}(u,k) := \{ v\in\T \,:\, d(u,v) \leq k\}$ of $u$
	such that $\Neighbor_u$ contains $u$.
In that case, we must assume that  the distribution of the shape (\ie when seen up to translation) $\Shape(\Neighbor_{V_n})$
converges to some limit distribution.
Also note that we could allow more general choices as in \cite{weibelErgodicTheoremBranching2024} 
for the averaging sets $(A_n)_{n\in\N}$,
for the tree $\T$, and for the transition kernel $Q$ and initial distribution of the branching Markov chain $X$.
\end{remark}

\begin{proof}
\textbf{First case: we only have constant functions 
$f_{\ShapeValue} \equiv c(\ShapeValue)$ 
	for all $\ShapeValue\in\ShapeSetValues_k$.}
Then, for every $n\in\N$ we have:
\begin{equation}\label{eq_M_An_f_constant}
\bar M_{A_n}(\f) = \sum_{\ShapeValue\in\ShapeSetValues_k}
	c(\ShapeValue)\, \frac{\vert u\in A_n : \ShapeValue_u = \ShapeValue \vert}{\vert A_n \vert},
\end{equation}
where the right hand side converges in distribution (and thus in $L^2$) as $n\to\infty$ to
$\Esp_{U_k} \bigl[ c(\ShapeValue_{U_k}) \bigr]$
(remind that the distribution of $\ShapeValue_{V_n}$ converges to the uniform distribution on $\G_k$
when $n\to\infty$).
This concludes the proof in this first case.
\medskip

\textbf{General case:} 
We adapt the proof of \cite[Theorem~2.2]{weibelErgodicTheoremBranching2024}
to the case of neighborhood-shape-dependent functions.

Using the first case and the linearity in $\f$ of the empirical
	averages, and replacing $f_{\ShapeValue}$
	by $f_{\ShapeValue} - \scalProd{\pi}{Q^{\ShapeValue} f_{\ShapeValue}}$,
	we may assume that $\scalProd{\pi}{Q^{\ShapeValue} f_{\ShapeValue}} = 0$
	for all $\ShapeValue\in\ShapeSetValues_k$.
For all $n\in\N$, we have:
\begin{equation}\label{eq_Esp_M_A_f_Neigh}
\Esp_Q\Bigl[ \bar M_{A_n}(\f)^2\Bigr]
= \inv{\vert A_n \vert^2} \sum_{u,v \in A_n} 
	\Esp_Q\Bigl[ f_{\ShapeValue_u}(X_{\Delta(u,k)}) f_{\ShapeValue_v}(X_{\Delta(v,k)}) \Bigr].
\end{equation}

Using \Cref{lemma_assump_Thm_ergo_Q_unif_geom_ergo},
as the transition kernel $Q$ is uniformly geometrically ergodic and has unique invariant probability measure $\pi$,
and as the function $Q^{\ShapeValue} f_{\ShapeValue}$ for $\ShapeValue\in\ShapeSetValues_k$
are all in $L^2(\pi)$
(see the comment just after \eqref{eq_def_Q_S_u}),
we have that
$C_{\ShapeValue} := \sup_{n\in\N} \pi Q^n (Q^{\ShapeValue} f_{\ShapeValue})^2 < \infty$ for all $\ShapeValue\in\ShapeSetValues_k$.
Define $C := \max_{\ShapeValue\in\ShapeSetValues_k} C_{\ShapeValue} < \infty$ (remind that $\ShapeSetValues_k$ is finite).
Thus, for $u\in\T$, we have:
\begin{align*}
\Esp_Q\Bigl[ f_{\ShapeValue_u}(X_{\Delta(u,k)})^2  \Bigr]
& = \pi Q^{(\height{u}-k)_+} ( Q^{\ShapeValue_u} f_{\ShapeValue_u} )^2
\leq C < \infty.
\end{align*}
Hence, for all $u,v\in\T$, using Cauchy-Schwarz inequality, we have:
\begin{equation}
	\label{eq_cov_Ou_Ov_upperbound}
\Esp_Q\Bigl[ f_{\ShapeValue_u}(X_{\Delta(u,k)}) f_{\ShapeValue_v}(X_{\Delta(v,k)}) \Bigr]
 \leq \left( \Esp_Q\Bigl[ f_{\ShapeValue_u}(X_{\Delta(u,k)})^2  \Bigr] \Esp_Q\Bigl[ f_{\ShapeValue_v}(X_{\Delta(v,k)})^2 \Bigr] \right)^{1/2}
 \leq C < \infty.
\end{equation}

Let $u,v\in\T$ such that $d(u,v)>2k$,
which implies that $\Delta(u,k) \cap \Delta(v,k) = \emptyset$.
Without loss of generality, assume that $\height{u}\geq \height{v}$.
Then, we have that $\height{u\land v} < \height{u} -k$.
Denote by $v_0$ the last ancestor of $u$ in $\Delta(v,k) \cup \{ u\land v\}$.
Remark that $u\land v$ is an ancestor of $v_0$.
Then, we have:
\begin{align}
\Esp_Q\Bigl[ f_{\ShapeValue_u}(X_{\Delta(u,k)})  f_{\ShapeValue_v}&(X_{\Delta(v,k)})  \Bigr]  \nonumber\\
& = \Esp_Q\Bigl[ \Esp_Q\bigl[ f_{\ShapeValue_u}(X_{\Delta(u,k)}) \bigm\vert X_{\Delta(v,k)},\, X_{u\land v} \bigr]  f_{\ShapeValue_v}(X_{\Delta(v,k)}) \Bigr] \nonumber\\
& \leq \left( \Esp_Q\Bigl[ \Esp_Q\bigl[ f_{\ShapeValue_u}(X_{\Delta(u,k)}) \bigm\vert X_{\Delta(v,k)},\, X_{u\land v} \bigr]^2 \Bigr]
		\Esp_Q\Bigl[ f_{\ShapeValue_v}(X_{\Delta(v,k)})^2 \Bigr]
	\right)^{1/2} \nonumber\\
& \leq C^{1/2} \left( 
		\Esp_Q\Bigl[ \Esp_Q\bigl[ f_{\ShapeValue_u}(X_{\Delta(u,k)}) \bigm\vert X_{v_0} \bigr]^2 \Bigr]
	\right)^{1/2} \nonumber\\
& \leq C^{1/2} \left( 
		\Esp_Q\Bigl[ \Esp_Q\bigl[ f_{\ShapeValue_u}(X_{\Delta(u,k)}) \bigm\vert X_{u\land v} \bigr]^2 \Bigr]
	\right)^{1/2} \nonumber\\
& = C^{1/2} \left(
		\pi Q^{\height{u\land v}} \Bigl( Q^{d(u\land v,u)-k} Q^{\ShapeValue_u} f_{\ShapeValue_u} \Bigr)^2
	\right)^{1/2} \nonumber\\
& \leq C^{1/2} 
	\left( \max_{\ShapeValue\in\ShapeSetValues_k} \
		\pi Q^{\height{u\land v}} \Bigl( Q^{\tilde d(u, v) -k} Q^{\ShapeValue} f_{\ShapeValue} \Bigr)^2 
	\right)^{1/2}
	,	\label{eq_cov_Ou_Ov_upperbound_conv2}
\end{align}
where we used Cauchy-Schwarz inequality in the first inequality,
we used \eqref{eq_cov_Ou_Ov_upperbound} and the Markov property of the process $X$ in the second inequality,
and we used Jensen's inequality in the third inequality.
Remark that $\tilde d(u,v) = \max( d(u\land v, u), d(u\land v, v))$ is a distance on $\T$
that satisfies $d / 2 \leq \tilde d \leq d$.

Let $U_n$ and $V_n$ be uniformly distributed over $A_n$,
and independent of each other and of the branching Markov chain $X$,
and denote by $\Prb_{U_n, V_n}$ their joint probability distribution.
Using again \Cref{lemma_assump_Thm_ergo_Q_unif_geom_ergo}, there exist finite constants $M < \infty$ and $\alpha\in (0,1)$
such that for all $\ShapeValue\in\ShapeSetValues$ we have 
\begin{equation}\label{eq_geom_bound_Q_f_S}
\forall m\in\N,\qquad
\sup_{j\in\N} \pi Q^j \bigl(Q^m Q^{\ShapeValue} f_{\ShapeValue} \bigr)^2  \leq M^2 \alpha^{2m}.
\end{equation}
Hence, combining \eqref{eq_Esp_M_A_f_Neigh}, \eqref{eq_cov_Ou_Ov_upperbound}, \eqref{eq_cov_Ou_Ov_upperbound_conv2}
and \eqref{eq_geom_bound_Q_f_S},
we get for any $K\geq k$:
\begin{align}
\Esp_Q\Bigl[ \bar M_{A_n}(\f)^2\Bigr]
& \leq C\, \Prb_{U_n, V_n}( \tilde d(U_n, V_n) \leq 2K) \nonumber\\
& \qquad 	+ C^{1/2} \,  \vert A_n \vert^{-2} 
	\sum_{u,v\in A_n : \tilde d(u,v) > 2K} \left( \max_{\ShapeValue\in\ShapeSetValues_k} \
		\pi Q^{\height{u\land v}} \Bigl( Q^{\tilde d(u, v) -K} Q^{\ShapeValue} f_{\ShapeValue} \Bigr)^2 
	\right)^{1/2} \nonumber\\
& \leq C\, \Prb_{U_n, V_n}( d(U_n, V_n) \leq 4K) \nonumber\\
& \qquad 	+ C^{1/2} \,  \vert A_n \vert^{-2} 
	\sum_{u,v\in A_n : \tilde d(u,v) > 2K} 
		M \alpha^{ \tilde d(u,v) - K}
		\label{eq_M_An_upper_bound_thm_ergo} \\
& \leq C\, \Prb_{U_n, V_n}( d(U_n, V_n) \leq 4K) 
	+ C^{1/2} M \alpha^{K}
	. 	\label{eq_M_An_upper_bound_thm_ergo_2}
\end{align}
Let $\eps >0$.
Let $K\geq k$ be such that $C^{1/2} M \alpha^K < \eps$.
Using \cite[Lemma~3.1]{weibelErgodicTheoremBranching2024},
we get that the first term in the right hand side 
of \eqref{eq_M_An_upper_bound_thm_ergo_2} goes to zero as $n\to\infty$.
Thus, for $n$ large enough, the right hand side of \eqref{eq_M_An_upper_bound_thm_ergo_2}
is upper bounded by $2\eps$.
This being true for all $\eps>0$, we get that $\lim_{n\to\infty} \Esp_Q\bigl[ \bar M_{A_n}(\f)^2\bigr] = 0$.
This concludes the proof.
\end{proof}

We now state and prove a strong law of large numbers for branching Markov chains indexed by
the infinite complete binary tree $\T$ and with neighborhood-shape-dependent functions.
This result uses the same assumptions as in \Cref{Ergodic_theorem_with_neighborhood_functions}.

\begin{theo}[Strong law of larger numbers with neighborhood-dependent function]
	\label{Strong_LLN_with_neighborhood_functions}

Let $Q$ be a transition kernel on $(\SpaceX, \BorelX)$
which is uniformly geometrically ergodic and has a unique invariant probability measure $\pi$.
Let $X = (X_u, u\in\T)$ be a branching Markov chain 	with transition kernel $Q$ and initial distribution $\pi$.

Let $k\in\N$ be fixed.
Let $U_k$ be uniformly distributed over $\G_k$ and independent of the process $X$,
and let $\Esp_{U_k} \otimes \Esp_Q$ denote the joint expectation over $U_k$ and $X$.
Let $\f = (f_{\ShapeValue} : \SpaceX^{\ShapeValue} \to \R)_{\ShapeValue\in\ShapeSetValues_k}$ 
be a collection of neighborhood-shape-dependent Borel functions that are in $L^2(X)$.

Then, we have: \begin{equation*}
\text{\as}\qquad
\lim_{n\to\infty} \bar M_{\G_n}(\f)
= \lim_{n\to\infty} \bar M_{\T_n\setminus \T_{k-1}}(\f)
= \Esp_{U_k} \otimes \Esp_Q \bigl[ f_{\ShapeValue_{U_k}}(X_{\Delta(U_k)}) \bigr] .
\end{equation*}
Moreover, there exist finite constants $C_0 < \infty$ and $\beta\in(0,1)$ such that:
\begin{equation}\label{eq_LLN_upper_bound_Esp_M_Gn}
\forall n\geq k, \qquad \Esp_Q \left[ \left(\bar M_{\G_n}(\f) - 
		\Esp_{U_k} \otimes \Esp_Q \bigl[ f_{\ShapeValue_{U_k}}(X_{\Delta(U_k)}) \bigr]
		\right)^2 \right] \leq C_0 \beta^n .
\end{equation}
\end{theo}

\begin{proof}
After using \eqref{eq_M_An_upper_bound_thm_ergo},
the proof is an easy adaptation of the proof of \cite[Theorem~14]{GuyonLimitTheorem}.

The case of $\bar M_{\T_n \setminus \T_{k-1}}(\f)$ follows directly from the case of $\bar M_{\G_n}(\f)$ as:
\begin{equation*}
\vert \bar M_{\T_n}(\f) \vert \leq \sum_{j=k}^n \frac{\vert \G_j\vert}{\vert \T_n \setminus \T_{k-1} \vert} \bar M_{\G_j}(\f) . 
\end{equation*} 

Thus, it it enough to treat the case of $\bar M_{\G_n}(\f)$.
In the case where all functions $f_{\ShapeValue}$ for $\ShapeValue\in\ShapeSetValues_k$ are constant,
writing  $\bar M_{\G_n}(\f)$ as in \eqref{eq_M_An_f_constant},
and using the convergence in distribution of $(\ShapeValue_{U_n})_{n\in\N}$ the uniform distribution over $\ShapeSetValues_k$,
we get that the sought convergence holds \as for $\bar M_{\G_n}(\f)$.
Thus, without loss of generality, we assume that $\scalProd{\pi}{Q^{\ShapeValue} f_{\ShapeValue}} = 0$ 	for all $\ShapeValue\in\ShapeSetValues_k$.
	
Remark that it is enough to prove that $\sum_{n\geq k} \Esp[ \bar M_{\G_n}(\f)^2 ] < \infty$, as then we can immediately conclude
	using Borel-Cantelli lemma with Markov's inequality.
Thus, for $n\geq k$, using \eqref{eq_M_An_upper_bound_thm_ergo} with $n' = \vert \T_n \vert$ (such that $A_{n'} =\G_n$) and $K=k$, 
we get:		\begin{align*}
\Esp_Q \bigl[ \bar M_{\G_n}(\f)^2 \bigr]
& \leq C\, 2^{-(n-2k)} 
	+ C^{1/2} M \,  
	2^{-2n} \sum_{u,v\in \G_n : {d}(u,v) > 2k} \alpha^{\tilde d(u, v) -k} \nonumber\\
& = C\, 2^{-(n-2k)} 
	+ C^{1/2} M \, \sum_{j=k}^n 2^{-(n-j)- \ind_{\{j>0\}}} \alpha^{j-k} \nonumber\\
& \leq C\, 2^{-(n-2k)} 
	+ C^{1/2} M \, 2^{-(n-k)}  \sum_{j=k}^n 2^{j-k} \alpha^{j-k} \nonumber\\
& \leq C\, 2^{-(n-2k)} 
	+ C^{1/2} M \, 2^{-(n-k)} C' \max(n+1, (2\alpha)^{n-k}) \nonumber\\
& \leq C\, 2^{-(n-2k)} 
	+ C^{1/2} M C' \, \max((n+1)2^{-(n-k)}, \alpha^{n-k}) , 
	\end{align*}
where $C'$ is a constant whose value only depends on the value of $2\alpha$. 
Hence, there exist finite constants $C_0 < \infty$ and $\beta\in(0,1)$ 
	such that \eqref{eq_LLN_upper_bound_Esp_M_Gn} holds.
In particular, we get that $\sum_{n\geq k} \Esp[ \bar M_{\G_n}(\f)^2 ] < \infty$.
This concludes the proof of the theorem.
\end{proof}

\section{Proof of the ‘‘backward’’ coupling Lemma~\ref{lemma_backward_mixing_bound}}
	\label{appendix_proof_backward_coupling_lemma}

We now prove \Cref{lemma_backward_mixing_bound}.

\begin{proof}[Proof of \Cref{lemma_backward_mixing_bound}]
The proof relies on a ‘‘backward in time’’ bound from $u$ to $u\land v$,
and then a ‘‘forward in time’’ bound from $u\land v$ to $v$.
We divide the proof in two cases: first when $v$ is an ancestor of $u$, and then the general case.

\newcommand{\DeltaRminusSetUj}[1]{\DeltaR^-(u,k,#1)}

For all $j\leq k$, define the vertex $U_{j} = \parent^j(u)$ which is random for $j > \height{u}$
	(in which case, it depends on $\cU$).
Write $x_{U_{k}}=x$.
For all $j\in \{ 1, \cdots, k\}$, define the random set  (which depends on $\cU$):
\begin{equation*}
\DeltaRminusSetUj{j} = (\DeltaR^*(u,k) \setminus \{U_k\}) \cap (\Tpast(U_k) \setminus \Tpast(U_{j-1})),
\end{equation*}
and in particular remark that $U_0 = u \not\in \DeltaRminusSetUj{j}$ and $U_k = \parent^k(u) \not\in \DeltaRminusSetUj{j}$.

\textbf{Case 1: $v$ is an ancestor of $u$.}
We mimic the proof of \cite[Proposition~12.5.4]{CappeHMM}.
The proof of the first case relies on the observation that conditioned 
	on $X_{\parent^k(u)}$ and $Y_{\DeltaR(u,k)}$,
the backward ancestral process $X$ from $U_0 = u$ to $U_k = \parent^k(u)$ is a non-homogeneous Markov chain
satisfying a uniform mixing condition.
The fact that $(X_{U_j})_{0\leq j \leq k}$ is a Markov chain
comes from the Markov property of the HMT $(X,Y)$
(remind the discussion around \eqref{eq_illustration_Markov_prop} on page~\pageref{eq_illustration_Markov_prop})
which gives for all $j\in \{ 1, \cdots, k\}$:
\begin{align}
\cL( X_{U_j} \,\vert\, Y_{\DeltaR(u,k)}, X_{U_k}, X_{\Tpast(U_{j-1})} ) 
& = \cL( X_{U_j} \,\vert\, Y_{\DeltaR(u,k)}, X_{U_k}, X_{U_{j-1}} ) \nonumber\\
& = \cL( X_{U_j} \,\vert\, Y_{\DeltaRminusSetUj{j}}, X_{U_k}, X_{U_{j-1}} ) .
	\label{eq_Markov_prop_backward_kernel}
\end{align}

For all integers $j\in \{ 1, \cdots, k\}$,
the backward transition kernel (which depends on $\cU$) from $X_{U_{j-1}}$ to $X_{U_j}$ is defined as:
\begin{align*}
\text{B}_{x_{U_{k}},j}[y_{\DeltaR(u,k)}] (x_{U_{j-1}};f)  
= \Esp_\theta\bigl[ f(X_{U_{j}})  \,\bigm\vert\, Y_{\DeltaR(u,k)} = y_{\DeltaR(u,k)}, X_{U_k} = x_{U_k}, X_{U_{j-1}}=x_{U_{j-1}} \bigr] ,
\end{align*}
for any $x_{U_{j-1}}\in\SpaceX$ and any bounded Borel function $f$ on $\SpaceX$.
By the Markov property (see \eqref{eq_Markov_prop_backward_kernel}), 
note that $\text{B}_{x_{U_{k}},j}[y_{\DeltaR(u,k)}] (x_{U_{j-1}},f)$
only depends on $y_{\DeltaRminusSetUj{j}}$ instead of $y_{\DeltaR(u,k)}$, that is:
\begin{align}
\text{B}_{x_{U_{k}},j}[y_{\DeltaR(u,k)}] (x_{U_{j-1}};&f)  \nonumber\\
&  = \Esp_\theta\bigl[ f(X_{U_{j}})  \,\bigm\vert\, Y_{\DeltaRminusSetUj{j}} = y_{\DeltaRminusSetUj{j}}, X_{U_k} = x_{U_k}, X_{U_{j-1}}=x_{U_{j-1}} \bigr]
	\nonumber\\
&  = \frac{
	\int_{\SpaceX} f(x_{U_{j}}) \, p_\theta( y_{\DeltaRminusSetUj{j}}, x_{U_j} \,\vert\, X_{U_{k}} = x_{U_{k}})
		 q_\theta(x_{U_{j}},x_{U_{j-1}}) \,  \lambda(\drv x_{U_j}) 
}
{
	\int_{\SpaceX}  p_\theta( y_{\DeltaRminusSetUj{j}}, x_{U_j} \,\vert\, X_{U_{k}} = x_{U_{k}})
		 q_\theta(x_{U_{j}},x_{U_{j-1}}) \,  \lambda(\drv x_{U_j}) 
} , \label{eq_def_backward_kernel}
\end{align}
where:
\begin{multline*}
p_\theta( y_{\DeltaRminusSetUj{j}}, x_{U_{j}} \,\vert\, X_{U_{k}} = x_{U_{k}}) = \\
\int_{\SpaceX^{\DeltaRminusSetUj{j}}} 
		\prod_{w\in \DeltaRminusSetUj{j}} q_\theta(x_{\parent(w)},x_w) g_\theta(x_w, y_w)  
			\prod_{w\in \DeltaRminusSetUj{j} \setminus \{ U_{j} \}} \lambda(\drv x_w)  .
\end{multline*}
To simplify notations, we will keep the dependence on $y_{\DeltaR(u,k)}$ for all indices $j$.
Note that the integral in the denominator in the right hand side of \eqref{eq_def_backward_kernel} is lower bounded by:
\begin{equation*}
\prod_{w\in \DeltaRminusSetUj{j}} \sigma^- \int_{\SpaceX} g_\theta(x_w, y_w) \lambda(\drv x_w) ,
\end{equation*}
and is thus positive $\Prb_{\theta}$-\as under \Cref{assump_HMM_2}.

Using \Cref{assump_HMM_2}, we get that those backward transition kernels satisfy the following Doeblin condition
(remind \Cref{def_Doeblin_condition}):
\begin{equation*}
\frac{\sigma^-}{\sigma^+}  \nu_{x_{U_{k}},j}[y_{\DeltaR(u,k)}] (f)
\leq \text{B}_{x_{U_{k}},j}[y_{\DeltaR(u,k)}] (x_{U_{j-1}};f),
\end{equation*}
where for any bounded Borel function $f$ on $\SpaceX$, we have:
\begin{align*}
\nu_{x_{U_{k}},j}[y_{\DeltaR(u,k)}] (f) 
& = \Esp_\theta\bigl[ f(X_{U_{j}})  \,\bigm\vert\, Y_{\DeltaRminusSetUj{j}} = y_{\DeltaRminusSetUj{j}}, X_{U_k} = x_{U_k} \bigr]
\nonumber \\
 & = \frac{
	\int_{\SpaceX} f(x_{U_{j}}) \, p_\theta( y_{\DeltaRminusSetUj{j}}, x_{U_j} \,\vert\, X_{U_{k}} = x_{U_{k}})
		  \,  \lambda(\drv x_{U_j}) 
}
{
	\int_{\SpaceX}  p_\theta( y_{\DeltaRminusSetUj{j}}, x_{U_j} \,\vert\, X_{U_{k}} = x_{U_{k}})
		  \,  \lambda(\drv x_{U_j}) 
} ,
\end{align*}
where note that the only difference with the definition of $\text{B}_{x_{U_{k}},j}[y_{\DeltaR(u,k)}] (x_{U_{j-1}};f)$
is that the term $q_\theta(x_{U_j},x_{U_{j-1}})$ has disappeared in both the numerator and the denominator of $\nu_{x_{U_{k}},j}[y_{\DeltaR(u,k)}] (f)$.
Thus, \Cref{lemma_Doeblin_Dobrushin} shows that the Dobrushin coefficient $\delta(B_{x_{U_k},j})$ (defined in \eqref{eq_def_Dobrushin_coefficient})
of the backward transition kernel $B_{x_{U_k},j}$
is upper bounded by $\rho = 1 - \sigma^- / \sigma^+$.

Note that the Markov property in \eqref{eq_Markov_prop_backward_kernel} (with $j=1$) also gives us:
\begin{align}
\cL( X_{U_1} \,\vert\, Y_{\DeltaR(u,k)}, X_{U_k}, X_{U_{0}} ) 
& = \cL( X_{U_1} \,\vert\, Y_{\DeltaRminusSetUj{j}}, X_{U_k}, X_{U_{0}} ) \nonumber\\
& = \cL( X_{U_1} \,\vert\, Y_{\DeltaR^*(u,k)}, X_{U_k}, X_{U_{0}} ) .
	\label{eq_Markov_prop_backward_kernel_2}
\end{align}
Finally, if we write:
\begin{multline}
 \Prb_\theta( X_v \in \cdot \,\vert\, Y_{\DeltaR(u,k)} = y_{\DeltaR(u,k)}, X_{U_{k}} = x_{U_{k}} ) \\
\begin{aligned}
&  \qquad = \int \Prb_\theta( X_v \in \cdot \,\vert\, 
		Y_{\DeltaRminusSetUj{1}} = y_{\DeltaRminusSetUj{1}}, X_{U_{k}} = x_{U_{k}}, X_{U_1}=x_{U_1} ) \\
& \qquad \qquad \qquad
	\times \Prb_\theta( X_{U_1} \in \drv x_{U_1} \,\vert\, Y_{\DeltaR(u,k)} = y_{\DeltaR(u,k)}, X_{U_{k}} = x_{U_{k}} ),
\end{aligned}\label{eq_backward_coupling_decomp_proba_1}
\end{multline}
and we also write (using  \eqref{eq_Markov_prop_backward_kernel_2}):
\begin{multline}
\Prb_\theta( X_v \in \cdot \,\vert\, Y_{\DeltaR^*(u,k)} = y_{\DeltaR^*(u,k)}, X_{U_{k}} = x_{U_{k}} ) \\
\begin{aligned}
& \qquad = \int \Prb_\theta( X_v \in \cdot \,\vert\, 
		Y_{\DeltaRminusSetUj{1}} = y_{\DeltaRminusSetUj{1}}, X_{U_{k}} = x_{U_{k}}, X_{U_1}=x_{U_1} ) \\
& \qquad \qquad \qquad
	\times \Prb_\theta( X_{U_1} \in \drv x_{U_1} \,\vert\, Y_{\DeltaR^*(u,k)} = y_{\DeltaR^*(u,k)}, X_{U_{k}} = x_{U_{k}} ),
\end{aligned}\label{eq_backward_coupling_decomp_proba_2}
\end{multline}
then the two distributions (for $X_v$) on the left hand sides of those displayed equations
can be considered as obtained through running $\dgr(u,v)-1$ iterations of the backward ancestral conditional Markov chain
described above, using two different initial conditions.
Therefore, as the Dobrushin coefficient is sub-multiplicative (remind \Cref{lemma_Dobrushin_sub_multiplicative}),
we get that those two probability distribution differ by at most $2\, \rho^{d(u,v)-1}$ in total variation.
This concludes the proof of the first case.
\medskip

\textbf{Case 2: general case.}
The proof of the second case relies on the observation that conditioned 
	on $X_{\parent^k(u)}$ and $Y_{\DeltaR(u,k)}$,
if we consider the process $X$ backward from $u$ to $u\land v$ (remind that $v\in\DeltaR^*(u,k)$) and then forward from $u\land v$ to $v$,
we get a non-homogeneous Markov chain satisfying uniform mixing rate $\rho$.
Note that as $v\in \Tpast(\parent^k(u),k) \setminus\{u\}$, we have that $u\land v \in \{U_1, \cdots, U_k\}$.
Using the first case, it only remains to check those observations for the forward segment,
which were already proved in the proof of \Cref{lemme_exp_coupling_HMT}.

Hence, if we use the same decomposition 
as in \eqref{eq_backward_coupling_decomp_proba_1} and \eqref{eq_backward_coupling_decomp_proba_1},
which corresponds to run $\dgr(u,v)-1$ iterations of the backward-forward conditional chain
described above ($\dgr(u, u\land v) -1$ backward iterations and $\dgr(u\land v,v)$ forward iterations),
we get as in the first case that those two probability distribution differ 
by at most $2\, \rho^{d(u,v)-1}$ in total variation.
This concludes the proof of the lemma.
\end{proof}

\section{Proof of \texorpdfstring{\eqref{eq_lim_Y_Tm_cond_past}}{} (used in the proof of Proposition~\ref{prop_global_max_l})}
	\label{section_convergence_Y_triangles}

Let $m\in\N^*$ be fixed through this section.

For ease of read, we restate some notation definitions used only in the proof of \Cref{prop_global_max_l}.
For $u,v\in\Tpast$ with $h(u) \equiv h(v)  \mod m+1$, 
we write $T(u,m) < T(v,m)$ if $u < v$.
Moreover for $u,v\in\Tpast$, we write $u < T(v,m)$
if $h(u)< h(v)$ or 
$h(u) \leq h(v)+m$ and for all $w\in T(v,m)$ with $h(w)=h(u)$ (note that such $w$ must exist),
	we have $u < w$.
Informally $u$ is ‘‘above or on the left of $T(v,m)$’’.
For all $u\in\T$, $k\in\N$, define the random subtrees which depend on $\cU$: 
\begin{equation*}
\DeltaR^*(\T(u,m),k) = \bigcup \bigl\{ T(v,m) : v \in \DeltaR^*(u, k(m+1)) \text{ such that }  \height{v}\equiv \height{u} \mod m+1 \bigr\} ,
\end{equation*}
and $\DeltaR(\T(u,m),k) = \DeltaR^*(\T(u,m),k) \cup \T(u,m)$.
When $\height{u}\geq k(m+1)$, then those subtrees do not depend on $\cU$,
and we write $\Delta^*(\T(u,m),k) = \DeltaR^*(\T(u,m),k)$ and $\Delta(\T(u,m),k) = \DeltaR(\T(u,m),k)$
to indicate it.
See \Cref{fig_illustration_Delta_T_u_k} on page~\pageref{fig_illustration_Delta_T_u_k} for an illustration
of the ‘‘past’’ subtree $\DeltaR^*(\T(u,m),k)$ of the block subtree $\T(u,m)$.

The goal of this section is to prove \eqref{eq_lim_Y_Tm_cond_past} for all $\theta\in\Theta$ and $x\in\SpaceX$,
which we restate here for ease of read:
\begin{equation*}
\lim_{k\to\infty} \Esp_\cU \otimes \Etrue \left[
	\log p_\theta(Y_{\T_m} \,\vert\, Y_{\DeltaR^*(\T_m,k)} , X_{\parent^{k(m+1)}(\rooot)} =x )
	\right] 
= \vert \T_m \vert \, \ell(\theta) .
\end{equation*}

\subsection{Decomposition of the log-likelihood into subtree increments}

Following \eqref{eq_def_H_ukx}, 
for all $u\in\T$, $k\in\N$, $x\in\SpaceX$ and $\theta\in\Theta$, using the conditional probabilities formula, define:
\begin{align}\label{eq_def_H_ukx:appendix_triangle}
\Hterm_{\T(u,m),k,x}(\theta) 
& = \frac{p_\theta(Y_{\DeltaR(\T(u,m),k)} \,\vert\, X_{\parent^k(u)} = x)}
	{p_\theta(Y_{\DeltaR^*(\T(u,m),k)} \,\vert\, X_{\parent^k(u)} = x)} \\ 
& = \int p_\theta(Y_{\T(u,m)} \,\vert\, X_u =x_u)  \, 
	\Prb_\theta ( X_u\in \drv x_u \,\vert\, Y_{\DeltaR^*(\T(u,m),k)} , X_{\parent^{(k-1)(m+1)}(u)} =x ) , \nonumber
\end{align}
where:
\begin{equation*}
p_\theta(Y_{\T(u,m)} \,\vert\, X_u =x_u) 
= \int_{\SpaceX^{\vert \T(u,m)\vert}}
	g_\theta(x_u,Y_u) \prod_{w\in \T(u,m)\setminus\{u\}} g_\theta(x_w, Y_w) q_\theta(x_{\parent(w)},x_w) \, \lambda(\drv x_w) .
\end{equation*}
We then define the log-likelihood contribution of the subtree $\T(u,m)$ with past over $k\in\N$ subtree generations
	(that is, $k(m+1)$ (node) generations) as:
\begin{equation*}
\h_{\T(u,m),k,x}(\theta) = \log \Hterm_{\T(u,m),k,x}(\theta) 
\end{equation*}

For all $n\in\N^*$, we decompose the tree $\T_{n(m+1)-1}$ into subtrees of height $m$ (such as $\T_m$),
and we order those subtrees according to $<$.
Hence, using \eqref{eq_def_l_nx_theta}, \eqref{eq_def_l_nx_theta_2} and \eqref{eq_def_H_ukx:appendix_triangle} 
	and a telescopic sum argument,
the log-likelihood of the observed variables $Y_{\T_{n(m+1)-1}}$ can be rewritten as:
\begin{equation}\label{eq_l_sum_incr:appendix_triangle}
\ell_{n(m+1)-1,x}(\theta)  = \sum_{k=0}^{n-1} \sum_{u\in\G_{k(m+1)}}  \h_{\T(u,m),k,x}(\theta). \end{equation}

\subsection{Construction of the log-likelihood increments with infinite past for subtrees}

In this subsection, we construct the log-likelihood increments with infinite past for subtrees.

To construct the limit of the functions $\h_{\T(u,m),k,x}(\theta)$ we first prove the following
lemma which states some uniform bound about the asymptotic behavior of those
functions when $k \to \infty$.

\begin{lemme}[Uniform bounds for $\h_{\T(u,m),k,x}(\theta)$]
	Assume that Assumptions \ref{assump_HMM_1}--\ref{assump_HMM_2}
	and \ref{assump_HMM_3}-\ref{assump_HMM_3:item2} hold.
For all vertices $u\in\T$ and all integers $k,k' \in\N^*$,
the following assertions hold true:
\begin{equation}\label{eq_h_ukx_unif_Cauchy:appendix_triangle}
\sup_{\theta\in\Theta} \sup_{x,x'\in\SpaceX}
	\vert \h_{\T(u,m),k,x}(\theta) - \h_{\T(u,m),k',x'}(\theta) \vert
\leq \frac{\rho^{(k\land k')(m+1) -1} }{(1-\rho)^{\vert \T_m \vert}} ,
\end{equation}
\begin{equation}\label{eq_h_ukx_unif_bounded:appendix_triangle}
\sup_{\theta\in\Theta} \sup_{k\in\N^*} \sup_{x\in\SpaceX}
	\vert \h_{\T(u,m),k,x}(\theta) \vert
\leq \bigl( {\vert \T_m \vert} \log b^+ \bigr)	\lor	\left\vert \sum_{w\in\T(u,m)} \log(\sigma^- b^-(Y_w) ) \right\vert .
\end{equation}
\end{lemme}

\begin{proof}
{[The proof is a straightforward adaptation of the proof of \cite[Lemma 12.3.2]{CappeHMM} 
	using \Cref{lemme_exp_coupling_HMT} for the coupling.]}
Let $k' \geq k \geq 1$,
and write  $v = \parent^{k(m+1)}(u)$, $v' = \parent^{k'(m+1)}(u)$.
Then, write:
\begin{align}
\Hterm_{\T(u,m),k,x}(\theta)
 =  \int_{\SpaceX^2} & \left[ \int_{\SpaceX^{\T(u,m)}} \prod_{w\in \T(u,m)} g_\theta(x_w, Y_w) q_\theta(x_{\parent(w)}, x_w) \lambda(\drv x_w) \right]
 											\label{eq_H_ukxpi_1:appendix_triangle}\\
	&  \times \Prb_\theta( X_{\parent(u)} \in \drv x_{\parent(u)}
			\,\vert\,  Y_{\DeltaR^*(\T(u,m),k)} , X_v =x_v )
	\times \delta_x(\drv x_v) , \nonumber
\end{align}
and using the Markov property at $X_v$, write:
\begin{align}
\Hterm_{\T(u,m),k',x'}(\theta)
 =  \int_{\SpaceX^2} & \left[ \int_{\SpaceX^{\T(u,m)}} \prod_{w\in \T(u,m)} g_\theta(x_w, Y_w) q_\theta(x_{\parent(w)}, x_w) \lambda(\drv x_w) \right] 
 											\label{eq_H_ukxpi_2:appendix_triangle}\\
	&  \times \Prb_\theta( X_{\parent(u)} \in \drv x_{\parent(u)}
			\,\vert\,  Y_{\DeltaR^*(\T(u,m),k)} , X_v =x_v ) \nonumber\\
	& \times \Prb_\theta( X_v \in \drv x_v
		\,\vert\,  Y_{\DeltaR^*(\T(u,m),k')\setminus \DeltaR(\T(u,m),k)} , X_{v'} =x' ) . \nonumber
\end{align}
Applying \Cref{lemme_exp_coupling_HMT}, we get 
(note that the integrands in \eqref{eq_H_ukxpi_1:appendix_triangle} and \eqref{eq_H_ukxpi_2:appendix_triangle} are non-negative):
\begin{align}
\vert \Hterm_{\T(u,m),k,x}(\theta) & - \Hterm_{\T(u,m),k',x'}(\theta) \vert \nonumber\\
& \leq \rho^{k(m+1)-1} \sup_{x_{\parent(u)}\in\SpaceX} 
		\int \prod_{w\in \T(u,m)} g_\theta(x_w, Y_w) q_\theta(x_{\parent(w)}, x_w) \lambda(\drv x_w) \nonumber\\
& \leq \rho^{k(m+1)-1} (\sigma^+)^{\vert \T_m \vert}   \prod_{w\in \T(u,m)} \int  g_\theta(x_w, Y_w) \lambda(\drv x_w) . 
	\label{eq_upper_bound_H_ukxpi:appendix_triangle}
\end{align}
The integral in \eqref{eq_H_ukxpi_1:appendix_triangle} can be lower bounded giving us:
\begin{equation}\label{eq_H_ukxpi_lower_bound:appendix_triangle}
\Hterm_{u,k,x}(\theta) \geq (\sigma^-)^{\vert \T_m \vert}   \prod_{w\in \T(u,m)} \int  g_\theta(x_w, Y_w) \lambda(\drv x_w) ,
\end{equation}
where the right hand side is positive by \Cref{assump_HMM_2}-\ref{assump_HMM_2:item2};
and similarly for \eqref{eq_H_ukxpi_2:appendix_triangle}.
Combining \eqref{eq_upper_bound_H_ukxpi:appendix_triangle} with \eqref{eq_H_ukxpi_lower_bound:appendix_triangle},
and with the inequality $\vert \log x - \log y \vert \leq \vert x-y \vert / (x\land y)$, we get the first assertion of the lemma:
\begin{equation*}
	\vert \h_{\T(u,m),k,x}(\theta) - \h_{\T(u,m),k',x'}(\theta) \vert
	\leq \left( \frac{\sigma^+}{\sigma^-} \right)^{\vert \T_m \vert} \rho^{k(m+1)-1} 
	= \frac{ \rho^{k(m+1)-1} }{ (1-\rho)^{\vert \T_m \vert} } \cdot
\end{equation*}
Combining \eqref{eq_def_H_ukx:appendix_triangle} and \eqref{eq_H_ukxpi_lower_bound:appendix_triangle},
we get:
\begin{equation*}
\prod_{w\in\T(u,m)} \sigma^- b^-(Y_w)  \leq  \Hterm_{\T(u,m),k,x}(\theta)  \leq  (b^+)^{\vert \T_m \vert} ,
\end{equation*}
which yields the second assertion of the lemma 
(remind that $b^-(Y_w) > 0$ for all $w\in\Tpast$ by \Cref{assump_HMM_2}-\ref{assump_HMM_2:item2}).
\end{proof}

We are now ready to construct the limit of the functions $\h_{\T(u,m),k,x}(\theta)$
and state some properties of this limit.
Note that this result is stated for every $u\in\T$, but we will only need it for $u=\rooot$.
Remind that we are in the stationary case, and that the HMT process $(X,Y)$ is defined on $\Tpast$.

\begin{prop}[Properties of the limit function $\h_{\T(u,m),\infty}(\theta)$]
	\label{prop_existence_stationary_likelihood:appendix_triangle}
Assume that Assumptions \ref{assump_HMM_0}--\ref{assump_HMM_3} hold.
For every $u\in\T$ and $\theta\in\Theta$,
there exists $\h_{\T(u,m),\infty}(\theta) \in L^1(\P_\cU \otimes \Ptrue)$
such that for all $x\in\SpaceX$, the sequence $(\h_{\T(u,m),k,x}(\theta))_{k\in\N}$
converges $\P_\cU \otimes \Ptrue$-\as and in $L^1(\P_\cU \otimes \Ptrue)$ to $\h_{\T(u,m),\infty}(\theta)$.

Furthermore, this convergence is uniform over $\theta\in\Theta$ and $x\in\SpaceX$, that is,
we have that $\lim_{k\to\infty} \sup_{\theta\in\Theta} \sup_{x\in\SpaceX} \vert \h_{\T(u,m),k,x}(\theta) - \h_{\T(u,m),\infty}(\theta)\vert = 0$
$\P_\cU \otimes \Ptrue$-\as and in $L^1(\P_\cU \otimes \Ptrue)$.
\end{prop}

The limit function $\h_{\T(u,m),\infty}(\theta)$ can be interpreted as 
$\log p_\theta( Y_{\T(u,m)} \,\vert\, Y_{\DeltaR^*(\T(u,m),\infty)} ) $,
where $\DeltaR^*(\T(u,m),\infty) = \{ v\in\Tpast \,:\, v <_\cU \T(u,m) \}$ is a random subset of vertices.
Note that $\h_{\T(u,m),\infty}(\theta)$ is a function 
of the random set of variables  $(Y_v, v\in \DeltaR(\T(u,m),\infty))$,
where we define $\DeltaR(\T(u,m),\infty) = \DeltaR^*(\T(u,m),\infty) \cup \T(u,m)$,
and thus implicitly depend on $\cU$ trough $\DeltaR(\T(u,m),\infty)$.

\begin{proof}
Fix some $u\in\T$. 
Note that \eqref{eq_h_ukx_unif_Cauchy:appendix_triangle} shows
that the sequence $( \h_{\T(u,m),k,x}(\theta) )_{k\in\N}$ is Cauchy uniformly in $\theta$ and $x$,
and thus has $\P_\cU \otimes \Ptrue$-almost surely a limit when $k \to \infty$
which does not depend on $x$; we denote this limit by $\h_{\T(u,m),\infty}(\theta)$.
Furthermore, we get from \eqref{eq_h_ukx_unif_bounded:appendix_triangle} that 
$( \h_{\T(u,m),k,x}(\theta) )_{k\in\N}$ is uniformly bounded in $L^1(\P_\cU\otimes\Ptrue)$,
and thus $\h_{\T(u,m),\infty}(\theta)$ is in $L^1(\P_\cU\otimes\Ptrue)$ and
the convergence also holds in $L^1(\P_\cU\otimes\Ptrue)$.
Finally, as the bound in \eqref{eq_h_ukx_unif_Cauchy:appendix_triangle} is uniform in $\theta$ and $x$,
we get that the convergence holds uniform over $\theta$ and $x$
both $\P_\cU \otimes \Ptrue$-almost surely and in $L^1(\P_\cU\otimes\Ptrue)$.
\end{proof}

\subsection{Properties of the contrast function}

As the functions $\h_{\T(u,m),\infty}(\theta)$ are in $L^1(\Prb_\cU\otimes\Ptrue)$ 
under the assumptions used in \Cref{prop_existence_stationary_likelihood:appendix_triangle},
we can now define the \emph{contrast function} $\ell^{(m)}$ 
(which is deterministic) for block subtree of height $m$
as: 
\begin{equation*}
	\ell^{(m)}(\theta)
= \Esp_\cU \otimes \Etrue \bigl[ \h_{\T_m,\infty}(\theta) \bigr] ,
\end{equation*}
where remind  $\Esp_\cU \otimes \Etrue$ is the expectation corresponding to $\Prb_\cU\otimes\Ptrue$.
We prove under the $L^2$ regularity assumption \Cref{assump_HMM_3bis}
the convergence of the normalized log-likelihood to this contrast function.

\begin{prop}[Ergodic convergence for the log-likelihood]
	\label{prop_conv_likelihood_to_contrast_func:appendix_triangle}
Assume that Assumptions~\ref{assump_HMM_0}--\ref{assump_HMM_3bis} hold.
Then, for all $x\in\SpaceX$, the normalized log-likelihood $\vert \T_{n(m+1)-1} \vert^{-1} \ell_{n(m+1)-1,x}(\theta)$
converges $\Ptrue$-\as  to the contrast function $\ell^{(m)}(\theta)$ as $n\to \infty$.

\begin{equation}\label{eq_conv_ergodic_ell:appendix_triangle}
\lim_{n\to\infty}
\frac{\vert \T_m\vert}{\vert \T_{n(m+1)-1} \vert} \, \ell_{n(m+1)-1,x}(\theta)
= \ell^{(m)}(\theta)
\qquad \text{$\Ptrue$-\as}
\end{equation}

In particular, we get that $\ell^{(m)}(\theta) = \vert \T_m \vert \ell(\theta)$.
\end{prop}

\begin{proof}
Let $\theta\in\Theta$ be some parameter.
Fix some $k\in\N^*$ and $x\in\SpaceX$.
Remind \eqref{eq_l_sum_incr:appendix_triangle}. Applying \eqref{eq_h_ukx_unif_Cauchy:appendix_triangle} for each vertex $u\in\G_{j(m+1)}$ with $j\in\{k,\cdots, n-1\}$, we get:
\begin{multline}
\frac{\vert \T_m\vert}{\vert \T_{n(m+1)-1} \vert}
	\left\vert
		\ell_{n(m+1)-1,x}(\theta)  - \sum_{j=k}^{n-1} \sum_{u\in\G_{j(m+1)}}  \h_{\T(u,m),k,x}(\theta)
	\right\vert
					\\
 \qquad \leq  \frac{\rho^{k(m+1)-1}}{(1-\rho)^{\vert \T_m\vert}}
	+ \frac{\vert \T_m\vert}{\vert \T_{n(m+1)-1} \vert} 
		\sum_{j=0}^{k-1} \sum_{u\in\G_{j(m+1)}} \vert \h_{\T(u,m),j,x}(\theta) \vert
.	\label{eq_bound_likelihood_and_approx:appendix_triangle}
\end{multline}
Note that by \eqref{eq_h_ukx_unif_bounded:appendix_triangle}, we have that $\vert \h_{\T(u,m),j,x}(\theta) \vert < \infty$ $\Ptrue$-\as
	for all $j\in\N^*$ and $u\in\G_{j(m+1)}$.
For $u=\rooot$, we have $\h_{\T_m,0,x}(\theta) = \log p_\theta(Y_{\T_m} \,\vert\, X_{\rooot}=x)$ which is finite $\Ptrue$-\as by \Cref{assump_HMM_2}-\ref{assump_HMM_2:item3}.

The definition of the \emph{shape} 
from Section~\ref{subsection_ergodic_theorem} 
can straightforwardly  be adapted to the (deterministic) subtrees
$\Delta(\T(u,m),k)$ for vertices $u\in\G_{j(m+1)}$ with $j\geq k$, 
where $u$ is seen as a distinguished vertex of $\Delta(\T(u,m),k)$.
Following \eqref{eq_def_shape_subtree} (on page \pageref{eq_def_shape_subtree}) in the initial vertex-by-vertex decomposition setting,
for a vertex $u\in\G_{j(m+1)}$ with $j\geq k$, let $v_u\in\G_{k(m+1)}$ be the unique vertex $\G_{k(m+1)}$
such that $\Delta(\T(u,m),k)$ and $\Delta(\T(v_u,m),k)$ have the same shape.
Then, we have:
\begin{multline}\label{eq_equal_h_ukx_up_to_shape:appendix_triangle}
\h_{\T(u,m),k,x}\bigl(\theta; Y_{\Delta(\T(u,m),k)}=y_{\Delta(\T(u,m),k)}\bigr)  \\
	=  \h_{\T(v_u,m),k,x}\bigl(\theta; Y_{\Delta(\T(v_u,m),k)}=y_{\Delta(\T(u,m),k))}\bigr) .
\end{multline}
Moreover, using \eqref{eq_h_ukx_unif_bounded:appendix_triangle} together with \Cref{assump_HMM_3bis},
	we get for every $u\in\G_{j(m+1)}$ with $j\geq k$
	that the random variable $\h_{\T(u,m),k,x}(\theta; Y_{\Delta(\T(u,m),k)})$ is in $L^2(\Ptrue)$.
Hence, applying a straightforward modification of \Cref{lemma_ergodic_convergence} for subtree blocks $\T(u,m)$
to the collection of neighborhood-shape-dependent functions $( \h_{\T(v,m),k,x}(\theta; Y_{\Delta(\T(v,m))}=\cdot) )_{v\in\G_{k(m+1)}}$
(remind that indexing functions with $\G_{k(m+1)}$ or with the set of possible shapes 	is equivalent by \eqref{eq_def_set_possible_shapes}),
and using \eqref{eq_equal_h_ukx_up_to_shape:appendix_triangle} and \eqref{eq_equality_expectation_root_and_U_k} (in \Cref{rem_rerooting_delta}),
we get:
\begin{equation}
	\label{eq_conv_ergo_likelihood:appendix_triangle}
\frac{\vert \T_m\vert}{\vert \T_{n(m+1)-1} \vert}
 	\sum_{j=k}^{n-1} \sum_{u\in\G_{j(m+1)}}  \h_{\T(u,m),k,x}(\theta)
\underset{n\to\infty}{\longrightarrow}
	\Esp_\cU \otimes \Etrue \bigl[ \h_{\T_m,k,x}(\theta) \bigr]
\qquad \text{$\Ptrue$-\as }\end{equation}

Using \eqref{eq_h_ukx_unif_Cauchy:appendix_triangle} with \Cref{prop_existence_stationary_likelihood:appendix_triangle},
	we get:
\begin{equation*}
\bigl\vert \Esp_\cU \otimes \Etrue \bigl[ \h_{\T_m,k,x}(\theta) \bigr] 
		- \Esp_\cU \otimes \Etrue \bigl[ \h_{\T_m,\infty}(\theta) \bigr] \bigr\vert
		\leq  \frac{\rho^{k(m+1)-1}}{(1-\rho)^{\vert\T_m\vert}} \cdot
\end{equation*}
Thus, combining this bound with \eqref{eq_bound_likelihood_and_approx:appendix_triangle} and \eqref{eq_conv_ergo_likelihood:appendix_triangle},
we get $\Ptrue$-\as that:
\begin{equation*}
\limsup_{n\to\infty} \left\vert
		\frac{\vert \T_m\vert}{\vert \T_{n(m+1)-1} \vert} \ell_{n(m+1)-1,x}(\theta) 	
		- \Esp_\cU \otimes \Etrue \bigl[ \h_{\T_m,\infty}(\theta) \bigr]
	\right\vert
\leq 2\, \frac{\rho^{k(m+1)-1}}{(1-\rho)^{\vert\T_m\vert}} \cdot
\end{equation*}
As the left hand side does not depend on $k$,
letting $k\to\infty$, we get that \eqref{eq_conv_ergodic_ell:appendix_triangle} in the lemma holds.
Lastly, as the limit must be the same as in \Cref{prop_conv_likelihood_to_contrast_func}, 
we get that $\ell^{(m)}(\theta) = \vert \T_m \vert \ell(\theta)$.
This concludes the proof.
\end{proof}

We are now ready to close this section by proving that \eqref{eq_lim_Y_Tm_cond_past} holds.

\begin{prop}
Assume that Assumptions~\ref{assump_HMM_0}--\ref{assump_HMM_3bis} hold.
Then, \eqref{eq_lim_Y_Tm_cond_past} holds for all $\theta\in\Theta$ and $x\in\SpaceX$.
\end{prop}

\begin{proof}
Applying \Cref{prop_existence_stationary_likelihood:appendix_triangle},
we get that the left hand side of \eqref{eq_lim_Y_Tm_cond_past} is equal to $\ell^{(m)}(\theta) = \Esp_\cU \otimes \Etrue[h_{\T(\rooot,m),\infty}(\theta)]$,
which is equal to $\vert \T_m \vert \ell(\theta)$ by \Cref{prop_conv_likelihood_to_contrast_func:appendix_triangle}.
\end{proof}

\section{Details of the proof of Proposition \ref{prop_conv_Gamma_term}}
	\label{appendix_proof_prop_conv_Gamma_term}

Remind that the proof of \Cref{prop_conv_Lambda_term} can be straightforwardly adapted to  \Cref{prop_conv_Gamma_term}
	except for Lemma~\ref{lemma_Lambda_incr_L2_bound}.
In \Cref{section_proof_prop_conv_Gamma_term}, for brevity, 
	we have only presented the adaptation of Lemma~\ref{lemma_Lambda_incr_L2_bound} 
	to the terms $\Gamma_{u,k,x}(\theta)$.
In this appendix, we present all the details of the adaptation of the rest of the proof of \Cref{prop_conv_Lambda_term}
	to the terms $\Gamma_{u,k,x}(\theta)$.

\medskip

The following lemma gives an exponential bound on the $L^2(\Ptrue)$ norm
uniformly in $x\in\SpaceX$ for the the average of the quantities $\Gamma_{u,\height{u},x}(\thetaTrue)$
over $u\in\T_n^*$.

\begin{lemme}
	\label{lemma_Gamma_conv_unif_x}	
Under the assumptions of \Cref{prop_conv_Gamma_term}, for all $x\in\SpaceX$ and $\theta\in\Theta_0$,
	there exist finite constants $C<\infty$ and $\alpha \in (0,1)$ such that for all $n\in\N^*$
	we have:
\begin{equation}\label{eq_Gamma_exp_decay_unif_x}
\Etrue\left[  \sup_{x\in\SpaceX} \left\vert
	\inv{\vert \T_n \vert} \sum_{u\in\T_n^*} \Gamma_{u,\height{u},x}(\theta) 
	- \Esp_\cU \otimes \Etrue\bigl[  \Gamma_{\rooot,\infty}(\theta) \bigr]
\right\vert^2 \right]^{1/2}
\leq C \alpha^n .
\end{equation}
\end{lemme}

\begin{proof}
Let $x'\in\SpaceX$ and $\theta\in\Theta_0$.
Using Minkowski's inequality and Jensen's inequality, for all $n,k\in\N^*$, we get:
\begin{multline}
\Etrue\left[  \sup_{x\in\SpaceX} \left\vert
		\inv{\vert \T_n \vert} \sum_{u\in\T_n^*} \Gamma_{u,\height{u},x}(\theta) 
		- \Esp_\cU \otimes \Etrue \bigl[  \Gamma_{\rooot,\infty}(\theta) \bigr]
	\right\vert^2 \right]^{1/2} \displaybreak[3] \\
\begin{aligned} 
& \leq  \Etrue\left[  \sup_{x,x'\in\SpaceX} \left\vert  \inv{\vert \T_n \vert}
	 		 \sum_{u\in\T_{k-1}^*} \Gamma_{u,\height{u},x}(\theta) 
		\right\vert^2 \right]^{1/2} \\*
& \qquad +  \Etrue\left[  \sup_{x,x'\in\SpaceX}  \left\vert \inv{\vert \T_n \vert}
	 		 \sum_{u\in\T_n\setminus\T_{k-1}} \Gamma_{u,\height{u},x}(\theta)  - \Gamma_{u,k,x'}(\theta) 
		\right\vert^2 \right]^{1/2} \\*
& \qquad + \Etrue\left[   \left\vert
		\inv{\vert \T_n \vert} \sum_{u\in\T_n\setminus\T_{k-1}} \Gamma_{u,k,x'}(\theta) 
		- \Esp_\cU \otimes \Etrue \bigl[  \Gamma_{\rooot,k,x'}(\theta) \bigr]
	\right\vert^2 \right]^{1/2}  \\* 
& \qquad + \Esp_\cU \otimes \Etrue \bigl[  
		\vert \Gamma_{\rooot,k,x'}(\theta) - \Gamma_{\rooot,\infty}(\theta) \vert^2
	\bigr]^{1/2} 
.	
\end{aligned}\label{eq_sum_Gamma_conv_L2_approx}
\end{multline}
Using \Cref{lemma_Gamma_incr_L2_bound} 
together with \eqref{eq_unif_bound_dot_hukx} on page~\pageref{eq_unif_bound_dot_hukx}
(which, remind, are both immediate consequences of \Cref{lemma_score_incr_L2_bound}), 
there exists a finite constant $C<\infty$  and $\beta\in (0,1)$ such that
the first term in the right hand side of \eqref{eq_sum_Gamma_conv_L2_approx}
is upper bounded by $C 2^{-(n-k)}$
(note that $\frac{\vert T_{k-1} \vert}{\vert \T_n \vert} \leq 2^{-(n-k)}$),
and the second and fourth terms in the right hand side of \eqref{eq_sum_Gamma_conv_L2_approx}
are both upper bounded by $C \beta^{k/2}$.

We now give an upper bound for the second term in the right hand side of \eqref{eq_sum_Gamma_conv_L2_approx}.
For a vertex $u$ in $\T\setminus \T_{k-1}$, 
let $v_u\in\G_k$ be the unique vertex that satisfies the shape equality constraint \eqref{eq_def_shape_subtree} (on page \pageref{eq_def_shape_subtree}),
then we have:
\begin{equation}\label{eq_equal_Gamma_ukx_up_to_shape}
\Gamma_{u,k,x'}(\theta; Y_{\Delta(u,k)}=y_{\Delta(u,k)}) 
	=  \Gamma_{v_u,k,x'}(\theta; Y_{\Delta(v_u)}=y_{\Delta(u,k)}) .
\end{equation}
Moreover, using the definition of $\Gamma_{u,k,x}(\theta)$ in \eqref{eq_def_Gamma_ukx} 
	together with the assumption on $\phi_\theta$ in \Cref{prop_conv_Gamma_term},
	we get that the random variable $\Gamma_{u,k,x'}(\theta; Y_{\Delta(u,k)}=y_{\Delta(u,k)})$ 
	is in $L^2(\Ptrue)$ for every $u\in\T\setminus \T_{k-1}$.
Thus, we can apply \Cref{lemma_ergodic_convergence} (see in particular \eqref{eq_LLN_upper_bound_Esp_M_Gn_main_body})
to the collection of neighborhood-shape-dependent functions $( \Gamma_{v_u,k,x'}(\theta; Y_{\Delta(v)}=\cdot) )_{v\in\G_k}$
(remind that indexing functions with $\G_k$ or with $\ShapeSetValues_k$ is equivalent by \eqref{eq_def_set_possible_shapes}).
Using \eqref{eq_LLN_upper_bound_Esp_M_Gn_main_body} in \Cref{lemma_ergodic_convergence}
together with \eqref{eq_equal_h_ukx_up_to_shape} and \eqref{eq_equality_expectation_root_and_U_k} in \Cref{rem_rerooting_delta},
	we get that
	there exist $\gamma\in (0,1)$ and a finite constant $C'<\infty$
	(note that they both do not depend on $k$ and $n$) such that for all $n,k\in\N^*$ with $n\geq k$,
the second term in the right hand side of \eqref{eq_sum_Gamma_conv_L2_approx} 
	is upper bounded by $C' \gamma^{n-k}$.

Hence, taking $k = \ceil{n / 2}$, we get that the left hand side of \eqref{eq_sum_Gamma_conv_L2_approx} 
is upper bounded by $2 C \beta^{n/4}  + C' \alpha^{n/2} + C 2^{-n/2+1}$,
and thus decays at exponential rate as desired.
This concludes the proof.
\end{proof}

\Cref{lemma_Gamma_conv_unif_x} implies as a corollary
the convergence $\Ptrue$-\as and in $L^2(\Ptrue)$ uniformly in $x\in\SpaceX$
for the the sum of the quantities $\Gamma_{u,\height{u},x}(\thetaTrue)$ over $u\in\T_n^*$.

\begin{corol}
	\label{corol_Gamma_conv_unif_x}	
Under the assumptions of \Cref{prop_conv_Gamma_term}, for all $x\in\SpaceX$ and $\theta\in\Theta_0$, 
	we have:
\begin{equation*}\lim_{n\to\infty}
  \sup_{x\in\SpaceX} \left\vert
	\inv{\vert \T_n \vert} \sum_{u\in\T_n^*} \Gamma_{u,\height{u},x}(\theta) 
	- \Esp_\cU \otimes \Etrue \bigl[  \Gamma_{\rooot,\infty}(\theta) \bigr]
\right\vert
= 0
\quad \text{$\Ptrue$-\as and in $L^2(\Ptrue)$.}
\end{equation*}
\end{corol}

\begin{proof}
The convergence in $L^2(\Ptrue)$ follows immediately from \Cref{lemma_Gamma_conv_unif_x}.
Moreover, using again \Cref{lemma_Gamma_conv_unif_x}, we have:
\begin{equation*}
\sum_{n\in\N^*} \Etrue\left[ \sup_{x\in\SpaceX} \left\vert
		\inv{\vert\T_n\vert} \sum_{u\in\T_n^*} \Gamma_{u,\infty}(\theta)
		- \Esp_\cU \otimes \Etrue [  \Gamma_{\rooot,\infty}(\theta) ]
	\right\vert^2 \right] < \infty .
\end{equation*}
Hence, Borel-Cantelli lemma and Markov's inequality imply that the convergence 
in the lemma also holds $\Ptrue$-\as
\end{proof}

The following lemma gives some continuity properties of the function $\theta \mapsto \Gamma_{\rooot,k,x}(\theta)$.

\begin{lemme}\label{lemma_Gamma_root_continuous}
Under the assumptions of \Cref{prop_conv_Gamma_term}, for all $x\in\SpaceX$ and $k\in\N$,
the random function $\theta \mapsto \Gamma_{\rooot,k,x}(\theta)$ is $\Prb_\cU \otimes \Ptrue$-\as continuous on $\Theta_0$.
Moreover, for all $\theta\in\Theta_0$, we have:
\begin{equation*}
\lim_{\delta\to 0} \Esp_\cU \otimes \Etrue\left[
		\sup_{\theta' \in\Theta_0 : \Vert \theta' - \theta \Vert \leq \delta} 
			\vert \Gamma_{\rooot,k,x}(\theta') - \Gamma_{\rooot,k,x}(\theta) \vert^2
	\right] = 0.
\end{equation*}
\end{lemme}

\begin{proof}
We mimic the proof of \cite[Lemma 14]{doucAsymptoticPropertiesMaximum2004}.

For all $v\in\Tpast$, define the random variable 
$\norm{\phi^v}_\infty = \sup_{\theta'\in\Theta_0} \sup_{x,x'\in\SpaceX} \vert \phi_{\theta'}(x',x,Y_v) \vert$.
Remind that under the assumptions of \Cref{prop_conv_Gamma_term},
the HMT process $(X,Y)$ is stationary and the random variable $\norm{\phi^{\rooot}}_\infty$ is in $L^4(\Ptrue)$.
Thus, for all $v\in\Tpast$, the random variable $\norm{\phi^v}_\infty$ is in $L^4(\Ptrue)$.
Remind from \eqref{eq_inclusion_Delta_Tpast} on page~\pageref{eq_inclusion_Delta_Tpast}
that $\DeltaR(\rooot,k)$ is a random subtree of the deterministic subtree $\Tpast(\parent^k(u),k)$.
Then, note that we have: 
\begin{equation*}
\sup_{\theta\in\Theta_0} \vert\Gamma_{\rooot,k,x}(\theta)\vert \leq 4 \left( \sum_{v\in\Tpast(\parent^k(\rooot),k)} \norm{\phi^v}_\infty \right)^2 ,
\end{equation*}
where the upper bound is a random variable in $L^2(\Ptrue)$ (and thus in $L^2(\Prb_\cU\otimes\Ptrue)$)
which depends on $Y_{\Tpast(\parent^k(u),k)}$ but not on $\cU$.
Hence, to prove the lemma, it suffices to prove that for all $v_1, v_2\in \Tpast(\parent^k(u),k)\setminus\{\parent^k(\rooot)\}$ 
and $\epsilon\in\{0,1\}$,
we have $\Prb_\cU \otimes \Ptrue$-\as:
\begin{align*}
\lim_{\delta\to 0} \sup_{\theta' \in\Theta_0 : \Vert \theta' - \theta \Vert \leq \delta} \
	\Bigl\vert \Bigr.& \Esp_{\theta'}[ \phi_{\theta'}^{(2,\epsilon)}(X_{\parent({v_1})},X_{v_1},Y_{v_1},X_{\parent({v_2})},X_{v_2},Y_{v_2}) 
			\,\vert\, Y_{\DeltaR(\rooot,k)}, X_{\parent^k(\rooot)}=x ] \\
	& \Bigl. - \Esp_{\theta}[ \phi_{\theta}^{(2,\epsilon)}(X_{\parent({v_1})},X_{v_1},Y_{v_1},X_{\parent({v_2})},X_{v_2},Y_{v_2})
			\,\vert\, Y_{\DeltaR(\rooot,k)}, X_{\parent^k(\rooot)}=x ]
	\Bigr\vert
= 0,\end{align*}
where:
\begin{equation*}
\phi_{\theta'}^{(2,\epsilon)}(X_{\parent({v_1})},X_{v_1},Y_{v_1},X_{\parent({v_2})},X_{v_2},Y_{v_2}) 
:= \phi_{\theta'}(X_{\parent({v_1})},X_{v_1},Y_{v_1})
		\phi_{\theta'}(X_{\parent({v_2})},X_{v_2},Y_{v_2})^\epsilon .
\end{equation*}

Denote $x_{\parent^k(\rooot)}=x$, and write:
\begin{multline}
\Esp_{\theta}[ \phi_{\theta'}^{(2,\epsilon)}(X_{\parent({v_1})},X_{v_1},Y_{v_1},X_{\parent({v_2})},X_{v_2},Y_{v_2}) 
	 \,\vert\, Y_{\DeltaR(\rooot,k)}, X_{\parent^k(\rooot)}=x ] \\
 = \frac{
		\int_{\SpaceX^{\vert\DeltaR(\rooot,k)\vert-1}} 
			\phi_{\theta'}^{(2,\epsilon)}(x_{\parent({v_1})},x_{v_1},Y_{v_1},x_{\parent({v_2})},x_{v_2},Y_{v_2}) \,
		\Psi( \drv x_{\DeltaR(\rooot,k)\setminus\{\parent^k(\rooot)\}} )
	}
	{
		\int_{\SpaceX^{\DeltaR(\rooot,k)\setminus\{\parent^k(\rooot)\}}}
			1 \,  \Psi( \drv x_{\DeltaR(\rooot,k)\setminus\{\parent^k(\rooot)\}} )
	}\cdot 	\label{eq_control_phi_frac}
\end{multline}
where:
\begin{equation*}
\Psi( \drv x_{\DeltaR(\rooot,k)\setminus\{\parent^k(\rooot)\}} ) :=
\prod_{w\in\DeltaR(\rooot,k)\setminus\{\parent^k(\rooot)\}} q_\theta(x_{\parent(w)},x_w) g_\theta(x_w,Y_w)
			\lambda(\drv x_w) .
\end{equation*}
Using Assumptions~\ref{assump_HMM_1}-\ref{assump_HMM_3} 
(which are part of the assumptions in \Cref{prop_conv_Gamma_term}),
we know that the integrand in the numerator of the right hand side of \eqref{eq_control_phi_frac}
is continuous \wrt $\theta$ and is upper bounded by the random variable 
$\norm{\phi^{v_1}}_\infty (\norm{\phi^{v_2}}_\infty)^\epsilon 
(\sigma^+ b^+)^{\vert \Tpast(\parent^k(u),k)\vert-1}$ (remind that $\sigma^+\geq 1$ and $b^+\geq 1$).
And similarly, the denominator is continuous \wrt $\theta$, 
and, using \Cref{assump_HMM_2}-\ref{assump_HMM_2:item2}, is lower bounded by the random variable:
\begin{equation*}
\prod_{w\in\DeltaR(\rooot,k)\setminus\{\parent^k(\rooot)\}} \sigma^- \inf_{\theta'\in\Theta}\int g_{\theta'}(x_w,Y_w) \lambda(\drv x_w) > 0.
\end{equation*}
Hence, using dominated convergence, we conclude that $\Prb_\cU\otimes\Ptrue$-\as the left hand side of \eqref{eq_control_phi_frac}
is continuous \wrt $\theta$.
This concludes the proof.
\end{proof}

As a corollary of \Cref{lemma_Gamma_root_continuous},
we get that the function $\theta \mapsto \Gamma_{\rooot,\infty}(\theta)$ is continuous in $L^2(\Ptrue)$.

\begin{corol}\label{corol_Gamma_root_infinite_continuous}
Under the assumptions of \Cref{prop_conv_Gamma_term}, for all $\theta\in\Theta_0$, we have:
\begin{equation*}
\lim_{\delta\to 0} \Esp_\cU \otimes \Etrue \left[ 
		\sup_{\theta' \in\Theta_0 : \Vert \theta' - \theta \Vert \leq \delta} 
			\vert \Gamma_{\rooot,\infty}(\theta') - \Gamma_{\rooot,\infty}(\theta) \vert^2
	\right]  = 0.
\end{equation*}
In particular, the function $\theta \mapsto \Esp_\cU  \otimes \Etrue [ \Gamma_{\rooot,\infty}(\theta) ]$ is continuous on $\Theta_0$.
\end{corol}

\begin{proof}
Using Minkowski's inequality and \Cref{lemma_Gamma_incr_L2_bound}, 
there exist a finite constant $C<\infty$ and $\beta\in (0,1)$ such that for all $x\in\SpaceX$ and $k\in\N^*$, we have:
\begin{multline}
 \Esp_\cU \otimes \Etrue \left[ 
		\sup_{\theta' \in\Theta_0 : \Vert \theta' - \theta \Vert \leq \delta} 
			\vert \Gamma_{\rooot,\infty}(\theta') - \Gamma_{\rooot,\infty}(\theta) \vert^2
	 \right]^{1/2} \\
 \leq 	2 C \beta^{k/2}
	+ \Esp_\cU \otimes \Etrue \left[ 
		\sup_{\theta' \in\Theta_0 : \Vert \theta' - \theta \Vert \leq \delta} 
			\vert \Gamma_{\rooot,k,x}(\theta') - \Gamma_{\rooot,k,x}(\theta) \vert^2
	\right]^{1/2} . \label{eq_bound_cont_Gamma_root}
\end{multline}
Using \Cref{lemma_Gamma_root_continuous}, we get:
\begin{equation*}
\limsup_{\delta\to 0}
\Esp_\cU \otimes \Etrue \left[ 
		\sup_{\theta' \in\Theta_0 : \Vert \theta' - \theta \Vert \leq \delta} 
			\vert \Gamma_{\rooot,\infty}(\theta') - \Gamma_{\rooot,\infty}(\theta) \vert^2
	 \right]^{1/2}
\leq 2 C \beta^{k/2},
\end{equation*}
and taking $k\to\infty$, the upper bound vanishes.
This concludes the proof.
\end{proof}

We now prove a locally uniform law of large numbers for the quantities $\Gamma_{u,k,x}(\theta)$.

\begin{lemme}\label{lemma_conv_unif_sum_Gamma_k}
Under the assumptions of \Cref{prop_conv_Gamma_term}, for all $x\in\SpaceX$, we have:
\begin{align*}
\lim_{\delta\to 0} \lim_{n\to\infty} \sup_{\theta'\in\Theta_0 \,:\, \Vert\theta'-\theta\Vert \leq \delta}
	\left\vert
		\inv{\vert\T_n\vert} \sum_{u\in\T_n^*} \Gamma_{u,\height{u},x}(\theta')
		- \Esp_\cU \otimes \Etrue [  \Gamma_{\rooot,\infty}(\theta) ] 
	\right\vert
= 0,
\quad \Ptrue\text{-\as}
\end{align*}
\end{lemme}

\begin{proof}
First, write:
\allowdisplaybreaks[4]
\begin{multline}
 \sup_{\theta'\in\Theta_0 \,:\, \Vert\theta'-\theta\Vert \leq \delta} 
	\left\vert
		\inv{\vert\T_n\vert} \sum_{u\in\T_n^*} \Gamma_{u,\height{u},x}(\theta')
		- \Esp_\cU \otimes \Etrue [  \Gamma_{\rooot,\infty}(\theta) ]
	\right\vert \\
\begin{aligned}
& \leq 
		\inv{\vert\T_n\vert} \sum_{u\in\T_n^*} 
			\sup_{\theta'\in\Theta_0 \,:\, \Vert\theta'-\theta\Vert \leq \delta}
		\bigl\vert \Gamma_{u,\height{u},x}(\theta') - \Gamma_{u,\height{u},x}(\theta) \bigr\vert 
	 \\*
& \qquad\qquad +
	\left\vert
		\inv{\vert\T_n\vert} \sum_{u\in\T_n^*} \Gamma_{u,\height{u},x}(\theta)
		- \Esp_\cU \otimes \Etrue [  \Gamma_{\rooot,\infty}(\theta) ] 
	\right\vert . \label{eq_upper_bound_unif_sup_Gamma_k}
\end{aligned}
\end{multline}
\allowdisplaybreaks[1]

Then, we use the exact same argument as in the proofs of \Cref{lemma_Gamma_conv_unif_x} and \Cref{corol_Gamma_conv_unif_x}
where for all $u\in\T^*$, the random variable $\Gamma_{u,k,x}(\theta)$ is replaced by the random variable:
\begin{equation*}
\sup_{\theta'\in\Theta_0 \,:\, \Vert\theta'-\theta\Vert \leq \delta}
		\bigl\vert \Gamma_{u,\height{u},x}(\theta') - \Gamma_{u,\height{u},x}(\theta) \bigr\vert ,
\end{equation*}
	which are in $L^2(\Ptrue)$ using the assumptions of \Cref{prop_conv_Gamma_term}.
This gives us
that the first term in the upper bound of \eqref{eq_upper_bound_unif_sup_Gamma_k}
converges $\Ptrue$-\as as $n\to\infty$ to:
\begin{equation*}
\Esp_\cU \otimes \Etrue \left[
		\sup_{\theta' : \Vert \theta' - \theta \Vert \leq \delta} 
			\vert \Gamma_{\rooot,\infty}(\theta') - \Gamma_{\rooot,\infty}(\theta) \vert
	 \right],
\end{equation*}
which, by \Cref{corol_Gamma_root_infinite_continuous}, vanishes when $\delta\to 0$.
\Cref{corol_Gamma_conv_unif_x} implies that 
the second term in the upper bound of \eqref{eq_upper_bound_unif_sup_Gamma_k}
vanishes $\Ptrue$-\as when $n\to\infty$.
This concludes the proof. \end{proof}

Combining the previous lemmas in this appendix and \Cref{lemma_Gamma_incr_L2_bound}, 
we are now ready to prove \Cref{prop_conv_Gamma_term}.

\begin{proof}[Proof of \Cref{prop_conv_Gamma_term}]
By \Cref{lemma_Gamma_incr_L2_bound}, for all $u\in\T$, we have that $(\Gamma_{u,k,x}(\theta))_{k\in\N^*}$
is a Cauchy sequence uniformly \wrt $\theta\in\Theta_0$ in $L^2(\Prb_\cU\otimes\Ptrue)$
that converges to some limit $\Gamma_{u,\infty}(\theta)$
	(that does not depend on $x$).
By \Cref{corol_Gamma_conv_unif_x}, we have that $\Ptrue$-\as
the convergence for the the average of the quantities $\Gamma_{u,\height{u},x}(\thetaTrue)$ over $u\in\T_n^*$ holds uniformly in $x\in\SpaceX$,
that is, \eqref{eq_conv_as_sum_Gamma_ukx} in \Cref{prop_conv_Gamma_term} holds.
By \Cref{corol_Gamma_root_infinite_continuous},
we have that the function $\theta \mapsto \Esp_\cU  \otimes \Etrue [ \Gamma_{\rooot,\infty}(\theta) ]$ is continuous on $\Theta_0$.
Finally, the last part of the proposition is given by \Cref{lemma_conv_unif_sum_Gamma_k}.
\end{proof}

\end{document}